\documentclass[11pt]{amsart}
\usepackage{amsmath, amssymb, amsthm, esint,   mathtools, mathrsfs, bbm}
\usepackage[backend=biber]{biblatex}
\usepackage{geometry, mathabx}
\newcommand{\bd}{\boldsymbol}
\usepackage{cancel}
\usepackage{xcolor}
\usepackage[ colorlinks=true, citecolor=blue]{hyperref}
\usepackage[normalem]{ulem}

\def\R{{\mathbb R}}

\def\bd{\boldsymbol}
\newcommand{\cred}{\color{red}}

\newcommand{\bP}{{\mathbb P}}
\newcommand{\sF}{\mathcal F}

\newcommand{\bE}{\mathbb{E}}

\theoremstyle{definition}

\newtheorem{notation}{Notation}[section]
\newtheorem{theorem}{Theorem}[section]
\newtheorem{corollary}{Corollary}[section]
\newtheorem{definition}{Definition}[section]
\newtheorem{proposition}{Proposition}[section]
\newtheorem{remark}{Remark}[section]
\newtheorem{lemma}{Lemma}[section]

\date{}
\author[J. Kuan, K. Tawri and K. Trivisa]{Jeffrey Kuan$^1$, Krutika Tawri$^2$ and Konstantina Trivisa$^1$}

\address{\newline	$^1$ Department of Mathematics, University of Maryland, MD, USA.\\
\newline
$^2$  Department of Applied Mathematics, University of Washington, WA, USA.}
	
\email{jeffreyk@umd.edu (Jeffrey Kuan), ktawri@uw.edu (Krutika Tawri), trivisa@umd.edu (Konstantina Trivisa)}

\begin{document}
\title[Statistically stationary solutions to damped compressible Euler equations]{Statistically stationary solutions to the stochastic isentropic compressible Euler equations with linear damping}

\begin{abstract}
We study the long time behavior of isentropic compressible Euler equations with linear damping driven by a white-in-time noise, on a one-dimensional torus. We prove the existence of a statistically stationary solution in the class of weak martingale entropy solutions for any adiabatic constant $\gamma>1$, which satisfies an associated entropy inequality. To establish this result, we use a multi-level approximation scheme consisting of a truncation parameter $R$ and an artificial viscosity parameter $\epsilon$. {The truncated system preserves the structure of the regularized system with the artificial viscosity, thereby providing key properties such as an invariant region and non-existence of vacuum at the approximate level.}
{These properties allow us to} construct an invariant measure for the approximate system in both $R$ and $\epsilon$ associated to a Feller semigroup for the {well-posed} dynamics of the approximate system { for any $\gamma>1$}. This gives us a statistically stationary solution for the approximate problem, which we then successively pass to the limit as $R \to \infty$ and as $\epsilon \to 0$ to obtain a statistically stationary solution to the original stochastic system. Our analysis is novel, using new techniques for establishing uniform bounds on entropies of all orders, which allow us to pass to the limit in the parameters. We believe that this result is a valuable step towards further understanding the long-time statistical behavior of the stochastic Euler equations in one spatial dimension.
\end{abstract}
\maketitle

\section{Introduction}

In this manuscript, we study the stochastic compressible Euler equations with linear damping, given by
\begin{equation}\label{maineqn}
\begin{cases}
\partial_{t}\rho + \text{div}(\rho u) = 0 \\
\partial_{t}(\rho u) + \partial_{x}(\rho u^{2} + p(\rho)) = \Phi(\rho, \rho u) dW - \alpha \rho u
\end{cases}
\quad \text{ for } (t, x) \in \R^{+} \times \mathbb{T},
\end{equation}
where $\mathbb{T}$ is the one-dimensional torus (so that we impose periodic boundary conditions) and where $\rho$ and $u$ represent the density and the velocity of the compressible isentropic fluid, and hence $\rho u$ represents the momentum of the fluid. The damping is linear damping in the momentum equation, where the intensity of the damping is determined by the positive constant $\alpha > 0$. The constitutive relationship for the pressure in terms of the density is given by the power law:
\begin{equation*}
p(\rho) = \kappa \rho^{\gamma} \text{ for } \gamma > 1,
\end{equation*}
where we define the following constants in terms of $\gamma$:
\begin{equation*}
\kappa = \frac{\theta^{2}}{\gamma} \quad  \text{ and } \quad \theta = \frac{\gamma - 1}{2} > 0. 
\end{equation*}
For this system, we introduce the state variables for the fluid density $\rho$ and the fluid momentum $q$:
\begin{equation}\label{state}
\bd{U} := \begin{pmatrix} \rho \\ q \\ \end{pmatrix}, \quad \text{ for } q := \rho u,
\end{equation}
as we will often work with the momentum $q$ instead of the fluid velocity $u$ at various points throughout the analysis. {{We remark that a similar system is considered in \cite{BerthelinVovelle, LPS96} in the undamped case of $\alpha = 0$.}}

\medskip

\subsection{The stochastic noise} The stochastic noise is represented by the multiplicative noise term $\Phi(\bd{U}) dW$ in the momentum equation. Here, $\{W(t)\}_{t \ge 0}$ represents a cylindrical Wiener process with respect to a filtration $\{\mathcal{F}_{t}\}_{t \ge 0}$, taking values in a separable Hilbert space $\mathcal{U}$. Letting $\{\bd{e}_{k}\}_{k = 1}^{\infty}$ denote an orthonormal basis of $\mathcal{U}$, we can represent the cylindrical Wiener process taking values in $\mathcal{U}$ formally as
\begin{equation*}
W(t) := \sum_{k = 1}^{\infty} W_{k}(t) \bd{e}_{k},
\end{equation*}
where $\{W_{k}(t)\}_{t \ge 0}$ for positive integers $k$ is a collection of independent one-dimensional Brownian motions indexed by the positive integers $k$. 

Next, we define the noise coefficient $\Phi(\bd{U}): \mathcal{U} \to L^{2}(\mathbb{T})$ by defining its action on the orthonormal basis elements $\{\bd{e}_{k}\}_{k = 1}^{\infty}$ of $\mathcal{U}$. We define
\begin{equation}\label{gkdef}
\bd{\Phi}(\rho, q)\bd{e}_{k} := G_{k}(x, \rho, q) = \rho g_{k}(x, \rho, q),
\end{equation}
where for each $(\rho, q) \in [0, \infty) \times \mathbb{R}$, $g_{k}(x, \rho, q)$ is a continuous periodic function on $\mathbb{T}$. We make the Lipschitz assumption on the stochastic noise that
\begin{align}\label{noiseassumption}
|\nabla_{\rho, q}g_{k}(x, \rho, q)| + |g_{k}(x, \rho, q)| \le \alpha_{k} \qquad &\text{for all } (x, \rho, q) \in \mathbb{T} \times [0, \infty) \times \mathbb{R}, \\
&\text{for some positive constants $\alpha_{k}$ such that $\displaystyle \sum_{k = 1}^{\infty} \alpha_{k}^{2} \le A_{0}$.} \nonumber
\end{align}
We also introduce the notation:
\begin{equation*}
\bd{G}(x, \rho, q) := \left(\sum_{k = 1}^{\infty} |G_{k}(x, \rho, q)|^{2}\right)^{1/2},
\end{equation*}
and we note that the following bound is a direct consequence of the assumption \eqref{noiseassumption} on the noise:
\begin{equation}\label{Gbound}
|G_{k}(x, \rho, q)| \le \alpha_{k}\rho \quad \Longrightarrow \quad |\bd{G}(x, \rho, q)| \le A_{0}^{1/2}\rho.
\end{equation}

\medskip

\subsection{Summary of the system} We can express the given system in quasilinear form via a single equation, by recalling the vector $\bd{U}$ of state variables for the fluid density and the momentum from \eqref{state}. We can then rewrite the system of equations in \eqref{maineqn} as
\begin{equation}\label{eulerU}
d\bd{U} + \partial_{x}\bd{F}(\bd{U}) dt = \bd{A}(\bd{U}) dt + \bd{\Psi}(\bd{U}) dW,
\end{equation}
for an associated flux, forcing, and noise function:
\begin{equation*}
\bd{F}(\bd{U}) := \begin{pmatrix} q \\ \frac{q^{2}}{\rho} + p(\rho) \end{pmatrix}, \qquad \bd{A}(\bd{U}) = \begin{pmatrix} 0 \\ -\alpha q \\ \end{pmatrix}, \qquad \bd{\Psi}(\bd{U}) := \begin{pmatrix} 0 \\ \Phi(\bd{U}) \\ \end{pmatrix}.
\end{equation*}

\subsection{The entropy inequality and weak martingale entropy solutions} If we have sufficiently smooth functions $\eta: (\rho, {u}) \to \R$ and $H: (\rho, {u}) \to \R$ such that
\begin{equation}\label{entropyflux}
\nabla\eta(\rho, q) \nabla \bd{F}(\rho, q) = \nabla H(\rho, q)
\end{equation}
where the gradient is with respect to the variables $\rho$ and $q$, we formally have by applying It\^{o}'s formula with the functional $\bd{U} \to \eta(\bd{U})$ that
\begin{equation}\label{entropyito}
d\eta(\bd{U}) + \partial_{x} H(\bd{U}) dt + \alpha q\partial_{q} \eta(\bd{U}) dt = \partial_{q}\eta(\bd{U}) \bd{\Phi}(\bd{U}) dW + \frac{1}{2} \partial_{q}^{2} \eta(\bd{U}) G^{2}(\bd{U}) dt,
\end{equation}
where we refer to $({\eta}, H)$ satisfying the relation \eqref{entropyflux} as an \textbf{entropy-flux pair}. For weak entropy solutions, we expect the equation \eqref{entropyito} to hold in a weak sense (distributionally) with an inequality rather than equality, to account for potential increase in entropy due to the appearance of shocks. This is what is referred to as the entropy inequality. 

We will rigorously state the entropy inequality in the forthcoming definition in Definition \ref{entropydef}, which we require to hold for \textit{all entropy-flux pairs} $(\eta, H)$ satisfying \eqref{entropyflux}. However, we first discuss before stating the entropy inequality, a characterization of the specific entropy flux pairs $(\eta, H)$ that we will use for the entropy inequality, satisfying the relation \eqref{entropyflux}. Using a kinetic formulation of the system of conservation laws (see \cite{LPS96, LPT94}), it is well-known that one can use the following explicit formula to generate entropy-flux pairs $(\eta, H)$: 
\begin{equation}\label{etamformula}
\begin{split}
    &\eta(\bd{U}) = \rho c_{\lambda} \int_{-1}^{1} g(u + z\rho^{\theta}) (1 - z^{2})^{\lambda} dz, \\
&H(\bd{U}) = \rho c_{\lambda} \int_{-1}^{1} g(u + z\rho^{\theta}) (u + z\theta \rho^{\theta}) (1 - z^{2})^{\lambda} dz,
\end{split}
\end{equation}
where the constants $\theta, \lambda$ and $c_{\lambda}$ are defined, for $\gamma>1$, by
\begin{equation*}
 \theta = \frac{\gamma - 1}{2},\qquad \lambda = \frac{3 - \gamma}{2(\gamma - 1)}, \qquad c_{\lambda} = \left(\int_{-1}^{1} (1 - z^{2})^{\lambda} dz\right)^{-1},
\end{equation*}
and $g: \R \to \R$ is a function satisfying $g \in C^{2}(\R)$ is a convex function ($g''(z) \ge 0$ for all $z$). Namely, any such $C^{2}$ convex function $g: \R \to \R$ will generate an associated entropy-flux pair $(\eta, H)$ where the associated entropy $\eta(\bd{U})$ is a convex function of $U$, in the physically relevant region away from vacuum (namely $\rho > 0$ and $u \in \R$). We make the following technical assumption on the convex functions $g \in C^{2}(\R)$ that we use to generate the entropy-flux pairs, which is commonplace in the existing literature on compressible isentropic Euler equations \cite{BerthelinVovelle, LPS96}.

\begin{definition}\label{tildeG}
  \if 1 = 0 {\cred  We define the class $\mathcal{G}$ of admissible convex functions $g \in C^{2}(\R)$ to be all convex functions $g \in C^{2}(\R)$ that are \textbf{subquadratic} in the sense that for some constant $C$:
    \begin{equation}\label{subquad}
    |g(z)| \le C(1 + |z|^{2}), \qquad |g'(z)| \le C(1 + |z|), \qquad \text{ for all } z \in \R.
    \end{equation}} \fi
    We define the class $\tilde{\mathcal{G}}$ of admissible convex functions $g \in C^{2}(\R)$ to be all convex functions $g \in C^{2}(\R)$ that are \textbf{subpolynomial} in the sense that for some constant $C$ and some positive integer $m$:
    \begin{equation}\label{subpoly}
    |g(z)| \le C(1 + |z|^{2m}), \quad |g'(z)| \le C(1 + |z|^{2m - 1}), \quad |g''(z)| \le C(1 + |z|^{2m - 2}), \quad \text{ for all } z \in \R.
    \end{equation}
\end{definition}
%{\cred add comment about nontriviality.}
Then, the classical definition of a weak martingale entropy solution to the stochastic compressible isentropic (damped) Euler equations is as follows, see \cite{BerthelinVovelle}:

\begin{definition}\label{entropydef}
We say that $(\rho, q)$ on a probability space $(\Omega, \mathcal{F}, \mathbb{P})$ with a filtration $\{\mathcal{F}_{t}\}_{t \ge 0}$ and a $U$-valued cylindrical Wiener process $\{W_{t}\}_{t \ge 0}$ is a \textbf{weak martingale entropy solution} to \eqref{eulerU} if the following conditions are satisfied:
\medskip
\newline \noindent 1. $\{W_{t}\}_{t \ge 0}$ is a $\mathcal{U}$-valued cylindrical Wiener process on $(\Omega, \mathcal{F}, \mathbb{P})$ with respect to the filtration $\{\mathcal{F}_{t}\}_{t \ge 0}$.
\medskip
\newline \noindent 2. $\bd{U}=(\rho, q) \in C([0, \infty); H^{-2}(\mathbb{T}))$ almost surely.
\medskip
\newline \noindent 3. $(\rho, q)$ is locally integrable on $[0, \infty) \times \mathbb{T}$ and is of finite energy in the sense that for the energy $\eta_{E}(\rho,q) := \frac{1}{2}\frac{q^{2}}{\rho} + \frac{\kappa}{\gamma - 1}\rho^{\gamma}$, {for $\gamma>1$} and for all $T > 0$:
\begin{equation*}
\mathbb{E} \|\eta_{E}(\bd{U})\|_{L^{\infty}(0, T; L^{1}(\mathbb{T}))} < \infty.
\end{equation*}
\newline \noindent 4. The term ${\Phi}(\bd{U})$ is progressively measurable with ${\Phi}(\bd{U}) \in L^{2}(\Omega \times [0, T]; L_{2}(\mathcal{U}; L^{2}(\mathbb{T})))$ for all $T > 0$, where $L_{2}(\mathcal{U}; L^{2}(\mathbb{T}))$ denotes the space of Hilbert-Schmidt operators from $\mathcal{U}$ to $L^{2}(\mathbb{T})$. 
\medskip
\newline \noindent 5. For all entropy-flux pairs $(\eta, H)$ arising from functions $g \in \tilde{\mathcal{G}}$, and for all $\varphi(x) \in C^{\infty}(\mathbb{T})$ and \textit{non-negative} $\psi(t) \in C_{c}^{\infty}(0, \infty)$:
\begin{multline}\label{entropyineq}
\int_{0}^{\infty} \left(\int_{\mathbb{T}} \eta(\bd{U}) \varphi(x) dx\right) \partial_{t}\psi(t) dt + \int_{0}^{\infty} \left(\int_{\mathbb{T}} H(\bd{U}) \partial_{x}\varphi(x) dx\right) \psi(t) dt - \int_{0}^{\infty} \left(\int_{\mathbb{T}} \alpha q\partial_{q}\eta(\bd{U}) \varphi(x) dx\right) \psi(t) dt \\
+ \int_{0}^{\infty} \left(\int_{\mathbb{T}} \partial_{q}\eta(\bd{U}) \Phi(\bd{U}) \varphi(x) dx\right) \psi(t) dW(t) + \int_{0}^{\infty} \left(\int_{\mathbb{T}} \frac{1}{2} \partial^{2}_{q} \eta(\bd{U}) G^{2}(\bd{U}) \varphi(x) dx\right) \psi(t) dt \le 0.
\end{multline}

\end{definition}

The entropy inequality \eqref{entropyineq} is required to hold for all entropy-flux pairs $(\eta, H)$ arising from the formulas \eqref{etamformula} for all convex subpolynomial $g \in \tilde{\mathcal{G}}$. However, it will be useful to consider specific functions $g$ in the entropy formula. For example, the constant function $g(z) = 1$ produces the associated entropy $\eta_0 = 1$, the linear function $g(z) = z$ produces the momentum $\eta = \rho u$, and the quadratic function $g(z) =\frac{1}{2}z^{2}$ produces the natural energy 
\begin{equation*}
\eta_{E} := \frac{1}{2} \rho u^{2} + \frac{\kappa}{\gamma - 1} \rho^{\gamma}.
\end{equation*}
In particular, we make the observation that using the entropy $\eta_{E}$ in the entropy inequality \eqref{entropyineq} with $\varphi \equiv 1$ and $\psi := \mathbbm{1}_{[0, t]}$ recovers the usual energy inequality:
\begin{equation*}
\eta_{E}(\bd{U}(t)) \le \eta_{E}(\bd{U}_{0}) + \int_{0}^{t} \left(\int_{\mathbb{T}} \Phi(\rho, \rho u )dx\right) dW(s) + \int_{0}^{t}\left(\int_{\mathbb{T}} \rho^{-1} G^{2}(\bd{U}(s)) dx\right) ds.
\end{equation*}
Therefore, we see that the entropy inequality \eqref{entropyineq} \textit{generalizes the energy inequality}, and by substituting different choices of the entropy other than $\eta_{E}$, we can obtain more information. While the formula \eqref{etamformula} admits all convex $g \in C^{2}(\R)$, it will be helpful to consider specific choices of $g$. In particular, it is useful to define \textbf{higher-order entropy-flux pairs} $(\eta_{m}, H_{m})$, which are associated with the convex functions $g(z) = z^{2m}$ for nonnegative integers $m$ via the formula \eqref{etamformula}, so that $\eta_{1} \sim 2\eta_{E}$.

\begin{remark}[Remarks on the definition of a weak martingale entropy solution]
    We remark that Condition 5 is called the \textbf{entropy inequality}, which is the distributional form of \eqref{entropyito}, expressed as an inequality to take into account the influence of shock formation on entropy. We also remark that usually, a \textit{weak martingale (entropy) solution} is defined on an interval $[0, T]$, with \textit{initial data} $\rho(0) = \rho_{0}$ and $q(0) = q_{0}$ which must be satisfied. However, we state the weak martingale solution in the context of all $t \ge 0$, $t \in [0, \infty)$ without initial data, as statistically stationary solutions (which will be the focus of our manuscript), do {not} have a notion of initial data and are also defined for all time $t \ge 0$. Namely, their main defining feature, other than satisfying the entropy inequality, is stationarity, and hence, one does not a priori provide initial data when considering stationary solutions.
\end{remark}

\begin{remark}[Remark on admissible entropies]
Note that in Condition 5, we require the entropy inequality to hold for all entropy-flux pairs $(\eta, H)$, generated from \eqref{etamformula} by subpolynomial $g \in \tilde{\mathcal{G}}$. We remark that this is a stronger condition than usually required (see for example \cite{BerthelinVovelle, LPS96}), where only entropy-flux pairs $(\eta, H)$ arising from subquadratic $g \in \mathcal{G}$ are considered. However, in our case, since our statistically stationary solution will have bounded entropies of {all} orders, we will use the broader class arising from subpolynomial $g \in \tilde{\mathcal{G}}$ for our entropy inequality. { We also note here that, although not all of these higher-order entropy bounds are required for proving our main result, they are intrinsic to the system and arise naturally as a byproduct of our analysis.}
\end{remark}

\subsection{{Energy dissipation} and statistically stationary solutions} 

Assuming that $\rho(t, x)$ and $u(t, x)$ are smooth functions, one can show immediately that $(\rho, u)$ satisfying the following energy estimate:
\begin{equation}\label{energy}
\frac{1}{2} \int_{\mathbb{T}} \rho u^{2}(t) dx + \int_{\mathbb{T}} \rho^{\gamma}(t) dx + \alpha \int_{0}^{t} \int_{\mathbb{T}} \rho u^{2} dx ds = \frac{1}{2} \int_{\mathbb{T}} \frac{q_{0}^{2}}{\rho_{0}} + \int_{\mathbb{T}} \rho_{0}^{\gamma} + \int_{0}^{t} \left(\int_{\mathbb{T}} \Phi(\rho, u) u dx\right) dW(s),
\end{equation}
where the natural energy associated to the problem is
\begin{equation*}
E(t) := \frac{1}{2} \int_{\mathbb{T}} \rho u^{2}(t) + \int_{\mathbb{T}} \rho^{\gamma}(t).
\end{equation*}
Note that in the absence of damping $(\alpha = 0)$ and in the absence of stochasticity, the deterministic dynamics are \textit{isentropic}, namely energy is conserved $E(t) = E(0)$ for classical solutions. However, in the case of stochasticity and damping, there are two direct contributions changes in the energy to the system encoded in the energy estimate \eqref{energy}:
\begin{itemize}
\item Energy added to the system from the stochastic forcing, which has the potential to increase the energy of the system in expectation.
\item Energy, dissipated from the system as a result of the damping.
\end{itemize}
In order to obtain long-time behavior, one would expect the dynamics of the system to be bounded in long time. Namely, we would want the linear damping to dissipate energy in expectation in a way that balances out the potential energy increase due to the stochasticity. As an important step towards understanding the long-time statistical behavior of the stochastic damped Euler system \eqref{maineqn}, we claim that for this system, we have a \textit{statistically stationary solution}. Heuristically, the existence of stationary solutions intuitively gives information about the long-time statistics of a stochastic dynamical system, namely one would expect that the laws of a stochastic solution in time (or the time-averaged laws) would converge weakly in long time to the law of a statistically stationary solution. 

We remark that a particularly interesting question is whether such a statistically stationary solution exists in the undamped case where $\alpha = 0$, in which case some other mechanism must be present to dissipate the energy added by the stochasticity. In this case, the formation of shocks heuristically should be the mechanism that can dissipate energy (which is accounted for by the fact that the energy balance in \eqref{energy} should really be an energy inequality, in agreement with the entropy inequality of which the energy inequality is a special case). However, this problem, while fundamentally important and mathematically interesting, is beyond the scope of this manuscript. 

\medskip

We define the notion of a \textit{statistically stationary solution} that we will consider in this manuscript. Intuitively, such a solution has statistics at all times that are the same. 
\begin{definition}\label{stationary}
    A weak martingale entropy solution $(\rho, q)$ to the main equation \eqref{maineqn} with the stochastic basis $(\Omega,\sF,\{\sF_t\}_{t\geq 0},W)$ (see Definition \ref{entropydef}) is a \textbf{statistically stationary solution} if $\rho \in C_{w}(\mathbb{R}^{+}; L^{\gamma}(\mathbb{T}))$ and $q \in C_{w}(\mathbb{R}^{+}; L^{\frac{2\gamma}{\gamma + 1}}(\mathbb{T}))$ almost surely and the law of $(\rho(t), q(t))$ in the state space $L^{\gamma}(\mathbb{T}) \times L^{\frac{2\gamma}{\gamma + 1}}(\mathbb{T})$ is independent of the time $t \ge 0$. Hence, $(\rho(s), q(s)) =_{d} (\rho(t), q(t))$ for any $s,t\geq 0$, where $=_{d}$ denotes equality in law. 
\end{definition}

We now can state the main theorem of the manuscript.

\begin{theorem}[Main theorem]\label{mainthm}
Under the assumption \eqref{noiseassumption} on the noise coefficient $\bd{\Phi}(\rho, q): \mathcal{U} \to L^{2}(\mathbb{T})$, there exists a statistically stationary weak martingale entropy solution $(\rho, q)$ with an associated stochastic basis $(\tilde{\Omega}, \tilde{\mathcal{F}}, \tilde{\mathbb{P}}, \{\tilde{\mathcal{F}}_{t}\}_{t \ge 0},\tilde W)$ to the damped compressible stochastic Euler equations in \eqref{maineqn}, for any fixed but arbitrary damping parameter $\alpha > 0$.
\end{theorem}

\begin{remark}\label{rem:anyp}
    Using the methodology of this work, we can obtain the existence of a statistically stationary solution in any state space $L^p(\mathbb{T})\times L^{p}(\mathbb{T})$, for any $1<p<\infty.$
\end{remark}

\begin{remark}
    Note that the system \eqref{maineqn} conserves total mass in time. So more generally, there exists at least one statistically stationary solution corresponding to every possible value of the total mass $\displaystyle \int_{\mathbb{T}} \rho(t, x) dx = M > 0$. For the purpose of the proof however, we will only consider $M = 1$, since the generalization to arbitrary $M$ is immediate.
\end{remark}

\subsection{Significance of results and literature review}

The long-time statistical behavior of stochastic physical systems is a question of inherent physical interest, in addition to being a mathematically interesting problem. Namely, in a system subject to random perturbations, individual observed outcomes of how the system evolves may appear to be disordered, yet when considering an ensemble of repeated experiments, there may be convergence of the overall statistics of the system in long time. This is a fundamental question, especially in fluid dynamics, where the long time statistical behavior of fluid flows is of particular interest for experiments and for the engineering of real-life physical systems.

The long-time statistical behavior of a stochastic system can be encoded in two ways. At a first level, one can show the existence of a \textit{statistically stationary solution}, see for example Definition \ref{stationary}, which has a law that is constant in time. Heuristically, one would expect the time-averaged statistics of a stochastic system to converge to the statistics of a statistically stationary solution. A stronger notion of stationarity is the \textit{existence of an invariant measure}, which involves showing that the random dynamics for the stochastic system are well-posed (existence, uniqueness, continuous dependence). Thus, it makes sense to describe an associated Feller semigroup to the dynamics of the stochastic system, and an invariant measure, which in this case is a probability measure on the state space whose overall statistics are unchanged by the dynamics of the random system.

The question of long-time statistical behavior of stochastic fluid systems is classical, tracing back to results on existence of statistically stationary solutions and invariant measures for the stochastic incompressible Navier-Stokes equations. In 2D, the Navier-Stokes equations have global existence and uniqueness in the deterministic case, which allows for analysis of invariant measures for the 2D stochastic Navier-Stokes equations \cite{BKL01, BKL02, EMS01, FM95, Hai02, HM06, KS12, KS00, MY02, M02, DPZ96}. This includes work on existence/uniqueness of invariant measures for 2D Navier-Stokes equations with discrete-in-time random ``kick" forces \cite{BKL01, KS12, KS00, MY02} and more general noise (such as white noise in time or cylindrical Wiener processes) \cite{BKL02, EMS01, FM95, Hai02, HM06, M02, DPZ96}, and work on exponential mixing and exponential convergence of statistics \cite{BKL02, Hai02, MY02, M02}. These results were extended to (the weaker notion of) statistically stationary solutions in 3D \cite{FlandG95, FM08, DPD03, DPD08}, but invariant measures are in general still an open problem due to the lack of well-posedness for 3D Navier-Stokes, though some approaches are able to construct transition semigroups for dynamics even in the absence of uniqueness \cite{FM08, DPD03, DPD08}. 

The extension of the analysis of long-time behavior to the case of stochastic compressible fluids is more recent. This work was made possible first from new existence results for weak solutions to stochastic compressible Navier-Stokes equations for compressible viscous fluids \cite{BFH18, BreitHofmanova, Smith, SmithTrivisa}, which were natural stochastic extensions of deterministic existence results for global weak solutions in the spirit of Leray-Hopf for (deterministic) compressible Navier-Stokes equations \cite{FeireislCompressible, LionsBook}. The work \cite{FeireislCompressible} on the deterministic compressible Navier-Stokes equations involves a four-layer approximation scheme including an artificial pressure parameter $\delta$ and an artificial viscosity parameter $\epsilon$. This methodology has been robust, and has been used to analyze existence of global in time weak martingale solutions to stochastic compressible flows \cite{BFH18, BreitHofmanova, Smith, SmithTrivisa}. 

These developments in the \textit{existence theory} for stochastic compressible fluid flows then led to the study of the \textit{long-time behavior} of stochastic compressible viscous fluid flows. In fact, a similar four-layer approximation scheme was used to show the existence of a statistically stationary solution to the stochastic compressible Navier-Stokes dynamics in \cite{BFHM19}, where an invariant measure for an approximate system is constructed and is passed through the various approximation layers. We remark however that the proof of existence of a statistically stationary solution is much more delicate than the existence proof, since one must obtain uniform-in-time estimates for the approximate system, and in addition, one can appeal only to stationarity to obtain uniform estimates at each approximation level. Namely, stationary solutions have no ``initial data", and thus one can only appeal to stationarity to define any uniform estimates. A stronger result on existence of an invariant measure was obtained for the specific case of 1D compressible Navier-Stokes equations, with a barotropic fluid for linear pressure law and a linear pressure law in \cite{TrivisaInvariant}. The challenge in this work is establishing well-posedness (existence, uniqueness, continuous dependence on initial data) for the stochastic compressible Navier-Stokes system in 1D, which involves careful a priori estimates. This well-posedness is necessary for defining an associated Feller semigroup for the stochastic compressible Navier-Stokes dynamics, which is needed to properly define a notion of invariant measure.

These results on existence of stochastic solutions and long-time statistics discussed so far have been for compressible \textit{viscous} fluids, but there have not been many developments in the study of stochastic compressible inviscid fluids.
{For a study of statistically stationary solutions or invariant measures to the incompressible Euler equations with linear damping, see \cite{BF20, B08} and for fractionally dissipated Euler equations, see \cite{CGV14}.}
In the context of inviscid {\it compressible} fluid flows, global existence of weak martingale entropy solutions, namely stochastic solutions that satisfy not just a weak formulation but a general entropy inequality, was accomplished in the work \cite{BerthelinVovelle}. This work uses important techniques from the deterministic study of inviscid compressible flows (most importantly \cite{LPS96}, where existence of weak entropy solutions is established for isentropic Euler equations in 1D, and related works \cite{Che86, DCL85, DP831, LPT94}) and invariant regions \cite{CCS77} (see also \cite{DP832, DP831} for applications to deterministic isentropic Euler equations), to obtain uniform estimates on approximate solutions and to pass to the limit in an artificial viscosity parameter using Young measure compactness results and then reducing Young measures {to a Dirac mass due to the presence of a large family of
entropies}. While some numerical evidence is provided in \cite{BerthelinVovelle} to support the existence of an invariant measure for the 1D stochastic inviscid isentropic Euler equations, to the best of our knowledge, progress has not been made on this question. Hence, we believe that the result of the current manuscript, namely the existence of a statistically stationary solution to the \textit{damped} 1D stochastic compressible Euler equations, is a significant step towards the eventual goal of showing invariant measures to the undamped stochastic compressible Euler equations.

\subsection{Algebraic bounds on the $m$-order entropies}\label{etamsection}  In this section we will find algebraic bounds for the entropy functions $\eta_m$, defined in \eqref{etamformula} arising from polynomial $g(z)=z^{2m}$, in terms of density $\rho$ and velocity $u$.

\begin{proposition}\label{etamboundalg}
    We have the following bounds for $\eta_{m}$, where the constants $c_{m}$ and $C_{m}$ depend only on $m$:
    \begin{equation*}
    c_{m}\rho(u^{2m} + \rho^{m(\gamma-1)}) \le \eta_{m}(\bd{U}) \le C_{m}\rho(u^{2m} + \rho^{m(\gamma-1)}),
    \end{equation*}
    and in addition, for some positive constants $c_{m}$ and $C_{m}$:
    \begin{equation*}
    c_{m}\rho(u^{2m} + \rho^{(m - 1)(\gamma-1)}u^{2}) \le q\partial_{q}\eta_{m}(\bd{U}) \le C_{m}\rho(u^{2m} + \rho^{(m - 1)(\gamma-1)}u^{2}),
    \end{equation*}
    \begin{equation*}
    \partial_{q}^{2}\eta_{m}(\bd{U}) \le C_{m}\rho^{-2}\eta_{m - 1}(\bd{U}).
    \end{equation*}
\end{proposition}

\begin{proof}
Recall that $\theta = \frac{\gamma - 1}{2}$.
We obtain by substituting $g(z) = z^{2m}$ into \eqref{etamformula} and the binomial theorem that
\begin{align}\label{expansion}
\eta_{m}(\bd{U}) &= \rho c_{\lambda} \int_{-1}^{1} (u + z\rho^{\theta})^{2m} (1 - z^{2})^{\lambda} dz = c_{\lambda} \rho \sum_{j = 0}^{2m} \left(\int_{-1}^{1} z^{2m - j} (1 - z^{2})^{\lambda} dz\right){ C^m_j} u^{j} \rho^{(2m - j) \theta} \nonumber \\
&= c_{\lambda} \rho \sum_{j = 0}^{m} \left(\int_{-1}^{1} z^{2(m - j)} (1 - z^{2})^{\lambda} dz\right) { C^m_j}  u^{2j} \rho^{2(m - j)\theta} := \sum_{j = 0}^{m} a_{m - j} { C^m_j}  u^{2j} \rho^{1 + 2(m - j)\theta},
\end{align}
where we define the (strictly positive) coefficients $\displaystyle a_{k} := c_{\lambda} \int_{-1}^{1} z^{2k}(1 - z^{2})^{\lambda} dz$. From the expansion \eqref{expansion}, we immediately deduce that $\eta_{m}(U) \ge \min(a_{0}, a_{m}) \cdot \rho(u^{2m} + \rho^{2m\theta})$, and by using Young's inequality with exponents $m/j$ and $m/(m - j)$ applied to the terms $u^{2j}$ and $\rho^{2(m - j)\theta}$ in \eqref{expansion}, we also obtain $\eta_{m}(U) \le C_{m}\rho(u^{2m} + \rho^{2m\theta})$. We can rewrite \eqref{expansion} in terms of $(\rho, q)$ using $q = \rho u$, as $\displaystyle \eta_{m}(\bd{U}) = \sum_{j = 0}^{m} a_{m - j} q^{2j} \rho^{1 + 2(m - j)\theta - 2j}$, and hence, using Young's inequality:
%\begin{equation*}\partial_{q}\eta_{m}(\bd{U}) = \sum_{j = 1}^{m} 2ja_{m - j}q^{2j - 1}\rho^{1 + 2(m - j)\theta - 2j} = \sum_{j = 1}^{m} 2ja_{m - j} \rho^{2(m - j)\theta}u^{2j - 1},\end{equation*}
\begin{equation*}
q\partial_{q}\eta_{m}(\bd{U}) = \sum_{j = 1}^{m} 2ja_{m - j} \rho^{1 + 2(m - j)\theta} u^{2j} \leq C_m \rho(u^{2m} + \rho^{2(m - 1)\theta}u^{2}),
\end{equation*}
\begin{equation*}
\partial_{q}^{2}\eta_{m}(\bd{U}) = \frac{1}{\rho^{2}}\sum_{j = 0}^{m - 1} (2j + 2)(2j + 1) a_{m - j - 1} \rho^{2((m - 1) - j)\theta + 1}u^{2j} \le C_{m}\rho^{-2} \eta_{m - 1}(\bd{U}).
\end{equation*}
\end{proof}
\if 1 = 0
Next, we derive a similar algebraic identity for the flux functions $H_{m}$, associated to the functions $g(z) = z^{2m}$. {\cred delete following}
{\cred
\begin{proposition}
    The flux functions $H_{m}$ satisfy the following algebraic bounds for constants $c_{m}$ and $C_{m}$ depending only on $m$: 
    \begin{equation*}
    c_{m}\rho|u|\Big(u^{2m} + \rho^{(\gamma - 1)m}\Big) \le |H_{m}(\bd{U})| \le C_{m}\rho|u|\Big(u^{2m} + \rho^{(\gamma - 1)m}\Big).
    \end{equation*}
\end{proposition}
}
\begin{proof}
    For $\theta = \frac{\gamma - 1}{2}$ we compute using \eqref{etamformula}:
    \begin{equation*}
    H(\bd{U}) = \rho c_{\lambda} \int_{-1}^{1} g(u + z\rho^{\theta}) (u + z\theta \rho^{\theta}) (1 - z^{2})^{\lambda} dz. 
    \end{equation*}
    Note that for $g(z) = z^{2m}$, the expansion of $g(u + z\rho^{\theta}) (u + z\theta \rho^{\theta})$ is the same as the expansion of $(u + z\rho^{\theta})^{2m + 1}$ up to constant multiples on the coefficients, due to the extra coefficient of $\theta$. Hence, for some constants $a_{k}$:
    \begin{equation*}
    H_{m}(\bd{U}) = \sum_{k = 0}^{2m + 1} \alpha_{k} u^{2m + 1 - k} \rho^{1 + \theta k} \left(\int_{-1}^{1} z^{k} (1 - z^{2})^{\lambda}_{} dz\right) = \sum_{k = 0}^{m} \alpha_{k} u^{2(m - k) + 1} \rho^{1 + 2\theta k}.
    \end{equation*}
    Hence, we conclude that
    \begin{equation*}
    |H_{m}(\bd{U})| \sim \rho |u| \Big(u^{2m} + \rho^{(\gamma-1) m} \Big)
    \end{equation*}
\end{proof}
\fi
It will also be important to deduce certain algebraic bounds on the higher moments of the state variables $\bd{U} = (\rho, q)$. For this, we recall the following algebraic identities that bound higher powers of the density and momentum in terms of the higher entropies $\eta_{m}$ corresponding to $g(z) = z^{2m}$. This will allow us to relate bounds on the higher-order entropies to higher integrability ($L^{s}$ norms) of the lower-order entropies. The following lemma is from Lemma 3.12 from \cite{BerthelinVovelle}.

\begin{lemma}\label{momentidentity}
    Let $(\eta_{m}, H_{m})$ denote the entropy-flux pair corresponding to the convex function $g(z) = z^{2m}$ for nonnegative integers $m$. Then, for any $s \ge 1$,
    \begin{equation*}
    |\eta_{m}(\bd{U})|^{s} \le C(m, s, p) \Big(\eta_{0}(\bd{U}) + \eta_{p}(\bd{U})\Big) \quad \text{ for } p \ge ms + \frac{s - 1}{(\gamma-1)},
    \end{equation*}
    \begin{equation*}
    |H_{m}(\bd{U})|^{s} \le C(m, s, p) \Big(\eta_{0}(\bd{U}) + \eta_{p}(\bd{U})\Big) \quad \text{ for } p \ge \left(m + \frac{1}{2}\right)s + \frac{s - 1}{(\gamma-1)},
    \end{equation*}
    for a constant $C(m, s, p)$ depending only on $s \ge 1$, and the nonnegative integers $m$ and $p$. 
\end{lemma}

Finally, we will establish the following uniform algebraic bound on the entropy dissipation.

\begin{proposition}\label{dissipationalg}
    Suppose that we have that for some deterministic $\bd{U} := (\rho, q)$ which is a function from $\mathbb{T}$ to $[0, \infty) \times \mathbb{T}$ and for some nonnegative integer $m$:
    \begin{equation*}
    \int_{\mathbb{T}} \langle D^{2}\eta_{m + 1}(\bd{U}) \partial_{x}\bd{U}, \partial_{x}\bd{U} \rangle dx \le C. 
    \end{equation*}
    Then, for a constant $C_{m}$ depending only on $m$:
    \begin{equation*}
    \int_{\mathbb{T}} \Big(u^{2m} + \rho^{2m\theta}\Big)\rho^{\gamma - 2}(\partial_{x}\rho)^{2} dx \le C_{m}\int_{\mathbb{T}} \langle D^{2}\eta_{m + 1}(\bd{U}) \partial_{x}\bd{U}, \partial_{x}\bd{U} \rangle dx,
    \end{equation*}
    \begin{equation*}
    \int_{\mathbb{T}} \Big(u^{2m} + \rho^{2m\theta}\Big) \rho(\partial_{x}u)^{2} dx \le C_{m} \int_{\mathbb{T}} \langle D^{2}\eta_{m + 1}(\bd{U}) \partial_{x}\bd{U}, \partial_{x}\bd{U} \rangle dx.
    \end{equation*}
\end{proposition}

\begin{proof}
    This is a direct consequence of the algebraic computations in Proposition 3.14 and Corollary 3.15 in \cite{BerthelinVovelle}, if one keeps track of the constants in the proofs.
\end{proof}

\subsection{Algebraic bounds on general entropy functions} The entropy inequality \eqref{entropyineq} is required to hold for all general entropy-flux pairs $(\eta, H)$ generated by general convex functions $g \in C^{2}(\R)$ via \eqref{entropyineq} which are in $\tilde{\mathcal{G}}$. However,  because of the subpolynomial bound on admissible functions $g \in \tilde{\mathcal{G}}$ in \eqref{subpoly}, a general entropy can be bounded above by the entropies of the form $\eta_{m}$ that we defined for the special functions $g(z) = z^{2m}$. {{In contrast to the previous Section \ref{etamsection} which derives entropy bounds specifically for entropy functions of the form $\eta_m$, we derive some algebraic bounds on \textit{general} entropies $\eta$ (for arbitrary $g \in \tilde{\mathcal{G}}$) in terms of the special entropies $\eta_{m}$ in this section, to justify this reasoning.}}

First, we derive a bound on the general entropies $\eta$ for $g \in \tilde{\mathcal{G}}$ under the assumption of bounded densities. This will be useful for passing to the limit in approximation levels where we still have some uniform control over the maximum value of the density.

\begin{proposition}\label{generaletabound}
    For a general $g \in \tilde{\mathcal{G}}$ satisfying \eqref{subpoly} for some positive integer $m$, there exists a constant $C_{g, M}$ depending only on $g$ and $M$ such that for all $(\rho, q)$ satisfying:
    \begin{equation*}
    0 < \rho \le M, \qquad \left|\frac{q}{\rho}\right| \le M
    \end{equation*}
    for some positive constant $M > 0$, we have the bounds:
    \begin{equation}\label{CgM}
    |\eta(\bd{U})| \le C_{g, M} \rho, \qquad |\nabla_{\rho, q}\eta(\bd{U})| \le C_{g, M}, \qquad \rho |\nabla^{2}_{\rho, q} \eta(\bd{U})| \le C_{g, M}.
    \end{equation}
\end{proposition}

\begin{proof}
For $\displaystyle \theta = \frac{\gamma - 1}{2}$ we compute the first derivatives of the entropy using \eqref{etamformula} as follows:
    \begin{equation*}
    \partial_{\rho}\eta(\bd{U}) = c_{\lambda} \int_{-1}^{1} g\left(\frac{q}{\rho} + z\rho^{\theta}\right)(1 - z^{2})^{\lambda} dz + c_{\lambda} \int_{-1}^{1}g'\left(\frac{q}{\rho} + z\rho^{\theta}\right) \left(\theta z\rho^{\theta} - \frac{q}{\rho}\right) (1 - z^{2})^{\lambda} dz,
    \end{equation*}
    \begin{equation}\label{dqeta}
    \partial_{q}\eta(\bd{U}) = c_{\lambda} \int_{-1}^{1} g'\left(\frac{q}{\rho} + z\rho^{\theta}\right) (1 - z^{2})^{\lambda} dz.
    \end{equation}
    For the second derivatives of the entropy, we compute:
    \begin{equation*}
    \partial_{\rho}^{2}\eta(\bd{U}) = c_{\lambda}\theta(1 - \theta) \rho^{\theta - 1} \int_{-1}^{1} g'\left(\frac{q}{\rho} + z\rho^{\theta}\right) z(1 - z^{2})^{\lambda} dz + c_{\lambda} \rho \int_{-1}^{1} g''\left(\frac{q}{\rho} + z\rho^{\theta}\right) \left(\theta z\rho^{\theta - 1} - \frac{q}{\rho^{2}}\right)^{2} (1 - z^{2})^{\lambda} dz.
    \end{equation*}
    \begin{equation*}
    \partial_{q}\partial_{\rho}\eta(\bd{U}) = c_{\lambda} \int_{-1}^{1} g''\left(\frac{q}{\rho} + z\rho^{\theta}\right) \left(\theta z\rho^{\theta - 1} - \frac{q}{\rho^{2}}\right) (1 - z^{2})^{\lambda} dz 
    \end{equation*}
    \begin{equation}\label{d2qeta}
    \partial_{q}^{2}\eta(\bd{U}) = c_{\lambda} \rho^{-1} \int_{-1}^{1} g''\left(\frac{q}{\rho} + z\rho^{\theta}\right) (1 - z^{2})^{\lambda} dz.
    \end{equation}
    Using the fact that $g(z), g'(z), g''(z)$ are continuous functions on $\R$, and hence are bounded on compact subsets of $\R$, we have that the inequalities \eqref{CgM} follow immediately.
\end{proof}

For later approximation level passages, where we do not have uniform bounds on the maximum of the density, we have the following more general bound.

\begin{proposition}\label{generaletaH}
    Suppose that $g \in \tilde{\mathcal{G}}$ satisfies \eqref{subpoly} for some positive integer $m$. Then, for some constant $C_{g}$ depending on $g$ and for all $\bd{U} := (\rho, q) \in [0, \infty) \times \R$:
    \begin{equation}\label{etaH}
    |\eta(\bd{U})| \le C_{g}\Big(\eta_{0}(\bd{U}) + \eta_{m}(\bd{U})\Big) \quad \text{ and } \quad |H(\bd{U})| \le C_{g}\Big(\eta_{0}(\bd{U}) + \eta_{m + 1}(\bd{U})\Big).
    \end{equation}
\end{proposition}

\begin{proof}
Using \eqref{etamformula} and the subpolynomial bound \eqref{subpoly} for $g$ and $\displaystyle \theta = \frac{\gamma - 1}{2}$ we have,
    \begin{align*}
    |\eta(\bd{U})| &\le C\rho \int_{-1}^{1} (1 - z^{2})^{\lambda} dz + C \rho \sum_{j = 0}^{2m}  \left(\int_{-1}^{1} |z|^{2m - j} (1 - z^{2})^{\lambda} dz\right) C^{2m}_{j} |u|^{j} \rho^{(2m - j)\theta} \\
    &\le C\rho + C\rho\sum_{j = 0}^{2m} b_{2m - j} C^{2m}_{j} |u|^{j}\rho^{(2m - j)\theta},
    \end{align*}
    where $C^{m}_{j}$ are binomial coefficients and $\displaystyle 
    b_{2m - j} := \int_{-1}^{1} |z|^{2m - j} (1 - z^{2})^{\lambda} dz$. Using Young's inequality with exponents $2m/j$ and $2m/(2m - j)$, we obtain that,
    \begin{equation*}
    |\eta(\bd{U})| \le C\rho + C\rho \Big(u^{2m} + \rho^{2m\theta}\Big) = C\rho + C\rho\Big(u^{2m} + \rho^{m(\gamma - 1)}\Big) \le C\Big(\eta_{0}(\bd{U}) + \eta_{m}(\bd{U})\Big),
    \end{equation*}
    where we used the fact that $\eta_{0}(\bd{U}) = \rho$, and the bounds in Proposition \ref{etamboundalg}. Similarly, for $H(\bd{U})$, corresponding to $g \in \tilde{\mathcal{G}}$ satisfying \eqref{subpoly}, we estimate using \eqref{subpoly} and the formula in \eqref{etamformula} that
    \begin{align*}
    |H(\bd{U})| &\le \rho c_{\lambda} \int_{-1}^{1} (|u| + |z|\theta \rho^{\theta}) (1 - z^{2})^{\lambda} dz + C_{\lambda, \theta} \rho \int_{-1}^{1} \Big(|u| + |z\rho^{\theta}|\Big)^{2m + 1} (1 - z^{2})^{\lambda} dz \\
    &\le C\Big(\rho|u| + \rho^{\frac{\gamma + 1}{2}}\Big) + C_{\lambda, \theta} \rho \sum_{j = 0}^{2m + 1} \left(\int_{-1}^{1} |z|^{2m + 1 - j}(1 - z^{2})^{\lambda} dz\right) |u|^{j} \rho^{(2m + 1 - j)\theta} \\
    &\le C\Big(\rho + \rho|u| + \rho^{\gamma}\Big) + C\rho\Big(|u|^{2m + 1} + \rho^{(2m + 1)\theta}\Big),
    \end{align*}
    by a similar Young's inequality argument. Since $\theta = \frac{\gamma - 1}{2}$, we have by Young's inequality that
    \begin{equation*}
    \rho|u| \le C\Big(\rho + \rho|u|^{2m}\Big), \qquad \rho\Big(|u|^{2m + 1} + \rho^{(2m + 1)\theta}\Big) \le C\rho\Big(1 + |u|^{2(m + 1)} + \rho^{(m + 1)(\gamma - 1)}\Big),
    \end{equation*}
    which gives the bound for $H(\bd{U})$ in \eqref{etaH} via Proposition \ref{etamboundalg}.
    
\end{proof}

We also obtain similar algebraic bounds for other terms involving entropies that will appear in the weak formulation.

\begin{proposition}\label{otheretaH}
    Suppose that $g \in \tilde{\mathcal{G}}$ satisfies \eqref{subpoly} for some positive integer $m$. Then, for some constant $C_{g}$ depending on $g$ and for all $\bd{U} := (\rho, q) \in [0, \infty) \times \R$:
    \begin{align*}
    |q\partial_{q}\eta(\bd{U})| &\le C_{g}\Big(\eta_{0}(\bd{U}) + \eta_{m}(\bd{U})\Big), \qquad |G(\bd{U})\partial_{q}\eta(\bd{U})| \le C_{g}\Big(\eta_{0}(\bd{U}) + \eta_{m}(\bd{U})\Big),\\
    &|G^{2}(\bd{U})\partial_{q}^{2}\eta(\bd{U})| \le C_{g}\Big(\eta_{0}(\bd{U}) + \eta_{m - 1}(\bd{U})\Big).
    \end{align*}
    
\end{proposition}

\begin{proof}
    This proceeds similarly to the previous proof of Proposition \ref{generaletaH}, except we use \eqref{dqeta} and \eqref{d2qeta}. Using the bounds on $|g'(z)|$ and $|g''(z)|$ in \eqref{subpoly}, the binomial theorem, and Young's inequality: 
    \begin{align*}
    |\partial_{q}\eta(\bd{U})| &\le c_{\lambda} \int_{-1}^{1} |g'(u + z\rho^{\theta})| (1 - z^{2})^{\lambda} dz \\
    &\le C \int_{-1}^{1} (1 - z^{2})^{\lambda} + C\sum_{j = 0}^{2m - 1} \left(\int_{-1}^{1} |z|^{2m - 1 - j}(1 - z^{2})^{\lambda} dz\right) C^{2m - 1}_{j} |u|^{j} \rho^{(2m - 1 - j)\theta} \\
    &\le C\Big(1 + |u|^{2m - 1} + \rho^{(2m - 1)\theta}\Big).
    \end{align*}
    \begin{align*}
    |\partial_{q}^{2}\eta(\bd{U})| &\le c_{\lambda}\rho^{-1} \int_{-1}^{1} |g''(u + z\rho^{\theta})| (1 - z^{2})^{\lambda} dz \\
    &\le C \rho^{-1} \int_{-1}^{1} (1 - z^{2})^{\lambda} + C\rho^{-1}\sum_{j = 0}^{2m - 2} \left(\int_{-1}^{1} |z|^{2m - 1 - j}(1 - z^{2})^{\lambda} dz\right) C^{2m - 1}_{j} |u|^{j} \rho^{(2m - 1 - j)\theta} \\
    &\le C\rho^{-1}\Big(1 + |u|^{2m - 2} + \rho^{(2m - 2)\theta}\Big).
    \end{align*}

    We then compute using Young's inequality, the bound \eqref{Gbound} on $G(\bd{U})$, and Proposition \ref{etamboundalg} that
    \begin{equation*}
    |q\partial_{q}\eta(\bd{U})| \le C\rho\Big(|u| + |u|^{2m} + \rho^{(2m - 1)\theta}|u|\Big) \le C\rho\Big(1 + |u|^{2m} + \rho^{2m\theta}\Big) \le C\Big(\eta_{0}(\bd{U}) + \eta_{m}(\bd{U})\Big),
    \end{equation*}
    \begin{equation*}
    |G(\bd{U})\partial_{q}\eta(\bd{U})| \le C\rho\Big(1 + |u|^{2m - 1} + \rho^{(2m - 1)\theta}\Big) \le C\rho\Big(1 + |u|^{2m} + \rho^{2m\theta}\Big) \le C\Big(\eta_{0}(\bd{U}) + \eta_{m}(\bd{U})\Big),
    \end{equation*}
    \begin{equation*}
    |G^{2}(\bd{U})\partial_{q}^{2}\eta(\bd{U})| \le C\rho\Big(1 + |u|^{2m - 2} + \rho^{(2m - 2)\theta}\Big) \le C\Big(\eta_{0}(\bd{U}) + \eta_{m - 1}(\bd{U})\Big).
    \end{equation*}
\end{proof}

\section{The approximate system}

In this section, we define an approximate system for which we can find an \textit{invariant measure} using a standard tightness and time-averaging procedure. For an invariant measure to exist for the approximate system, we require strong well-posedness properties for the approximate system: namely existence, uniqueness, and continuous dependence. Hence, we will have to use a sufficient number of approximations and truncations in order to have a suitable approximate system. For this, we use two parameters: (1) a truncation parameter $R$ and (2) an artificial viscosity parameter $\epsilon$. With each parameter, we also appropriately regularize the noise coefficients, defined in \eqref{gkdef}, {in order to obtain high-order derivative estimates for the approximate solutions. Moreover, the noise coefficient approximation is also compactly supported which lets us obtain uniform $L^\infty$ bounds for the approximate solutions.} In particular, we consider a regularized noise coefficient $\bd{\Phi}^{R, \epsilon_{N}}(\bd{U}_{R, \epsilon}): \mathcal{U} \to L^{2}(\mathbb{T})$, where analogously to \eqref{gkdef}, we now define
\begin{equation*}
\bd{\Phi}^{R, \epsilon_{N}}(\rho, q) \bd{e}_{k} := G_{k}^{R, \epsilon_{N}}(x, \rho, q) = \rho g_{k}^{R, \epsilon_{N}}(x, \rho, q),
\end{equation*}
and we assume that, compared to \eqref{noiseassumption}, the noise satisfies stronger assumptions, which we will discuss in Section \ref{noiseregularize}. Here, we note that to define the regularized noise, it will be easier to consider a discrete sequence of artificial viscosity parameters $\{\epsilon_{N}\}_{N = 1}^{\infty}$ with $\epsilon_{N} \searrow 0$, which is why we use the notation of $\epsilon_{N}$ in the approximate system. \textit{Often, for ease of notation, we will omit the explicit sequence dependence on $N$ in the sequence $\epsilon_{N} \to 0$, and just use the parameter $\epsilon > 0$.} 

We will first state the approximate system in Section \ref{approxsystemsection}. Then, we define the noise coefficient approximations in Section \ref{noiseregularize}, and we establish Hadamard well-posedness for the approximate system in Sections \ref{existenceapprox} and \ref{continuousdependence}.

\subsection{Statement of the approximate system.}\label{approxsystemsection} Let $\chi_R \in C_c^{\infty}(\mathbb{R})$ be a smooth function such that $\chi_R(s) = 1$ for $s \le \frac{R}2$, ${\chi_R(s)} = 0$ for $s > R$, and $\chi$ is strictly decreasing on the interval $[\frac{R}2, R]$. For technical reasons, we also choose $\chi_{R}$ so that $\sqrt{\chi_{R}}$ is also smooth and compactly supported on $\R$ (see the proof of Lemma \ref{diffest}, where this assumption is used). Furthermore, given spatial functions $q \in H^{2}(\mathbb{T})$ and $\rho \in H^{2}(\mathbb{T})$, define the truncation
\begin{equation}\label{truncatedq}
[q]_{R} := \chi_R\Big(\|\rho^{-1}\|_{L^{\infty}(\mathbb{T})}\Big) \chi_R\Big(\|q\|_{H^{2}(\mathbb{T})} \Big)q.
\end{equation}
At times, for shorthand, we will use the abbreviation
\begin{equation}\label{chi}
\chi_{R}(\rho, q) := \chi_{R}\Big(\|\rho^{-1}\|_{L^{\infty}(\mathbb{T})}\Big) \chi_{R}\Big(\|q\|_{H^{2}(\mathbb{T})}\Big),
\end{equation}
so that
\begin{equation*}
[q]_{R} = \chi_{R}(\rho, q) q.
\end{equation*}

Then, we consider the following approximate system for $(t, x) \in \mathbb{R}^{+} \times \mathbb{T}$:
\begin{equation}\label{approxsystem}
\begin{cases}
    \partial_{t}\rho + \partial_{x}([q]_{R}) = \epsilon \Delta \rho, \\
    \displaystyle dq + \partial_{x}\left(\frac{[q]_{R}q}{\rho}\right) + \chi_{R}(\rho, q) \partial_{x}(\kappa \rho^{\gamma}) = \chi_{R}(\rho, q)\bd{\Phi}^{R, \epsilon}(\rho, q) dW -\alpha q + \epsilon \Delta q.
\end{cases}
\end{equation}
The goal will be to show existence of an invariant measure on an appropriate phase space to this system. This will involve showing well-posedness (existence, uniqueness, and continuous dependence), and obtaining uniform bounds for the fluid density and fluid velocity uniformly in time.
\subsection{Approximations of the noise coefficient}\label{noiseregularize}

We now discuss the noise coefficients and their approximations at the various levels.

\medskip

\noindent \textbf{Original noise coefficient.} We recall from \eqref{gkdef} and \eqref{noiseassumption} that the noise coefficient $\bd{\Phi}(\bd{U}): \mathcal{U} \to L^{2}(\mathbb{T})$ acts on orthonormal basis elements $\{\bd{e}_{k}\}_{k = 1}^{\infty}$ of $\mathcal{U}$ via
\begin{equation*}
\bd{\Phi}(\rho, q) \bd{e}_{k} := G_{k}(x, \rho, q) = \rho g_{k}(x, \rho, q),
\end{equation*}
where for each positive integer $k$ and for each $(\rho, q) \in [0, \infty) \times \R$, $g_{k}(x, \rho, q)$ is a continuous real-valued function on $\mathbb{T}$. We have the growth assumption:
\begin{equation*}
|g_{k}(x, \rho, q)| + |\nabla_{\rho, q}g_{k}(x, \rho, q)| \le \alpha_{k}, \qquad \text{ for all } (x, \rho, q) \in \mathbb{T} \times [0, \infty) \times \mathbb{R},
\end{equation*}
for some constants $\alpha_{k}$ satisfying:
\begin{equation}\label{A0}
\sum_{k = 1}^{\infty} \alpha_{k}^{2} = A_{0} < \infty.
\end{equation}
Hence,
\begin{equation*}
|\bd{G}(x, \rho, q)| \le A_{0}^{1/2}\rho, \qquad \text{ for } \bd{G}(x, \rho, q) := \left(\sum_{k = 1}^{\infty} |G_{k}(x, \rho, q)|^{2}\right)^{1/2}.
\end{equation*}

\medskip

\noindent \textbf{The regularized $\epsilon$-level noise coefficient.} At the $\epsilon$ level, the goal is to localize the noise coefficient so that it is compactly supported on some invariant region $\Lambda_{\kappa}$ for some $\kappa$ related to $\epsilon > 0$. For the $\epsilon$-level approximation, we will hence truncate the noise so that it is compactly supported in an invariant region. %Recall that we can also express the invariant region \eqref{invariant} in terms of the $(\rho, q)$ variables: 
We define the following sets that will later be important as invariant regions for the approximate system (see Section \ref{uniforminvariant}):
\begin{align}\label{invariant}
\Lambda_{\kappa} := &\{(\rho, u) \in [0, \infty) \times \mathbb{R} : -\kappa \le z \le w \le \kappa\} \\
= &\{(\rho, u) \in [0, \infty) \times \mathbb{R} : 0 \le \rho \le \kappa^{1/\theta}, -\kappa + \rho^{\theta} \le u \le \kappa - \rho^{\theta}\} \nonumber,
\end{align}
and
\begin{equation*}
\tilde{\Lambda}_{\kappa} := \{(\rho, q) \in [0, \infty) \times \mathbb{R} : 0 \le \rho \le \kappa^{1/\theta}, -\rho(\kappa - \rho^{\theta}) \le q \le \rho(\kappa - \rho^{\theta})\},
\end{equation*}
where we note that $(\rho, u) \in \Lambda_{\kappa}$ if and only if $(\rho, q) \in \tilde{\Lambda}_{\kappa}$. Note that just for the current construction of the regularized noise, we will distinguish between the sets $\Lambda_{\kappa}$ in the $(\rho, u)$ plane and the sets $\tilde{\Lambda}_{\kappa}$ in the $(\rho, q)$ plane, but we will later denote both by $\Lambda_{\kappa}$ for notational simplicity.

The goal will be to localize the noise coefficient $\bd{\Phi}(\rho, q)$ to the sets $\tilde{\Lambda}_{\kappa}$ in the $(\rho, q)$ state space, on the $\epsilon$ level. To define the localization of the noise on these sets, we define the following compactly supported function. Let $T(x, y): \R^{2} \to \R$ be a compactly supported function such that 
\begin{equation}\label{Tprops}
    T \text{ is smooth, radially symmetric, and strictly decreasing radially on $1/2 \le |(x, y)| \le 1$},
\end{equation}
\begin{equation*}
    T(x, y) = 1 \text{ for } |(x, y)| \le 1/2,  \quad T(x, y) = 0 \text{ for } |(x, y)| \ge 1.
\end{equation*}

For each positive integer $N$, define the open ball of radius $2N$ centered at $(\rho, q) = (2N + 1/N, 0)$:
\begin{equation}\label{BNdef}
B_{N} := \{(\rho, q) \in [0, \infty) \times \R : \|(\rho, q) - (2N + 1/N, 0)\| \le 2N\},
\end{equation}
and note that these open ball $\{B_{N}\}_{N \in \mathbb{Z}^{+}}$ increase to all of $(0, \infty) \times \R$. There exists an associated increasing sequence of positive real numbers $\{\kappa_{N}\}_{N = 1}^{\infty}$ such that
\begin{equation}\label{kappaN}
B_{N} \subset \tilde{\Lambda}_{\kappa_{N}}.
\end{equation}
We then define the following localized noise coefficient functions, with support in $\overline{B_{N}}$:
\begin{equation}\label{noiseeps}
\begin{split}
   & g_{k}^{\epsilon_{N}}(x, \rho, q) = T\Big(N^{-1}(\rho - (2N + 1/N), q)\Big) g_{k}(x, \rho, q), \\ 
   &\text{ and hence } \ \ \bd{\Phi}^{\epsilon}(\rho, q) \bd{e}_{k} := G^{\epsilon_{N}}_{k}(x, \rho, q) = \rho g^{\epsilon}_{k}(x, \rho, q).
\end{split}
\end{equation}
for any sequence of positive $\{\epsilon_{N}\}_{N = 1}^{\infty}$ strictly decreasing to zero as $N \to \infty$.

Importantly, the regularized noise coefficient has the following essential property, due to \eqref{kappaN}:
\begin{equation}\label{compactsupp}
\text{supp}\Big(\bd{\Phi}^{\epsilon_{N}}(x, \rho, q)\Big) \subset \mathbb{T} \times \tilde{\Lambda}_{\kappa_{N}}, \qquad \text{ for all positive integers $N$}.
\end{equation}
This is because by construction,
\begin{equation}\label{tildeBN}
\text{supp}\Big(\bd{\Phi}^{\epsilon_{N}}(x, \rho, q)\Big) \subset \mathbb{T} \times \tilde{B}_{N},
\end{equation}
where
\begin{equation}\label{tildeBNdef}
\tilde{B}_{N} := \{(\rho, q) \in [0, \infty) \times \mathbb{R} : \|(\rho, q) - (2N + 1/N, 0)\| \le N\} \subset B_{N} \subset \tilde{\Lambda}_{\kappa_{N}}.
\end{equation}
We remark that sometimes, we will will be imprecise with our notation and denote the property \eqref{compactsupp} by:
\begin{equation*}
\text{supp}\Big(\bd{\Phi}^{\epsilon_{N}}(x, \rho, q)\Big) \subset \mathbb{T} \times \Lambda_{\kappa_{N}},
\end{equation*}
in terms of the invariant region $\Lambda_{\kappa_{N}}$ in the $(\rho, u)$ plane, as we actually mean that the values of $(\rho, u)$, and not the values of $(\rho, q)$, which are in the support of $\bd{\Phi}^{\epsilon_N}(x, \rho, q)$ lie in $\Lambda_{\kappa}$. 

We can verify from the formula \eqref{noiseeps} and the properties of the compactly supported function $T$ in \eqref{Tprops}, that we still have the following properties:
\begin{equation*}
|g_{k}^{\epsilon}(x, \rho, q)| \le \alpha_{k}, \qquad |\nabla_{\rho, q} g^\epsilon_{k}(x, \rho, q)| \le \alpha_{k}(1 + N^{-1}),
\end{equation*}
for the same constants $\alpha_{k}$ as in \eqref{noiseassumption}, and hence, for the same constant $A_{0}$ in \eqref{A0}:
\begin{equation}\label{Gepsbound}
|\bd{G}^{\epsilon}(x, \rho, q)| \le A_{0}^{1/2}\rho, \qquad \text{ for } \bd{G}^{\epsilon}(x, \rho, q) := \left(\sum_{k = 1}^{\infty} |G_{k}^{\epsilon}(x, \rho, q)|^{2}\right)^{1/2}.
\end{equation}
Furthermore, using the fact that the region $\tilde{\Lambda}_{\kappa_{N}}$ is bounded in the $(\rho, q)$ plane depending on $\kappa_{N}$, we hence have 
\begin{equation}\label{GN}
|G^{\epsilon}_{k}(x, \rho, q)| + |\nabla_{\rho, q} G^{\epsilon}_{k}(x, \rho, q)| \le C_{N}\alpha_{k},
\end{equation}
for a constant $C_{N}$ depending only on $N$.

\medskip

\noindent \textbf{The regularized $\epsilon$-$R$ noise coefficient.} {{At the $R$ level, we additionally regularize the $\epsilon$-level noise coefficient via convolution, so that the $\epsilon$-$R$ level approximations of the noise coefficients have spatial derivatives of all orders, which will help us show well-posedness for our $\epsilon$-$R$ approximate system.}} For this approximation, let $\overline{\zeta}(z): \R \to \R$ be a standard smooth nonnegative convolution kernel with support in $[-1, 1]$ with integral equal to one. Then, define
\begin{equation*}
\zeta_{\alpha}(x, \rho, q) = \frac{1}{\alpha^{3}} \overline{\zeta}\left(\frac{x}{\alpha}\right) \overline{\zeta}\left(\frac{\rho}{\alpha}\right) \overline{\zeta}\left(\frac{q}{\alpha}\right).
\end{equation*}
Define for positive integers $R \ge 1$:
\begin{equation}\label{GRepsconv}
G^{R,\epsilon_{N}}_{k}(x, \rho, u) = G^{\epsilon_{N}}_{k}(x, \rho, u) * \zeta_{N/R} 1_{k \le R},
\end{equation}
for positive integers $R$. Since $G^{\epsilon_{N}}_{k}(x, \rho, q)$ is compactly supported on $\mathbb{T} \times \tilde{B}_{N}$ where $\tilde{B}_{N}$ as in \eqref{tildeBNdef} is the ball of radius $N$ centered at $(2N + 1/N, 0)$, we note that this convolution is well-defined (extend $G^{R_{\epsilon}}_{k}(x, \rho, q)$ by zero for $\rho \le 0$) and 
\begin{equation*}
\text{supp}\Big(G^{R, \epsilon_{N}}_{k}(x, \rho, q)\Big) \subset \mathbb{T} \times B_{N},
\end{equation*}
where $B_{N}$ is the ball of radius $2N$ centered at $(2N + 1/N, 0)$, as in \eqref{BNdef}. Then, define
\begin{equation*}
\bd{\Phi}^{R, \epsilon_{N}}(x, \rho, q) \bd{e}_{k} = G^{R, \epsilon_{N}}_{k}(x, \rho, q),
\end{equation*}
and note that
\begin{equation}\label{suppReps}
\text{supp}\Big(\bd{\Phi}^{R, \epsilon}(x, \rho, q)\Big) \subset \mathbb{T} \times \Lambda_{\kappa_{\epsilon}},
\end{equation}
by \eqref{kappaN}. Then, for each positive integer $m$, there exists a constant $C(m, R, N)$ such that
\begin{equation}\label{noiseRepsbound}
|\nabla_{\rho, q}^{m} G^{R,\epsilon}_{k}(x, \rho, q)| \le C(m, R, N)\alpha_{k}.
\end{equation}
By \eqref{GN} and the properties of the convolution in \eqref{GRepsconv}, note that this constant is independent of $R$ when $m = 0, 1$, namely:
\begin{equation}\label{alphakeps}
|G_k^{R, \epsilon}(x, \rho, q)| + |\nabla_{\rho, q} G^{R, \epsilon}_{k}(x, \rho, q)| \le C_{N}\alpha_{k},
\end{equation}
for a constant $C_{N}$ depending only on $N$. Hence, for the same constant $A_{0}$ in \eqref{A0}:
\begin{equation}\label{A0eps}
|\bd{G}^{R,\epsilon}(x, \rho, q)| \le A_{0}^{1/2}\rho, \qquad \text{ for } \bd{G}^{R,\epsilon}(x, \rho, q) := \left(\sum_{k = 1}^{\infty} |G_{k}^{R,\epsilon}(x, \rho, q)|^{2}\right)^{1/2}.
\end{equation}
\subsection{Existence and uniqueness for the approximate system and preliminary estimates}\label{existenceapprox}

To show existence of an invariant measure for the approximate system \eqref{approxsystem}, we will need a notion of Hadamard well-posedness for this system (global existence, uniqueness, and continuous dependence). In this subsection, we consider existence and uniqueness of the approximate system \eqref{approxsystem} for initial data $(\rho_0, u_0)$ satisfying $\displaystyle \int_{\mathbb{T}} \rho_{0} dx = 1$ with $\rho > 0$, in terms of the following phase space:
\begin{equation}\label{path}
\mathcal{X} := \left\{(\rho, q) \in H^{2}(\mathbb{T}) \times H^{2}(\mathbb{T}) : \int_{\mathbb{T}} \rho(x) dx = 1 \text{ and } \rho \ge \frac{1}{R}\right\},
\end{equation}
where the function space $\mathcal{X}$ is endowed with the usual norm of $H^{2}(\mathbb{T}) \times H^{2}(\mathbb{T})$. We verify that with respect to this phase space $\mathcal{X}$, we have existence, uniqueness, and continuous dependence for the approximate system \eqref{approxsystem}, which will allow us to form a Feller semigroup $\mathcal{P}_{t}$ (see Section \ref{fellersemigroup}) for the evolution of solutions to \eqref{approxsystem}. We first show the following existence and uniqueness result.

\begin{proposition}\label{existunique}
    Given $\bd{U}_0=(\rho_{0}, q_{0}) \in \mathcal{X}$ { and a probability space $(\Omega,\mathcal{F},\bP)$ along with a filtration $(\sF_t)_{t\geq 0}$ and an $(\sF_t)_{t\geq 0}$-Wiener process $W$}, there exists a unique (classical) pathwise solution $\bd{U}^{R,\epsilon}=(\rho^{R,\epsilon}, q^{R,\epsilon})$ to the approximate problem \eqref{approxsystem} in the class ${ L^2}(\Omega;C(0, { \infty}; \mathcal{X}))$ that satisfies, for all nonnegative test functions $\varphi \in C^{2}(\mathbb{T})$ and  $\psi \in C^{\infty}_{c}(0, \infty)$, and all entropy-entropy flux pair $(\eta, H)$, the following equation $\bP$-almost surely:
\begin{multline}\label{entropyRepsilon}
\int_{0}^{\infty} \left(\int_{\mathbb{T}} \eta(\bd{U}^{R,\epsilon}(t)) \varphi(x) dx\right) \partial_{t}\psi(t) dt + \int_{0}^{\infty} \chi_R(\rho^{R,\epsilon},q^{R,\epsilon}) \left(\int_{\mathbb{T}} H(\bd{U}^{R,\epsilon})\partial_{x}\varphi(x) dx\right) \psi(t) dt \\
- \int_{0}^{\infty} \left(\int_{\mathbb{T}} \alpha q^{R,\epsilon} \partial_{q}\eta(\bd{U}^{R,\epsilon}) \varphi(x) dx\right) \psi(t) dt + \int_{0}^{\infty} \left(\int_{\mathbb{T}} \partial_{q}\eta(\bd{U}^{R,\epsilon}) \bd{\Phi}(\bd{U}^{R,\epsilon}) \varphi(x) dx\right) \psi(t) dW(t) \\
+ \int_{0}^{\infty} \left(\int_{\mathbb{T}} \frac{1}{2} \partial_{q}^{2} \eta(\bd{U}^{R,\epsilon}) G^{2}(\bd{U}^{R,\epsilon}) \varphi(x) dx\right) \psi(t) =  \epsilon \int_{0}^{\infty} \left(\int_{\mathbb{T}} \langle D^{2}\eta(\bd{U}^{R,\epsilon})\partial_{x}\bd{U}^{R,\epsilon}, \partial_{x}\bd{U}^{R,\epsilon} \rangle \varphi(x) dx\right) \psi(t) dt \\
- \epsilon \int_{0}^{\infty} \left(\int_{\mathbb{T}} \eta(\bd{U}^{R,\epsilon}) \partial^{2}_{x}\varphi dx\right) \psi(t) dt. 
\end{multline}

%    In addition, for any $\delta \ge 0$, there exists a unique (classical) solution $(\rho_{\delta}^{R,\epsilon}, q_{\delta}^{R,\epsilon})$ to the generalized approximate problem \eqref{approxsystemdelta} in the class $L^2(\Omega;C(0, \infty; \mathcal{X}))$.  
\end{proposition}

\begin{proof}
The proof is identical to that of Theorem 3.2 in \cite{BerthelinVovelle}.
\end{proof}
Next, we verify some useful preliminary estimates on the global unique solution to the approximate system, which will be useful for the upcoming proof of continuous dependence. Specifically, we prove the following preliminary a priori estimate for the density. 
\begin{lemma}\label{rhoH2}
    Let $(\rho_{}^{R,\epsilon}, q_{}^{R,\epsilon})$ be the unique solution to \eqref{approxsystem} in $C(0, T; H^2(\mathbb{T}))$ for initial data $(\rho_0, q_0) \in \mathcal{X}$. Then,
    for some (deterministic) constant $C_{R, \epsilon, T}$ depending only on $R$, $\epsilon$, and the final time $T$, we have the following almost sure bound:
    \begin{equation*}
    \|\rho_{}^{R,\epsilon}\|_{C(0, T; H^{2}(\mathbb{T}))}^{2} \le \|\rho_0\|_{H^{2}(\mathbb{T})}^{2} + C_{R, \epsilon, T}. 
    \end{equation*}
  %  Moreover, $(\rho_{\delta}^{R,\epsilon}, q_{\delta}^{R,\epsilon})\in C(0, T; \mathcal{X})$ almost surely, that is $\rho_{\delta}^{R,\epsilon}>\frac1R$ almost surely for all $t\in [0,T]$.

\end{lemma}
\begin{proof}
By taking zero, one, and two spatial derivatives of the continuity equation in \eqref{dampedtruncate}, we obtain for $i = 0, 1, 2$:
\begin{equation*}
\partial_{t}(\partial_{x}^{i}\rho_{}^{R,\epsilon}) + \chi_{R}(\rho_{}^{R,\epsilon}, q_{}^{R,\epsilon}) \partial_{x}^{i + 1} q_{}^{R,\epsilon} = \epsilon \partial_{x}^{i + 2} \rho_{}^{R,\epsilon} %-  \partial_{x}^{i} \rho_{}^{R,\epsilon} 
\qquad \text{ on } [0,T]\times \mathbb{T}. 
\end{equation*}
By testing with $\partial_{x}^{i}\rho_{}^{R,\epsilon}$ and integrating by parts, we obtain:
\begin{equation*}
\int_{\mathbb{T}} (\partial_{x}^{i}\rho_{}^{R,\epsilon})^{2}(t) %+ \delta \int_{0}^{t} \int_{\mathbb{T}} (\partial_{x}^{i} \rho_{\delta}^{R,\epsilon})^{2} 
+ \epsilon \int_{0}^{t} \int_{\mathbb{T}} |\partial_{x}^{i + 1} \rho_{}^{R,\epsilon}|^{2} = \int_{\mathbb{T}} (\partial_{x}^{i}\rho_{0})^{2} + \int_{0}^{t} \int_{\mathbb{T}} \chi_{R}(\rho, q) \partial^{i}_{x} q_{}^{R,\epsilon} \partial_{x}^{i + 1}\rho_{}^{R,\epsilon}.
\end{equation*}
By estimating
\begin{align*}
\left|\int_{0}^{t} \int_{\mathbb{T}} \chi_{R}(\rho_{}^{R,\epsilon}, q_{}^{R,\epsilon}) \partial_{x}^{i}q_{}^{R,\epsilon} \partial_{x}^{i + 1}\rho_{}^{R,\epsilon}\right| \le \frac{\epsilon}{2} \int_{0}^{t} \int_{\mathbb{T}} (\partial_{x}^{i + 1}\rho_{}^{R,\epsilon})^{2} + C(\epsilon) \int_{0}^{t} \int_{\mathbb{T}} |\chi_{R}(\rho_{}^{R,\epsilon}, q_{}^{R,\epsilon})\partial_{x}^{i}q_{}^{R,\epsilon}|^{2} \\
\le \frac{\epsilon}{2} \int_{0}^{t} \int_{\mathbb{T}} (\partial^{i + 1}_{x}\rho_{}^{R,\epsilon})^{2} + C(\epsilon, R)t, 
\end{align*}
since $\|\chi(\rho_{}^{R,\epsilon}, q_{}^{R,\epsilon}) q_{}^{R,\epsilon}\|_{H^{2}(\mathbb{T})} \le R$ by the definition of the truncation in \eqref{chi}. Thus,
\begin{equation*}
\int_{\mathbb{T}} (\partial_{x}^{i}\rho_{}^{R,\epsilon})^{2}(t) %+ \delta \int_{0}^{t} \int_{\mathbb{T}} (\partial^{i}_{x}\rho_{\delta}^{R,\epsilon})^{2} 
+ \frac{\epsilon}{2} \int_{0}^{t} \int_{\mathbb{T}} |\partial_{x}^{i + 1} \rho_{}^{R,\epsilon}|^{2} \le \int_{\mathbb{T}} (\partial_{x}^{i}\rho_{0})^{2} + C(\epsilon, R)t
\end{equation*}
almost surely, from which we obtain the desired (almost sure) estimate from Gronwall's inequality. %Note that since the $\delta$ parameter appears as a dissipative term, the resulting almost sure estimate of $\|\rho_{}^{R,\epsilon}\|_{C(0, T; H^{2}(\mathbb{T}))}$ is independent of $\delta$. 

\end{proof}

We also have the following minimum principle for initial data $(\rho_0, q_0) \in \mathcal{X}$, which we recall from the definition of the phase space $\mathcal{X}$ in \eqref{path} must satisfy $\rho_{0} \ge \frac1{R}$. This result is important because it shows that the unique solution $(\rho^{R,\epsilon}, q^{R,\epsilon})$ to \eqref{approxsystem} in $C(0, T; H^{2}(\mathbb{T}))$ for initial data $(\rho_0, q_0) \in \mathcal{X}$ is more specifically in $C(0, T; \mathcal{X})$ since by the definition of the phase space $\mathcal{X}$ in \eqref{path}, {functions $(\rho^{R,\epsilon}, q^{R,\epsilon})$ in $\mathcal{X}$ must satisfy the pointwise lower bound for the density $\rho^{R,\epsilon} \ge \frac1{R}$. This also importantly gives us uniform control in time on the density $\rho^{R, \epsilon}$ away from vacuum}. 

\begin{proposition}\label{rholowerbound}
    Let $\rho^{R,\epsilon} \in C(0, T; H^{2}(\mathbb{T}))$ and $q^{R,\epsilon} \in C(0, T; H^{2}(\mathbb{T}))$ satisfy
    \begin{equation}\label{continuityR}
    \partial_{t} \rho^{R,\epsilon} + \partial_{x}([q^{R,\epsilon}]_{R}) = \epsilon \Delta \rho^{R,\epsilon}, \qquad \text{ on } \mathbb{T}, 
    \end{equation}
    for initial data $\rho(0) = \rho_{0} \in \mathcal{X}$. Then, $\rho^{R,\epsilon}(t, \cdot) \ge \frac1{R}$ for all $t \in [0, T]$, that is we have $\rho^{R,\epsilon} \in C(0, T; \mathcal{X})$.
\end{proposition}

\begin{proof}
\if 1 = 0
    To verify this minimum principle claim, we use a standard perturbation technique. We first consider, for the given $q \in C(0, T; H^{2}(\mathbb{T}))$, a solution $\rho_{\delta}^{R,\epsilon} \in C(0, T; H^{2}(\mathbb{T}))$ to the following equation:
    \begin{equation}\label{deltacontinuity}
    \partial_{t}\rho_{\delta}^{R,\epsilon} + \chi_{R}\Big(\|\rho_{\delta}^{R,\epsilon}^{-1}\|_{L^{\infty}(\mathbb{T})}\Big) \chi_{R}\Big(\|q\|_{H^{2}(\mathbb{T})}\Big) \partial_{x}q = \epsilon \Delta \rho_{\delta}^{R,\epsilon} + \delta \rho_{\delta}^{R,\epsilon}, \qquad \text{ on } [0, T] \times \mathbb{T}.
    \end{equation}
    Using usual well-posedness arguments,  a solution $\rho_{\delta}^{R,\epsilon} \in C(0, T; H^{2}(\mathbb{T}))$ with such a regularity exists.

    \medskip

    \noindent \textbf{Step 1. Show that the solution to the perturbed equation $\rho_{\delta}^{R,\epsilon} > 0$.} 
    \fi 

   % We first show that $\rho_{\delta}^{R,\epsilon} \ge 0$ for all $t \in [0, T]$ to illustrate the general technique, and then generalize this to show that $\rho_{\delta}^{R,\epsilon} \ge 1/R$. 
   %To show that $\rho_{\delta}^{R,\epsilon} > 0$ {\cred are we proving this for $\delta$-solutions?}, we can test the equation \eqref{continuityR} by $-\text{sgn}(-\rho^{R,\epsilon})$ to obtain for the negative part of $\rho^{R,\epsilon}$ denoted by $(\rho^{R,\epsilon})^{-}$, that \begin{equation*}\int_{\mathbb{T}} (\rho^{R,\epsilon})^{-}(t) = \epsilon \int_{0}^{t} \int_{\mathbb{T}} \Delta(-\rho^{R,\epsilon}) \text{sgn}(-\rho^{R,\epsilon}), \end{equation*} where we use the fact that $\rho_0 \ge 1/R$ and hence, $\rho_{0}^{-} = 0$ on $\mathbb{T}$, and also {\cred check} $\chi_{R}\Big(\|(\rho^{R,\epsilon})^{-1}\|_{L^{\infty}(\mathbb{T})}\Big) \text{sgn}(-\rho^{R,\epsilon}) = 0$ by definition of the truncation and the sgn function. Since $\rho^{R,\epsilon} \in C(0, T; H^{2}(\mathbb{T}))$, by applying the result on pg.~64 of \cite{BFH18}, we have that $\displaystyle \epsilon \int_{0}^{t} \int_{\mathbb{T}} \Delta(-\rho^{R,\epsilon})\text{sgn}(-\rho^{R,\epsilon}) \le 0$, and hence: $\displaystyle \int_{\mathbb{T}} (\rho^{R,\epsilon})^{-}(t) \le 0$. Thus, $\displaystyle \int_{\mathbb{T}} (\rho^{R,\epsilon})^{-}(t) = 0$ for all $t \in [0, T]$. This shows that $\rho^{R,\epsilon} \ge 0$. 

   {{Note that the initial data $\rho_0 \in \mathcal{X}$.}} Now observe that \eqref{approxsystem}$_1$ can be written as
    \begin{equation}\label{shifted}
    \partial_{t}(\rho^{R,\epsilon} - R^{-1}) + \partial_{x}([q^{R,\epsilon}]_{R}) = \epsilon \Delta (\rho^{R,\epsilon} - R^{-1}).
    \end{equation}
First, we claim for any $t\in[0,T]$, that
 $\chi_{R}\Big(\|(\rho^{R,\epsilon}(t))^{-1}\|_{L^{\infty}(\mathbb{T})}\Big) \cdot \text{sgn}^{+}(R^{-1} - \rho^{R,\epsilon}(t)) = 0$. 
Here $\text{sgn}^{+}(x)=1$ if $x\geq0$ and is equal to 0 if $x<0$. 
Notice that this claimed equality is true if for all times we have $\rho^{R,\epsilon}(t)> -\frac1{R}$.
Hence, for a contradiction, we assume that for a fixed but arbitrary $\omega\in\Omega$ there exists a time $\tau(\omega)<T$ such that $\tau(\omega)$ is the first time when $\rho^{R,\epsilon}(\tau(\omega))=-\frac1{R}$ and hence the first time when $\chi_{R}\Big(\|(\rho^{R,\epsilon}(\tau))^{-1}\|_{L^{\infty}(\mathbb{T})}\Big) \cdot \text{sgn}^{+}(R^{-1} - \rho^{R,\epsilon}(\tau)) \neq 0$. Note that since $\rho^{R,\epsilon}\in C([0,T]\times\mathbb{T})$ and since $\rho_0\geq \frac1R$, we must have $0<\tau(\omega)$ for every $\omega\in \Omega$.
 
    Now we multiply \eqref{shifted} by $\text{sgn}^{ +}({R}^{-1} - \rho^{R,\epsilon})$ and integrate in space and then in time on $[0,\tau(\omega)]$ to obtain:
    \begin{equation*}
    \int_{\mathbb{T}} (\rho^{R,\epsilon} - R^{-1})^{-}(\tau) = \epsilon \int_{0}^{\tau} \int_{\mathbb{T}} \Delta(R^{-1} - \rho^{R,\epsilon}) \cdot \text{sgn}^+(R^{-1} - \rho^{R,\epsilon}),
    \end{equation*}
    since $\rho_0 \ge 1/R$, so that $\displaystyle \int_{\mathbb{T}} (\rho_0 - R^{-1})^{-} = 0$, and since by assumption
    $\chi_{R}\Big(\|(\rho^{R,\epsilon})^{-1}\|_{L^{\infty}(\mathbb{T})}\Big) \cdot \text{sgn}^+(R^{-1} - \rho^{R,\epsilon}) = 0$ for $t<\tau(\omega)$. Since $\rho^{R,\epsilon} - R^{-1} \in C(0, T; H^{2}(\mathbb{T}))$, we have that $\displaystyle \epsilon \int_{0}^{\tau} \int_{\mathbb{T}} \Delta (R^{-1} - \rho^{R,\epsilon}) \text{sgn}^+(R^{-1} - \rho^{R,\epsilon}) \le 0$ by pg.~64 of \cite{FN17}. So we have that $\displaystyle \int_{\mathbb{T}} (\rho^{R,\epsilon} - R^{-1})^{-}(\tau) = 0$ and thus $\rho^{R,\epsilon} \ge 1/R$ for all $(t, x) \in [0, \tau(\omega)] \times \mathbb{T}$. This contradicts our assumption.
\end{proof}
    \if 1 = 0
    
    \noindent \textbf{Step 2. Show that the solution to the perturbed equation $\rho_{\delta}^{R,\epsilon} \ge 1/R$.} We can then verify that the solution $\rho_{\delta}^{R,\epsilon} \in C(0, T; H^{2}(\mathbb{T}))$ to \eqref{deltacontinuity} satisfies $\rho_{\delta}^{R,\epsilon} \ge 1/R$ using a minimum principle argument. Namely, assume $\rho_{\delta}^{R,\epsilon}$ and $q$ are smooth (since for the general case, we can use an approximation argument to regularize $\rho_{\delta}^{R,\epsilon}$ and $q$ and pass to a limit). We can argue by contradiction to show that $\rho_{\delta}^{R,\epsilon}(t, \cdot) \ge 1/R$, where we note that since $\rho_{0} \in \mathcal{X}$, we have that the initial data $\rho_{0} \ge 1/R$. Assume for contradiction that $\rho_{\delta}^{R,\epsilon}$ attains a strict minimum at $(t_0, x_0) \in [0, T] \times \mathbb{T}$ for some $t_0 > 0$. Then, $\chi_{R}\Big(\|\rho_{\delta}^{R,\epsilon}(t_0)^{-1}\|_{L^{\infty}(\mathbb{T})}\Big) = 0$, $\Delta \rho_{\delta}^{R,\epsilon}(t_0, x_0) \ge 0$, and $-\delta \rho_{\delta}^{R,\epsilon}(t_0, x_0) < 0$ since $\rho_{\delta}^{R,\epsilon} > 0$ on $[0, T]$. Therefore, using the equation \eqref{deltacontinuity}, we have that $\partial_{t}\rho_{\delta}^{R,\epsilon}(t_0, x_0) > 0$, which contradicts that $(t_0, x_0)$ is a minimum. Thus, $\rho_{\delta}^{R,\epsilon}(t, \cdot) \ge 1/R$ for all $t \in [0, T]$
    
    \medskip
    
    \noindent \textbf{Step 3. Show that the solution $\rho$ to the original equation $\rho \ge 1/R$.} We claim that for the original equation where $\delta = 0$, we have that $\rho \ge 1/R$ also for $[0, T] \times \mathbb{T}$, and to show this, we use the fact that $\rho_{\delta}^{R,\epsilon} \ge 1/R$ on $[0, T] \times \mathbb{T}$ and pass to a limit as $\delta \to 0$. Consider the solution $\rho_{\delta}^{R,\epsilon} \in C(0, T; \mathcal{X})$ to the following perturbed equation:
    \begin{equation*}
    \partial_{t}\rho_{\delta}^{R,\epsilon} + \chi_{R}\Big(\|\rho_{\delta}^{R,\epsilon}\|_{L^{\infty}(\mathbb{T})}^{-1}\Big) \chi_{R}\Big(\|q\|_{H^{2}(\mathbb{T})}\Big)\partial_{x}q = \epsilon \Delta \rho_{\delta}^{R,\epsilon} + \delta \rho_{\delta}^{R,\epsilon}.
    \end{equation*}
    We then compare this to the given solution $\rho$ to the original equation:
    \begin{equation*}
    \partial_{t}\rho + \chi_{R}\Big(\|\rho\|_{L^{\infty}(\mathbb{T})}^{-1}\Big) \chi_{R}\Big(\|q\|_{H^{2}(\mathbb{T})}\Big) \partial_{x} q = \epsilon \rho.
    \end{equation*}
    Here, we consider $q \in C(0, T; \mathcal{X})$ to be given a priori, and we note that $q$ is the same in both the perturbed and original equation.

    We can estimate the difference $\rho - \rho_{\delta}^{R,\epsilon}$ by noting that for $j = 0, 1$ (by differentiating the equations once for $j = 1$), we have that:
    \begin{equation*}
    \partial_{t}\partial_{x}^{j}(\rho - \rho_{\delta}^{R,\epsilon}) + \Big[\chi_{R}\Big(\|\rho^{-1}\|_{L^{\infty}(\mathbb{T})}\Big) - \chi_{R}\Big(\|\rho_{\delta}^{R,\epsilon}^{-1}\|_{L^{\infty}(\mathbb{T})}\Big)\Big] \chi_{R}\Big(\|q\|_{H^{2}(\mathbb{T})}\Big) \partial_{x}^{j + 1}q = \epsilon \partial_{x}^{j + 2} (\rho - \rho_{\delta}^{R,\epsilon}) - \delta \partial_{x}^{j} \rho_{\delta}^{R,\epsilon}.
    \end{equation*}
    We can then test this equation by $\partial_{j}(\rho - \rho_{\delta}^{R,\epsilon})$ for $j = 0, 1$ and sum the result. We obtain the following terms:
    \begin{itemize}
        \item \textbf{Term 1.} We use the fact that $\chi_{R}$ is smooth and $\|\rho\|_{L^{\infty}(\mathbb{T})} > 0$ (need a continuity argument for this) to obtain:
        \begin{align*}
        \Big|\chi_{R}\Big(\|\rho^{-1}\|_{L^{\infty}(\mathbb{T})}\Big) - \chi_{R} \Big(\|\rho_{\delta}^{R,\epsilon}^{-1}\|_{L^{\infty}(\mathbb{T})}\Big)\Big| &\le C_{R}\Big|\|\rho^{-1}\|_{L^{\infty}(\mathbb{T})} - \|\rho_{\delta}^{R,\epsilon}^{-1}\|_{L^{\infty}(\mathbb{T})}\Big| \\
        &\le C_{R}\|\rho^{-1} - \rho_{\delta}^{R,\epsilon}^{-1}\|_{L^{\infty}(\mathbb{T})} \le C_{R}\|\rho - \rho_{\delta}^{R,\epsilon}\|_{H^{1}(\mathbb{T})}.
        \end{align*}
        Therefore,
        \begin{align*}
        \Bigg|\int_{0}^{t} \int_{\mathbb{T}} \Big[\chi_{R}&\Big(\|\rho^{-1}\|_{L^{\infty}(\mathbb{T})}\Big) - \chi_{R}\Big(\|\rho_{\delta}^{R,\epsilon}^{-1}\|_{L^{\infty}(\mathbb{T})}\Big)\Big] \chi_{R}\Big(\|q\|_{H^{2}(\mathbb{T})}\Big) \partial^{j + 1}_{x} q \partial^{j}_{x}(\rho - \rho_{\delta}^{R,\epsilon})\Bigg| \\
        &\le C_{R}\int_{0}^{t} \|\rho - \rho_{\delta}^{R,\epsilon}\|_{H^{1}(\mathbb{T})}^{2} \|q\|_{H^{2}(\mathbb{T})} ds \le C_{R} \|q\|_{C(0, T; H^{2}(\mathbb{T}))} \int_{0}^{t} \|\rho - \rho_{\delta}^{R,\epsilon}\|_{H^{1}(\mathbb{T})}^{2}.
        \end{align*}
        \item \textbf{Term 3.} We estimate that
        \begin{equation*}
        \int_{0}^{t} \int_{\mathbb{T}} \delta \partial_{x}^{j} \rho_{\delta}^{R,\epsilon} \partial_{x}^{j}(\rho - \rho_{\delta}^{R,\epsilon}) \le \frac{\epsilon}{2} \int_{0}^{t} \int_{\mathbb{T}} |\partial_{x}(\rho - \rho_{\delta}^{R,\epsilon})|^{2} + C(\epsilon) \int_{0}^{t} \int_{\mathbb{T}} |\delta \partial_{x}^{j} \rho_{\delta}^{R,\epsilon}|^{2}.
        \end{equation*}
    \end{itemize}
    Therefore, we obtain:
    \begin{align}\label{deltaineq1}
    \|(\rho - \rho_{\delta}^{R,\epsilon})(t)\|_{H^{1}(\mathbb{T})}^{2} &+ \frac{\epsilon}{2} \int_{0}^{t} \int_{\mathbb{T}} \Big(|\partial_{x}(\rho - \rho_{\delta}^{R,\epsilon})(s)|^{2} + |\partial_{x}^{2}(\rho - \rho_{\delta}^{R,\epsilon})(s)|^{2}\Big) dx ds \\
    &\le C(\epsilon) \delta^{2} \|\rho_{\delta}^{R,\epsilon}\|_{L^{2}(0, T; H^{1}(\mathbb{T}))}^{2} + C_{R}\|q\|_{C(0, T; H^{1}(\mathbb{T}))} \int_{0}^{t} \|(\rho - \rho_{\delta}^{R,\epsilon})(s)\|_{H^{1}(\mathbb{T})}^{2} ds.
    \end{align}
    By usual energy estimate arguments, we can show that for $0 \le \delta \le 1$, $\|rho_{\delta}\|_{L^{2}(0, T; H^{1}(\mathbb{T}))}^{2} \le C_{R, T}$ for a constant depending only on $R$ and $T$, using Gronwall's inequality. Therefore, by using Gronwall's inequality in \eqref{deltaineq1}, we conclude that
    \begin{equation*}
    \|\rho - \rho_{\delta}^{R,\epsilon}\|_{C(0, T; H^{1}(\mathbb{T}))}^{2} \to 0, \qquad \text{ as } \delta \to 0.
    \end{equation*}
    So $\|\rho - \rho_{\delta}^{R,\epsilon}\|_{C([0, T] \times \mathbb{T}} \to 0$, and since $\rho_{\delta}^{R,\epsilon} \ge 1/R$ on $[0, T] \times \mathbb{T}$, we have $\rho \ge 1/R$ on $[0, T] \times \mathbb{T}$ also.
    \fi

\subsection{Continuous dependence on the initial condition of the approximate system.}\label{continuousdependence}

To complete the verification of Hadamard well-posedness for the approximate system, we finally show the following continuous dependence result for the approximate system \eqref{approxsystem}. %{\cred we already have weak feller property, must comment on why strong is required in our case: }

\begin{proposition}\label{continuous}
    Let $(\rho, q)$ and $(\tilde{\rho}, \tilde{q})$ be two solutions in $C(0, T; \mathcal{X})$ to \eqref{approxsystem} with initial data $(\rho_{0}, q_{0}) \in \mathcal{X}$ and $(\tilde{\rho}_{0}, \tilde{q}_{0}) \in \mathcal{X}$. Then, for all $t \in [0, T]$,
    \begin{equation*}
    \mathbb{E}\|(\rho(t), q(t)) - (\tilde{\rho}(t), \tilde{q}(t))\|_{C(0, T; \mathcal{X})} \le C(T, R, \epsilon) \|(\rho_0, q_0) - (\tilde{\rho}_0, \tilde{q}_0)\|_{\mathcal{X}},
    \end{equation*}
    where the constant $C(T, R, \epsilon)$ depends only on $T$, $R$, and $\epsilon$.
\end{proposition}

This will be accomplished via a priori estimates, on all spatial derivatives up to the second spatial derivative. Before doing the a priori estimates, we make the following observations, which will be useful for estimating the difference between two solutions to \eqref{approxsystem}.

\begin{lemma}\label{diffest}
For $(\rho, q)$ and $(\tilde{\rho}, \tilde{q})$ in $H^{2}(\mathbb{T}) \times H^{2}(\mathbb{T})$ with $\rho, \tilde{\rho} \ge 1/R$ and $\max\Big(\|\rho\|_{H^{2}(\mathbb{T})}, \|\tilde{\rho}\|_{H^{2}(\mathbb{T})}\Big) \le M$ for some constant $M > 0$, the following estimates hold:
\begin{align}
\Big|\chi_{R}\Big(\|\rho^{-1}\|_{L^{\infty}(\mathbb{T})}\Big) - \chi_{R}\Big(\|\tilde{\rho}^{-1}\|_{L^{\infty}(\mathbb{T})}\Big)\Big| &\le C_{R}\|\rho - \tilde{\rho}\|_{H^{2}(\mathbb{T})}, \label{diff1} \\
\Big|\chi_{R}\Big(\|q\|_{H^{2}(\mathbb{T})}\Big) - \chi_{R}\Big(\|\tilde{q}\|_{H^{2}(\mathbb{T})}\Big)\Big| &\le C_{R} \|q - \tilde{q}\|_{H^{2}(\mathbb{T})}, \label{diff2} \\
\Big|\chi_{R}(\rho, q) - \chi_{R}(\tilde{\rho}, \tilde{q})\Big| &\le C_{R}\Big(\|\rho - \tilde{\rho}\|_{H^{2}(\mathbb{T})} + \|q - \tilde{q}\|_{H^{2}(\mathbb{T})}\Big), \label{diff3} \\
\Big\|\chi_{R}\Big(\|q\|_{H^{2}(\mathbb{T})}\Big) q - \chi_{R}\Big(\|\tilde{q}\|_{H^{2}(\mathbb{T})}\Big)\tilde{q}\Big\|_{H^{2}(\mathbb{T})} &\le C_{R}\|q - \tilde{q}\|_{H^{2}(\mathbb{T})}, \label{diff4} \\
\Big\|\chi_{R}(\rho, q) q - \chi_{R}(\tilde{\rho}, \tilde{q}) \tilde{q}\Big\|_{H^{2}(\mathbb{T})} &\le C_{R}\Big(\|\rho - \tilde{\rho}\|_{H^{2}(\mathbb{T})} + \|q - \tilde{q}\|_{H^{2}(\mathbb{T})}\Big), \label{diff5} \\
\Big\|\chi_{R}(\rho, q) \rho - \chi_{R}(\tilde{\rho}, \tilde{q}) \tilde{\rho}\Big\|_{H^{2}(\mathbb{T})} &\le C_{R, M}\Big(\|\rho - \tilde{\rho}\|_{H^{2}(\mathbb{T})} + \|q - \tilde{q}\|_{H^{2}(\mathbb{T})}\Big), \label{diff5b} \\
\label{diff6} \|\rho^{-\alpha} - \tilde{\rho}^{-\alpha}\|_{L^{\infty}(\mathbb{T})} &\le C_{R, M, \alpha}\|\rho - \tilde{\rho}\|_{H^{2}(\mathbb{T})} \qquad \text{ for } \alpha \ge 1.
\end{align}
Furthermore, for $i, j = 0, 1, 2$:
\begin{equation}\label{diffquad}
\begin{cases}
\Big\|\chi_{R}(\rho, q) \partial_{x}^{i}q \partial_{x}^{j}q - \chi_{R}(\tilde{\rho}, \tilde{q}) \partial_{x}^{i} \tilde{q} \partial_{x}^{j} \tilde{q}\Big\|_{L^{2}(\mathbb{T})} &\le C_{R}\Big(\|\rho - \tilde{\rho}\|_{H^{2}(\mathbb{T})} + \|q - \tilde{q}\|_{H^{2}(\mathbb{T})}\Big), \\
\Big\|\chi_{R}(\rho, q) \partial_{x}^{i}\rho\partial_{x}^{i}q - \chi_{R}(\tilde{\rho}, \tilde{q})\partial_{x}^{i}\tilde{\rho}\partial_{x}^{i}\tilde{q}\Big\|_{L^{2}(\mathbb{T})} &\le C_{R, M}\Big(\|\rho - \tilde{\rho}\|_{H^{2}(\mathbb{T})} + \|q - \tilde{q}\|_{H^{2}(\mathbb{T})}\Big), \\
\Big\|\chi_{R}(\rho, q) \partial_{x}^{i}\rho \partial_{x}^{j} \rho - \chi_{R}(\tilde{\rho}, \tilde{q}) \partial_{x}^{i}\tilde{\rho} \partial_{x}^{j}\tilde{\rho}\Big\|_{L^{2}(\mathbb{T})} &\le C_{R, M}\Big(\|\rho - \tilde{\rho}\|_{H^{2}(\mathbb{T})} + \|q - \tilde{q}\|_{H^{2}(\mathbb{T})}\Big).
\end{cases}
\end{equation}
\end{lemma}

\begin{proof}
We prove the series of inequalities as follows.
\medskip
\newline
\noindent \textbf{Proof of truncation estimates \eqref{diff1}--\eqref{diff3}.} Note that since $\chi_{R}$ is a smooth function satisfying $0 \le \chi_{R} \le 1$ and $\sup_{z \ge 0} |\chi_{R}'(z)| \le C_{R}$ for a constant $C_{R}$ depending only on $R$, we can estimate:
\begin{align*}
\Big|\chi_{R}\Big(\|\rho^{-1}\|_{L^{\infty}(\mathbb{T})}\Big) - \chi_{R}\Big(\|\tilde{\rho}^{-1}\|_{L^{\infty}(\mathbb{T})}\Big)\Big| &\le C_{R}\Big|\|\rho^{-1}\|_{L^{\infty}(\mathbb{T})} - \|\tilde{\rho}^{-1}\|_{L^{\infty}(\mathbb{T})}\Big| \\
&\le C_{R}\|\rho^{-1} - \tilde{\rho}^{-1}\|_{L^{\infty}(\mathbb{T})} \le C_{R} \left\|\frac{\rho - \tilde{\rho}}{\rho\tilde{\rho}}\right\|_{L^{\infty}(\mathbb{T})} \\
&\le C_{R} \|\rho - \tilde{\rho}\|_{L^{\infty}(\mathbb{T})} \le C_{R} \|\rho - \tilde{\rho}\|_{H^{2}(\mathbb{T})},
\end{align*}
since $\rho, \tilde{\rho} \ge 1/R$ by assumption. Similarly, we estimate that
\begin{equation*}
\Big|\chi_{R}\Big(\|q\|_{H^{2}(\mathbb{T})}\Big) - \chi_{R} \Big(\|\tilde{q}\|_{H^{2}(\mathbb{T})}\Big)\Big| \le C_{R} \Big|\|q\|_{H^{2}(\mathbb{T})} - \|\tilde{q}\|_{H^{2}(\mathbb{T})}\Big| \le C_{R}\|q - \tilde{q}\|_{H^{2}(\mathbb{T})}. 
\end{equation*}
By combining these two estimates \eqref{diff1} and \eqref{diff2}, and using the fact that $0 \le \chi_{R} \le 1$, we also obtain the estimate \eqref{diff3} via
\begin{multline*}
|\chi_{R}(\rho, q) - \chi_{R}(\tilde{\rho}, \tilde{q})| \le \Big|\chi_{R}\Big(\|\rho^{-1}\|_{L^{\infty}(\mathbb{T})}\Big) - \chi_{R}\Big(\|\tilde{\rho}^{-1}\|_{L^{\infty}(\mathbb{T})}\Big)\Big| \cdot \chi_{R}\Big(\|q\|_{H^{2}(\mathbb{T})}\Big) \\
+ \chi_{R}\Big(\|\tilde{\rho}^{-1}\|_{L^{\infty}(\mathbb{T})}\Big) \Big|\chi_{R}\Big(\|q\|_{H^{2}(\mathbb{T})}\Big) - \chi_{R}\Big(\|\tilde{q}\|_{H^{2}(\mathbb{T})}\Big)\Big|.
\end{multline*}

\medskip

\noindent \textbf{Proof of \eqref{diff4} and \eqref{diff5}.} Next, we prove the fourth inequality \eqref{diff4}. To show this, note that if both $\|q\|_{H^{2}(\mathbb{T})} \ge R + 1$ and $\|\tilde{q}\|_{H^{2}(\mathbb{T})} \ge R + 1$, then the left-hand side of \eqref{diff4} is zero and the inequality is trivial. So suppose that at least one of $q$ and $\tilde{q}$ has $H^{2}(\mathbb{T})$ norm less than $R + 1$, and without loss of generality, let $\|q\|_{H^{2}(\mathbb{T})} < R + 1$. Then, by \eqref{diff2}:
\begin{align*}
\Big\|\chi_{R}\Big(\|q\|_{H^{2}(\mathbb{T})}\Big) q &- \chi_{R}\Big(\|\tilde{q}\|_{H^{2}(\mathbb{T})}\Big)\tilde{q}\Big\|_{H^{2}(\mathbb{T})} \\
&\le \Big|\chi_{R}\Big(\|q\|_{H^{2}(\mathbb{T})}\Big) - \chi_{R}\Big(\|\tilde{q}\|_{H^{2}(\mathbb{T})}\Big)\Big| \cdot \|q\|_{H^{2}(\mathbb{T})} + \chi_{R}\Big(\|\tilde{q}\|_{H^{2}(\mathbb{T})}\Big) \|q - \tilde{q}\|_{H^{2}(\mathbb{T})} \\
&\le C_{R}(R + 1)\|q - \tilde{q}\|_{H^{2}(\mathbb{T})} + \|q - \tilde{q}\|_{H^{2}(\mathbb{T})} \le C_{R}\|q - \tilde{q}\|_{H^{2}(\mathbb{T})}.
\end{align*}
\if 1 = 0
The same argument works for establishing the last inequality, once we note that $1/R \le \rho, \tilde{\rho} \le C_{R}$ pointwise almost surely. So we can estimate:
\begin{equation*}
\|\rho^{-\alpha} - \tilde{\rho}^{-\alpha}\|_{L^{\infty}(\mathbb{T})} = \left\|\frac{\rho^{\alpha} - \tilde{\rho}^{\alpha}}{\rho^{\alpha}\tilde{\rho}^{\alpha}}\right\|_{L^{\infty}(\mathbb{T})} \le R^{2\alpha}\|\rho^{\alpha} - \tilde{\rho}^{\alpha}\|_{L^{\infty}(\mathbb{T})} \le C_{R}\|\rho - \tilde{\rho}\|_{L^{\infty}(\mathbb{T})},
\end{equation*}
by the mean value theorem and the pointwise upper and lower bounds on the density. 
\fi

To show the fifth inequality \eqref{diff5}, we estimate the left-hand side of the fourth inequality as
\begin{equation*}
\Big\|\chi_{R}\Big(\|\rho^{-1}\|_{L^{\infty}(\mathbb{T})}\Big) \chi_{R}\Big(\|q\|_{H^{2}(\mathbb{T})}\Big) q - \chi_{R}\Big(\|\tilde{\rho}^{-1}\|_{L^{\infty}(\mathbb{T})}\Big) \chi_{R}\Big(\|\tilde{q}\|_{H^{2}(\mathbb{T})}\Big) \tilde{q}\Big\|_{H^{2}(\mathbb{T})} \le I_{1} + I_{2},
\end{equation*}
for $I_{1}$ and $I_{2}$ defined by:
\begin{equation*}
I_{1} := \Big|\chi_{R}\Big(\|\rho^{-1}\|_{L^{\infty}(\mathbb{T})}\Big) - \chi_{R} \Big(\|\tilde{\rho}^{-1}\|_{L^{\infty}(\mathbb{T})}\Big) \Big| \cdot \Big\|\chi_{R}\Big(\|q\|_{H^{2}(\mathbb{T})}\Big) q\Big\|_{H^{2}(\mathbb{T})},
\end{equation*}
\begin{equation*}
I_{2} := \chi_{R}\Big(\|\tilde{\rho}^{-1}\|_{L^{\infty}(\mathbb{T})}\Big) \Big\|\chi_{R}\Big(\|q\|_{H^{2}(\mathbb{T})}\Big) q - \chi_{R}\Big(\|\tilde{q}\|_{H^{2}(\mathbb{T})}\Big) \tilde{q}\Big\|_{H^{2}(\mathbb{T})}. 
\end{equation*}
We then immediately obtain the desired inequality \eqref{diff5} from the past results \eqref{diff1} and \eqref{diff4}, the bound $0 \le \chi_{R} \le 1$, and the fact that $\Big\|\chi_{R}\Big(\|q\|_{H^{2}(\mathbb{T})}\Big) q\Big\|_{H^{2}(\mathbb{T})} \le R + 1$. 

\medskip

\noindent \textbf{Proof of \eqref{diff5b}.} We estimate that
\begin{equation*}
\Big\|\chi_{R}\Big(\|\rho^{-1}\|_{L^{\infty}(\mathbb{T})}\Big) \chi_{R}\Big(\|q\|_{H^{2}(\mathbb{T})}\Big)\rho - \chi_{R}\Big(\|\tilde{\rho}^{-1}\|_{L^{\infty}(\mathbb{T})}\Big) \chi_{R}\Big(\|\tilde{q}\|_{H^{2}(\mathbb{T})}\Big)\tilde{\rho}\Big\|_{H^{2}(\mathbb{T})} \le I_{1} + I_{2},
\end{equation*}
where
\begin{equation*}
I_{1} := \Big|\chi_{R}\Big(\|\rho^{-1}\|_{L^{\infty}(\mathbb{T})}\Big) \chi_{R}\Big(\|q\|_{L^{\infty}(\mathbb{T})}\Big) - \chi_{R}\Big(\|\tilde{\rho}^{-1}\|_{L^{\infty}(\mathbb{T})}\Big) \chi_{R}\Big(\|\tilde{q}\|_{H^{2}(\mathbb{T})}\Big)\Big| \cdot \|\rho\|_{H^{2}(\mathbb{T})},
\end{equation*}
\begin{equation*}
I_{2} := \chi_{R}\Big(\|\tilde{\rho}^{-1}\|_{L^{\infty}(\mathbb{T})}\Big) \chi_{R}\Big(\|\tilde{q}\|_{H^{2}(\mathbb{T})}\Big) \cdot \|\rho - \tilde{\rho}\|_{H^{2}(\mathbb{T})}.
\end{equation*}
By the assumption $\|\rho\|_{H^{2}(\mathbb{T})} \le M$ and the estimate \eqref{diff3}:
\begin{equation*}
I_{1} \le C_{R, M} \Big(\|\rho - \tilde{\rho}\|_{H^{2}(\mathbb{T})} + \|q - \tilde{q}\|_{H^{2}(\mathbb{T})}\Big).
\end{equation*}
Since $0 \le \chi_{R} \le 1$, we immediately have $I_{2} \le \|\rho - \tilde{\rho}\|_{H^{2}(\mathbb{T})}$, which establishes the estimate.

\medskip

\noindent \textbf{Proof of density estimates \eqref{diff6}.} We calculate that
\begin{equation*}
\|\rho^{-\alpha} - \tilde{\rho}^{-\alpha}\|_{L^{\infty}(\mathbb{T})} = \left\|\frac{\rho^{\alpha} - \tilde{\rho}^{\alpha}}{\rho^{\alpha}\tilde{\rho}^{\alpha}}\right\|_{L^{\infty}(\mathbb{T})} \le R^{2\alpha} \|\rho^{\alpha} - \tilde{\rho}^{\alpha}\|_{L^{\infty}(\mathbb{T})} \le C_{R, M, \alpha} \|\rho - \tilde{\rho}\|_{L^{\infty}(\mathbb{T})} \le C_{R, M, \alpha} \|\rho - \tilde{\rho}\|_{H^{2}(\mathbb{T})},
\end{equation*}
using the almost sure lower bound on $\rho, \tilde{\rho} \ge 1/R$ and the almost sure upper bound $\rho, \tilde{\rho} \le C_{M}$ from the assumption $\max\Big(\|\rho\|_{H^{2}(\mathbb{T})}, \|\tilde{\rho}\|_{H^{2}(\mathbb{T})}\Big) \le M$.

\medskip

\noindent \textbf{Proof of quadratic estimates \eqref{diffquad}.} To show the last estimate, note that 
\begin{align*}
\Big\|\chi_{R}(\rho, q) &\partial_{x}^{i}q \partial_{x}^{j}q - \chi_{R}(\tilde{\rho}, \tilde{q}) \partial_{x}^{i} \tilde{q} \partial_{x}^{j} \tilde{q}\Big\|_{L^{2}(\mathbb{T})} \le \Big\|\Big(\sqrt{\chi_{R}(\rho, q)} \partial_{x}^{i}q - \sqrt{\chi_{R}(\tilde{\rho}, \tilde{q})} \partial_{x}^{i}\tilde{q}\Big)\Big\|_{L^{\infty}(\mathbb{T})} \Big\|\sqrt{\chi_{R}(\rho, q)}\partial_{x}^{j}q\Big\|_{L^{2}(\mathbb{T})} \\
&+ \Big\|\sqrt{\chi_{R}(\tilde{\rho}, \tilde{q})} \partial_{x}^{i}\tilde{q}\Big\|_{L^{\infty}(\mathbb{T})} \Big\|\sqrt{\chi_{R}(\rho, q)}\partial_{x}^{j}q - \sqrt{\chi_{R}(\tilde{\rho}, \tilde{q})} \partial_{x}^{j}\tilde{q}\Big\|_{L^{2}(\mathbb{T})} \\
&\le \Big\|\Big(\sqrt{\chi_{R}(\rho, q)} q - \sqrt{\chi_{R}(\tilde{\rho}, \tilde{q})} \tilde{q}\Big)\Big\|_{H^{2}(\mathbb{T})} \cdot \Big(\Big\|\sqrt{\chi_{R}(\rho, q)}q\Big\|_{H^{2}(\mathbb{T})} + \Big\|\sqrt{\chi_{R}(\tilde{\rho}, \tilde{q})} \tilde{q}\Big\|_{H^{2}(\mathbb{T})}\Big) \\
&\le C_{R}\Big(\|\rho - \tilde{\rho}\|_{H^{2}(\mathbb{T})} + \|q - \tilde{q}\|_{H^{2}(\mathbb{T})}\Big).
\end{align*}
Here, we use the fifth estimate \eqref{diff5} (which still holds in this case since $\sqrt{\chi_{R}}$ is also smooth and hence the proof of the estimate \eqref{diff5} remains unchanged in this case), and the support properties of $\chi_{R}$ and hence $\sqrt{\chi_{R}}$ to obtain this estimate. We can similarly obtain the other estimates in \eqref{diffquad} analogously, using \eqref{diff5} and \eqref{diff6}.
\end{proof}

We then use these inequalities to estimate differences of nonlinear terms in the equations.

\begin{lemma}\label{nonlineardiff}
For nonnegative $\rho, \tilde{\rho} \in H^{2}(\mathbb{T})$ satisfying the pointwise lower bounds $\rho \ge 1/R$, $\tilde{\rho} \ge 1/R$, and $\max\Big(\|\rho\|_{H^{2}(\mathbb{T})}, \|\tilde{\rho}\|_{H^{2}(\mathbb{T})}\Big) \le M$ for some positive constant $M$, we have the following estimates:
\begin{align*}
\|[q]_{R} - [\tilde{q}]_{R}\|_{H^{2}(\mathbb{T})} &\le C_{R}\Big(\|\rho - \tilde{\rho}\|_{H^{2}(\mathbb{T})} + \|q - \tilde{q}\|_{H^{2}(\mathbb{T})}\Big), \\
\left\|\frac{[q]_{R}q}{\rho} - \frac{[\tilde{q}]_{R}\tilde{q}}{\tilde{\rho}}\right\|_{H^{2}(\mathbb{T})} &\le C_{R, M}\Big(\|\rho - \tilde{\rho}\|_{H^{2}(\mathbb{T})} + \|q - \tilde{q}\|_{H^{2}(\mathbb{T})}\Big), \\
\Big\|\chi_{R}(\rho, q) \rho^{\gamma} - \chi_{R}(\tilde{\rho}, \tilde{q})\tilde{\rho}^{\gamma}\Big\|_{H^{2}(\mathbb{T})} &\le C_{R, M}\Big(\|\rho - \tilde{\rho}\|_{H^{2}(\mathbb{T})} + \|q - \tilde{q}\|_{H^{2}(\mathbb{T})}\Big).
\end{align*}
\end{lemma}

\begin{proof}
The first inequality is the fifth inequality \eqref{diff5} proved in Lemma \ref{diffest}. For the second inequality, we have the following estimates, using \eqref{diff6} and \eqref{diffquad} in Lemma \ref{diffest} and $\Big\|\chi_{R}(\rho, q)\partial_{x}^{i}q \partial_{x}^{j}q\Big\|_{L^{\infty}(\mathbb{T})} \le C_{R}$ for $i, j = 0, 1$ by the Sobolev embedding $H^{2}(\mathbb{T}) \subset W^{1, \infty}(\mathbb{T})$:
\begin{equation*}
\left\|\frac{\chi_{R}(\rho, q)q^{2}}{\rho} - \frac{\chi_{R}(\tilde{\rho}, \tilde{q})\tilde{q}^{2}}{\tilde{\rho}}\right\|_{L^{2}(\mathbb{T})} \le C_{R}\Big(\|\rho - \tilde{\rho}\|_{H^{2}(\mathbb{T})} + \|q - \tilde{q}\|_{H^{2}(\mathbb{T})}\Big).
\end{equation*}
We can similarly obtain the same estimate for the following quantities:
\begin{equation*}
\left\|\frac{\chi_{R}(\rho, q) (\partial_{x}q)^{2}}{\rho} - \frac{\chi_{R}(\tilde{\rho}, \tilde{q})(\partial_{x}\tilde{q})^{2}}{\tilde{\rho}}\right\|_{L^{2}(\mathbb{T})}, \qquad \left\|\frac{\chi_{R}(\rho, q) q(\partial_{x}q)}{\rho} - \frac{\chi_{R}(\tilde{\rho}, \tilde{q})\tilde{q} (\partial_{x}\tilde{q})}{\tilde{\rho}}\right\|_{L^{2}(\mathbb{T})},
\end{equation*}
\begin{equation*}
\left\|\frac{\chi_{R}(\rho, q)q\partial_{x}^{2}q}{\rho} - \frac{\chi_{R}(\tilde{\rho}, \tilde{q})\tilde{q}\partial_{x}^{2}\tilde{q}}{\tilde{\rho}}\right\|_{L^{2}(\mathbb{T})},
\end{equation*}
where we use $\Big\|\chi_{R}(\rho, q) q\partial_{x}^{2}q\Big\|_{L^{2}(\mathbb{T})} \le C_{R}$ by the truncation to estimate the last quantity. It remains to estimate two more quantities. For $i = 1, 2$:
\begin{align*}
\Bigg\|\frac{\chi_{R}(\rho, q)q \partial_{x}^{i}\rho}{\rho^{2}} &- \frac{\chi_{R}(\tilde{\rho}, \tilde{q}) \tilde{q} \partial_{x}^{i}\tilde{\rho}}{\tilde{\rho}^{2}}\Bigg\|_{L^{2}(\mathbb{T})} \le \|\rho^{-2}\|_{L^{\infty}(\mathbb{T})} \|\partial_{x}^{i}(\rho - \tilde{\rho})\|_{L^{2}(\mathbb{T})} \|\chi_{R}(\rho, q)q\|_{L^{\infty}(\mathbb{T})} \\
&+ \|\rho^{-2} - \tilde{\rho}^{-2}\|_{L^{\infty}(\mathbb{T})} \|\partial_{x}^{i}\tilde{\rho}\|_{L^{2}(\mathbb{T})} \|\chi_{R}(\rho, q) q\|_{L^{\infty}(\mathbb{T})} + \left\|\frac{\partial_{x}^{i} \tilde{\rho}}{\tilde{\rho}^{2}}\right\|_{L^{2}(\mathbb{T})} \|\chi_{R}(\rho, q) q - \chi_{R}(\tilde{\rho}, \tilde{q}) \tilde{q}\|_{L^{\infty}(\mathbb{T})} \\
&\le C_{R, M}\Big(\|\rho - \tilde{\rho}\|_{H^{2}(\mathbb{T})} + \|q - \tilde{q}\|_{H^{2}(\mathbb{T})}\Big),
\end{align*}
where we use the assumptions $\|\rho\|_{H^{2}(\mathbb{T})} \le M$ and $\rho, \tilde{\rho} \ge 1/R$ with \eqref{diff5} and \eqref{diff6}. Finally, we estimate:
\begin{align*}
&\Bigg\|\frac{\chi_{R}(\rho, q) q (\partial_{x}\rho)^{2}}{\rho^{3}} - \frac{\chi_{R}(\tilde{\rho}, \tilde{q}) \tilde{q}(\partial_{x}\tilde{\rho})^{2}}{\tilde{\rho}^{3}}\Bigg\|_{L^{2}(\mathbb{T})} \le \|\rho^{-3}\|_{L^{\infty}(\mathbb{T})} \|(\partial_{x}\rho)^{2} - (\partial_{x}\tilde{\rho})^{2}\|_{L^{2}(\mathbb{T})} \|\chi_{R}(\rho, q)q\|_{L^{\infty}(\mathbb{T})} \\
&+ \|\rho^{-3} - \tilde{\rho}^{-3}\|_{L^{\infty}(\mathbb{T})} \|(\partial_{x}\tilde{\rho})^{2}\|_{L^{2}(\mathbb{T})} \|\chi_{R}(\rho, q) q\|_{L^{\infty}(\mathbb{T})} + \left\|\frac{(\partial_{x}^{i} \tilde{\rho})^{2}}{\tilde{\rho}^{3}}\right\|_{L^{2}(\mathbb{T})} \|\chi_{R}(\rho, q) q - \chi_{R}(\tilde{\rho}, \tilde{q}) \tilde{q}\|_{L^{\infty}(\mathbb{T})}.
\end{align*}
We estimate, using the assumption that $\|\rho\|_{H^{2}(\mathbb{T})} \le M$ (and similarly with $\tilde{\rho}$) along with Sobolev embedding, that 
\begin{equation*}
\|(\partial_{x}\tilde{\rho})^{2}\|_{L^{2}(\mathbb{T})} \le \|\partial_{x}\tilde{\rho}\|_{L^{4}(\Omega)}^{2} \le C_{R, M}, 
\end{equation*}
\begin{equation*}
\|(\partial_{x}\rho)^{2} - (\partial_{x}\tilde{\rho})^{2}\|_{L^{2}(\mathbb{T})} \le \|\partial_{x}(\rho + \tilde{\rho})\|_{L^{\infty}(\mathbb{T})} \|\partial_{x}(\rho - \tilde{\rho})\|_{L^{2}(\mathbb{T})} \le C_{R, M}\|\rho - \tilde{\rho}\|_{H^{2}(\mathbb{T})}. 
\end{equation*}
Hence, by \eqref{diff5} and \eqref{diff6}:
\begin{equation*}
\Bigg\|\frac{\chi_{R}(\rho, q) q (\partial_{x}\rho)^{2}}{\rho^{3}} - \frac{\chi_{R}(\tilde{\rho}, \tilde{q}) \tilde{q}(\partial_{x}\tilde{\rho})^{2}}{\tilde{\rho}^{3}}\Bigg\|_{L^{2}(\mathbb{T})} \le C_{R, M}\Big(\|\rho - \tilde{\rho}\|_{H^{2}(\mathbb{T})} + \|q - \tilde{q}\|_{H^{2}(\mathbb{T})}\Big).
\end{equation*}

To prove the last inequality, we estimate it as
\begin{align*}
\Big\|\chi_{R}(\rho, q) \rho^{\gamma} - \chi_{R}(\tilde{\rho}, \tilde{q})\tilde{\rho}^{\gamma}\Big\|_{L^{2}(\mathbb{T})} &\le \|(\chi_{R}(\rho, q) - \chi_{R}(\tilde{\rho}, \tilde{q}))\rho^{\gamma}\|_{L^{2}(\mathbb{T})} + \|\chi_{R}(\tilde{\rho}, \tilde{q}) \cdot (\rho^{\gamma} - \tilde{\rho}^{\gamma})\|_{L^{2}(\mathbb{T})} \\
&\le C_{R}\|\chi_{R}(\rho, q) - \chi_{R}(\tilde{\rho}, \tilde{q})\|_{L^{2}(\mathbb{T})} + \|\rho^{\gamma} - \tilde{\rho}^{\gamma}\|_{L^{2}(\mathbb{T})} \\
&\le C_{R}\Big(\|\rho - \tilde{\rho}\|_{H^{2}(\mathbb{T})} + \|q - \tilde{q}\|_{H^{2}(\mathbb{T})}\Big),
\end{align*}
where we use \eqref{diff3} and the assumption that $1/R \le \rho \le C_{M}$ combined with the mean value theorem. For the first derivative, we estimate:
\begin{align*}
\Big\|\chi_{R}(\rho, q) &\rho^{\gamma - 1} \partial_{x}\rho - \chi_{R}(\tilde{\rho}, \tilde{q}) \tilde{\rho}^{\gamma - 1} \partial_{x}\tilde{\rho}\Big\|_{L^{2}(\mathbb{T})} \le |\chi_{R}(\rho, q) - \chi_{R}(\tilde{\rho}, \tilde{q})| \cdot \|\rho^{\gamma - 1}\|_{L^{\infty}(\mathbb{T})} \|\partial_{x}\rho\|_{L^{2}(\mathbb{T})} \\
&+ \chi_{R}(\tilde{\rho}, \tilde{q}) \|\rho^{\gamma - 1} - \tilde{\rho}^{\gamma - 1}\|_{L^{\infty}(\mathbb{T})} \|\partial_{x}\rho\|_{L^{2}(\mathbb{T})} + \|\chi_{R}(\tilde{\rho}, \tilde{q}) \tilde{\rho}^{\gamma - 1}\|_{L^{\infty}(\mathbb{T})} \|\partial_{x}(\rho - \tilde{\rho})\|_{L^{2}(\mathbb{T})}.
\end{align*}
We use \eqref{diff3}, the assumption that $1/R \le \rho \le C_{M}$, the mean value theorem, and the assumption that $\|\rho\|_{H^{2}(\mathbb{T})} \le M$ to conclude that
\begin{equation*}
\Big\|\chi_{R}(\rho, q) \rho^{\gamma - 1} \partial_{x}\rho - \chi_{R}(\tilde{\rho}, \tilde{q}) \tilde{\rho}^{\gamma - 1} \partial_{x}\tilde{\rho}\Big\|_{L^{2}(\mathbb{T})} \le C_{R, M}\Big(\|\rho - \tilde{\rho}\|_{H^{2}(\mathbb{T})} + \|q - \tilde{q}\|_{H^{2}(\mathbb{T})}\Big).
\end{equation*}
Finally, for the second derivative, we estimate two quantities. We first estimate:
\begin{align*}
\Big\|\chi_{R}(\rho, q) &\rho^{\gamma - 1} \partial_{x}^{2}\rho - \chi_{R}(\tilde{\rho}, \tilde{q}) \tilde{\rho}^{\gamma - 1} \partial_{x}^{2}\tilde{\rho}\Big\|_{L^{2}(\mathbb{T})} \le |\chi_{R}(\rho, q) - \chi_{R}(\tilde{\rho}, \tilde{q})| \cdot \|\rho^{\gamma - 1}\|_{L^{\infty}(\mathbb{T})} \|\partial_{x}^{2}\rho\|_{L^{2}(\mathbb{T})} \\
&+ \chi_{R}(\tilde{\rho}, \tilde{q}) \|\rho^{\gamma - 1} - \tilde{\rho}^{\gamma - 1}\|_{L^{\infty}(\mathbb{T})} \|\partial_{x}^{2}\rho\|_{L^{2}(\mathbb{T})} + \|\chi_{R}(\tilde{\rho}, \tilde{q}) \tilde{\rho}^{\gamma - 1}\|_{L^{\infty}(\mathbb{T})} \|\partial_{x}^{2}(\rho - \tilde{\rho})\|_{L^{2}(\mathbb{T})} \\
&\le C_{R, M} \Big(\|\rho - \tilde{\rho}\|_{H^{2}(\mathbb{T})} + \|q - \tilde{q}\|_{H^{2}(\mathbb{T})}\Big),
\end{align*}
similarly to the first derivative computation. We also estimate:
\begin{align*}
\Big\|\chi_{R}(\rho, q) &\rho^{\gamma - 2} (\partial_{x}\rho)^{2} - \chi_{R}(\tilde{\rho}, \tilde{q}) \tilde{\rho}^{\gamma - 2} (\partial_{x}\tilde{\rho})^{2}\Big\|_{L^{2}(\mathbb{T})} \le |\chi_{R}(\rho, q) - \chi_{R}(\tilde{\rho}, \tilde{q})| \cdot \|\rho^{\gamma - 2}\|_{L^{\infty}(\mathbb{T})} \|(\partial_{x}\rho)^{2}\|_{L^{2}(\mathbb{T})} \\
&+ \chi_{R}(\tilde{\rho}, \tilde{q}) \|\rho^{\gamma - 2} - \tilde{\rho}^{\gamma - 2}\|_{L^{\infty}(\mathbb{T})} \|(\partial_{x}\rho)^{2}\|_{L^{2}(\mathbb{T})} + \|\chi_{R}(\tilde{\rho}, \tilde{q}) \tilde{\rho}^{\gamma - 2}\|_{L^{\infty}(\mathbb{T})} \|(\partial_{x}\rho)^{2} - (\partial_{x}\tilde{\rho})^{2}\|_{L^{2}(\mathbb{T})} \\
&\le C_{R, M} \Big(\|\rho - \tilde{\rho}\|_{H^{2}(\mathbb{T})} + \|q - \tilde{q}\|_{H^{2}(\mathbb{T})}\Big).
\end{align*}
Here, we estimate using $\|\rho\|_{H^{2}(\mathbb{T})} \le M$ and $\|\tilde{\rho}\|_{H^{2}(\mathbb{T})} \le M$, along with Sobolev embedding, that 
\begin{equation*}
\|(\partial_{x}\rho)^{2}\|_{L^{2}(\mathbb{T})} \le \|\partial_{x}\rho\|_{L^{4}(\mathbb{T})}^{2} \le \|\rho\|_{H^{2}(\mathbb{T})}^{2} \le C_{M},
\end{equation*}
\begin{equation*}
\|(\partial_{x}\rho)^{2} - (\partial_{x}\tilde{\rho})^{2}\|_{L^{2}(\mathbb{T})} \le \|\partial_{x}\rho + \partial_{x}\tilde{\rho}\|_{L^{\infty}(\mathbb{T})} \cdot \|\partial_{x}(\rho - \tilde{\rho})\|_{L^{2}(\mathbb{T})} \le C_{R, M} \|\rho - \tilde{\rho}\|_{H^{2}(\mathbb{T})}.
\end{equation*}
This completes the proof of the last inequality.
\end{proof}

We also use the estimates in Lemma \ref{diffest} to prove estimates on the difference of the terms arising from the stochastic noise.

\begin{lemma}\label{noisedifflemma}
    Suppose that $\rho, \tilde{\rho} \in H^{2}(\mathbb{T})$ with $\rho \ge 1/R$, $\tilde{\rho} \ge 1/R$, and $\max\Big(\|\rho\|_{H^{2}(\mathbb{T})}, \|\tilde{\rho}\|_{H^{2}(\mathbb{T})}\Big) \le M$ for some positive constants $R, M > 0$. We then have the following estimates for $i = 0, 1, 2$, for a constant $C_{R, M, \epsilon}$ that depends only on $R$, $M$, and $\epsilon$ (and independent of $k$):
    \begin{multline}\label{noisediffk}
    \Bigg\|\chi_{R}(\rho,q) \partial_{x}^{i}\Big[G_{k}^{R, \epsilon}(\rho, q)\Big] - \chi_{R}(\tilde\rho,\tilde q)\partial_{x}^{i}\Big[G_{k}^{R, \epsilon}(\tilde{\rho}, \tilde{q})\Big]\Bigg\|_{L^{2}(\mathbb{T})} 
    \le C_{R, M, \epsilon} \alpha_{k} \Big(\|\rho - \tilde{\rho}\|_{H^{2}(\mathbb{T})} + \|q - \tilde{q}\|_{H^{2}(\mathbb{T})}\Big),
    \end{multline}
    for any positive integer $k$, where $\alpha_{k}$ is defined in \eqref{noiseassumption}. Hence, for $A_{0}$ defined in \eqref{A0}, we have that
    \begin{equation}\label{noisediffA0}
    \sum_{k = 1}^{\infty} \Bigg\|\chi_{R}(\rho, q) \partial_{x}^{i}\Big[G^{R, \epsilon}_{k}(\rho, q)\Big] - \chi_{R}(\tilde{\rho}, \tilde{q})\partial_{x}^{i}\Big[G^{R, \epsilon}_{k}(\tilde{\rho}, \tilde{q}\Big]\Bigg\|_{L^{2}(\mathbb{T})}^{2} \le C_{R, M, \epsilon} A_{0}\Big(\|\rho - \tilde{\rho}\|_{H^{2}(\mathbb{T})} + \|q - \tilde{q}\|_{H^{2}(\mathbb{T})}\Big).
    \end{equation}
\end{lemma}

\begin{proof}
    The second statement \eqref{noisediffA0} follows immediately from the first statement \eqref{noisediffk} using the definition of $A_0$ in \eqref{A0}. We hence prove the first statement \eqref{noisediffk}, and we focus on the case of $i = 2$, since the other cases of $i = 0, 1$ are easier variations of the same argument. By the Chain Rule:
    \begin{equation*}
    \partial_{x}^{2}\Big[G_{k}^{R, \epsilon}(\rho, q)\Big] = \nabla_{\rho, q}G_{k}^{R, \epsilon}(\rho, q) \cdot \partial_{x}^{2}(\rho, q) + \langle \nabla_{\rho, q}^{2}G_{k}^{R, \epsilon}(\rho, q) \partial_{x}(\rho, q), \partial_{x}(\rho, q)\rangle. 
    \end{equation*}
    So it suffices to estimate the following two terms:
    \begin{equation}\label{noiseterm1}
    \Big\|\chi_{R}(\rho, q) \nabla_{\rho, q} G_{k}^{R, \epsilon}(\rho, q) \cdot \partial_{x}^{2}(\rho, q) - \chi_{R}(\tilde{\rho}, \tilde{q}) \nabla_{\rho, q} G_{k}^{R, \epsilon}(\tilde{\rho}, \tilde{q}) \cdot \partial_{x}^{2}(\tilde{\rho}, \tilde{q})\Big\|_{L^{2}(\mathbb{T})},
    \end{equation}
    \begin{equation}\label{noiseterm2}
    \Big\|\chi_{R}(\rho, q) \langle \nabla_{\rho, q}^{2}G_{k}^{R, \epsilon}(\rho, q) \partial_{x}(\rho, q), \partial_{x}(\rho, q) \rangle - \chi_{R}(\tilde{\rho}, \tilde{q}) \langle \nabla_{\rho, q}^{2}G_{k}^{R, \epsilon}(\tilde{\rho}, \tilde{q}) \partial_{x}(\tilde{\rho}, \tilde{q}), \partial_{x}(\tilde{\rho}, \tilde{q}) \rangle\Big\|_{L^{2}(\mathbb{T})}.
    \end{equation}

    For the term \eqref{noiseterm1}, we estimate using \eqref{noiseRepsbound}, \eqref{diff5}, and \eqref{diff5b}:
    \begin{equation*}
    \Big\|\Big(\chi_{R}(\rho, q) \partial_{x}^{2}(\rho, q) - \chi_{R}(\tilde{\rho}, \tilde{q}) \partial^{2}_{x}(\tilde{\rho}, \tilde{q})\Big) \nabla_{\rho, q} G_{k}^{R, \epsilon}(\rho, q)\Big\|_{L^{2}(\mathbb{T})} \le C_{R, M, \epsilon} \alpha_{k} \Big(\|\rho - \tilde{\rho}\|_{H^{2}(\mathbb{T})} + \|q - \tilde{q}\|_{H^{2}(\mathbb{T})}\Big),
    \end{equation*}
    and we use \eqref{noiseRepsbound}, along with the fact that $\|\chi_{R}(\tilde{\rho}, \tilde{q}) (\rho, q)\|_{H^{2}(\mathbb{T})} \le C_{R, M}$ by the definition of the truncation and $\|\tilde{\rho}\|_{H^{2}(\mathbb{T})} \le M$, to conclude that
    \begin{align*}
    \Big\|\chi_{R}(\tilde{\rho}, \tilde{q})\partial_{x}^{2}(\tilde{\rho}, \tilde{q}) \Big(\nabla_{\rho, q}G_{k}^{R, \epsilon}(\rho, q) &- \nabla_{\rho, q} G_{k}^{R, \epsilon}(\tilde{\rho}, \tilde{q})\Big)\Big\|_{L^{2}(\mathbb{T})} \\
    &\le C_{R}\alpha_{k}\Big\|\chi_{R}(\tilde{\rho}, \tilde{q})\partial_{x}^{2}(\tilde{\rho}, \tilde{q})\Big\|_{L^{2}(\mathbb{T})} \Big(\|\rho - \tilde{\rho}\|_{L^{\infty}(\mathbb{T})} + \|q - \tilde{q}\|_{L^{\infty}(\mathbb{T})}\Big) \\
    &\le C_{R, M}\alpha_{k} \Big(\|\rho - \tilde{\rho}\|_{H^{2}(\mathbb{T})} + \|q - \tilde{q}\|_{H^{2}(\mathbb{T})}\Big).
    \end{align*}

    For the term \eqref{noiseterm2}, we use the quadratic estimates \eqref{diffquad} and the boundedness of the derivatives of the truncated noise coefficient \eqref{noiseRepsbound} to conclude that:
    \begin{multline*}
    \Big\|\chi_{R}(\rho, q) \langle \nabla^{2}_{\rho, q} G_{k}^{R, \epsilon}(\rho, q) \partial_{x}(\rho, q), \partial_{x}(\rho, q) - \chi_{R}(\tilde{\rho}, \tilde{q}) \langle \nabla^{2}_{\rho, q} G_{k}^{R, \epsilon}(\rho, q) \partial_{x}(\tilde{\rho}, \tilde{q}), \partial_{x}(\tilde{\rho}, \tilde{q}) \rangle \Big\|_{L^{2}(\mathbb{T})} \\
    \le C_{R, M, \epsilon} \alpha_{k} \Big(\|\rho - \tilde{\rho}\|_{H^{2}(\mathbb{T})} + \|q - \tilde{q}\|_{H^{2}(\mathbb{T})}\Big).
    \end{multline*}
    Furthermore, by a Lipschitz estimate on the truncated noise which follows from \eqref{noiseRepsbound}, and by the fact that $\Big\|\chi_{R}(\rho, q) |\partial_{x}(\rho, q)|^{2}\Big\|_{L^{2}(\mathbb{T})} \le C_{R, M}$ by Sobolev embedding, the definition of the truncation, and $\|\rho\|_{H^{2}(\mathbb{T})} \le M$, we conclude that
    \begin{equation*}
    \Big\|\chi(\tilde{\rho}, \tilde{q}) \langle \nabla^{2}_{\rho, q} G_{k}^{R, \epsilon}(\rho, q) - \nabla^{2}_{\rho, q} G_{k}^{R, \epsilon}(\tilde{\rho}, \tilde{q}) \partial_{x}(\tilde{\rho}, \tilde{q}), \partial_{x}(\tilde{\rho}, \tilde{q}) \rangle\Big\|_{L^{2}(\mathbb{T})} \le C_{R, M, \epsilon} \alpha_{k} \Big(\|\rho - \tilde{\rho}\|_{H^{2}(\mathbb{T})} + \|q - \tilde{q}\|_{H^{2}(\mathbb{T})}\Big).
    \end{equation*}
\end{proof}

Now, we proceed with the a priori estimates for continuous dependence, and we will often use the inequalities in Lemma \ref{nonlineardiff} and Lemma \ref{noisedifflemma} in these estimates. For these a priori estimates, we consider initial data $(\rho_0, q_0), (\tilde{\rho}_{0}, \tilde{q}_{0}) \in \mathcal{X}$ (see \eqref{path}), and we solve the approximate system \eqref{approxsystem} with this initial data to obtain corresponding unique solutions $(\rho, q)$ and $(\tilde{\rho}, \tilde{q})$ that are both in $C(0, T; \mathcal{X})$ almost surely. We want to estimate difference between these two solutions $(\rho - \tilde{\rho}, q - \tilde{q})$ in terms of the difference of the initial data. These a priori estimates will be the content of the proof of Proposition \ref{continuous}.

\begin{proof}[Proof of Proposition \ref{continuous}]
The goal of the a priori estimates is to obtain a Gronwall-type inequality for the quantity 
\begin{equation*}
\mathbb{E}\Big(\|\rho(t) - \tilde{\rho}(t)\|^2_{H^{2}(\mathbb{T})} + \|q(t) - \tilde{q}(t)\|^2_{H^{2}(\mathbb{T})}\Big).
\end{equation*}
We recall from Lemma \ref{rhoH2} and Proposition \ref{rholowerbound} that the solutions $(\rho, q)$ and $(\tilde{\rho}, \tilde{q})$ satisfy the following almost sure bounds:
\begin{equation*}
\|\rho\|_{C(0, T; H^{2}(\mathbb{T}))} \le C_{R, T}, \quad \|\tilde{\rho}\|_{C(0, T; H^{2}(\mathbb{T}))} \le C_{R, T}, \quad \rho(t, \cdot) \ge 1/R, \quad \tilde{\rho}(t, \cdot) \ge 1/R.
\end{equation*}
This will allow us to apply Lemma \ref{nonlineardiff} and Lemma \ref{noisedifflemma}, where these types of bounds are required as assumptions in order to derive the estimates. 

We obtain the following equations for the difference $(\rho - \tilde{\rho}, q - \tilde{q})$:
\begin{equation}\label{firstdiff0}
    \partial_{t}(\rho - \tilde{\rho}) + \partial_{x}([q]_{R} - [\tilde{q}]_{R}) = \epsilon \Delta (\rho - \tilde{\rho}),
\end{equation}
\begin{multline}\label{seconddiff0}
    \partial_{t}(q - \tilde{q}) + \partial_{x}\left(\frac{[q]_{R}q}{\rho} - \frac{[\tilde{q}]_{R}\tilde{q}}{\tilde{\rho}}\right) + \chi_{R}(\rho, q)\partial_{x}(\kappa \rho^{\gamma}) - \chi_{R}(\tilde{\rho}, \tilde{q})\partial_{x}(\kappa \tilde{\rho}^{\gamma}) \\
    = \Big(\chi_{R}(\rho, q) \bd{\Phi}^{R, \epsilon}(\rho, q) - \chi_{R}(\tilde{\rho}, \tilde{q}) \bd{\Phi}^{R, \epsilon}(\tilde{\rho}, \tilde{q})\Big) dW - \alpha(q - \tilde{q}) + \epsilon\Delta(q - \tilde{q}). 
\end{multline}

We differentiate each equation \eqref{firstdiff0} and \eqref{seconddiff0} by $\partial_{x}^{j}$ for $j = 0, 1, 2$, and then use the estimates in Lemma \ref{nonlineardiff} to prove the continuous dependence result in Proposition \ref{continuous} via Gronwall's inequality. From the first equation \eqref{firstdiff0}, we obtain:
\begin{equation}\label{testdiff1}
\int_{\mathbb{T}} |\partial_{x}^{j}(\rho - \tilde{\rho})(t)|^{2} dx + \epsilon \int_{0}^{t} \int_{\mathbb{T}} |\partial^{j + 1}_{x} (\rho - \tilde{\rho})|^{2} = \int_{\mathbb{T}} |\partial_{x}^{j}(\rho_0 - \tilde{\rho}_{0})|^{2} + \int_{0}^{t} \int_{\mathbb{T}} \partial_{x}^{j} \Big([q]_{R} - [\tilde{q}]_{R}\Big) \partial_{x}^{j + 1} (\rho - \tilde{\rho}).
\end{equation}
Using It\"{o}'s formula in \eqref{seconddiff0}, we obtain:
\begin{align}\label{testdiff2}
\int_{\mathbb{T}} |\partial_{x}^{j}(q - \tilde{q})(t)|^{2} &+ \alpha \int_{0}^{t} \int_{\mathbb{T}} |\partial^{j}_{x}(q - \tilde{q})|^{2} + \epsilon \int_{0}^{t} \int_{\mathbb{T}} |\partial^{j + 1}_{x}(q - \tilde{q})|^{2} = \int_{0}^{t} \int_{\mathbb{T}} \partial^{j}_{x}\left(\frac{[q]_{R}q}{\rho} - \frac{[\tilde{q}]_{R}\tilde{q}}{\tilde{\rho}}\right) \partial^{j + 1}_{x} (q - \tilde{q}) \nonumber \\
&+ \int_{0}^{t} \int_{\mathbb{T}} \Big(\chi_{R}(\rho, q) \partial_{x}^{j}(\kappa \rho^{\gamma}) - \chi_{R}(\tilde{\rho}, \tilde{q}) \partial_{x}^{j} (\kappa \tilde{\rho}^{\gamma})\Big) \partial_{x}^{j + 1}(q - \tilde{q}) \nonumber \\
&+ \int_{0}^{t} \int_{\mathbb{T}} \partial_{x}^{j}\Big(\chi_{R}(\rho, q) \bd{\Phi}^{R, \epsilon}(\rho, q) - \chi_{R}(\tilde{\rho}, \tilde{q}) \bd{\Phi}^{R, \epsilon}(\tilde{\rho}, \tilde{q})\Big) \partial_{x}^{j}(q - \tilde{q}) dW \nonumber \\
&+ \frac{1}{2} \int_{0}^{t} \sum_{k = 1}^{\infty} \left(\int_{\mathbb{T}} \partial_{x}^{j}\Big(\chi_{R}(\rho, q) G_{k}^{R, \epsilon}(\rho, q) - \chi_{R}(\tilde{\rho}, \tilde{q}) G_{k}^{R, \epsilon}(\tilde{\rho}, \tilde{q})\Big)\right)^{2}.
\end{align}
By using Lemma \ref{nonlineardiff}, we estimate for \eqref{testdiff1} that
\begin{align}\label{continuouscalc1}
\left|\int_{0}^{t} \int_{\mathbb{T}} \partial_{x}^{j}\Big([q]_{R} - [\tilde{q}]_{R}\Big) \partial^{j + 1}_{x}(\rho - \tilde{\rho})\right| &\le \frac{\epsilon}{2} \int_{0}^{t} \int_{\mathbb{T}} |\partial^{j + 1}_{x}(\rho - \tilde{\rho})|^{2} + C_{\epsilon, R, T} \int_{0}^{t} \Big\|[q]_{R} - [\tilde{q}]_{R}\Big\|^{2}_{H^{2}(\mathbb{T})} \\
&\le \frac{\epsilon}{2} \int_{0}^{t} \int_{\mathbb{T}} |\partial^{j + 1}_{x}(\rho - \tilde{\rho})|^{2} + C_{\epsilon, R, T} \int_{0}^{t} \Big(\|\rho - \tilde{\rho}\|^{2}_{H^{2}(\mathbb{T})} + \|q - \tilde{q}\|^{2}_{H^{2}(\mathbb{T})}\Big) \nonumber,
\end{align}
where by using Cauchy's inequality (with $\epsilon$), we can absorb the term $\displaystyle \frac{\epsilon}{2} \int_{0}^{t} \int_{\mathbb{T}} |\partial^{j + 1}_{x}(\rho - \tilde{\rho})|^{2}$ into the dissipation term on the left-hand side of \eqref{testdiff1}. For \eqref{testdiff2}, we note that by taking \textit{expectation} of both sides, the stochastic integral has expectation zero and hence vanishes from the computation. For the first two nonlinear difference terms on the right-hand side of \eqref{testdiff2}, we use Cauchy with epsilon along with the estimates in Lemma \ref{nonlineardiff} to estimate:
\begin{multline}\label{continuouscalc2}
\left|\mathbb{E} \int_{0}^{t} \int_{\mathbb{T}} \partial_{x}^{j}\left(\frac{[q]_{R}q}{\rho} - \frac{[\tilde{q}]_{R}\tilde{q}}{\tilde{\rho}}\right) \partial^{j + 1}_{x}(q - \tilde{q}) + \mathbb{E} \int_{0}^{t} \int_{\mathbb{T}} \Big(\chi_{R}(\rho, q) \partial_{x}^{j}(\kappa \rho^{\gamma}) - \chi_{R}(\tilde{\rho}, \tilde{q})\partial_{x}^{j}(\kappa \tilde{\rho}^{\gamma})\Big) \partial^{j + 1}_{x}(q - \tilde{q})\right| \\
\le \frac{\epsilon}{2} \mathbb{E} \int_{0}^{t} \int_{\mathbb{T}} |\partial^{j + 1}_{x}(q - \tilde{q})|^{2} + C_{\epsilon, R, T} \mathbb{E} \int_{0}^{t} \Big(\|\rho - \tilde{\rho}\|^{2}_{H^{2}(\mathbb{T})} + \|q - \tilde{q}\|^{2}_{H^{2}(\mathbb{T})}\Big).
\end{multline}
We can also use the estimate in Lemma \ref{noisedifflemma} to estimate the quadratic variation term:
\begin{equation}\label{continuouscalc3}
\mathbb{E} \int_{0}^{t} \sum_{k = 1}^{\infty} \int_{\mathbb{T}} \Big|\partial^{j}_{x}\Big(\chi_{R}(\rho, q) G_{k}^{R, \epsilon}(\rho, q) - \chi_{R}(\tilde{\rho}, \tilde{q}) G_{k}^{R, \epsilon}(\tilde{\rho}, \tilde{q})\Big)\Big|^{2} \le C_{\epsilon, R, T} A_{0} \mathbb{E} \int_{0}^{t} \Big(\|\rho - \tilde{\rho}\|_{H^{2}(\mathbb{T})}^{2} + \|q - \tilde{q}\|_{H^{2}(\mathbb{T})}^{2}\Big) ds.
\end{equation}
By taking expectations in \eqref{firstdiff0} and \eqref{seconddiff0} and adding over $j = 0, 1, 2$, and then applying the estimates \eqref{continuouscalc1}, \eqref{continuouscalc2}, and \eqref{continuouscalc3}, we obtain the desired estimate by absorbing terms into the dissipation term on the left-hand side:
\begin{multline}\label{gronwallcontinuous}
\mathbb{E} \Big(\|(\rho - \tilde{\rho})(t)\|^{2}_{H^{2}(\mathbb{T})} + \|(q - \tilde{q})(t)\|^{2}_{H^{2}(\mathbb{T})}\Big) + \frac{\epsilon}{2} \mathbb{E} \int_{0}^{t} \Big(\|\partial_{x}(\rho - \tilde{\rho})(s)\|_{H^{2}(\mathbb{T})}^{2} + \|\partial_{x}(q - \tilde{q})(s)\|_{H^{2}(\mathbb{T})}^{2}\Big) \\
\le \Big(\|\rho_{0} - \tilde{\rho}_{0}\|_{H^{2}(\mathbb{T})}^{2} + \|q_{0} - \tilde{q}_{0}\|_{H^{2}(\mathbb{T})}^{2}\Big) + C_{\epsilon, R, T} \mathbb{E} \int_{0}^{t} \Big(\|(\rho - \tilde{\rho})(s)\|_{H^{2}(\mathbb{T})}^{2} + \|(q - \tilde{q})(s)\|_{H^{2}(\mathbb{T})}\Big) ds.
\end{multline}
The result then follows from an application of Gronwall's inequality.
\end{proof}

\subsection{Feller semigroup for the approximate system.}\label{fellersemigroup} Since the approximate system \eqref{approxsystem} has a notion of Hadamard well-posedness, we claim that we can define an associated \textit{Feller semigroup} $\{\mathcal{P}_{t}\}_{t \ge 0}$ to the evolution of the approximate system. Recall the definition of the phase space $\mathcal{X}$ from \eqref{path}:
\begin{equation*}
\mathcal{X} := \left\{(\rho, q) \in H^{2}(\mathbb{T}) \times H^{2}(\mathbb{T}) : \int_{\mathbb{T}} \rho(x) dx = 1 \text{ and } \rho \ge \frac{1}{R}\right\}.
\end{equation*}
Consider any deterministic state $(\rho_0, q_0) \in \mathcal{X}$. Note that by the global existence and uniqueness result in Proposition \ref{continuous} for the approximate system, we can define the evolution
\begin{equation*}
\Big(\rho^{R,\epsilon}(t; (\rho_0, q_0)), q^{R,\epsilon}(t; (\rho_0, q_0))\Big) \in \mathcal{X} \qquad \text{ for } t \ge 0,
\end{equation*}
which we define to be the (random) solution at time $t$ to the approximate system \eqref{approxsystem} with initial data $(\rho_0, q_0)$. 

Let $C_{b}(\mathcal{X})$ denote the space of bounded continuous functions $\varphi: \mathcal{X} \to \mathbb{R}$. Then, for each time $t \ge 0$, we can define an operator $\mathcal{P}_{t}$ acting on functions in $C_{b}(\mathcal{X})$, defined by
\begin{equation}\label{ptdef}
(\mathcal{P}_{t}\varphi)(\rho_0, q_0) = \mathbb{E}\Big[\varphi\Big(\rho^{R,\epsilon}(t; (\rho_0, q_0)), q^{R,\epsilon}(t; (\rho_0, q_0))\Big)\Big].
\end{equation}

We claim that $\{\mathcal{P}_{t}\}_{t \ge 0}$ is a Feller semigroup, namely that it has the properties listed in the following proposition.

\begin{proposition}\label{fellerprop}
The collection $\{\mathcal{P}_{t}\}_{t \ge 0}$ is a Feller semigroup in the sense that $\mathcal{P}_{t}: C_{b}(\mathcal{X}) \to C_{b}(\mathcal{X})$ for each $t \ge 0$, and for all $s, t \ge 0$:
\begin{equation*}
\mathcal{P}_{0}\varphi = \varphi \ \ \text{ and } \ \ \mathcal{P}_{s}(\mathcal{P}_{t}\varphi) = \mathcal{P}_{s + t}\varphi, \qquad \text{ for all } \varphi \in C_{b}(\mathcal{X}),
\end{equation*}
\end{proposition}

\begin{proof}
    It is immediate that $\mathcal{P}_{0} = \text{Id}$, and the semigroup property $\mathcal{P}_{s} \circ \mathcal{P}_{t} = \mathcal{P}_{s + t}$ follows from the uniqueness result in Proposition \ref{existunique}. The fact that $\mathcal{P}_{t}$ is a bounded continous operator on $C_{b}(\mathcal{X})$, namely that $\mathcal{P}_{t}\varphi \in C_{b}(\mathcal{X})$ for all $\varphi \in C_{b}(\mathcal{X})$, is a direct consequence of the continuous dependence property in Proposition \ref{continuous}.
\end{proof}

\section{Results on uniform invariant regions}\label{uniforminvariant}
Now that we have a Feller semigroup $\{\mathcal{P}_{t}\}_{t \ge 0}$ for the approximate system \eqref{approxsystem}, we will use a time-averaging procedure to obtain an invariant measure (and hence statistically stationary solution) on $\mathcal{X}$ for the approximate system. To carry out this time averaging procedure, we need uniform estimates on the approximate system in time for a general (stochastic) solution $(\rho^{R,\epsilon}(t), q^{R,\epsilon}(t))$ to the initial value problem in \eqref{approxsystem}. This will be the content of the current section and also Section \ref{uniformsection}.

In this section, we will deduce uniform in time $L^{\infty}$ bounds on the solution $(\rho^{R,\epsilon}(t), q^{R,\epsilon}(t))$ to the initial value problem \eqref{approxsystem}, using the structure of \textbf{invariant regions} for the isentropic Euler equations. In particular, we note that the structure of the truncations we use in the approximate system \eqref{approxsystem} preserve the structure of invariant regions to the damped compressible Euler equations, which we can use to obtain uniform $L^{\infty}$ bounds independently of $T$ and $R$. 

It is well-known that the (undamped) deterministic compressible isentropic Euler equations in one spatial dimension with artificial viscosity
\begin{equation}\label{undamped}
\begin{cases}
    \partial_{t}\rho + \partial_{x}q = \epsilon \Delta \rho, \\
    \partial_{t}q + \partial_{x}\left(\frac{q^{2}}{\rho}\right) + \partial_{x}(\kappa \rho^{\gamma}) = \epsilon \Delta q.
\end{cases}
\end{equation}
possess the Riemann invariants
\begin{equation*}
z = u - \rho^{\theta}, \qquad w = u + \rho^{\theta},
\end{equation*}
and an associated invariant region
\begin{equation*}
\Lambda_{\kappa} = \{(\rho, u) \in (0, \infty) \times \R : -\kappa \le z \le w \le \kappa\}.
\end{equation*}
This region is invariant in the sense that any classical solution to \eqref{undamped} with initial data $(\rho_0, q_0) \in \Lambda_{\kappa}$ for some $\kappa > 0$ will remain within $\Lambda_{\kappa}$ for all times $t \ge 0$. This is the result of \cite{CCS77} (see also \cite{DP832}), which essentially involves a change of variables via the Riemann invariants with a maximum/minimum principle type argument. 

In this section, we consider the following deterministic \textit{truncated} system with damping, which differs from the classical scenario in \eqref{undamped} and gives the approximate problem we are considering in \eqref{approxsystem} when $\delta = 0$:
\begin{equation}\label{dampedtruncate}
\begin{cases}
    \partial_{t}\rho +  \chi_{R}(\rho,q) 
    \partial_{x}q = \epsilon \Delta \rho, \\
    \partial_{t}q + \chi_{R}(\rho,q)\partial_{x}\left(\frac{q^{2}}{\rho} + \kappa \rho^{\gamma} \right) = -\alpha q + \epsilon \Delta q.
\end{cases}
\end{equation}
We claim that this system has the same invariant region $\Lambda_{\kappa}$ (for arbitrary $\kappa > 0$) \textit{independently of the truncation parameter $R$}. Moreover, the fact that the truncated equations have the same invariant region will then give us the following uniform bound on the classical solutions $(\rho^{R,\epsilon}(t), q^{R,\epsilon}(t))$ to the initial value problem for initial data $(\rho_0, q_0) \in \mathcal{X}$, that are constructed in Proposition \ref{existunique}.

\begin{proposition}\label{Linfuniform1}
    For initial data $(\rho_0, q_0) \in \mathcal{X}$, the classical solution $(\rho^{R,\epsilon}, q^{R,\epsilon})$ to \eqref{approxsystem} in Proposition \ref{existunique} satisfies the following uniform bounds:
    \begin{equation*}
    \left\|\left(\rho^{R,\epsilon}, q^{R,\epsilon}, \frac{q^{R,\epsilon}}{\rho^{R,\epsilon}}\right)\right\|_{ C(0,T;C(\mathbb{T}))} \le C_\epsilon \quad \text{ almost surely. }
    \end{equation*}
    for a deterministic constant $C_\epsilon$ that is independent of $T$ and $R$ but dependent on $\epsilon$.
\end{proposition}

This proposition will be a consequence of the following result on invariant regions for the deterministic damped truncated system in \eqref{dampedtruncate}. (Note that even though the approximate system is stochastic, due to the regularization of the noise in Section \ref{noiseregularize} at the $\epsilon$ level, the $\epsilon$-regularized noise is compactly supported and hence also respects the invariant region structure of the damped Euler equations, see Proposition \ref{stochinv}.) 

\begin{proposition}\label{invariantregion}
The region
\begin{equation}\label{Lambdakappa}
\Lambda_{\kappa} = \{(\rho, q) \in (0, \infty) \times \R : -\kappa \le z \le w \le \kappa\}, \qquad \text{ for } z = \frac{q}{\rho} - \rho^{\theta} \text{ and } w = \frac{q}{\rho} + \rho^{\theta}
\end{equation}
is an invariant region for \eqref{dampedtruncate} in the sense that for any spatially smooth initial data $(\rho_0, q_0)$ such that $(\rho_0(x), q_0(x)) \in \Lambda_{\kappa}$ for all $x \in \mathbb{T}$, the unique global smooth solution with initial data $(\rho_0, q_0)$ has $(\rho(t, x), q(t, x)) \in \Lambda_{\kappa}$ for all $t \ge 0$ and $x \in \mathbb{T}$. 
\end{proposition}

The proof of this proposition is essentially a minimum/maximum principle type argument, but with a nonlinear transformation given by the Riemann invariants. As is customary in some minimum and maximum principle arguments, to have the strict inequalities required for such arguments, we use a perturbation by a parameter $\delta > 0$ to help with the proof of the result, and then pass to the limit as $\delta \to 0$ to obtain a result for the original set of equations. Namely,  we consider
\begin{equation}\label{perturbed}
\begin{cases}
\partial_{t}\rho + \chi_{R}(\rho,q) \partial_{x}q = - \delta \rho + \epsilon \Delta \rho, \\
\partial_{t}q + \chi_{R}(\rho,q) \partial_{x}\Big(\frac{q^{2}}{\rho} + \kappa \rho^{\gamma}\Big) = -\alpha q + \epsilon \Delta q.
\end{cases}
\end{equation}
\iffalse
{\cred
observe that  if $(\rho,q)$ satisfies \eqref{dampedtruncate} then $(\bar\rho,\bar q)=e^{-\delta t}(\rho,q)$ satisfies %{\cred this isn't exactly what we were looking for}
\begin{equation}%\label{perturbed}
\begin{cases}
\partial_{t}\bar\rho + \chi_{R}(\bar\rho e^{\delta t},\bar q e^{\delta t}) \partial_{x}\bar q = - \delta \bar \rho + \epsilon \Delta \bar \rho, \\
\partial_{t}\bar q + \chi_{R}(\bar\rho e^{\delta t},\bar q e^{\delta t}) \partial_{x}\Big(\frac{\bar q^{2}}{\bar \rho} + \kappa e^{(\gamma-1)\delta t}{\bar \rho}^{\gamma}\Big) = -(\alpha +\delta) \bar q + \epsilon \Delta \bar q.
\end{cases}
\end{equation}
}
\fi 

We hence first show the following invariant region result for the perturbed $\delta$-system \eqref{perturbed} and then pass to the limit as $\delta \to 0$ to prove Proposition \ref{invariantregion}.
This is done in the spirit of the invariant region results in \cite{CCS77}.

\begin{lemma}\label{perturbedinvariant}
The region $\Lambda_{\kappa}$ is an invariant region for \eqref{perturbed}, whenever $0 < \delta < \alpha$.
\end{lemma}

\begin{proof}[Proof of Lemma \ref{perturbedinvariant}]

    We rewrite the system \eqref{perturbed} as
    \begin{equation}\label{quasitruncate}
    \partial_{t}\bd{U} + \chi_R(\rho,q) \bd{F}'(\bd{U}) \partial_{x}\bd{U} = \epsilon \Delta \bd{U} + \bd{G}_{\delta}(\bd{U}) \text{ on } \R^{+} \times \mathbb{T}, \qquad \bd{U}(0) = \bd{U}_{0} := \begin{pmatrix} \rho_0 \\ q_0 \\ \end{pmatrix},
    \end{equation}
    where
    \begin{equation*}
    \bd{U} := \begin{pmatrix} \rho \\ q \\ \end{pmatrix}, \quad \bd{F}(\bd{U}) = \begin{pmatrix} q \\ \frac{q^{2}}{\rho} + \kappa \rho^{\gamma} \end{pmatrix}, \quad \bd{F}'(\bd{U}) = \begin{pmatrix} 0 & 1 \\ -\frac{q^{2}}{\rho^{2}} + \kappa \gamma \rho^{\gamma - 1} & \frac{2q}{\rho} \\ \end{pmatrix}, \quad \bd{G}_{\delta}(\bd{U}) := \begin{pmatrix} -\delta \rho \\ -\alpha q \\ \end{pmatrix}.
    \end{equation*}
    In terms of the state variables $(\rho, q)$, we can rewrite the Riemann invariants as
    \begin{equation}\label{rinv}
    z = \frac{q}{\rho} - \rho^{\theta}, \qquad w = \frac{q}{\rho} + \rho^{\theta}.
    \end{equation}
    Hence, we compute that
    \begin{equation*}
    \nabla_{\rho, q}z = \left(-\frac{q}{\rho^{2}} - \theta \rho^{\theta - 1}, \frac{1}{\rho}\right), \qquad \nabla_{\rho, q}w = \left(-\frac{q}{\rho^{2}} + \theta \rho^{\theta - 1}, \frac{1}{\rho}\right)
    \end{equation*}
    and we observe that $\nabla_{\rho, q} z$ and $\nabla_{\rho, q} w$ are left eigenvectors of the matrix $\bd{F}'(\bd{U})$ with eigenvalues $\displaystyle \lambda_{z} = \frac{q}{\rho} - \theta \rho^{\theta}$ and $\displaystyle \lambda_{w} = \frac{q}{\rho} + \theta \rho^{\theta}$ respectively.

    \medskip

    We now prove the desired claim about invariant regions by using the truncated system written in quasilinear form \eqref{quasitruncate}. Consider initial data $(\rho_0, q_0) \in \Lambda_{\kappa}$, defined in \eqref{Lambdakappa}, and let $\bd{U}_\delta=(\rho_\delta(t, x), q_\delta(t, x))$ be the solution to \eqref{quasitruncate}. Suppose that there exists $t_{0} > 0$ such that 
    \begin{equation*}
    z(\rho_\delta(t_0, x_0), q_\delta(t_0, x_0)) = -\kappa \ \ \text{ or } \ \  w(\rho_\delta(t_0, x_0), q_\delta(t_0, x_0)) = \kappa, \quad \text{ for some $x_0 \in \mathbb{T}$},
    \end{equation*}
    and
    \begin{equation}\label{zw}
    -\kappa \le z(\rho_\delta(t, x), q_\delta(t, x)) \le w(\rho_\delta(t, x), q_\delta(t, x)) \le \kappa, \qquad \text{ for all } 0 \le t < t_0 \text{ and } x \in \mathbb{T}.
    \end{equation}
    We claim that
    \begin{equation*}
    \begin{cases}
    \frac{\partial}{\partial t}\Big(z(\rho_\delta(t, x), q_\delta(t, x))\Big)\Big|_{(t, x) = (t_0, x_0)} > 0 \text{ in the case where } z(\rho_\delta(t_0, x_0), q_\delta(t_0, x_0)) = -\kappa, \\
    \frac{\partial}{\partial t} \Big(w(\rho_\delta(t, x), q_\delta(t, x))\Big)\Big|_{(t, x) = (t_0, x_0)} < 0 \text{ in the case where $w(\rho_\delta(t_0, x_0), q_\delta(t_0, x_0)) = \kappa$}.
    \end{cases}
    \end{equation*}

    To prove this, we consider the case where 
    \begin{equation}\label{wkappa}
    w(\rho_\delta(t_0, x_0), q_\delta(t_0, x_0)) = \kappa
    \end{equation}
    with the objective of showing that $\displaystyle \frac{\partial}{\partial t} w\Big(\rho_\delta(t_0, x_0), q_\delta(t_0, x_0)\Big) < 0$, as the case of showing the claim for $z$ follows analogously. We use \eqref{quasitruncate} and the Chain Rule to compute that
    \begin{equation*}
    \frac{\partial}{\partial t}w(\bd{U}_\delta) = \nabla_{\rho, q} w(\bd{U}_\delta) \Big(\epsilon \Delta \bd{U}_\delta + \bd{G}_{\delta}(\bd{U}_\delta) - \chi_R(\rho_\delta,q_\delta ) \bd{F}'(\bd{U}_\delta) \partial_{x}\bd{U}_\delta\Big).
    \end{equation*}
    We will perform a sign analysis on the various terms on the right-hand side, evaluated at the point $(t_0, x_0)$, in order to show the desired result.

    \medskip

    \noindent \textbf{Term 1.} It can be shown (see Section 4 in \cite{DP832}) that $w$ is \textit{quasiconvex}, meaning that
    \begin{equation*}
    \nabla_{\rho, q} w \cdot \partial_{x}\bd{U}_\delta|_{(t, x)} = 0 \text{ at some point $(t, x)$} \quad \Longrightarrow \quad \langle (\nabla^{2}_{\rho, q}w)\partial_{x}\bd{U}_\delta, \partial_{x}\bd{U}_\delta \rangle|_{(t, x)} \ge 0.
    \end{equation*}
    Recall that there is a local maximum in space of $w(\rho_\delta(t, x), q_\delta(t, x))$ at $(t_0, x_0)$ by assumption in \eqref{zw} and \eqref{wkappa}. So by considering the quadratic term in the Taylor expansion in the $x$ variable around $x_0$ (for fixed $t_0$):
    \begin{equation*}
    \nabla_{\rho, q}w(\bd{U}_\delta) \cdot \partial_{x}^{2}\bd{U}_\delta + \langle (\nabla^{2}_{\rho, q} w)\partial_{x}\bd{U}_\delta, \partial_{x}\bd{U}_\delta \rangle|_{(t_0, x_0)} \le 0.
    \end{equation*}
    So the quasiconvexity of $w$ implies that
    \begin{equation*}
    \epsilon \nabla_{\rho, q} w(\bd{U}_\delta) \cdot \partial_{x}^{2}\bd{U}_\delta = \epsilon \nabla_{\rho, q}w(\bd{U}_\delta) \cdot \Delta \bd{U}_\delta |_{(t_0, x_0)} \le 0.
    \end{equation*}

    \medskip

    \noindent \textbf{Term 2.} We compute that
    \begin{equation*}
    \nabla_{\rho, q} w(\bd{U}_\delta) \bd{G}_{\delta}(\bd{U}_\delta) = -(\alpha - \delta) \frac{q_\delta}{\rho_\delta} - \delta \theta \rho^{\theta}_\delta. 
    \end{equation*}
    Since $w(\rho_\delta(t_0, x_0), q_\delta(t_0, x_0)) = \kappa$ for $w = \frac{q_\delta}{\rho_\delta} + \rho_\delta^{\theta}$ and $-\kappa \le z(\rho_\delta(t, x), q_\delta(t, x)) \le w(\rho_\delta(t, x), q_\delta(t, x)) \le \kappa$ for all $0 \le t < t_0$ by assumption in \eqref{zw} and \eqref{wkappa}, we conclude by the geometry of the invariant region $\Lambda_{\kappa}$ that $0 < \rho_\delta(t_0, x_0) \le \kappa^{1/\theta}$ and $q_\delta(t_0, x_0) \ge 0$ by solving the inequalities for $z$ and $w$, using \eqref{rinv}. 
    \iffalse 
    Note that if $(\rho_\delta,q_\delta)$ satisfies \eqref{perturbed}, then $(\bar\rho_\delta,\bar q_\delta)=e^{-\delta t}(\rho_\delta,q_\delta)$ satisfies 
$$\partial_{t}\bar\rho_\delta + \chi_{R}(\bar\rho_\delta e^{\delta t},\bar q_\delta e^{\delta t}) \partial_{x}\bar q_\delta = - \delta \bar \rho_\delta + \epsilon \Delta \bar \rho_\delta. $$
 An easy generalization of the result in Proposition \ref{rholowerbound} via a minimum principle argument gives us $\frac{e^{-\delta t}}{R+1} \leq \bar\rho_\delta(t,x)$ a.s. {almost surely for all} $t\in [0,T], x\in\mathbb{T}$. Hence we obtain:
 \begin{align}%\label{uniformrho_inv}
   0<  \frac{1}{R+1} \leq \rho_\delta(t,x),\qquad \text{almost surely for all }t\in [0,T], x\in\mathbb{T},
\end{align}
\fi

Next,  note that if $(\rho_\delta,q_\delta)$ satisfies \eqref{perturbed}, then $(\bar\rho_\delta,\bar q_\delta)=e^{\delta t}(\rho_\delta,q_\delta)$ satisfies 
$$\partial_{t}\bar\rho_\delta + \chi_{R}(\bar\rho_\delta e^{-\delta t},\bar q_\delta e^{-\delta t}) \partial_{x}\bar q_\delta =\epsilon \Delta \bar \rho_\delta. $$

 An easy generalization of the result in Proposition \ref{rholowerbound} via a minimum principle argument gives us $\frac{1}{R} \leq \bar\rho_\delta(t,x)$ a.s. {almost surely for all} $t\in [0,T], x\in\mathbb{T}$. Hence we obtain:
 \begin{align}\label{uniformrho_inv}
   0<  \frac{e^{-\delta t}}{R} \leq \rho_\delta(t,x),\qquad \text{almost surely for all }t\in [0,T], x\in\mathbb{T},
\end{align}
    
  This gives us that $\rho_\delta(t_0, x_0)$ cannot be zero, i.e. the approximate system \eqref{perturbed} does not have vacuum.

    Hence, for $\delta > 0$ sufficiently small (namely $0 < \delta < \alpha$), we have that
    \begin{equation*}
    \nabla_{\rho, q} w(\bd{U}_\delta) \bd{G}_{\delta}(\bd{U}_\delta)\Big\vert_{(t_0, x_0)} < 0.
    \end{equation*}
    Note that this is where we needed the extra $\delta$ approximation, to get this derivative to be strictly negative (otherwise, if $\delta = 0$, it could be zero at $\rho_\delta(t_0, x_0) = \kappa^{1/\theta}$ and $q_\delta = 0$).

    \medskip

    \noindent \textbf{Term 3.} Since $\nabla_{\rho, q}w$ is a left eigenvector of $\bd{F}'(\bd{U}_\delta)$ with eigenvector $\displaystyle \lambda_{w} = \frac{q_\delta}{\rho_\delta} + \theta \rho_\delta^{\theta}$, we compute that
    \begin{equation*}
    -\chi_R(\rho,q)\nabla_{\rho, q}w(\bd{U}_\delta) \bd{F}'(\bd{U}_\delta) \partial_{x}\bd{U}_\delta = -\lambda_{w}(\bd{U}_\delta) \chi_R(\rho_\delta,q_\delta) \nabla_{\rho, q} w(\bd{U}_\delta) \partial_{x}\bd{U}_\delta.
    \end{equation*}
    At $(t_0, x_0)$, $w(\rho_\delta(t_0, x), q_\delta(t_0, x))$ has a local maximum in space by the assumptions \eqref{zw} and \eqref{wkappa}. So by the Chain Rule, $\nabla_{\rho, q}w(\bd{U}) \partial_{x}\bd{U}_\delta\Big\vert_{(t_0, x_0)} = 0$. Hence,
    \begin{equation*}
    -\chi_R(\rho_\delta,q_\delta)\nabla_{\rho, q} w(\bd{U}_\delta) \bd{F}'(\bd{U}_\delta) \partial_{x}\bd{U}_\delta\Big\vert_{(t_0, x_0)} = 0.
    \end{equation*}

    This establishes the desired result that $\displaystyle \frac{\partial}{\partial t} \Big(w(\rho_\delta(t, x), q_\delta(t, x))\Big)\Big\vert_{(t, x) = (t_0, x_0)} < 0$.

    \medskip
    
    Finally, we make a few comments about the other case in which $z(\rho_\delta(t_0, x_0), q_\delta(t_0, x_0)) = -\kappa$, and the assumption \eqref{zw} holds. In this case, we would want to show that
    \begin{equation*}
    \frac{\partial}{\partial t}\Big(z(\rho_\delta(t, x), q_\delta(t, x))\Big)\Big|_{(t_0, x_0)} > 0.
    \end{equation*}
    By \eqref{quasitruncate} and the Chain Rule:
    \begin{equation}\label{zchain}
    \frac{\partial}{\partial t}z(\bd{U}_\delta) = \nabla_{\rho, q}z(\bd{U}_\delta) \Big(\epsilon \Delta \bd{U}_\delta + G_{\delta}(\bd{U}_\delta) - \chi_R(\rho_\delta,q_\delta)\bd{F}'(\bd{U}_\delta) \partial_{x}\bd{U}_\delta\Big)
    \end{equation}
    In this case, we can estimate the terms on the right-hand side similarly to the case of $w$ above. For the first term on the right-hand side, $-z$ is quasiconvex (see Section 4 of \cite{DP832}), so therefore,
    \begin{equation*}
    \epsilon \nabla_{\rho, q} z(\bd{U}_\delta) \Delta \bd{U}_\delta|_{(t_0, x_0)} \ge 0.
    \end{equation*}
    For the second term, we compute that for $0 < \delta < \alpha$:
    \begin{equation*}
    \nabla_{\rho, q}z(\bd{U}_\delta) \bd{G}_{\delta}(\bd{U}_\delta) = -(\alpha - \delta) \frac{q_\delta}{\rho_\delta} + \delta \theta \rho_\delta^{\theta} > 0,
    \end{equation*}
    since $z(\rho_\delta(t_0, x_0), q_\delta(t_0, x_0)) = -\kappa$ and the assumption \eqref{zw} together imply that $0 < \rho_\delta(t_0, x_0) \le \kappa^{1/\theta}$ and $q_\delta(t_0, x_0) \le 0$, from the definition of the Riemann invariants. Finally, in the same way as for Term 3 of the computation for $w$ (namely, using the fact that $\nabla_{\rho, q}z$ is a left eigenvector of $\bd{F}'(\bd{U}_\delta)$), we have that $\nabla_{\rho, q} z(\bd{U}_\delta) \partial_{x}\bd{U}_\delta\Big\vert_{(t_0, x_0)} = 0$, and hence:
    \begin{equation*}
    -\chi_R(\rho_\delta,q_\delta) \nabla_{\rho, q}z(\bd{U}_\delta) \bd{F}'(\bd{U}_\delta) \partial_{x}\bd{U}_\delta\Big\vert_{(t_0, x_0)} = 0.
    \end{equation*}
    Thus, using \eqref{zchain}, this completes the proof of the fact that $\displaystyle \frac{\partial}{\partial t} z(\rho_\delta(t, x), q_\delta(t, x))\Big\vert_{(t_0, x_0)} > 0$.
\end{proof}

Our next goal is to extend this result in Proposition \ref{perturbedinvariant} on invariant regions to the equations \eqref{perturbed} with $\delta > 0$ to the case of $\delta = 0$, by taking a limit as $\delta \to 0$. This will hence prove the invariant region result in Proposition \ref{invariantregion} for $\delta = 0$. 

\begin{proof}[Proof of Proposition \ref{invariantregion}]
    Let $(\rho_{\delta}, q_{\delta})$ be the solution to \eqref{perturbed} for a perturbation parameter $\delta > 0$ and some initial data $(\rho_0, q_0) \in \Lambda_{\kappa}$, and let $(\rho, u)$ be the solution to \eqref{dampedtruncate} for the same initial data. The proposition will be established if we show that 
    \begin{equation}\label{deltastability}
        \|\rho_{\delta} - \rho\|_{ C([0, T] \times \mathbb{T})} + \|u_{\delta} - u\|_{C([0, T] \times \mathbb{T})} \to 0, \qquad \text{ almost surely.}
    \end{equation}
 %\cancel{   This is because this almost sure convergence \eqref{deltastability} would imply that $\rho_{\delta} \to \rho$ and $u_{\delta} \to u$ in $L^{2}(\mathbb{T})$ almost everywhere on $\Omega \times [0, T] \times \mathbb{T}$ {\cred needs fixing}, and} t
 The result thus follows when we combine this with the fact that $(\rho_{\delta}, u_{\delta}) \in \Lambda_{\kappa}$ for all $t \ge 0$ and for all $0 < \delta < \alpha$ almost surely.

    Subtracting the equations for $(\rho, q)$ and $(\rho_{\delta}, q_{\delta})$ in \eqref{dampedtruncate} and \eqref{perturbed}, we obtain:
    \begin{align*}
    &\partial_{t}(\rho - \rho_{\delta}) + \partial_{x}([q]_{R} - [q_{\delta}]_{R}) = \epsilon \Delta (\rho - \rho_{\delta}) - \delta \rho_{\delta},\\
   & \partial_{t}(q - q_{\delta}) + \partial_{x}\left(\frac{[q]_{R}q}{\rho} - \frac{[q_{\delta}]_{R}q_{\delta}}{\rho_{\delta}}\right) + \chi_{R}(\rho, q) \partial_{x}(\kappa \rho^{\gamma}) - \chi_{R}(\rho_{\delta}, q_{\delta})\partial_{x}(\kappa \rho_{\delta}^{\gamma}) \\
    &= \Big(\chi_{R}(\rho, q) \bd{\Phi}^{R, \epsilon}(\rho, q) - \chi_{R}(\rho_{\delta}, q_{\delta}) \bd{\Phi}^{R, \epsilon}(\rho_{\delta}, q_{\delta})\Big) dW - \alpha(q - q_{\delta}) + \epsilon \Delta (q - q_{\delta}).
    \end{align*}
    Note that these difference equations are the same as those for the continuous dependence proof in Proposition \ref{continuous}, see \eqref{firstdiff0} and \eqref{seconddiff0}, with an extra $\delta \rho_{\delta}$ term in the first equation.
    Moreover, observe due to the lower bounds on the density $\rho_\delta$ given in \eqref{uniformrho_inv}
we get analogues of the estimates in Lemma \ref{diffest}. 

So as in the proof of Proposition \ref{continuous}, we can obtain the following analogue of inequality \eqref{gronwallcontinuous}, where we account for the extra $\delta \rho_{\delta}$ term and note that $\rho$ and $\rho_{\delta}$ have the same initial data:
    \begin{multline}\label{deltaineq2}
    \Big(\|\rho - \rho_{\delta}(t)\|^{2}_{H^{2}(\mathbb{T})} + \|(q - q_{\delta})(t)\|_{H^{2}(\mathbb{T})}^{2}\Big) + \frac{\epsilon}{2} \mathbb{E} \int_{0}^{t} \Big(\|\partial_{x}(\rho - \rho_{\delta})(s)\|^{2}_{H^{2}(\mathbb{T})} + \|\partial_{x}(q - q_{\delta})(s)\|_{H^{2}(\mathbb{T})}^{2}\Big) ds \\
    \le \delta C(\epsilon) \mathbb{E}\Big(\|\rho_{\delta}\|_{L^{2}(0, T; H^{2}(\mathbb{T}))}^{2}\Big) + C_{R, \epsilon, T} \mathbb{E} \int_{0}^{t} \Big(\|(\rho - \rho_{\delta})(s)\|_{H^{2}(\mathbb{T})}^{2} + \|(q - q_{\delta})(s)\|^{2}_{H^{2}(\mathbb{T})}\Big) ds.
    \end{multline}
    Here, we used the estimate for $j = 0, 1, 2$ that
    \begin{equation*}
    \left|\int_{0}^{t} \int_{\mathbb{T}} \delta \partial_{x}^{j}\rho_{\delta} \partial_{x}^{j}(\rho - \rho_{\delta})\right| \le \frac{\epsilon}{4} \int_{0}^{t} \int_{\mathbb{T}} |\partial_{x}^{j}(\rho - \rho_{\delta})|^{2} + \delta C(\epsilon) \int_{0}^{t} \int_{\mathbb{T}} |\partial^{j}_{x}\rho_{\delta}|^{2}, 
    \end{equation*}
    where the term $\displaystyle \frac{\epsilon}{4} \int_{0}^{t} \int_{\mathbb{T}} |\partial^{j}_{x}(\rho - \rho_{\delta})|^{2}$ can be absorbed into the dissipation term on the left-hand side of the estimate, as in the proof of Lemma \ref{continuous}.

   Identically to the proof of Lemma \ref{rhoH2} we can show that $\|\rho_{\delta}\|_{C(0, T; H^{2}(\mathbb{T}))} \le C_{R, T}$ almost surely, for a constant $C_{R, T}$ that is independent of $\delta$. Therefore, using this in \eqref{deltaineq2}, we obtain by Gronwall's inequality that 
    \begin{equation*}
    \|\rho - \rho_{\delta}\|_{C(0, T; H^{2}(\mathbb{T}))} + \|q - q_{\delta}\|_{C(0, T; H^{2}(\mathbb{T}))} \to 0,
    \end{equation*}
    so since $(\rho_{\delta}, q_{\delta}) \in \Lambda_{\kappa}$, we conclude by the continuous embedding $C(0, T; H^{2}(\mathbb{T})) \subset C([0, T] \times \mathbb{T})$ that $(\rho, q) \in \Lambda_{\kappa}$ also by passing to the limit as $\delta \to 0$. 
\end{proof}

Now that we have shown that the deterministic dynamics of the approximate system \eqref{approxsystem}, represented by \eqref{dampedtruncate}, have an invariant region of $\Lambda_{\kappa}$ whenever the initial data $(\rho_0, q_0) \in \Lambda_{\kappa}$, we show a corresponding invariant region result for the stochastic problem.

\begin{proposition}\label{stochinv}
    The region $\Lambda_{\kappa_{\epsilon}}$ is an invariant region for the stochastic equation:
    \begin{equation*}
    \begin{cases}
        \partial_{t}\rho = 0, \\
        \partial_{t}q = \chi_{R}(\rho, q) \bd{\Phi}^{R, \epsilon}(\rho, q) dW.
    \end{cases}
    \end{equation*}
   Namely, given (potentially random) initial data $(\rho_0, q_0)$ which is in $\Lambda_{\kappa_{\epsilon}}$, the solution $(\rho(t), q(t))$ exists and is in $\Lambda_{\kappa_{\epsilon}}$ for all $t\geq 0$ almost surely.   (See Section \ref{noiseregularize} and \eqref{suppReps} for the construction of $\kappa_{\epsilon}$).
\end{proposition}

\begin{proof}
    This is a direct consequence of the compact support assumption on the noise coefficient $\bd{\Phi}^{R, \epsilon}(\rho, q)$ in \eqref{suppReps}, where we construct the regularized and truncated noise coefficient $\bd{\Phi}^{R, \epsilon}(\rho, q)$ so that it has support in $\Lambda_{\kappa_{\epsilon}}$.
\end{proof}

Since both the deterministic dynamics (Proposition \ref{invariantregion}) and the stochastic dynamics (Proposition \ref{stochinv}) of the approximate system \eqref{approxsystem} have an invariant region, the coupled dynamics of the full stochastic damped compressible Euler equations give rise to an invariant region too. This allows us to prove Proposition \ref{Linfuniform1} on uniform $L^{\infty}([0, T] \times \mathbb{T})$ bounds on $(\rho, q)$ as follows.

\begin{proof}[Proof of Proposition \ref{Linfuniform1}]
We split the problem \eqref{approxsystem} into its deterministic and stochastic components (see e.g. the splitting scheme used in \cite{B14} or \cite{BGR}). By combining Proposition \ref{invariantregion} and Proposition \ref{stochinv}, we conclude that the splitting scheme and thus the system \eqref{approxsystem} has an invariant region of $\Lambda_{\kappa_{\epsilon}}$. Namely, given initial data $(\rho_0, q_0) \in \Lambda_{\kappa_{\epsilon}}$, the stochastic solution to the approximate system \eqref{approxsystem} is in $\Lambda_{\kappa_{\epsilon}}$ for all $t \ge 0$, almost surely. Therefore, since $\rho$ and $u$ are bounded for all $(\rho, u) \in \Lambda_{\kappa_{\epsilon}}$, we conclude that $\rho$ and $q := \rho u$ are bounded almost surely, namely:
\begin{equation*}
\|(\rho, u, q)\|_{C([0, T] \times \mathbb{T})} \le C_{\epsilon}.
\end{equation*}
\end{proof}

\section{Uniform-in-time bounds for the approximate system}\label{uniformsection}
In this section, we will obtain additional higher order uniform bounds on the fluid density and the fluid velocity in time, in preparation for the time averaging procedure, see \eqref{muT}, that will give the existence of an invariant measure to \eqref{approxsystem}. We will consider the approximate solution \eqref{approxsystem} with initial data
\begin{equation*}
(\rho_0, u_0) = (1, 0),
\end{equation*}
and derive bounds on the resulting solution $(\rho, u)$ that are sublinear in time at the expense of their dependence on the approximating and regularizing parameters $R$ and $\epsilon$. 

{In Propositions \ref{densityest} and \ref{fluidapprox}, we will obtain $H^3(\mathbb{T})$ bounds for the fluid density and the momentum, as a result of the truncation and additional regularization, which will let us establish tightness of the time-averaged laws in Proposition \ref{timeaverage} via standard compactness arguments. We note that the choice of $H^{3}(\mathbb{T})$ is because of the fact that $H^{3}(\mathbb{T})$ compactly embeds into $H^{2}(\mathbb{T})$, which is the space for the path space $\mathcal{X}$ defined in \eqref{path}.}

\subsection{Bounds on the density.} 
We first start by showing uniform bounds for the density.
\begin{proposition}\label{densityest}
Let $(\rho^{R,\epsilon}, q^{R,\epsilon})$ be the unique solution in $C(0, T; \mathcal{X})$ to the approximate system \eqref{approxsystem} with initial data $(\rho_0, u_0) = (1, 0)$. Then, for all $T \ge 0$:
\begin{equation*}
\|\rho^{R,\epsilon}\|_{C(0, T; H^{1}(\mathbb{T}))} \le C_{R,\epsilon} \quad \text{ and } \quad \frac{1}{T}\int_{0}^{T}\|\rho^{R,\epsilon}(t, \cdot)\|^{2}_{H^{3}(\mathbb{T})} dt \le C_{R,\epsilon} \quad \text{ almost surely},
\end{equation*}
where the constant $C_{R,\epsilon}$, depending on $R, \epsilon$, is \textit{deterministic} and is \textit{independent of the time $T > 0$.}
\end{proposition}

\begin{proof} %{\cred Breit-Hoff}
    By integrating the approximate continuity equation in \eqref{approxsystem} over the spatial domain $\mathbb{T}$, we obtain conservation of mass:
    \begin{equation}\label{massconserve}
    \int_{\mathbb{T}} \rho^{R,\epsilon}(t, \cdot) dx = 1, \quad \text{ almost surely, for all } t \ge 0.
    \end{equation}
    We use this (almost sure) uniform $L^{1}(\mathbb{T})$ bound on $\rho^{R,\epsilon}$ to bootstrap the uniform-in-time bounds to higher regularity via maximal regularity. Namely, we have, for some constant $C>0$, %independent of $R$ and $\epsilon$, that
    \begin{align*}
    \|\partial_{t}\rho^{R,\epsilon}\|_{L^{2}(T, T + 1; H^{-2}(\mathbb{T}))} &+ \|\Delta \rho^{R,\epsilon}\|_{L^{2}(T, T + 1; H^{-2}(\mathbb{T}))} \le C\Big(\|\rho^{R,\epsilon}(T, \cdot)\|_{H^{-1}(\mathbb{T})} + \|\partial_{x}[q^{R,\epsilon}]_{R}\|_{L^{2}(T, T + 1; H^{-2}(\mathbb{T})}\Big) \\
    &\le C\Big(\|\rho^{R,\epsilon}(T, \cdot)\|_{L^{1}(\mathbb{T})} + \|[q^{R,\epsilon}]_R\|_{L^2(T,T+1;H^{-1}(\mathbb{T}))}\Big) 
    %\\ &\le C\Big(1 + \|\rho^{R,\epsilon}[u]_{R}\|_{L^{1}(\mathbb{T})}\Big) 
    \le C_R, 
    \end{align*}  
   by the Sobolev embedding $L^{1}(\mathbb{T}) \subset H^{-1}(\mathbb{T})$ and the definition of the truncation \eqref{truncatedq}. 
   \iffalse and the fact that by the definition {\cred $\chi \leq 1$} and of $[u]_{R}$ in \eqref{truncateduR}:
    \begin{equation*}
    \|[u]_{R}\|_{L^{\infty}(\mathbb{T})} \le \|[u]_{R}\|_{H^{2}(\mathbb{T})} \le R + 1.
    \end{equation*}
    \fi
    Therefore, for a deterministic constant $c_0$ that is independent of $T$:
    \begin{equation}\label{c0bound}
    \|\rho^{R,\epsilon}\|_{L^{2}(T, T + 1; L^{2}(\mathbb{T}))} \le c_0 \quad \text{ almost surely.}
    \end{equation}

    Using this uniform bound, we deduce that in every interval $[N/2, (N + 1)/2]$ for nonnegative integers $N$, there exists a corresponding random $\tau_{0, N}(\omega) \in [N/2, (N + 1)/2]$ depending on the random outcome in the probability space $\omega \in \Omega$, such that
    \begin{equation*}
    \|\rho^{R,\epsilon}(\omega, \tau_{0, N}, \cdot)\|_{L^{2}(\mathbb{T})} \le 2c_0.
    \end{equation*}
    and we note that by maximal regularity (and for a fixed $\omega$ in a measurable set of probability one):
    \begin{align*}
    \|\partial_{t}\rho^{R,\epsilon}\|_{L^{2}(\tau_{0, N}, \tau_{0, N} + 1; H^{-1}(\mathbb{T}))} &+ \|\Delta \rho^{R,\epsilon}\|_{L^{2}(\tau_{0, N}, \tau_{0, N} + 1; H^{-1}(\mathbb{T}))} \\
    &\le C\Big(\|\rho^{R,\epsilon}(\tau_{0, N}, \cdot)\|_{L^{2}(\mathbb{T})} + \|\partial_{x}[q^{R,\epsilon}]_{R}\|_{L^{2}(\tau_{0, N}, \tau_{0, N} + 1; H^{-1}(\mathbb{T}))}\Big) \\
    &\le C_{R, \epsilon}\Big(c_{0} + \|[q^{R,\epsilon}]_{R}\|_{L^{2}(\tau_{0, N}, \tau_{0, N} + 1; L^{2}(\mathbb{T}))}\Big) \le C(1 + 2c_0).
    \end{align*}
    Note that for (almost every) fixed $\omega \in \Omega$, the corresponding $\{\tau_{0, N}\}_{N = 1}^{\infty}$ is a monotonically increasing sequence of times for which $|\tau_{0, N + 1} - \tau_{0, N}| \le 1$. Thus, every interval $[T, T + 1]$ for arbitrary $T$ can be fully covered by five such intervals $[\tau_{0, N}, \tau_{0, N} + 1]$ and hence:
    \begin{equation}\label{c1bound}
    \|\rho^{R,\epsilon}\|_{L^{2}(T, T + 1; H^{1}(\mathbb{T}))} \le 5C(1 + 2c_0 +R) := c_1 \quad \text{ almost surely},
    \end{equation}
    for a deterministic constant $c_1$ that is independent of $T$. %and $R,\epsilon$.

    We can then iterate this procedure to bootstrap uniform bounds for higher regularity. We can construct $\tau_{1, N}(\omega) \in [N/2, (N + 1)/2]$ for each nonnegative integer $N$ and outcome $\omega$ as before, such that
    \begin{equation}\label{tau1N}
    \|\rho^{R,\epsilon}(\omega, \tau_{1, N}, \cdot)\|_{H^{1}(\mathbb{T})} \le 2c_1,
    \end{equation}
    and by maximal regularity, as before:
    \begin{align*}
    \|\partial_{t}\rho^{R,\epsilon}\|_{L^{2}(\tau_{1, N}, \tau_{1, N} + 1, L^{2}(\mathbb{T}))} &+ \|\Delta \rho^{R,\epsilon}\|_{L^{2}(\tau_{1, N}, \tau_{1, N} + 1; L^{2}(\mathbb{T}))} \\
    &\le C\Big(\|\rho^{R,\epsilon}(\tau_{1, N}, \cdot)\|_{H^{1}(\mathbb{T})} + \|\partial_{x}([q^{R,\epsilon}]_{R})\|_{L^{2}(\tau_{1, N}, \tau_{1, N} + 1; L^{2}(\mathbb{T}))}\Big) \\
    &\le C_{R,\epsilon}\Big(1  + c_1\Big),
    \end{align*}
    by \eqref{c0bound} and \eqref{c1bound}.
%Note that $C$ here is independent of $R$ but depends on $\epsilon$. 
    A similar covering argument shows that for a deterministic constant $c_2$ independent of $T > 0$, 
    \begin{equation}\label{c2bound}
    \|\rho^{R,\epsilon}\|_{L^{2}(T, T + 1; H^{2}(\mathbb{T}))} \le c_2 \quad \text{ almost surely.}
    \end{equation}
    Given the definition of the truncation we have, for any $T > 0$, that
    \begin{align*}
    \|\partial_{x}( [q^{R,\epsilon}]_{R})\|^2_{L^{2}(T, T + 1; H^{1}(\mathbb{T}))} &\le C_R.
\end{align*}
By the Sobolev embedding $H^{2}(\mathbb{T}) \subset W^{1, \infty}(\mathbb{T})$, we can iterate the above procedure once more to obtain a uniform bound for a deterministic constant $c_3$ that is independent of $T$:
    \begin{equation}\label{c3bound}
    \|\partial_{t}\rho^{R,\epsilon}\|_{L^{2}(T, T + 1; H^{1}(\mathbb{T}))} + \|\rho^{R,\epsilon}\|_{L^{2}(T, T + 1; H^{3}(\mathbb{T}))} \le c_{3} \quad \text{ almost surely}.
    \end{equation}
    Then, \eqref{c3bound} immediately implies that
    \begin{equation*}
    \frac{1}{T} \int_{0}^{T} \|\rho^{R,\epsilon}\|_{H^{3}(\mathbb{T})}^{2} dt \le C, \qquad \text{ almost surely}
    \end{equation*}
    for a deterministic constant $C$ that is independent of $T$. 
    \iffalse
    {\cred NOTE:
    
    In fact if we change the scheme and consider $[u]_R= \chi(\|u\|_{H^3(\mathbb{T})}-R)u$ then, again since $H^2(\mathbb{T})$ is Banach algebra, we would have
     \begin{align*}
    \|\partial_{x}([\rho]_R [u]_{R})\|^2_{L^{2}(T, T + 1; H^{2}(\mathbb{T}))} &\le 
    \int_T^{T+1} \|\partial_x\rho\|^2_{H^{2}(\mathbb{T})}\|[u]_R\|^2_{H^{2}(\mathbb{T})}+\|\rho\|^2_{H^{2}(\mathbb{T})}\|\partial_x [u]_R\|^2_{H^{2}(\mathbb{T})}\\
    &\le C_{R}\|\rho\|^2_{L^{2}(T, T + 1; H^{3}(\mathbb{T}))} 
\end{align*}
and thus,
\begin{equation}\label{c4bound}
    \|\partial_{t}\rho\|_{L^{2}(T, T + 1; H^{2}(\mathbb{T}))} + \|\rho\|_{L^{2}(T, T + 1; H^{4}(\mathbb{T}))} \le c_{4} \quad \text{ almost surely}.
    \end{equation}
    This would imply that we have for some $C$ independent of $T$:
      \begin{equation}\label{rho_h2}
    \sup_{t \ge 0} \|\rho(t, \cdot)\|_{H^{2}(\mathbb{T})} \le C, \quad \text{ almost surely.}
    \end{equation}
    
    }
    \fi

    Furthermore, by construction, for every $t \ge 0$, there exists some $\tau_{N} \le t$ such that $t - \tau_{N} \le 1$. Hence, by \eqref{tau1N} and \eqref{c3bound}, there exists a uniform in time constant $C$, dependent on $R$ and $\epsilon$, such that the following inequality follows from Fundamental Theorem of Calculus:
    \begin{equation*}
    \sup_{t \ge 0} \|\rho^{R,\epsilon}(t, \cdot)\|_{H^{1}(\mathbb{T})} \le C, \quad \text{ almost surely.}
    \end{equation*}
\end{proof}
\if 1 = 0
Next, in order to obtain tightness, we need a uniform bound on $\rho$ from below, away from vacuum. This is provided by the following proposition.
 
\begin{proposition}\label{vacuumbound}
Let $(\rho, u)$ be the unique solution to \eqref{approxsystem} with initial data $(\rho_0, u_0) = (1, 0)$. Then, there exists a deterministic constant $C$ that is independent of $T > 0$, but depending on $R$ and $\epsilon$, such that
\begin{equation*}
\frac{1}{T} \int_{0}^{T} \int_{\mathbb{T}} \frac{(\partial_{x}\rho)^2}{\rho^{2}} \le C.
\end{equation*}
\end{proposition}
\begin{proof}
We consider the regularized continuity equation in \eqref{approxsystem}:
\begin{equation*}
\partial_{t}\rho + \chi_R(\rho,q)\partial_{x}q = \epsilon \Delta \rho.
\end{equation*}
Thanks to Proposition \ref{rholowerbound} we know that there is no vacuum for all time, and we can test the approximate continuity equation by $1/\rho$ to obtain:
\begin{equation*}
\int_{\mathbb{T}} \text{log}(\rho(T)) dx + \int_{0}^{T} \int_{\mathbb{T}}\chi_R(\rho,q) \frac{\partial_{x}\rho}{\rho}\frac{q}{\rho} dx dt + \int_{0}^{T} \int_{\mathbb{T}} \chi_R(\rho,q) \partial_{x}(\frac{q}{\rho}) dx dt = \epsilon \int_{0}^{T} \int_{\mathbb{T}} \frac{(\partial_{x}\rho)^{2}}{\rho^{2}} dx dt.
\end{equation*}
Using Cauchy's inequality, we obtain:
\begin{equation*}
\frac{\epsilon}{2} \int_{0}^{T} \int_{\mathbb{T}} \frac{(\partial_{x}\rho)^{2}}{\rho^{2}} dx dt \le \int_{\mathbb{T}} |\log(\rho(T))| dx + C_{\epsilon}\int_{0}^{T} \int_{\mathbb{T}} \chi_R({\rho,q})\Big(1 + (\frac{q}{\rho}) ^{2} + (\partial_{x}(\frac{q}{\rho}) )^{2}\Big) dx dt.
\end{equation*}
By using \eqref{expbound} to estimate the term involving $|\text{log}(\rho(T))| \le CRt$ and the definition of $\chi_R(\rho,q)$ in \eqref{chi} along with the Sobolev embedding $H^{2}(\mathbb{T}) \subset W^{1, \infty}(\mathbb{T})$, we obtain:
\begin{equation*}
\frac{\epsilon}{2} \int_{0}^{T} \int_{\mathbb{T}} \frac{(\partial_{x}\rho)^{2}}{\rho^{2}} dx dt \le C_{R, \epsilon} T
\end{equation*}
almost surely, which establishes the desired bound.
\iffalse 
{\cred NOT NECESSARY ANYMORE: Now we will test with $\frac1{\rho^3}$:
\begin{align*}
    3\epsilon\int_0^T\int_{\mathbb{T}}\frac{(\partial_x\rho)^2}{\rho^4} + \frac12\int_{\mathbb{T}}\frac1{\rho^2(T)}&=  \frac12\int_{\mathbb{T}}\frac1{\rho^2(0)} + \int_{0}^{T} \int_{\mathbb{T}}\chi_\rho \frac{\partial_{x}\rho}{\rho^3}[u]_{R} dx dt + \int_{0}^{T} \int_{\mathbb{T}} \chi_{R,\rho}\frac{\partial_{x}[u]_{R}}{\rho^2} dx dt\\
   & =  \frac12\int_{\mathbb{T}}\frac1{\rho^2(0)}  + \frac32\int_{0}^{T} \int_{\mathbb{T}} \chi_{R,\rho}\frac{\partial_{x}[u]_{R}}{\rho^2} dx dt\\
   & \leq  \frac12\int_{\mathbb{T}}\frac1{\rho^2(0)}  + C_{R,\epsilon}T.\\
\end{align*}
where $\chi_{R,\rho}=\chi_R\left({{\|\frac1{\rho}\|_{L^\infty(\mathbb{T})}}}\right)$.
}
\fi 
\end{proof}

We note that the previous bound in Proposition \ref{vacuumbound} is useful for giving a pointwise lower bound on the fluid density $\rho$ away from vacuum for the following reason. {\cred I'm assuming that the point of the following lemma is to get bounds on density independent of time? Otherwise maximal regularity will give us bounds..}

\begin{lemma}\label{vacuumlemma}
Given $\rho \in \mathcal{X}$ such that $\displaystyle \int_{\mathbb{T}} \frac{(\partial_{x}\rho)^{2}}{\rho^{2}} dx \le M$, we have that
\begin{equation*}
e^{-(1 + M)} \le \rho \le e^{1 + M}.
\end{equation*}
\end{lemma}

\begin{proof}
    Since $\displaystyle \int_{\mathbb{T}} \rho(x) dx = 1$ and $\rho \in H^{2}(\mathbb{T}) \subset C(\mathbb{T})$, there exists $x_0 \in \mathbb{T}$ such that $\rho(x_0) = 1$, and hence $\log(\rho(x_0)) = 0$. By assumption,
    \begin{equation*}
    \int_{\mathbb{T}} |\partial_{x}(\log(\rho))|^{2} dx \le M.
    \end{equation*}
    Hence, for any $x_1 \in \mathbb{T}$:
    \begin{equation*}
    |\log(\rho(x_1))| = \left|\int_{x_0}^{x_1} \partial_{x}\log(\rho(x)) dx\right| \le M^{1/2}  {\color{red}\text{why not stop here? }} \le 1 + M. 
    \end{equation*}
    This establishes the desired bound.
\end{proof}
\fi
\subsection{Uniform bounds on the fluid velocity.} Finally, we show a uniform bound on the fluid velocity.
\begin{proposition}\label{fluidapprox}
    Let, for $T>0$, $(\rho^{R,\epsilon}, q^{R,\epsilon})$ be the unique solution to \eqref{approxsystem} for initial data $(\rho_0, u_0) = (1, 0)$. Then, there exists a constant $C$ that is independent of $T > 0$ such that
    \iffalse
\begin{equation*}
      \mathbb{E} \int_0^{{T}} \|\rho (t)\|_{H^{1}(\mathbb{T})}^{2} dt \le C_\epsilon,\qquad    \mathbb{E} \int_0^{{T}} \|q (t)\|_{H^{2}(\mathbb{T})}^{2} dt \le C_\epsilon.
    \end{equation*}
and,
\fi
    \begin{equation*}
   %   \frac{1}{T} \mathbb{E} \int_0^{{T}} \|\rho (t)\|_{H^{3}(\mathbb{T})}^{2} dt \le C,\qquad 
   \frac{1}{T} \mathbb{E} \int_0^{{T}} \|q^{R,\epsilon} (t)\|_{H^{3}(\mathbb{T})}^{2} dt \le C.
    \end{equation*}
\end{proposition}

\begin{proof} 
    \textbf{Zeroth order derivative estimate on the momentum.} We consider the momentum equation for $q^{R,\epsilon}=\rho^{R,\epsilon} u^{R,\epsilon}$ in \eqref{approxsystem}: 
    \begin{multline*}
        \partial_{t}q^{R,\epsilon} +\chi_R(\rho^{R,\epsilon},q^{R,\epsilon}) \partial_{x}\left(\frac{(q^{R,\epsilon})^2}{\rho^{R,\epsilon}}\right) + \kappa \chi_R(\rho^{R,\epsilon},q^{R,\epsilon}) \partial_{x}((\rho^{R,\epsilon})^{\gamma}) \\
        = \chi_R(\rho^{R,\epsilon},q^{R,\epsilon})\Phi^{R, \epsilon}(\rho^{R,\epsilon}, q^{R,\epsilon}) dW - \alpha q^{R,\epsilon} + \epsilon \Delta q^{R,\epsilon}.
    \end{multline*}
  By applying It\^{o}'s formula with the functional $q^{R,\epsilon} \to \frac{1}{2}\|q^{R,\epsilon}\|^{2}_{L^2(\mathbb{T})}$, we obtain for any $t\in [0,T]$ that
    \begin{align*}
    &\frac{1}{2} \int_{\mathbb{T}} |q^{R,\epsilon}(t)|^{2} dx + \alpha \int_{0}^{t} \int_{\mathbb{T}} (q^{R,\epsilon})^{2} dx dt + \epsilon  \int_{0}^{t} \int_{\mathbb{T}} |\partial_{x} q^{R,\epsilon}|^{2} dx dt \\
    &= \int_{0}^{t} \int_{\mathbb{T}} \chi_R\left(\rho^{R,\epsilon},q^{R,\epsilon}\right)\frac{(q^{R,\epsilon})^2}{\rho^{R,\epsilon}}(\partial_{x}q^{R,\epsilon}) dx dt %-\int_{0}^{t} \int_{\mathbb{T}} \chi_R\left(\rho^{R,\epsilon},q^{R,\epsilon}\right)\partial_{x}(\frac{q^{R,\epsilon}}{\rho^{R,\epsilon}})q^{R,\epsilon}^{2} dx dt \\
    +{ \kappa \gamma  \int_{0}^{t} \int_{\mathbb{T}} \chi_R\left(\rho^{R,\epsilon},q^{R,\epsilon}\right)(\rho^{R,\epsilon})^{\gamma} (\partial_x q^{R,\epsilon}) dx dt }\\
    &+ \frac{1}{2}  \int_{0}^{t} \int_{\mathbb{T}} |\chi(\rho^{R,\epsilon}, q^{R,\epsilon})|^{2} |\bd{G}^{R, \epsilon}(x, \rho^{R,\epsilon}, q^{R,\epsilon})|^{2} dx dt + \int_0^t\sum_{k=1}^\infty\int_{\mathbb{T}} q^{R,\epsilon}\chi_{R}(\rho^{R,\epsilon}, q^{R,\epsilon}) G_k^{R,\epsilon}(\rho^{R,\epsilon},q^{R,\epsilon})dW_k ,
    \end{align*}
    where $W_k= W\bd{e}_k$. Recall that $(\rho_0, q_0) = (1, 0)$, so that $\displaystyle \int_{\mathbb{T}} |q(0)|^{2} dx = 0$.
    Now we take $t=T$ and apply expectation on both sides and estimate the terms on the right-hand side.
First, observe that due to the uniform bounds on $q^{R,\epsilon},\rho^{R,\epsilon}$ and $\frac{q^{R,\epsilon}}{\rho^{R,\epsilon}}$ obtained in Proposition \ref{Linfuniform1}, we can immediately deduce by Young's inequality that
 \begin{align*}
    \left|\mathbb{E}   \int_{0}^{T} \int_{\mathbb{T}} \chi_R\left(\rho^{R,\epsilon},q^{R,\epsilon}\right)\left(\kappa \gamma(\rho^{R,\epsilon})^{\gamma} +\frac{(q^{R,\epsilon})^2}{\rho^{R,\epsilon}}\right)(\partial_x q^{R,\epsilon}) dx dt  \right| &\le    \mathbb{E} \int_{0}^{T} \int_{\mathbb{T}} \left|(\rho^{R,\epsilon})^{\gamma} +\frac{(q^{R,\epsilon})^2}{\rho^{R,\epsilon}}\right| |\partial_x q^{R,\epsilon}| dx dt  \\
    &\le \frac{\epsilon}2\bE \int_{0}^{T} \int_{\mathbb{T}} |\partial_xq^{R,\epsilon}|^2dxdt+ C_\epsilon T,
    \end{align*}

    \iffalse
    Note that the first three terms on the right-hand side, involving $\chi_R\left(\rho,q\right)$, are only nonzero when $\|q\|_{H^{2}(\mathbb{T})} \le R + 1$ and $\|\rho^{-1}\|_{L^\infty(\mathbb{T})} \leq R+1$. In particular, note that
    \begin{multline*}
    \left|\mathbb{E} \left(\int_{0}^{T} \int_{\mathbb{T}} \chi_R\left(\rho,q\right)\partial_{x}q \frac{q^2}{\rho} dx dt - \int_{0}^{T} \int_{\mathbb{T}} \chi_R\left(\rho,q\right)\partial_{x}(\frac{q}{\rho})q^{2} dx dt + \kappa \gamma \int_{0}^{T} \int_{\mathbb{T}} \chi_R\left(\rho,q\right)\rho^{\gamma-1} [q]_{R} dx dt \right) \right|  \le C_{R}T,
    \end{multline*}
    by the uniform bound on $\rho$ in $L^{\infty}(\mathbb{T})$ by Sobolev embedding and Proposition \ref{densityest}. 
    \fi 
    Similarly, using $|\bd{G}^{R, \epsilon}(\rho^{R,\epsilon}, q^{R,\epsilon})|^{2} \le A_{0}(\rho^{R,\epsilon})^{2}$ from \eqref{A0eps} on the final term, we thus conclude that
    \begin{equation}\label{uhigher1}
  %  \mathbb{E} \int_{\mathbb{T}} |q(T)|^{2} dx + 
  \epsilon \mathbb{E} \int_{0}^{T} \int_{\mathbb{T}} |\partial_{x}q^{R,\epsilon}|^{2} dx dt \le C_{ \epsilon}T,
    \end{equation}
    where $C_\epsilon$ is \textit{independent of both $R$ and $T$.}

  \medskip
    \noindent \textbf{First order derivative estimate on the momentum.} Next, we estimate the higher derivatives. We differentiate by $x$ and apply Ito formula with $\frac{1}{2}\|\cdot\|_{L^2(\mathbb{T})}$. By integrating by parts appropriately, we obtain:
    \begin{equation}\label{ito_qh2}
    \begin{split}
     \frac{1}{2}\int_{\mathbb{T}} &|\partial_{x}q^{R,\epsilon} (T)|^{2} dx + \alpha \int_{0}^{T} \int_{\mathbb{T}} |\partial_{x}q^{R,\epsilon}|^{2} dx dt +  \epsilon  \int_{0}^{T} \int_{\mathbb{T}} |\partial_{x}^{2}q^{R,\epsilon}|^{2} dx dt \\
    &=  \int_{0}^{T} \int_{\mathbb{T}} \chi_R({\rho^{R,\epsilon},q^{R,\epsilon}})\frac{q^{R,\epsilon}}{\rho^{R,\epsilon}} \partial_{x}q^{R,\epsilon} \partial_{x}^{2}q^{R,\epsilon} dx dt+  \int_{0}^{T} \int_{\mathbb{T}} \chi_R({\rho^{R,\epsilon},q^{R,\epsilon}}) \partial_{x}\left(\frac{q^{R,\epsilon}}{\rho^{R,\epsilon}}\right)q^{R,\epsilon} \partial_{x}^{2}q^{R,\epsilon} dx dt \\
    &+ \kappa \gamma  \int_{0}^{T} \int_{\mathbb{T}} \chi_R\left( \rho^{R,\epsilon},q^{R,\epsilon}\right) (\rho^{R,\epsilon})^{\gamma - 1} \partial_{x}\rho^{R,\epsilon} \partial_{x}^{2}q^{R,\epsilon} dx dt\\
    &+ \frac{1}{2} \sum_{k = 1}^{\infty} \int_{0}^{T} \int_{\mathbb{T}} \chi_{R}(\rho^{R,\epsilon}, q^{R,\epsilon})^{2} |\nabla_{\rho^{R,\epsilon}, q^{R,\epsilon}} G_{k}^{R, \epsilon}(\rho^{R,\epsilon}, q^{R,\epsilon})\cdot\partial_x(\rho^{R,\epsilon},q^{R,\epsilon}) |^{2} dx dt \\
    &+ \sum_{k = 1}^{\infty} \int_0^T\int_{\mathbb{T}}\chi_{R}(\rho^{R,\epsilon}, q^{R,\epsilon}) \nabla_{\rho^{R,\epsilon},q^{R,\epsilon}}G_k^{R,\epsilon}(\rho^{R,\epsilon},q^{R,\epsilon})\cdot\partial_{x}(\rho^{R,\epsilon},q^{R,\epsilon})\partial_xq^{R,\epsilon}dW_k,
    \end{split}
    \end{equation}
    where $W_k= W\bd{e}_k$. We again take expectation on both sides and find bounds for the terms on the right hand side.
Thanks to the definition of $\chi_R(\rho,q^{R,\epsilon})$ and Proposition \ref{Linfuniform1}, we have
    \begin{align*}
    &\left|\mathbb{E} \int_{0}^{T} \int_{\mathbb{T}} \chi_R({\rho^{R,\epsilon},q^{R,\epsilon}})\frac{q^{R,\epsilon}}{\rho^{R,\epsilon}}\partial_{x}q^{R,\epsilon} \partial_{x}^{2}q^{R,\epsilon} dx dt\right|\\
    &\le \frac{\epsilon}{4} \mathbb{E} \int_{0}^{T} \int_{\mathbb{T}} |\partial_{x}^{2}q^{R,\epsilon}|^{2} dx dt + C_{\epsilon} \mathbb{E} \int_{0}^{T} \int_{\mathbb{T}} \Big[\chi_R({\rho^{R,\epsilon},q^{R,\epsilon}})\Big]^{2}|\partial_{x}q^{R,\epsilon}|^{2} \left(\frac{q^{R,\epsilon}}{\rho^{R,\epsilon}}\right)^{2} \\
    &\le C_{R, \epsilon}T + \frac{\epsilon}{4} \mathbb{E} \int_{0}^{T} \int_{\mathbb{T}} |\partial_{x}^{2}q^{R,\epsilon}|^{2} dx.
    \end{align*}
    Similarly, for the second term, we additionally use Proposition \ref{densityest} and the definition of the truncation $\chi_{R}(\rho^{R, \epsilon}, q^{R, \epsilon}) = \chi_{R}\Big(\|(\rho^{R, \epsilon})^{-1}\|_{L^{\infty}(\mathbb{T})}\Big) \chi_{R}\Big(\|q^{R, \epsilon}\|_{H^{2}(\mathbb{T})}\Big)$ in \eqref{truncatedq}, to obtain
    \begin{equation*}
    \left|\mathbb{E} \int_{0}^{T} \int_{\mathbb{T}} q^{R,\epsilon} \chi_R({\rho^{R,\epsilon},q^{R,\epsilon}})\partial_{x}\left(\frac{q^{R,\epsilon}}{\rho^{R,\epsilon}} \right)\partial_{x}^{2}q^{R,\epsilon}\right| \le C_{R, \epsilon}T + \frac{\epsilon}{6} \mathbb{E} \int_{0}^{T} \int_{\mathbb{T}} |\partial_{x}^{2}q^{R,\epsilon}|^{2} dx.
    \end{equation*}
    Again from the uniform in time bound of $\rho^{R,\epsilon}$ in $H^{1}(\mathbb{T})$ for all time in Proposition \ref{densityest}, and the embedding $H^1(\mathbb{T})\hookrightarrow L^\infty(\mathbb{T})$, we obtain:
    \begin{align*}
    \kappa\gamma \mathbb{E} \int_{0}^{T} \int_{\mathbb{T}} |(\rho^{R,\epsilon})^{\gamma - 1} \partial_{x}\rho^{R,\epsilon} \partial_{x}^{2} q^{R,\epsilon}| dx dt &\le \frac{\epsilon}{6} \mathbb{E} \int_{0}^{T} \int_{\mathbb{T}} |\partial_{x}^{2}q^{R,\epsilon}|^{2}dx dt + C_{\epsilon} \mathbb{E} \int_{0}^{T} \int_{\mathbb{T}} (\rho^{R,\epsilon})^{2(\gamma - 1)} |\partial_{x}\rho^{R,\epsilon}|^{2} dx dt \\
    &\le C_{\epsilon} T + \frac{\epsilon}{6} \mathbb{E} \int_{0}^{T} \int_{\mathbb{T}} |\partial_{x}^{2}q^{R,\epsilon}|^{2} dx dt.
    \end{align*}
    Finally, by \eqref{A0eps}, \eqref{uhigher1}, Proposition \ref{densityest}, and the computation that $\nabla_{\rho, q}G_{k}^{R, \epsilon} = \rho \nabla_{\rho, q} g_{k}^{R, \epsilon} + (1, 0) g_{k}^{R, \epsilon}$:
    \begin{align*}
    \mathbb{E} \sum_{k = 1}^{\infty} \int_{0}^{T} \int_{\mathbb{T}} |\nabla_{\rho, q} G_{k}^{R, \epsilon}(\rho^{R,\epsilon}, q^{R,\epsilon}) &\cdot \partial_{x}(\rho^{R,\epsilon}, q^{R,\epsilon})|^{2} dx dt\\
    & \le 2A_{0}^{2} \mathbb{E} \int_{0}^{T} \int_{\mathbb{T}} (1 + (\rho^{R,\epsilon})^{2})\Big((\partial_{x}\rho^{R,\epsilon})^{2} + (\partial_{x}(q^{R,\epsilon})^{2})\Big) dx dt \\
    &\le C_{R, \epsilon}T.
    \end{align*}
    So we conclude that
    \begin{equation}\label{uhigher2}
    \mathbb{E} \int_{0}^{T} \int_{\mathbb{T}} |\partial_{x}^{2}q^{R,\epsilon}|^{2} dx dt \le C_{R, \epsilon} T. 
    \end{equation}

    \medskip

    \noindent \textbf{Second order derivative estimate on the momentum.} Finally, we differentiate the momentum equation twice and test with $\partial_{x}^{2}q^{R,\epsilon}$ to obtain:
    \begin{align*}
   & \mathbb{E} \int_{\mathbb{T}} |\partial_{x}^{2}q^{R,\epsilon}(T)|^{2} dx + \alpha \mathbb{E} \int_{0}^{T} \int_{\mathbb{T}} |\partial_{x}^{2}q^{R,\epsilon}|^{2} dx dt + \epsilon \mathbb{E} \int_{0}^{T} \int_{\mathbb{T}} |\partial_{x}^{3}q^{R,\epsilon}|^{2} dx dt \\
    =&\mathbb{E} \int_{0}^{T} \int_{\mathbb{T}} \left[(\partial_{x}^{2}q^{R,\epsilon} )\chi_{R}(\rho^{R,\epsilon},q^{R,\epsilon})\frac{q^{R,\epsilon}}{\rho^{R,\epsilon}} + 2 (\partial_{x}q^{R,\epsilon})\chi_R(\rho^{R,\epsilon},q^{R,\epsilon}) \partial_{x}\left(\frac{q^{R,\epsilon}}{\rho^{R,\epsilon}}\right) \right]\partial_{x}^{3}q^{R,\epsilon} dx dt\\
    &+ \mathbb{E} \int_{0}^{T} \int_{\mathbb{T}} \left[q^{R,\epsilon}\chi_R(\rho^{R,\epsilon},q^{R,\epsilon})\partial_{x}^{2}\left(\frac{q^{R,\epsilon}}{\rho^{R,\epsilon}}\right) \right]\partial_{x}^{3}q^{R,\epsilon} dx dt \\
    &+ \kappa \gamma \mathbb{E} \int_{0}^{T} \int_{\mathbb{T}} \chi_R\left(\rho^{R,\epsilon},q^{R,\epsilon} \right) \Big((\rho^{R,\epsilon})^{\gamma - 1} \partial_{x}^{2}\rho^{R,\epsilon} + (\gamma - 1) (\rho^{R,\epsilon})^{\gamma - 2}(\partial_{x}\rho^{R,\epsilon})^{2}\Big) \partial_{x}^{3}q^{R,\epsilon} dx dt \\
    &+ \frac{1}{2} \sum_{k = 1}^{\infty} \mathbb{E} \int_{0}^{T} \int_{\mathbb{T}} \chi_{R}(\rho^{R,\epsilon}, q^{R,\epsilon})^{2} |\langle \nabla^{2}_{\rho, q} G^{R, \epsilon}_{k}(\rho^{R,\epsilon}, q^{R,\epsilon}) \partial_{x}(\rho^{R,\epsilon}, q^{R,\epsilon}), \partial_{x}(\rho^{R,\epsilon}, q^{R,\epsilon}) \rangle|^{2} dx dt.
    \end{align*}
    In the same way as for the first order derivative estimate, we can estimate the first two terms in absolute values on the right hand side using Proposition \ref{densityest} and \eqref{uhigher1} as
    \begin{equation*}
    \le C_{R, \epsilon}T +  \frac{\epsilon}{6} \mathbb{E} \int_{0}^{T} \int_{\mathbb{T}} |\partial_{x}^{3}q^{R,\epsilon}|^{2} dx dt.
    \end{equation*}
    For the third term, the definition of $\chi_R(\rho^{R,\epsilon},q^{R,\epsilon})$ in \eqref{truncatedq} combined with Sobolev embedding, gives us:
    \begin{align*}
    &\mathbb{E}\int_0^T\int_{\mathbb{T}} \left|q^{R,\epsilon}\chi_R(\rho^{R,\epsilon},q^{R,\epsilon})\partial_{x}^{2}\left(\frac{q^{R,\epsilon}}{\rho^{R,\epsilon}}\right) \partial_{x}^{3}q^{R,\epsilon} \right|dx dt \\
    &=  \mathbb{E}\int_0^T\int_{\mathbb{T}} \left|q^{R,\epsilon}\chi_R(\rho^{R,\epsilon},q^{R,\epsilon})\left((\rho^{R,\epsilon})^{-1}\partial_{x}^{2}q^{R,\epsilon} -2\partial_xq^{R,\epsilon}\frac{\partial_x\rho^{R,\epsilon}}{(\rho^{R,\epsilon})^2}-\left(\frac{\partial^2_x\rho^{R,\epsilon}}{(\rho^{R,\epsilon})^2}-2\frac{(\partial_x\rho^{R,\epsilon})^2}{(\rho^{R,\epsilon})^3}\right)q^{R,\epsilon}\right)\partial_{x}^{3}q^{R,\epsilon} \right|dx dt \\
    & \leq C_{R,\epsilon}\mathbb{E}\int_0^T\int_{\mathbb{T}} \left(|\partial_x^2q^{R,\epsilon}|^2 dxdt + |\partial_x\rho^{R,\epsilon}|^2+ |\partial^2_x\rho^{R,\epsilon}|^2 + |\partial_x\rho^{R,\epsilon}|^4 \right)+ \frac{\epsilon}{6}\mathbb{E}\int_0^T\int_{\mathbb{T}} |\partial_{x}^{3}q^{R,\epsilon}|^{2} dx dt.
    \end{align*}
    Recall from Proposition \ref{densityest} that
 \begin{equation*}
    \sup_{t \ge 0} \|\rho^{R,\epsilon}(t, \cdot) \|_{C(\mathbb{T})} \le C_{R, \epsilon}, \qquad \|\rho^{R,\epsilon}(t, \cdot)\|_{L^{2}(0, T; H^{3}(\mathbb{T}))}^{2} \le C_{R, \epsilon}T.
    \end{equation*}
   Hence due to the following interpolation inequality
 \begin{align}\label{rhol4}
     \|\partial_x\rho^{R,\epsilon}\|_{L^4(\mathbb{T})}^4 \leq \|\rho^{R,\epsilon}\|^2_{L^\infty(\mathbb{T})}\|\partial_x^2\rho^{R,\epsilon}\|_{L^2(\mathbb{T})}^2,
 \end{align}
and the uniform bounds $\|\rho^{R,\epsilon}\|_{C(0, T; C(\mathbb{T}))} \le C_{R, \epsilon}$ and $\|\rho^{R,\epsilon}\|_{L^{2}(0, T; H^{2}(\mathbb{T}))}^{2} \le C_{R, \epsilon} T$ in Proposition \ref{densityest} and the bounds on $q^{R,\epsilon}$ found in \eqref{uhigher2}, we obtain
\begin{align*}
     &\mathbb{E}\int_0^T\int_{\mathbb{T}} \left|q^{R,\epsilon}\chi_R(\rho^{R,\epsilon},q^{R,\epsilon})\partial_{x}^{2}\left(\frac{q^{R,\epsilon}}{\rho^{R,\epsilon}}\right) \partial_{x}^{3}q^{R,\epsilon} \right|dx dt \\
     &\leq  C_{R,\epsilon}\left( T+ \mathbb{E}   \int_{0}^{T} \|\rho^{R,\epsilon}\|^2_{L^\infty(\mathbb{T})}\|\partial_x^2\rho^{R,\epsilon}\|_{L^2(\mathbb{T})}^2dxdt\right)+ \frac{\epsilon}{6}  \mathbb{E} \int_{0}^{T} \int_{\mathbb{T}} |\partial_{x}^{3}q^{R,\epsilon}|^2 dxdt \\
        &\leq C_{R,\epsilon}T+ \frac{\epsilon}{6}  \mathbb{E} \int_{0}^{T} \int_{\mathbb{T}} |\partial_{x}^{3}q^{R,\epsilon}|^2 dx dt.
\end{align*}
 Using again the bounds in Proposition \ref{densityest} we estimate that
    \begin{align*}
    \mathbb{E} \int_{0}^{T} \int_{\mathbb{T}} (\rho^{R,\epsilon})^{\gamma - 1} |\partial_{x}^{2}\rho^{R,\epsilon}| \cdot |\partial_{x}^{3}q^{R,\epsilon}| dx dt &\le \frac{\epsilon}{6} \mathbb{E} \int_{0}^{T} \int_{\mathbb{T}} |\partial_{x}^{3}q^{R,\epsilon}|^{2} dx dt + C_{\epsilon} \mathbb{E} \int_{0}^{T} \int_{\mathbb{T}} |\partial_{x}^{2}\rho^{R,\epsilon}|^{2} dx dt \\
    &\le \frac{\epsilon}{6} \mathbb{E} \int_{0}^{T} \int_{\mathbb{T}} |\partial_{x}^{3}q^{R,\epsilon}|^{2} dx dt + C_{R, \epsilon}T. 
    \end{align*}
    Next, we
    observe, due to the definition of $\chi_R$ and \eqref{rhol4}, that for any $\gamma \ge 1$, we can write:
     \begin{align*}
         \mathbb{E} \int_{0}^{T} \int_{\mathbb{T}} \chi_R(\rho^{R,\epsilon},q^{R,\epsilon}) (\rho^{R,\epsilon})^{\gamma - 2} (\partial_{x}\rho^{R,\epsilon})^{2} &|\partial_{x}^{3}q^{R,\epsilon}| dx dt 
        \leq  C_R\mathbb{E} \int_{0}^{T} \int_{\mathbb{T}}   |\partial_{x}\rho^{R,\epsilon}|^2|\partial_{x}^{3}q^{R,\epsilon}| dx dt \\
        &\leq  C_{R,\epsilon} \mathbb{E}   \int_{0}^{T} \int_{\mathbb{T}}  |\partial_x\rho^{R,\epsilon}|^4+ \frac{\epsilon}{6}  \mathbb{E} \int_{0}^{T} \int_{\mathbb{T}} |\partial_{x}^{3}q^{R,\epsilon}|^2\\
        &\leq  C_{R,\epsilon} \mathbb{E}   \int_{0}^{T} \|\rho^{R,\epsilon}\|^2_{L^\infty(\mathbb{T})}\|\partial_x^2\rho^{R,\epsilon}\|_{L^2(\mathbb{T})}^2+ \frac{\epsilon}{6}  \mathbb{E} \int_{0}^{T} \int_{\mathbb{T}} |\partial_{x}^{3}q^{R,\epsilon}|^2  \\
        &\leq C_{R,\epsilon}T+ \frac{\epsilon}{6}  \mathbb{E} \int_{0}^{T} \int_{\mathbb{T}} |\partial_{x}^{3}q^{R,\epsilon}|^2,
    \end{align*} 

    Finally, we can estimate the quadratic noise term using the noise assumption \eqref{A0eps}, the uniform bounds in Proposition \ref{densityest}, the interpolation inequality \eqref{rhol4}, the properties of the truncation \eqref{truncatedq}, and the computation 
    \begin{equation*}
    \nabla^{2}_{\rho, q} G_{k}(\rho^{R,\epsilon}, q^{R,\epsilon}) = \rho^{R,\epsilon} \nabla^{2}_{\rho, q} g_{k}(\rho^{R,\epsilon}, q^{R,\epsilon}) + 2(1, 0) \otimes \nabla_{\rho, q} g_{k}(\rho^{R,\epsilon}, q^{R,\epsilon}),
    \end{equation*}
    as follows, 
    \begin{align*}
    \sum_{k = 1}^{\infty} \mathbb{E} \int_{0}^{T} \int_{\mathbb{T}} \chi_{R}(\rho^{R,\epsilon}, q^{R,\epsilon})^{2} &|\langle \nabla^{2}_{\rho, q} G^{R, \epsilon}_{k}(\rho^{R,\epsilon}, q^{R,\epsilon}) \partial_{x}(\rho^{R,\epsilon}, q^{R,\epsilon}), \partial_{x}(\rho^{R,\epsilon}, q^{R,\epsilon}) \rangle|^{2} dx dt \\
    &\le C A_{0} \mathbb{E} \int_{0}^{T} \int_{\mathbb{T}} \chi_{R}(\rho^{R,\epsilon}, q^{R,\epsilon})^{2}(1 + (\rho^{R,\epsilon})^{2}) \Big((\partial_{x}\rho^{R,\epsilon})^{4} + (\partial_{x}q^{R,\epsilon})^{4}\Big) dx dt \\
    &\le C_{R, \epsilon} \mathbb{E} \int_{0}^{T} \int_{\mathbb{T}} \chi_{R}(\rho^{R,\epsilon}, q^{R,\epsilon}) \Big((\partial_{x}\rho^{R,\epsilon})^{4} + (\partial_{x}q^{R,\epsilon})^{4}\Big) dx dt \le C_{R, \epsilon} T,
    \end{align*}
    by the uniform bound on $\|\rho^{R,\epsilon}\|_{C(0, T; C(\mathbb{T}))} \le C_{R, \epsilon}$ and $\|\rho^{R,\epsilon}\|_{L^{2}(0, T; H^{2}(\mathbb{T}))}^{2} \le C_{R, \epsilon} T$ in Proposition \ref{densityest}{\if 1 = 0 and the bound in \eqref{uhigher2}\fi}, for a constant $C_{R, \epsilon}$ that is independent of $T$. Thus, we conclude that
    \begin{equation*}
    \frac{\epsilon}{2} \mathbb{E} \int_{0}^{T} \int_{\mathbb{T}} |\partial_{x}^{3}q^{R,\epsilon}|^{2} dx dt \le C_{R, \epsilon} T.
    \end{equation*}
\end{proof}

\section{Invariant measure for the approximate system}

In this section, we use the uniform-in-time estimates of Section \ref{uniformsection} to show that the approximate system \eqref{approxsystem} has an invariant measure, associated to the Feller semigroup $\{\mathcal{P}_{t}\}_{t \ge 0}$, defined in \eqref{ptdef} and Proposition \ref{fellerprop}. We use a standard time-averaging argument to establish existence of an invariant measure to \eqref{approxsystem} by averaging the laws of the stochastic solution $(\rho^{R,\epsilon}, q^{R,\epsilon})$ to \eqref{approxsystem} with initial condition $(1,0)$ over larger and larger time intervals $[0, T]$. This invariant measure corresponds to a stationary solution to the approximate system \eqref{approxsystem}, which we can denote by $(\rho_R, q_R)$, where we leave the $\epsilon_{N}$ dependence implicit. We then use the results on uniform invariant regions (independent of $R$) in Section \ref{uniforminvariant}, to deduce uniform $L^{\infty}(\mathbb{T}) \times L^{\infty}(\mathbb{T})$ bounds on the stationary solutions $(\rho_R, q_R)$ independently of the approximation parameters $R$. These uniform estimates (independent of $R$) will help us subsequently pass to the limit in the stationary solutions $(\rho_R, q_R)$ to the approximate system \eqref{approxsystem}, as $R \to \infty$. 

\subsection{Existence of an invariant measure.} We will show the existence of an invariant measure to the approximate problem \eqref{approxsystem}. To do this, we will define an appropriate path space, which we recall from \eqref{path}:
\begin{equation*}
\mathcal{X} := \left\{(\rho, q) \in H^{2}(\mathbb{T}) \times H^{2}(\mathbb{T}) : \int_{\mathbb{T}} \rho(x) dx = 1 \text{ and } \rho \geq \frac1R \right\}.
\end{equation*}
We will use a standard Krylov-Bogoliubov (time-averaging) procedure to show the existence of an invariant measure to the approximate system \eqref{approxsystem} on $\mathcal{X}$, associated to the Feller semigroup $\mathcal{P}_{t}$ in \eqref{ptdef}, generated by the (Hadamard well-posed) dynamics of the approximate system. See Proposition \ref{fellerprop}.

Hence, we consider the approximate system \eqref{approxsystem} with initial data $(\rho_0, q_0) = (1, 0)$, and we note that there is a unique solution $(\rho^{R,\epsilon},q^{R,\epsilon})$ that exists globally in time starting from this initial data $(\rho_0, (\rho u)_0) = (1, 0)$. We then define the \textit{time-averaged measures}:
\begin{equation}\label{muT}
\mu^{R,\epsilon}_T(B) = \frac{1}{T} \int_{0}^{T} \mathbb{P}\Big((\rho^{R,\epsilon}(t), q^{R,\epsilon}(t)) \in B \ \Big\vert \ (\rho_0, q_0) = (1, 0)\Big) dt.
\end{equation}
We will obtain an invariant measure for the approximate system \eqref{approxsystem} as a weak limit of the measures $\mu^{R,\epsilon}_T$ as $T \to \infty$, which will correspond to a statistically stationary solution with paths in $\mathcal{X}$ to the approximate problem \eqref{approxsystem}. To show that such a weak limit exists, we must show that the time-averaged measures $\{\mu^{R,\epsilon}_T\}_{T \in \mathbb{N}}$ are tight as measures on the phase space $\mathcal{X}$, which is the content of the following proposition. The tightness will be a direct consequence of the uniform bounds we have established for the fluid density and the fluid velocity, independently of time.

\begin{proposition}\label{timeaverage}
The time-averaged measures $\{\mu^{R,\epsilon}_T\}_{T \in \mathbb{N}}$ are tight as measures on $\mathcal{X}$. 
\end{proposition}

\begin{proof}
We recall from Proposition \ref{densityest} and Proposition \ref{fluidapprox} the following uniform bounds on the density and momentum, where the constants $C$ are independent of $T$ (but may depend on the parameters $R$ and $\epsilon$): 
\begin{align*}
&\frac{1}{T} \int_{0}^{T} \|\rho^{R,\epsilon}(t, \cdot)\|_{H^{3}(\mathbb{T})}^{2} dt \le C \quad \text{ almost surely},\\
&\frac{1}{T} \mathbb{E} \int_{0}^{T} \|q^{R,\epsilon}(t)\|_{H^{3}(\mathbb{T})}^{2} dt \le C.
\end{align*}
Therefore, we define the set
\begin{equation*}
K_{M} := \left\{(\rho, q) \in \mathcal{X} : \|\rho\|_{H^{3}(\mathbb{T})} \le M, \|q\|_{H^{3}(\mathbb{T})} \le M \right\}.
\end{equation*}
and we note that $K_{M}$ is a compact subset of the phase space $\mathcal{X}$ by standard compact embeddings. We can hence show tightness by considering an arbitrary $\varepsilon > 0$ and showing that for a uniform constant $M$ (depending potentially on $\epsilon$), we have that $\mu^{R,\epsilon}_T(K_{M}) \ge 1 - \varepsilon$ for all $T \in \mathbb{N}$. We hence calculate that
\begin{align*}
\mu^{R,\epsilon}_T(K_{M}) &= \frac{1}{T} \int_{0}^{T} \mathbb{P}\left(\|\rho^{R,\epsilon}(t)\|_{H^{3}(\mathbb{T})} \le M, \|q^{R,\epsilon}(t)\|_{H^{3}(\mathbb{T})} \le M\right) dt \\
&\ge 1 - \frac{1}{T} \int_{0}^{T} \left[\mathbb{P}\Big(\|\rho^{R,\epsilon}(t)\|_{H^{3}(\mathbb{T})} > M\Big) + \mathbb{P}\Big(\|q^{R,\epsilon}(t)\|_{H^{3}(\mathbb{T})} > M\Big)\right] dt \\
&\ge 1 - \frac{1}{TM^{2}} \int_{0}^{T} \mathbb{E} \Big(\|\rho^{R,\epsilon}(t)\|_{H^{3}(\mathbb{T})}^{2} + \|q^{R,\epsilon}(t)\|^{2}_{H^{3}(\mathbb{T})}\Big) dt \ge 1 - \frac{2C}{M^{2}}.
\end{align*}
Thus, choosing $M$ sufficiently large, since the constant $C$ is independent of $T$ in the previous estimate, we obtain the desired tightness result that $\mu^{R,\epsilon}_T(K_{M}) \ge 1 - \varepsilon$ for all $T \in \mathbb{N}$.
\end{proof}

Then, we can pass to the weak limit, via a standard Prokhorov theorem argument, in which tightness of probability measures implies weak convergence along a subsequence. 

\begin{corollary}\label{invmeasure}
    For fixed $R,\epsilon>0$, there exists an invariant measure $\mathcal{L}^{R,\epsilon}_{\rho,q}$ for $\mathcal{P}_t$ defined in \eqref{ptdef}, describing the dynamics of \eqref{approxsystem}.
\end{corollary}
\begin{proof}
    This is an immediate consequence of the Krylov–Bogoliubov Theorem (see Theorem 7.1 in \cite{DP06}) applied to Proposition \ref{timeaverage} and Proposition \ref{fellerprop}.
\end{proof}

Using the fact that the time-averaged laws $\mu^{R,\epsilon}_T$ defined in \eqref{muT} converge weakly to the invariant measure $\mathcal{L}^{R, \epsilon}_{\rho, q}$, we can deduce properties of the law of the invariant measure. The most important properties for the upcoming analysis are in the following proposition: (1) a uniform $L^{\infty}$ bound, which follows from the uniform-in-time bounds for the initial value problem in Proposition \ref{Linfuniform1}, and (2) the non-negativity of the density.
%{\cred what's the point of this proposition? we'll get the desired results for the stationary solution in the next corollary anyways since it solves the spde}
\begin{proposition}\label{Linfuniform}
    For the constant $C_{\epsilon}$ defined in Proposition \ref{Linfuniform1}, depending only on $\epsilon$, the invariant measure $\mathcal{L}^{R,\epsilon}_{\rho, q}$ satisfies:
    \begin{equation}\label{port1}
    \mathcal{L}^{R,\epsilon}_{\rho, q}\Big(\left\{(\rho, q) \in \mathcal{X}:\|(\rho, q)\|_{L^{\infty}(\mathbb{T}) \times L^{\infty}(\mathbb{T})} \le C_{\epsilon}\right\}\Big) = 1.
    \end{equation}
    In addition,
    \begin{equation}\label{port2}
    \mathcal{L}^{R, \epsilon}_{\rho, q}\left(\left\{(\rho, q) \in \mathcal{X}:\left\|\frac{q}{\rho}\right\|_{L^{\infty}(\mathbb{T})} \le C_{\epsilon} \text{ and } \rho(x) \ge { \frac1R} \text{ for all } x \in \mathbb{T}\right\}\right) = 1
    \end{equation}
    and
    \begin{equation}\label{port3}
    \mathcal{L}^{R, \epsilon}_{\rho, q}\left(\left\{(\rho, q) \in \mathcal{X} : \rho(x) \ge { \frac1R} \text{ and } \int_{\mathbb{T}} \rho(x) dx = 1\right\}\right) = 1.
    \end{equation}
\end{proposition}

\begin{proof}
    Note that the set $\{(\rho, q) \in \mathcal{X} : \|(\rho, q)\|_{L^{\infty}(\mathbb{T}) \times L^{\infty}(\mathbb{T})} \le C_{\epsilon}\}$ is a closed set in $\mathcal{X}$ defined in \eqref{path}, since $H^{2}(\mathbb{T}) \subset L^{\infty}(\mathbb{T})$. Hence, we can use weak convergence of the time-averaged measure $\mu^{R,\epsilon}_T$ in \eqref{muT} as $T \to \infty$ to the invariant measure $\mathcal{L}^{R,\epsilon}_{\rho, q}$, combined with Portmanteau's theorem, to obtain the desired result. Namely, for the initial value problem started with initial data $(1, 0)$, we have by Proposition \ref{Linfuniform1} that
    \begin{equation*}
    \mathbb{P}\Big(\|(\rho^{R,\epsilon}(t), q^{R,\epsilon}(t)\|_{L^{\infty}(\mathbb{T}) \times L^{\infty}(\mathbb{T})} \le C_{\epsilon}\Big) = 1, \qquad \text{ for all } t \ge 0.
    \end{equation*}
    Hence, for all $T \in \mathbb{N}$:
    \begin{equation*}
    \mu^{R,\epsilon}_T\Big(\{(\rho, q) \in \mathcal{X} : \|(\rho, q)\|_{L^{\infty}(\mathbb{T}) \times L^{\infty}(\mathbb{T})} \le C_{\epsilon}\Big) = 1.
    \end{equation*}
    So by Portmanteau's theorem for closed sets:
    \begin{equation*}
    \mathcal{L}^{R,\epsilon}_{\rho, q}\Big(\{(\rho, q) \in \mathcal{X} : \|(\rho, q)\|_{L^{\infty}(\mathbb{T}) \times L^{\infty}(\mathbb{T})} \le C_{\epsilon}\Big) \ge \lim_{T \to \infty} \mu^{R,\epsilon}_T\Big(\{(\rho, q) \in \mathcal{X} : \|(\rho, q)\|_{L^{\infty}(\mathbb{T}) \times L^{\infty}(\mathbb{T})} \le C_{\epsilon}\Big) = 1,
    \end{equation*}
    which establishes the desired result in \eqref{port1}.

    To prove the second result in \eqref{port2}, note that the set
    \begin{equation*}
    \left\{(\rho, q) \in \mathcal{X} : 1/R \le \rho(x) \le C_{\epsilon} \text{ and } \left\|\frac{q}{\rho}\right\|_{L^{\infty}(\mathbb{T})} \le C_{\epsilon}\right\}
    \end{equation*}
    is a closed set in $\mathcal{X}$. So by Proposition \ref{rholowerbound}, Proposition \ref{Linfuniform1}, and a similar Portmanteau theorem argument:
    \begin{equation*}
    \mathcal{L}^{R, \epsilon}_{\rho, q}\left(\left\{(\rho, q) \in \mathcal{X} : 1/R \le \rho(x) \le C_{\epsilon} \text{ and } \left\|\frac{q}{\rho}\right\|_{L^{\infty}(\mathbb{T})} \le C_{\epsilon}\right\}\right) = 1.
    \end{equation*}
    This establishes the result in \eqref{port2}. Finally, for \eqref{port3}, note that 
    \begin{equation*}
    \left\{(\rho, q) \in \mathcal{X} : \rho(x) \ge 1/R \text{ for all } x \in \mathbb{T} \text{ and } \int_{\mathbb{T}} \rho(x) dx = 1\right\}
    \end{equation*}
    is a closed set in $\mathcal{X}$, so a similar Portmanteau theorem argument works, once we note that $\rho^{R, \epsilon}(x) \ge 1/R$ for all $x \in \mathbb{T}$ almost surely by Proposition \ref{rholowerbound}, and furthermore, $\displaystyle \int_{\mathbb{T}} \rho^{R, \epsilon}(t, x) dx = 1$ for all $t \ge 0$ almost surely.
\end{proof}

We can directly translate the existence of an invariant measure $\mathcal{L}^{R,\epsilon}_{\rho, q}$ for the approximate system \eqref{approxsystem} into the existence of a statistically stationary solution to the approximate system \eqref{approxsystem}.

\begin{corollary}\label{stationaryReps}
There exists a statistically stationary solution $(\rho^S_{R, \epsilon}, q^S_{R, \epsilon})$ for the approximate system \eqref{approxsystem}, which furthermore satisfies the uniform bounds almost surely:
\begin{equation}\label{RepsLinf}
\|(\rho_{R, \epsilon}^S, q_{R, \epsilon}^S)\|_{C(\R_{+}; L^{\infty}(\mathbb{T}) \times L^{\infty}(\mathbb{T}))} \le C_{\epsilon}, %\qquad \|(\rho_{R, \epsilon}^S)^{-1}\|_{L^{\infty}(\mathbb{T})} \le R,
\qquad \left\|\frac{q^S_{R,\epsilon}}{\rho^S_{R,\epsilon}}\right\|_{C(\R_{+}; L^{\infty}(\mathbb{T}))} \le C_{\epsilon},
\end{equation}
for a constant $C_{\epsilon}$, depending only on $\epsilon$, with $\rho_{R, \epsilon}^S(t, x) \ge 1/R$ for all $x \in \mathbb{T}$ and $\displaystyle \int_{\mathbb{T}} \rho_{R, \epsilon}^S(t, x) dx = 1$ for all $t \ge 0$, almost surely.
\end{corollary}

\begin{proof}
    To construct the approximate solution, consider a stochastic basis $(\Omega,\sF,(\sF_t)_{t\geq 0},W)$ and (random) initial data $(\rho_{0}, q_{0})$ with law given by the invariant measure $\mathcal{L}^{R,\epsilon}_{\rho, q}$. Denote the resulting (unique) strong pathwise solution to \eqref{approxsystem} taking continuous paths in $\mathcal{X}$ by $(\rho_{R, \epsilon}^S(t), q_{R, \epsilon}^S(t))$. %{\cred Now, since $\mathcal{P}_t$ is Feller, the definition of invariant measure is equivalent to $\mathcal{P}^*_{t}\mathcal{L}^{R,\epsilon}_{\rho, q} = \mathcal{L}^{R,\epsilon}_{\rho, q}$ where $\mathcal{P}^*_t$ is the transpose of $\mathcal{P}_t$.

   % {\bf Or we can just do the following:}
    By the definition of invariant measure for $\mathcal{P}_t$ given in \eqref{ptdef} and the fact that $(\rho_0,q_0)\sim \mathcal{L}^{R,\epsilon}_{\rho,q}$, we know that for any $\varphi\in C_b(\mathcal{X})$:
  \begin{align*}
      \int_\mathcal{X} \varphi(x) \mathcal{L}^{R,\epsilon}_{\rho, q}(dx)&=\int_\mathcal{X} \mathcal{P}_t \varphi(x)\, \mathcal{L}^{R,\epsilon}_{\rho, q}(dx) =	\mathbb{E}[\mathcal{P}_t \varphi(\rho_0,q_0)] =\mathbb{E}[\varphi(\rho^S_{R,\epsilon}(t),q^S_{R,\epsilon}(t))],
  \end{align*}
  which means that $(\rho^S_{R,\epsilon}(t),q^S_{R,\epsilon}(t))\sim \mathcal{L}^{R,\epsilon}_{\rho, q}$ for all $t\geq 0$.

    Hence, we have that the resulting solution $(\rho_{R, \epsilon}^S(t), q_{R, \epsilon}^S(t))$ is a statistically stationary solution with continuous paths in $\mathcal{X}$. The bound in $C(\R_{+}; L^{\infty}(\mathbb{T}) \times L^{\infty}(\mathbb{T}))$ follows directly from Proposition \ref{Linfuniform}. We remark that we can establish the uniform lower bound $\rho^S_{R,\epsilon}(t, x) \ge 1/R$ for $x \in \mathbb{T}$ almost surely, where on this almost sure set, this lower bound holds for all $t \ge 0$ simultaneously. This is because of the continuity in time $\rho_{R, \epsilon}^S \in C(0, T; \mathcal{X}) \subset C(0, T; C(\mathbb{T}))$ by Proposition \ref{rholowerbound}.
\end{proof}

Finally, we observe the following uniform bound (independent of $R$), which is a consequence of the uniform $L^{\infty}(\mathbb{T})$ bound in \eqref{RepsLinf}. {{This will be important for obtaining control of the vacuum set where the density is zero, in the subsequent limit passage as $R \to \infty$ (see Proposition \ref{rhoepspositive}), since the uniform constants in this proposition are independent of $R$.}}

\begin{proposition}\label{logbounds}
    The stationary solution $(\rho_{R, \epsilon}^S, q_{R, \epsilon}^S)$ satisfies the following uniform bounds for all $t \ge 0$:
    \begin{equation*}
    \mathbb{E}\|\partial_{x}(\log(\rho_{R, \epsilon}^S(t, x)))\|_{L^{2}(\mathbb{T})}^{2} \le C_{\epsilon}, \qquad \mathbb{E}\|\log(\rho_{R, \epsilon}^S(t))\|_{L^{\infty}(\mathbb{T})}^{2} \le C_{\epsilon},
    \end{equation*}
    for a constant $C_{\epsilon}$ depending only on $\epsilon$ (and not on $R$ or $t \ge 0$).
\end{proposition}

\begin{proof}
    By Corollary \ref{stationaryReps}, we can test the continuity equation by $1/\rho_{\epsilon}$ since $\rho_{\epsilon}(t, x) \ge 1/R > 0$ for all $x \in \mathbb{T}$ and $t \ge 0$ almost surely. We thus obtain for arbitrary $T > 0$ that for all $t \in [0, T]$:
    \begin{equation*}
    \int_{0}^{T} \int_{\mathbb{T}} \frac{\partial_{t}\rho_{R, \epsilon}^S}{\rho_{R, \epsilon}^S} + \int_{0}^{T} \int_{\mathbb{T}} \frac{\partial_{x}([q_{R, \epsilon}^S]_{R})}{\rho_{R, \epsilon}^S} = \epsilon \int_{0}^{T} \int_{\mathbb{T}} \frac{\partial_{x}^{2}\rho_{R, \epsilon}^S}{\rho_{R, \epsilon}^S}.
    \end{equation*}
    By integrating by parts:
    \begin{equation*}
    \int_{\mathbb{T}} \log(\rho_{R, \epsilon}^S)(T) - \int_{\mathbb{T}} \log(\rho_{R, \epsilon}^S)(0) + \int_{0}^{T} \int_{\mathbb{T}} \frac{\partial_{x}\rho_{R, \epsilon}^S}{\rho_{R, \epsilon}^S} \cdot \frac{[q_{R, \epsilon}^S]_{R}}{\rho_{R, \epsilon}^S} = \epsilon \int_{0}^{T} \int_{\mathbb{T}} \frac{(\partial_{x}\rho_{R, \epsilon}^S)^{2}}{(\rho_{R, \epsilon}^S)^{2}}.
    \end{equation*}
    After taking expectation on both sides, we can use stationarity of $(\rho_{R, \epsilon}^S, q_{R, \epsilon}^S)$ to obtain for all $t \ge 0$:
    \begin{equation*}
    \epsilon \mathbb{E} \int_{\mathbb{T}} \frac{(\partial_{x}\rho_{R, \epsilon}^S)^{2}}{(\rho_{R, \epsilon}^S)^{2}} = \mathbb{E} \int_{\mathbb{T}} \frac{\partial_{x}\rho_{R, \epsilon}^S}{\rho_{R, \epsilon}^S} \cdot \frac{[q_{R, \epsilon}^S]_{R}}{\rho_{R, \epsilon}^S}.
    \end{equation*}
    So by Cauchy's inequality and the uniform $L^{\infty}(\mathbb{T})$ bounds in Proposition \ref{Linfuniform}:
    \begin{equation}\label{dxlog}
    \mathbb{E} \int_{\mathbb{T}} \Big(\partial_{x}\Big[\log(\rho_{R, \epsilon}^S(t))\Big]\Big)^{2} = \mathbb{E} \int_{\mathbb{T}} \frac{(\partial_{x}\rho_{R, \epsilon}^S)^{2}}{(\rho_{R, \epsilon}^S)^{2}} \le C_{\epsilon} \mathbb{E} \int_{\mathbb{T}} \left(\frac{[q_{R, \epsilon}^S]_{R}}{\rho_{R, \epsilon}^S}\right)^{2} \le C_{\epsilon}.
    \end{equation}
    This establishes the first estimate. For the second estimate, note that by Corollary \ref{stationaryReps}, $\displaystyle \int_{\mathbb{T}} \rho_{R, \epsilon}^S(t, x) dx = 1$, so there must exist some point $x_0 \in \mathbb{T}$ such that $\rho_{R, \epsilon}^S(x_0) = 1$. Then,
    \begin{equation*}
    |\log(\rho_{R, \epsilon}^S(t, x))|^{2} \le \left|\int_{x_{0}}^{x} \partial_{x}\Big(\log(\rho_{R, \epsilon}^S(t, x))\Big) dx\right|^{2} \le \int_{\mathbb{T}} \Big(\partial_{x}\Big[\log(\rho_{R, \epsilon}^S(t, \cdot))\Big]\Big)^{2},
    \end{equation*}
    so the result follows by taking expectations and using the bound \eqref{dxlog}.
\end{proof}

\section{Limit passage $R \to \infty$. }

In the previous section, we obtained an invariant measure for the approximate system \eqref{approxsystem}. This invariant measure corresponds to a (stochastic) stationary solution, see Proposition \ref{stationaryReps}. While these solutions depend also on the parameter $\epsilon$, we omit the explicit dependence on $\epsilon$ in this section for convenience of notation, as we focus on passing $R\to\infty$ for fixed but arbitrary $\epsilon > 0$. That is, 
\begin{notation} 
We denote by {$\boldsymbol{U}_R=(\rho_{R}, q_{R})$} the stationary solution constructed in Corollary \ref{stationaryReps} (denoted earlier by $(\rho^S_{R,\epsilon}, q^S_{R,\epsilon})$) which is the strong pathwise solution to \eqref{approxsystem} whose law at every time $t \ge 0$ is given by
the invariant measure $\mathcal{L}^{R,\epsilon}_{\rho, q}$ constructed in Corollary \ref{invmeasure}.
\end{notation}

 Most of the estimates obtained in the previous section depended on the parameter $R$. To prepare for the limit passage as $R \to \infty$, we thus need to obtain estimates on the approximate solutions $({\rho}_{R}, {q}_{R})$ independently of $R$. Note that upon obtaining the statistically stationary solution, the information about initial conditions is lost, as it no longer makes sense to talk about an initial value problem (since these statistically stationary solutions are obtained by time averaging). However, we will often use the fact that the law of the solution is equal for all times to recover uniform bounds on the approximate statistically stationary solutions.

In addition, we note that the approximation scheme \eqref{approxsystem} at the level of $R$ is still in conservation form. Namely, we note that an entropy flux pair $(\eta, H)$ for the original system \eqref{eulerU} corresponds to an entropy flux pair $\Big(\eta, \chi_R({ \rho,q}) H\Big)$ for the approximate system \eqref{approxsystem}. Hence, the approximate problem has the following entropy equality, which is satisfied for all entropy-flux pairs $(\eta, H)$ for the original problem, and for all test functions $\varphi \in C^{2}(\mathbb{T})$ and nonnegative $\psi \in C^{\infty}_{c}(0, \infty)$:
\begin{multline}\label{entropyR}
\int_{0}^{\infty} \left(\int_{\mathbb{T}} \eta(\bd{U}_{R}(t)) \varphi(x) dx\right) \partial_{t}\psi(t) dt + \int_{0}^{\infty} \chi_R(\rho,q) \left(\int_{\mathbb{T}} H(\bd{U}_{R})\partial_{x}\varphi(x) dx\right) \psi(t) dt \\
- \int_{0}^{\infty} \left(\int_{\mathbb{T}} \alpha q_{R} \partial_{q}\eta(\bd{U}_{R}) \varphi(x) dx\right) \psi(t) dt + \int_{0}^{\infty} \left(\int_{\mathbb{T}} \partial_{q}\eta(\bd{U}_{R}) \bd{\Phi}^{R, \epsilon}(\bd{U}_{R}) \varphi(x) dx\right) \psi(t) dW(t) \\
+ \int_{0}^{\infty} \left(\int_{\mathbb{T}} \frac{1}{2} \partial_{q}^{2} \eta(\bd{U}_{R}) G^{2}_{R, \epsilon}(\bd{U}_{R}) \varphi(x) dx\right) \psi(t) =  \epsilon \int_{0}^{\infty} \left(\int_{\mathbb{T}} \langle D^{2}\eta(\bd{U}_{R})\partial_{x}\bd{U}_{R}, \partial_{x}\bd{U}_{R} \rangle \varphi(x) dx\right) \psi(t) dt \\
- \epsilon \int_{0}^{\infty} \left(\int_{\mathbb{T}} \eta(\bd{U}_{R}) \partial^{2}_{x}\varphi dx\right) \psi(t) dt. 
\end{multline}
Our goal is to use uniform bounds on the solutions (independent of the parameter $R$) in order to pass to the limit as $R \to \infty$ in the approximate solutions $(\rho_{R}, q_{R})$ and also in the approximate entropy equality \eqref{entropyR}. 

We aim to obtain almost sure strong convergence as $R \to \infty$ for our approximate stationary solutions. For that purpose, we will apply Jakubowski's version of the classical Skorohod representation theorem \cite{J97}.

Consider the phase space:
\begin{equation}\label{phaseeps}
\begin{split}
    \mathcal{S} := [C_{\text{loc}}([0, \infty); L^{2}(\mathbb{T})) \cap L^2_{\text{loc}}([0,\infty);H^1(\mathbb{T})) &\cap (L^{2}_{loc}([0, \infty); H^{2}(\mathbb{T})), w)]^2 \\
&\times (L^{2}_{loc}([0, \infty); L^{2}(\mathbb{T})), w) \times C(0, \infty; \mathcal{{U}}),
\end{split}
\end{equation}
and let $\mu_{R}$ denote the law of the random variable 
\begin{equation*}
\mu_{R} = \text{Law}_{\mathcal{S}}\left(\rho_{R}, q_{R}, \frac{\partial_{x}\rho_{R}}{\rho_{R}}, W\right),
\end{equation*}
in $\mathcal{S}$, where $\bd{U}_{R} := (\rho_{R}, q_{R})$. Here, $(X, w)$ denotes the space $X$ equipped with the weak topology. The goal will be to show that the laws $\mu_{R}$ of the approximate statistically stationary solutions are tight as probability measures on the phase space $\mathcal{S}$.

\subsection{Uniform bounds on $(\rho_{R}, q_{R})$ in $R$.} To show tightness of the laws $\mu_{R}$, we will establish uniform bounds on the approximate statistically stationary solutions $(\rho_{R}, q_{R})$, independently of $R$. We first show that the following energy bounds are satisfied by the approximate statistically stationary solutions.
\begin{proposition}\label{energyR}
    For any $T>0$ the stationary solution $(\rho_R,q_R)$ to \eqref{approxsystem} satisfies the following bounds, where $C_{\epsilon, T}$ depends only on $\epsilon > 0$ and $T > 0$, and is independent of $R$:
    \begin{align*}
         \bE\|q_R\|^2_{L^2(0,T;H^1(\mathbb{T}))}\leq C_{\epsilon,T} \quad\text{ and }   \quad \bE\|\rho_R(t)\|^2_{H^1(\mathbb{T})} \leq C_\epsilon {\text{ for every $t\geq 0$}.} \end{align*}
\end{proposition}
\begin{proof}
    The proof for $ \bE\|q_R\|^2_{L^2(0,T;H^1(\mathbb{T}))} \leq C_{\epsilon,T} $ follows identically the proof of the bounds found in \eqref{uhigher1} which are independent of $R$.

    We test \eqref{approxsystem}$_{2}$ by $\rho_R$ and apply Corollary \ref{stationaryReps}:
    \begin{align*}
       \frac12 \frac{d}{dt}\|\rho_R(t)\|_{L^2(\mathbb{T})}^2 +\epsilon\int_{\mathbb{T}}|\partial_x\rho_R(t)|^2 &=  \int_{\mathbb{T}}\chi_{R}(\rho,q) q_R(t)\partial_x \rho_R(t) dx\\
       &
       \leq \int_{\mathbb{T}}q_R^2(t) + \frac{\epsilon}{2}\int_{\mathbb{T}}|\partial_x\rho_R(t)|^2 \leq   C_\epsilon + \frac{\epsilon}{2}\int_{\mathbb{T}}|\partial_x\rho_R(t)|^2  .
    \end{align*}
    We apply expectation on both sides and use stationarity of $(\rho_R,q_R)$ to obtain for any $t\geq0$ that
    \begin{align*}
\epsilon\bE \int_{\mathbb{T}}|\partial_x\rho_R(t)|^2 \leq  C_\epsilon.
    \end{align*}
\end{proof}
We will next upgrade these bounds and obtain the following tightness result for the approximate laws $\mu_{R}$ in $\mathcal{S}$, defined in \eqref{phaseeps}.

\begin{proposition}\label{tightnessR}
For any fixed $\epsilon>0$, the laws $\mu_{R}$ are tight in the path space $\mathcal{S}$, and hence, $\mu_{R}$ converges weakly along a subsequence as $R \to \infty$ to some limiting probability measure $\mu_\epsilon$ on $\mathcal{S}$.  
\end{proposition}

\begin{proof}
{\bf Part 1: Tightness of the laws of $\rho_R$:} We will begin by proving that for any $T>0$:
\begin{align}
%&\mathbb{E} \| q_{R}\|^2_{H^{\alpha}(0, T; L^{2}(\mathbb{T}) )} \le C_{\alpha,\epsilon, T}
&\mathbb{E} \|\rho_{R}\|^2_{H^1(0,T;L^2(\mathbb{T}))\cap L^{2}(0, T; H^2(\mathbb{T}))} \le C_{\epsilon, T}, \label{rhotightbound1}\\
&\mathbb{E} \|\rho_{R}\|^2_{C^{0,\frac14}(0, T; H^{\frac12}(\mathbb{T}))} \le C_{\epsilon, T} \label{rhotightbound},
%\qquad \mathbb{E} \|(\rho_{R}, q_{R})\|^2_{L^{2}(0, T; H^{2}(\mathbb{T}) \times H^{2}(\mathbb{T}))} \le C_{\alpha,\epsilon, T}.
\end{align}
for some constant $ C_{\epsilon, T}>0$ independent of $R$.

   Observe as before (see the proof of Proposition \ref{densityest}), that using maximal regularity for $\rho_R$ satisfying \eqref{approxsystem}$_1$ and applying the bounds found in Proposition \ref{energyR}, we have that for some $C_{\epsilon, T}>0$ depending only on $\epsilon$ and $T$:
\begin{equation}
 \begin{split}\label{rhoh2}
   \bE \|\partial_{t}\rho_R\|^2_{L^{2}(0,T; L^{2}(\mathbb{T}))} + \bE\|\partial_{xx} \rho_R\|^2_{L^{2}(0,T; L^{2}(\mathbb{T}))} &\le C\Big(\bE\|\rho_R(0, \cdot)\|^2_{H^{1}(\mathbb{T})} + \bE\|\partial_{x}q_R\|^2_{L^{2}(0,T; L^{2}(\mathbb{T}))}\Big) \le C_{\epsilon, T},
    \end{split} 
    \end{equation}
    where we additionally used Proposition \ref{energyR}. This finishes the proof of \eqref{rhotightbound1}.
%    In other words, we have uniform in $R$ bounds for $\rho_R$ in $L^2(\Omega;H^1(0,T;L^2(\mathbb{T})))$. Since for $\alpha<\frac12$ we have $H^1(0,T)\hookrightarrow C^{0,\alpha}(0,T)$, we obtain,
Hence, thanks to Sobolev interpolating inequalities, we have,
\begin{align*}
    \bE \|\rho_R\|^2_{C^{0,\frac14}(0,T;H^\frac12(\mathbb{T}))} &\leq \bE \|\rho_R\|^2_{H^\frac34(0,T;H^\frac12(\mathbb{T}))} \leq \bE\left(\|\rho_R\|^\frac32_{L^2(0,T;H^2(\mathbb{T}))}\|\rho_R\|^\frac12_{H^1(0,T;L^2(\mathbb{T}))}\right) \\
  &\leq \bE\left(\|\rho_R\|^2_{L^2(0,T;H^2(\mathbb{T}))}\right)^{\frac34}\bE\left(\|\rho_R\|^2_{H^1(0,T;L^2(\mathbb{T}))}\right)^{\frac14}\leq C_{\epsilon,T}.
\end{align*}
This finishes the proof of \eqref{rhotightbound} which is key in proving the tightness result for the laws of $\rho_R$ because of the following compact embedding, which follows from the Arzela-Ascoli theorem:
    \begin{equation*}
     C^{0,\frac14}(0, T; H^{\frac12}(\mathbb{T}) )\subset \subset C(0, T; L^{2}(\mathbb{T}) ). 
    \end{equation*}
     
    Recall that %$C_{\text{loc}}([0,\infty);L^2(\mathbb{T}))$ is a separable metric space and thus 
    a set $K$ is compact in $C_{\text{loc}}([0,\infty);L^2(\mathbb{T}))$ if $K_{T} = \{f|_{[0,T]};f\in K\}$ is compact in $C(0, T ;L^2(\mathbb{T}))$ for every $T\in\mathbb{N}$. This fact follows from a diagonalization argument i.e. by obtaining a subsequence $\{f_{j_k}^1\}$ converging in $K|_{[0,1]}$ of a sequence $\{f_j\}$ bounded in $K$ and then thinning it for every $T\in\mathbb{N}$ to extract a subsequence of $\{f_j\}$ converging in $C_{\text{loc}}([0,\infty);L^2(\mathbb{T}))$.
    Hence, we obtain tightness of the laws of $\rho_R$ in $C_{\text{loc}}([0,\infty);L^2(\mathbb{T}))$ by an application of the Chebyshev inequality. The same is true for tightness of the laws of $\rho_R$ in $L^2_{loc}([0,\infty);H^1(\mathbb{T}))$ thanks to the Aubin-Lions theorem that states that
    $$L^2(0,T;H^2(\mathbb{T}))\cap H^1(0,T;L^2(\mathbb{T})) \subset\subset L^2(0,T;H^1(\mathbb{T})).$$
    Finally, tightness of laws in $(L^{2}_{loc}([0, \infty); H^{2}(\mathbb{T})), w)$ is immediate from the uniform bounds in \eqref{rhotightbound1}.

    \medskip

    \iffalse
    Now, for any $M>0$, consider the set
    $$B_M:= \{\rho\in C_{\text{loc}}([0,\infty);L^2(\mathbb{T})): \|\rho\|_{  C^{0,\frac14}_{\text{loc}}([0, \infty); H^{\frac12}(\mathbb{T}) )}\leq M \}.$$
    
    Hence, for any $M>0$, thanks to the Chebyshev inequality we write
    \begin{align*}
        \bP\left( \rho_{R} \notin B_M\right) \leq \bP\left(\|\rho_R\|_{C^{0,\frac14}_{\text{loc}}([0, \infty);H^\frac12(\mathbb{T}))} \geq {M}\right) %\leq \frac{1}{M}\bE{\cred\left(\sum_{T\in\mathbb{N}}\frac1{2^T}\left(\frac{\|\rho_R\|_{C^{0,\frac14}([0, T];H^\frac12(\mathbb{T})) }}{\|\rho_R\|_{C^{0,\frac14}([0, T];H^\frac12(\mathbb{T})) }+ 1 }\right)\right)} 
        \leq \frac{C}{M}.
    \end{align*}
    \fi

\noindent{\bf Part 2: Tightness of the laws of $q_R$:}
To prove the second half of the tightness result, we will prove that for any $T>0$, we have
\begin{align}\label{qh2}
   \bE\sup_{t\in [0,T]}\|q_R(t)\|^2_{H^1(\mathbb{T})} +  \bE\|q_R\|^2_{L^2(0,T;H^2(\mathbb{T}))} \leq C_{\epsilon,T},
\end{align}
where the constant $C_{\epsilon,T}$ is independent of $R$.

We recall \eqref{ito_qh2} and take $\sup_{t\in[0,T]}$ and then expectation on both sides. We will next estimate each term on the right hand side, independently of $R$, of the resulting equation.
For that purpose, we will appeal to the uniform bounds obtained in Corollary \ref{stationaryReps} and the energy bounds Proposition \ref{energyR}. For the first term we thus obtain
    \begin{align*}
    \mathbb{E} \int_{0}^{T} \int_{\mathbb{T}} \left|\chi_R({\rho_R,q_R})\partial_{x}q_R \left(\frac{q_R}{\rho_R} \right)\partial_{x}^{2}q_R\right| dx dt &\le \frac{\epsilon}{4} \mathbb{E} \int_{0}^{T} \int_{\mathbb{T}} |\partial_{x}^{2}q_R|^{2} dx dt + C_{\epsilon} \mathbb{E} \int_{0}^{T} \int_{\mathbb{T}} |\partial_{x}q_R|^{2} \left(\frac{q_R}{\rho_R}\right)^{2} \\
    &\le C_{\epsilon,T} + \frac{\epsilon}{4} \mathbb{E} \int_{0}^{T} \int_{\mathbb{T}} |\partial_{x}^{2}q_R|^{2} dx.
    \end{align*}
    Similarly, by Corollary \ref{stationaryReps}  and Proposition \ref{energyR}, for the second term we obtain
    \begin{align*}
    \mathbb{E} \int_{0}^{T} \int_{\mathbb{T}} \left| q_R \chi_R({\rho_R,q_R})\partial_{x}\left(\frac{q_R}{\rho_R} \right)\partial_{x}^{2}q_R\right| dxdt&\le C_{\epsilon,T}\mathbb{E} \int_{0}^{T} \int_{\mathbb{T}}  \left|\frac{q_R\partial_{x}q_R}{\rho_R} - \frac{q_R^2\partial_{x}\rho}{\rho^2_R}\right|^2+ \frac{\epsilon}{6}\mathbb{E} \int_{0}^{T} \int_{\mathbb{T}}|\partial_{x}^{2}q_R|^2\\
     &\leq  C_{\epsilon,T}\mathbb{E} \int_{0}^{T} \int_{\mathbb{T}} |\partial_{x}q_R|^{2} + |\partial_{x}\rho_R|^{2} dx + \frac{\epsilon}{6} \mathbb{E} \int_{0}^{T} \int_{\mathbb{T}} |\partial_{x}^{2}q_R|^{2} dx\\
    &\leq  C_{\epsilon,T} + \frac{\epsilon}{6} \mathbb{E} \int_{0}^{T} \int_{\mathbb{T}} |\partial_{x}^{2}q_R|^{2} dx.
    \end{align*}
 Next, we  obtain the following estimate for any $\gamma\geq 1$ using Corollary \ref{stationaryReps}  and Proposition \ref{energyR}:
    \begin{align*}
    \kappa\gamma \mathbb{E} \int_{0}^{T} \int_{\mathbb{T}} \left|\rho_R^{\gamma - 1} \partial_{x}\rho_R \partial_{x}^{2} q_R \right|dx dt  &\le \frac{\epsilon}{6} \mathbb{E} \int_{0}^{T} \int_{\mathbb{T}} |\partial_{x}^{2}q_R|^{2}dx dt + C_{\epsilon} \mathbb{E} \int_{0}^{T} \int_{\mathbb{T}} \rho_R^{2(\gamma - 1)} |\partial_{x}\rho_R|^{2} dx dt \\
    &\le  \frac{\epsilon}{6} \mathbb{E} \int_{0}^{T} \int_{\mathbb{T}} |\partial_{x}^{2}q_R|^{2} dx dt+
    C_{\epsilon, T}  \mathbb{E} \int_{0}^{T} \int_{\mathbb{T}} |\partial_{x}\rho_R|^{2} dx dt \\
    &\le C_{\epsilon, T} + \frac{\epsilon}{6} \mathbb{E} \int_{0}^{T} \int_{\mathbb{T}} |\partial_{x}^{2}q_R|^{2} dx dt.
    \end{align*}
    Finally, we recall that $\nabla_{\rho, q}G_{k}^{R, \epsilon} = \rho \nabla_{\rho, q} g_{k}^{R, \epsilon} + (1, 0) g_{k}^{R, \epsilon}$ and we use \eqref{A0eps}, in addition to Corollary \ref{stationaryReps}  and Proposition \ref{energyR}:
 {  \begin{align*}
    \mathbb{E}  \int_{0}^{T} \sum_{k = 1}^{\infty}\int_{\mathbb{T}} |\nabla_{\rho, q} G_{k}^{R, \epsilon}(\rho_R, q_R) \cdot \partial_{x}(\rho_R, q_R)|^{2} dx dt &\le 2A_{0}^{2} \mathbb{E} \int_{0}^{T} \int_{\mathbb{T}} (1 + \rho_R^{2})\Big(|\partial_{x}\rho_R|^{2} + |\partial_{x}q_R|^{2}\Big) dx dt \\
    &\le C_{ \epsilon} T.
    \end{align*}}
    The final martingale term is treated by using the Burkholder-Davis-Gundy inequality,
    \begin{align}\label{bdgRest}
        \bE \sup_{t\in [0,T]} &\left| \sum_{k = 1}^{\infty} \int_0^t\int_{\mathbb{T}} \chi_{R}(\rho, q) \nabla_{\rho,q}G_k^{R,\epsilon}(\rho,q)\cdot \partial_x(\rho,q)q_x dxdW_{k}(t)\right| \nonumber \\
        &\leq C\bE\left( \int_0^T \sum_{k = 1}^{\infty} \left(\int_{\mathbb{T}}\chi_{R}(\rho, q) \nabla_{\rho,q}G_k^{R,\epsilon}(\rho,q)\cdot \partial_x(\rho,q)q_x\right)^2dt\right)^\frac12 \nonumber \\
        & \leq C\bE\left( \int_0^T \sum_{k = 1}^{\infty} \left(\Big\|\nabla_{\rho,q}G_k^{R,\epsilon}(\rho,q)\cdot \partial_x(\rho,q)\Big\|_{L^2(\mathbb{T})} \cdot \|q_x\|_{L^2(\mathbb{T})}\right)^2dt\right)^\frac12 \nonumber \\
        & \leq C\bE\left( \sup_{t\in[0,T]}\|q_x(t)\|_{L^2(\mathbb{T})}\left(\int_0^T \sum_{k = 1}^{\infty} \Big\|\nabla_{\rho,q}G_k^{R,\epsilon}(\rho,q)\cdot \partial_x(\rho,q)\Big\|_{L^2(\mathbb{T})}^2dt\right)^{\frac12}\right) \nonumber \\
        &\leq \frac14\bE\left(\sup_{t\in [0,T]}\|q_x\|^2_{L^2(\mathbb{T})}\right) + 8A_{0}^{2} \mathbb{E} \int_{0}^{T} \int_{\mathbb{T}} (1 + \rho_R^{2})\Big(|\partial_{x}\rho_R|^{2} + |\partial_{x}q_R|^{2}\Big) dx dt \nonumber\\
        &\leq \frac14\bE\left(\sup_{t\in [0,T]}\|q_x\|^2_{L^2(\mathbb{T})}\right) + C_{\epsilon,T}.
    \end{align}
    Thus, we conclude the proof of \eqref{qh2} by using the equation \eqref{ito_qh2} with supremum in time $t \in [0, T]$ and expectation applied to both sides, where we move the term $\displaystyle \frac{1}{4}\bE\left(\sup_{t \in [0, T]} \|q_x\|^{2}_{L^{2}(\mathbb{T})}\right)$ to the other side. Now we consider the approximate momentum equation in its integral form,
\begin{align*}
    q_R(t) &=
      \int_{0}^{t}\chi_R(\rho_R,q_R)\left(2\partial_{x}q_R \left(\frac{q_R}{\rho_R}\right) -\partial_{x}\rho_R \left(\frac{q_R^2}{\rho^2_R}\right) \right)dt
    + \kappa \gamma  \int_{0}^{t}  \chi_R(\rho_R,q_R)  \rho_R^{\gamma - 1} \partial_{x}\rho_R   dt\\
     &- \epsilon  \int_{0}^{t} \partial_{xx}q_R dt+ \int_0^t\chi_R\left(\rho_R,q_R\right)\bd{\Phi}^{R, \epsilon}(\rho_R, q_R)  dW =: I_{1}(t) + I_{2}(t) + I_{3}(t) + I_{4}(t).
    \end{align*}
    We will first show that the integrals $I_{i};i=1,2,3$ are bounded in $L^2(\Omega,H^1(0,T;L^2(\mathbb{T})))$.
    Again, we apply Corollary \ref{stationaryReps}  and the energy bounds $\bE\|q_R\|^2_{L^2(0,T;H^1(\mathbb{T}))} +  \bE\|\rho_R\|^2_{L^2(0,T;H^1(\mathbb{T}))}  \leq C_\epsilon$ from Proposition \ref{energyR} to obtain
\begin{align*}
    \bE\|\partial_tI_{1}\|^2_{L^2(0, T; L^{2}(\mathbb{T}))} &\leq C\bE\|\partial_xq_R\|^2_{L^2(0,T;L^2(\mathbb{T}))}\left\|\frac{q_R}{\rho_R}\right\|^2_{L^\infty(0,T;L^\infty(\mathbb{T}))} \\
    &\qquad+ C\bE\|\partial_x\rho_R\|^2_{L^2(0,T;L^2(\mathbb{T}))}\left\|\frac{q_R}{\rho_R}\right\|^4_{L^\infty(0,T;L^\infty(\mathbb{T}))} \leq C_{\epsilon,T}.
\end{align*}
Similarly for any $\gamma\geq 1$, we observe
\begin{align*}
   \bE \|\partial_tI_{2}\|^2_{L^2(0,T; L^{2}(\mathbb{T}))} \leq C\bE\|\partial_x\rho_R\|^2_{L^2(0,T;L^2(\mathbb{T}))}\|\rho_R^{\gamma-1}\|^2_{L^\infty(0,T;L^\infty(\mathbb{T}))}  \leq C_T.
\end{align*}
   Now we estimate $I_{3}$ by using \eqref{qh2} as follows:
   \begin{align*}
    \bE\|\partial_tI_{3}\|^2_{L^2(0,T; L^{2}(\mathbb{T}))}  \leq \epsilon \bE\|\partial_{xx}q_R\|^2_{L^2(0,T; L^{2}(\mathbb{T}))}  \leq C_{\epsilon,T}.
   \end{align*}
For the stochastic integral $I_4$, we use the fact that for any $\beta<\frac12$, $q \geq 2$ and Hilbert space $H$ (see Lemma 2.1 in \cite{FlandoliGatarek}),
\begin{align*}
    \bE\left\|\int_0^\cdot\Phi dW \right\|^q_{W^{\beta,q}(0,T;H)} \leq C\bE\int_{0}^T\|\Phi\|^q_{L_2(\mathcal{U},H)} dt.
\end{align*}
Then, Corollary \ref{stationaryReps} again and \eqref{A0eps} give us that for any $\beta < \frac{1}{2}$ and $q \geq 2$:
\begin{align*}
    \bE\|I_4\|^q_{W^{\beta,q}(0,T;L^2(\mathbb{T}))} \leq C\bE\int_{0}^T \|\boldsymbol{G}^{R,\epsilon}\|^q_{L^2(\mathbb{T})} \leq CA_0^q\bE\int_{0}^T \|\rho_{R,\epsilon}\|^q_{L^2(\mathbb{T})} \leq C_{\epsilon,T}.
\end{align*}
\iffalse
{\cred NO NEED (if $L^2$ is not needed):
Similarly, thanks to \eqref{rhol4}, we have
\begin{align*}
    \bE\|I_4\|^4_{W^{\beta,4}(0,T;H^1(\mathbb{T}))} \leq \bE\int_{0}^T \|\rho_{R,\epsilon}\|^4_{H^1(\mathbb{T})} \leq C_\epsilon.
\end{align*}
Here we used the following bounds that are due to interpolation and Proposition \ref{Linfuniform}, 
    \begin{align}\label{rhol4}
  \bE  \int_0^T \|\partial_x\rho_R\|_{L^4(\mathbb{T})}^4 \leq \bE\left( \sup_{t\in (0,T)}\|\rho_R(t)\|^2_{L^\infty(\mathbb{T})}\int_0^T\|\partial_{xx}\rho_R\|_{L^2(\mathbb{T})}^2\right) \leq C_\epsilon.
 \end{align}
}
\fi 
Hence, we define the following set for the stochastic integral:
\[
\mathcal{B}^M_{S}
:= \left\{\, Y \in C([0,T]; L^{2}(\mathbb{T})) \;:\;
\|Y\|_{W^{\beta,4}(0,T;L^2(\mathbb{T}))} \leq M \right\},
\]
for any $\beta<\frac12$. We also define the following set for the deterministic terms:
\[
\mathcal{B}^M_{D}
:= \left\{\, X \in C([0,T]; L^{2}(\mathbb{T})) \;:\;
\|X\|_{H^1(0,T;L^2(\mathbb{T}))} \leq M \right\}.
\]
For $q_R$, we define the set $\mathcal{B}^M:=\mathcal{B}^M_{S}+\mathcal{B}^M_{D}$.
That is,
\begin{align*}
    \{q_R\in \mathcal{B}^M\} \supset \{I_1+I_2+I_3 \in \mathcal{B}^M_D\} \cap \{I_4 \in \mathcal{B}^M_S\}.
\end{align*}
Hence, for $\frac14 < \beta<\frac12$, so that $4\beta >1$:
\begin{align}\label{BMR}
    \bP(\{q_R \notin \mathcal{B}^M\}) &\leq \bP(\{\|I_1+I_2+I_3\|_{H^1(0,T;L^{2}(\mathbb{T}))}>M\} )+ \bP(\{\|I_4\|_{W^{\beta,4}(0,T;L^2(\mathbb{T}))}>M\}) \nonumber \\
    &\leq \frac{1}{M}\bE\|I_1+I_2+I_3\|_{H^1(0,T;L^{2}(\mathbb{T}))} +\frac1M \bE \|I_4\|_{W^{\beta,4}(0,T;L^2(\mathbb{T}))} \leq \frac{C_{\epsilon, T}}M.
\end{align}
    Thus, the tightness of the laws of $q_R$ in the function space defined in $\mathcal{S}$ follows from the preceding calculations in \eqref{BMR} and \eqref{qh2}, since $\mathcal{B}^M$ is compact in $C_{loc}([0,\infty) ;L^2(\mathbb{T}))$ due to (see e.g. \cite{CKMT25, S87}):
    \begin{equation*}
    C(0,T;H^1(\mathbb{T}) )\cap W^{\beta,q}(0, T; L^{2}(\mathbb{T}))  \subset \subset C(0, T; L^2(\mathbb{T}) ),\qquad
    \text{for $0<\beta<1, q> 1$ such that $\beta q >1$}.
\end{equation*}
%{\cred can get compactness for $q$ in $C_w(L^2)$}
Similarly, tightness of the laws of $q_R$ in $L^2_{loc}(0,\infty;H^1(\mathbb{T}))$, is due to the Aubin-Lions-Simon theorem which, for any $0<\beta<1$, states that
$$L^2(0,T;H^2(\mathbb{T})) \cap H^{\beta}(0, T; L^{2}(\mathbb{T}))\subset\subset L^2(0,T;H^1(\mathbb{T})).$$

The uniform bounds in \eqref{rhotightbound1} and \eqref{qh2} also give tightness of $(\rho_{R}, q_{R})$ in $((L^{2}(0, T; H^{2}(\mathbb{T}))^{2}, w)$ immediately, see \eqref{phaseeps}. Finally, to observe tightness of the laws of $\rho_{R}^{-1}\partial_{x}\rho_{R}$ in $(L^{2}_{loc}([0, \infty); L^{2}(\mathbb{T})), w)$, we note that from Proposition \ref{logbounds}, we have that for each $T > 0$:
\begin{equation*}
\mathbb{E} \int_{0}^{T} \int_{\mathbb{T}} \frac{(\partial_{x}\rho_{R})^{2}}{\rho_{R}^{2}} dx dt \le C_{\epsilon}T,
\end{equation*}
for a constant $C_{\epsilon}$ independent of $R$.
\iffalse
{\color{red}{ I need to show almost surely:
\begin{equation*}
\int_{0}^{T} \int_{\mathbb{T}} \frac{(\partial_{x}q_{R})^{2}}{\rho_{R}} dx dt \to \int_{0}^{T} \int_{\mathbb{T}} \frac{(\partial_{x}q)^{2}}{\rho} dx dt.
\end{equation*}
(This is even in the dissipation term for $\eta$ = energy.)

Note that
\begin{equation*}
\int_{0}^{T} \int_{\mathbb{T}} \frac{(\partial_{x}q_{R})^{2}}{\rho_{R}} dx dt = -\int_{0}^{T} \int_{\mathbb{T}} u_{R} \partial_{x}^{2}q_{R} + \int_{0}^{T} \int_{\mathbb{T}} (\partial_{x}q_{R})u_{R} \frac{\partial_{x}\rho_{R}}{\rho_{R}}
\end{equation*}
}}
\fi
%We have $\partial_xq_R\to \partial_xq$ a.s. in $L^2(0,T;L^2(\mathbb{T}))$. Moreover since $\rho_R\to\rho$ a.s. in $C(0,T;L^2(\mathbb{T}))$ we have that $\rho_R^{-1}\to\rho^{-1}$ a.e. on the set $\{(x,t,\omega):\rho(x,t,\omega)>0\}$.

\end{proof}
\subsection{Skorohod representation theorem and limit passage as $R \to \infty$.} Given the tightness result in Proposition \ref{tightnessR}, we can obtain almost sure convergence of our approximate solutions, with the trade off of transferring to a potentially different probability space $(\tilde\Omega, \tilde{\mathcal{F}}, \tilde{\mathbb{P}})$ but with equivalence of laws. In particular, by the classical Skorohod representation theorem, we can construct new random variables $(\tilde{\rho}_{R}, \tilde{q}_{R}, \tilde{\ell}_{R}, \tilde{W}_{R})$ on a new probability space $(\tilde{\Omega}, \tilde{\mathcal{F}}, \tilde{\mathbb{P}})=([0,1), \text{Borel}([0,1)),\text{Leb}([0,1))$ with the same laws as $(\rho_{R}, q_{R}, \rho_{R}^{-1}\partial_{x}\rho_{R}, W)$ in $\mathcal{S}$ as defined in \eqref{phaseeps}, such that for any fixed $\epsilon>0$
\begin{equation}\label{skorohodR}
(\tilde{\rho}_{R}, \tilde{q}_{R}, \tilde{\ell}_{R}, \tilde{W}_{R}) \to ({\rho}_\epsilon, {q}_\epsilon, \ell_{\epsilon}, {W}_\epsilon), \qquad \text{ in } \mathcal{S}, \text{ $\tilde{\mathbb{P}}$-almost surely,}
\end{equation}
where the limiting random variable $({\rho}_{\epsilon}, {u}_{\epsilon}, \ell_{\epsilon}, {W}_{\epsilon})$ on the new probability space has law given by $\mu_\epsilon$, which is identified as the weak limit of the laws $\mu_{R}$ (along a subsequence), as in Proposition \ref{tightnessR}. See \eqref{phaseeps} for the definition of the phase space $\mathcal{S}$. %Note that we are using the variable $\ell$ for the derivative of $\log(\rho)$, namely $\rho^{-1}\partial_{x}\rho$ on the new probability space {\cred ???}.

We will now deduce several consequences of the convergence \eqref{skorohodR} that will help us pass to the limit as $R \to \infty$ in the approximate entropy balance equation \eqref{entropyR}.  %From the convergence \eqref{skorohodR}, note that $\tilde{\bd{U}}_{R} \to \bd{U}_{\epsilon}$ $\tilde{\mathbb{P}}$-almost surely in $C(0, T; L^{2}(\mathbb{T}))$. We claim that we more generally have the following almost everywhere convergence.
Let, 
$$\tilde {\bd{U}}_R:=(\tilde \rho_R,\tilde q_R)\quad\text{and}\quad\bd{U}_\epsilon:=(\rho_\epsilon,q_\epsilon).$$
\begin{proposition}\label{aeR}
For almost every $(\tilde{\omega}, t, x) \in \tilde{\Omega} \times [0, \infty) \times \mathbb{T}$, we have
that 
$$\tilde{\bd{U}}_{R}(\tilde{\omega}, t, x) \to \bd{U}_{\epsilon}(\tilde{\omega}, t, x)\quad\text{and}\quad\partial_{x}\tilde{\bd{U}}_{R}(\tilde{\omega}, t, x) \to \partial_{x}\bd{U}_{\epsilon}(\tilde{\omega}, t, x)\quad \text{ a.e. in }\tilde\Omega\times[0,\infty)\times\mathbb{T}.$$
\end{proposition}

\begin{proof}
    We first show that $\tilde{\bd{U}}_{R} \to \bd{U}_\epsilon$, up to a subsequence, for almost every $(\tilde\omega, t, x) \in \tilde\Omega \times [0, \infty) \times \mathbb{T}$. Since $\tilde{\bd{U}}_{R} \to \bd{U}_{\epsilon}$ in $C_{loc}([0, \infty); L^{2}(\mathbb{T}))$, $\tilde{\mathbb{P}}$-almost surely, we have that $\tilde{\bd{U}}_{R} \to \bd{U}_{\epsilon}$, $\tilde{\mathbb{P}}$-almost surely in $L^{1}([0, T] \times \mathbb{T})$ for all $T > 0$. Furthermore, by Proposition \ref{energyR} and equivalence of laws (by the Skorohod representation theorem), we have that for a uniform constant $C_{T}$ depending only on $T > 0$ and $\epsilon > 0$ (but independent of $R > 0$):
    \begin{equation*}
    \tilde{\mathbb{E}} \|(\tilde{\rho}_{R}, \tilde{q}_{R})\|_{L^{1}([0, T] \times \mathbb{T})}^{2} \le \tilde{\mathbb{E}} \|(\tilde{\rho}_{R}, \tilde{q}_{R})\|_{L^{2}([0, T] \times \mathbb{T})}^{2} \le C_{T, \epsilon}. 
    \end{equation*}
    So by Vitali's convergence theorem, $\tilde{\bd{U}}_{R} \to \bd{U}_{R}$ in $L^{1}(\tilde{\Omega} \times [0, T] \times \mathbb{T})$ for all $T > 0$, which establishes the desired result. 

    To show the result for the first spatial derivatives, namely that $\partial_{x}\tilde{\bd{U}}_{R}(\tilde\omega, t, x) \to \partial_{x}\bd{U}(\tilde\omega, t, x)$ for almost every $(\tilde\omega, t, x) \in \tilde\Omega \times [0, T] \times \mathbb{T}$, we can use the same Vitali convergence argument to show that $\partial_{x}\tilde{\bd{U}}_{R} \to \partial_{x}\bd{U}$ in $L^{1}(\tilde\Omega \times [0, T] \times \mathbb{T})$ for all $T > 0$. %This is because we have by the definition of the phase space \eqref{phaseeps} that $\partial_{x}\tilde{\bd{U}}_{R} \to \partial_{x}\bd{U}_{\epsilon}$ in $L^{2}(0, T; L^{2}(\mathbb{T}))$, $\tilde{\mathbb{P}}$-almost surely, which allows us to use the same Vitali convergence theorem argument.
\end{proof}

As a consequence of this almost everywhere convergence result in Proposition \ref{aeR}, we prove that there is no vacuum in the limiting system.
%make the following observation about the limiting solution $\bd{U}_{\epsilon} = (\rho_{\epsilon}, q_{\epsilon})$:

\begin{proposition}\label{rhoepspositive}
    For every fixed but arbitrary $\epsilon > 0$, the stationary solution $(\rho_{\epsilon}, q_{\epsilon})$ satisfies $\rho_{\epsilon} > 0$ almost everywhere on $\tilde\Omega \times [0, \infty) \times \mathbb{T}$. Furthermore,
    \begin{equation}\label{logfatou}
    \tilde{\mathbb{E}} \int_{0}^{T} \int_{\mathbb{T}} |\log(\rho_{\epsilon})|^{2} dx dt \le C_{\epsilon}T,
    \end{equation}
    \iffalse
    and
    \begin{equation}\label{dissipationfatou}
    \epsilon    \tilde{\mathbb{E}} \int_{0}^{\infty} \left(\int_{\mathbb{T}} \langle D^{2}\eta(\bd{U}_{\epsilon}) \partial_{x}\bd{U}_{\epsilon}, \partial_{x}\bd{U}_{\epsilon} \rangle \varphi(x) dx\right) \psi(t) dt < \infty,
    \end{equation}
    for all smooth $\varphi \in C^{\infty}(\mathbb{T})$ and $\psi \in C^{\infty}_{c}(0, \infty)$ with $\varphi, \psi \ge 0$.
    \fi
\end{proposition}

\begin{proof}
    Note that by Proposition \ref{logbounds}, we have that for all $R > 0$, almost surely:
    \begin{equation*}
    {\mathbb{E}} \|\log(\rho_{R}(t))\|_{L^{\infty}(\mathbb{T})}^{2} \le C_{\epsilon}, \qquad \text{ for all } t \ge 0.
    \end{equation*}
    So taking a time integral and transferring to the new probability space by equivalence of laws:
    \begin{equation*}
    \tilde{\mathbb{E}} \int_{0}^{T} \|\log(\tilde{\rho}_{R})\|^{2}_{L^{\infty}(\mathbb{T})} dx dt \le C_{\epsilon}T \qquad \text{ and hence} \qquad \tilde{\mathbb{E}} \int_{0}^{T} \int_{\mathbb{T}} |\log(\tilde{\rho}_{R})|^{2} dx dt \le C_{\epsilon}T,
    \end{equation*}
    for all $T > 0$. So \eqref{logfatou} follows from Proposition \ref{aeR} and Fatou's lemma.  %For the second result in \eqref{dissipationfatou}, note that $\langle D^{2} \eta(\tilde{\bd{U}}_{R}) \partial_{x}\tilde{\bd{U}}_{R}, \partial_{x}\tilde{\bd{U}}_{R} \rangle \ge 0$ by the fact that $\eta(\rho, q)$ is convex away from vacuum $\rho = 0$. So since we consider nonnegative $\varphi(x) \in C^{\infty}(\mathbb{T})$ and $\psi(t) \in C_{c}^{\infty}([0, \infty))$, we can use Proposition \ref{uniformRdissipation} together with Proposition \ref{aeR} to conclude the inequality \eqref{dissipationfatou} using Fatou's lemma. 
\end{proof}
\begin{remark}The observation that $\rho_{\epsilon} > 0$ almost everywhere on $\tilde{\Omega} \times [0, \infty) \times \mathbb{T}$, which follows from \eqref{logfatou}, is also essential for ensuring that the terms in the entropy inequality, such as the expectation of the dissipation integral involving $D^{2}\eta$ which is continuous away from the vacuum set $\rho = 0$, can be properly defined.
\end{remark}
Thanks to the previous two propositions, we can now identify the new random variables $\tilde \ell_R$ and $\ell_\epsilon$.
\begin{corollary}\label{lidentify}
    For the new random variable we have
  \begin{equation*}
    \tilde{\ell}_{R} = \frac{\partial_{x}\tilde{\rho}_{R}}{\tilde{\rho}_{R}},\qquad \ell_{\epsilon} = \frac{\partial_{x}\rho_{\epsilon}}{\rho_{\epsilon}}. 
    \end{equation*}
\end{corollary}
\begin{proof}
  From equivalence of laws $\displaystyle \left(\rho_{R}, \frac{\partial_{x}\rho_{R}}{\rho_{R}}\right) =_{d} (\tilde{\rho}_{R}, \tilde{\ell}_{R})$, it is immediate that $\displaystyle \tilde{\ell}_{R} = \frac{\partial_{x}\tilde{\rho}_{R}}{\tilde{\rho}_{R}}$. Furthermore, we know from \eqref{skorohodR} that
    \begin{equation*}
    \frac{\partial_{x}\tilde{\rho}_{R}}{\tilde{\rho}_{R}} \rightharpoonup \ell_{\epsilon}, \qquad \text{$\tilde{\mathbb{P}}$-almost surely and weakly in $L^{2}(0, T; L^{2}(\mathbb{T}))$}.
    \end{equation*}
  Then Propositions \ref{aeR} and \ref{rhoepspositive} allow us to uniquely identify the weak $\tilde{\mathbb{P}}$-almost sure limit as $\displaystyle \ell_{\epsilon} = \frac{\partial_{x}\rho_{\epsilon}}{\rho_{\epsilon}}$.
\end{proof}

As another consequence that will be important in the limit passage as $R \to \infty$ (in particular, for passing to the limit in the entropy dissipation terms), we observe the following convergence.

\begin{proposition}\label{weakstrongR}
  Let $\displaystyle \tilde{u}_{R} = \frac{\tilde q_{R}}{\tilde \rho_{R}}$ and $\displaystyle u_{\epsilon} = \frac{q_{\epsilon}}{\rho_{\epsilon}}$.  We have the following $\tilde{\mathbb{P}}$-almost sure strong convergence:
    \begin{equation}\label{ustrong}
    \tilde{u}_{R} \to u_{\epsilon}, \qquad \text{$\tilde{\mathbb{P}}$-almost surely and strongly in $L^{2}([0, T] \times \mathbb{T})$}.
    \end{equation}
   % and the following $\tilde{\mathbb{P}}$-almost sure weak convergence:
  %  \begin{equation}\tilde{u}_{R} \frac{\partial_{x}\tilde{\rho}_{R}}{\tilde{\rho}_{R}} \rightharpoonup u_{\epsilon} \frac{\partial_{x}\rho_{\epsilon}}{\rho_{\epsilon}}, \qquad \text{ weakly in $L^{2}([0, T] \times \mathbb{T}$)}. \end{equation}
Consequently, we have that for arbitrary $T > 0$:
    \begin{equation}\label{q2rho}
    \int_{0}^{T} \int_{\mathbb{T}} \frac{(\partial_{x}\tilde{q}_{R})^{2}}{\tilde{\rho}_{R}} dx dt \to \int_{0}^{T} \int_{\mathbb{T}} \frac{(\partial_{x}q_{\epsilon})^{2}}{\rho_{\epsilon}} dx dt, \qquad \text{$\tilde{\mathbb{P}}$-almost surely.}
    \end{equation}
\end{proposition}

\begin{proof}
%    We first prove \eqref{ulweak}, and we will establish \eqref{ustrong} in the process too. %Consider an arbitrary deterministic function $g \in L^{2}([0, T] \times \mathbb{T})$. 
    By the vacuum estimates in Propositions \ref{logbounds} and \ref{rhoepspositive}, and the almost everywhere convergence in Proposition \ref{aeR}, we have that
    \begin{equation*}
    \tilde{u}_{R}(\tilde\omega, t, x) \to u_{\epsilon}(\tilde\omega, t, x), \qquad \text{ for almost every $(\tilde\omega, t, x) \in \tilde\Omega \times [0, \infty) \times \mathbb{T}$}
    \end{equation*}
    Furthermore, by Corollary \ref{stationaryReps} and equivalence of laws:
    \begin{equation*}
    \|\tilde{u}_{R}\|_{L^{\infty}(\tilde\Omega \times [0, \infty) \times \mathbb{T})} \le C_{\epsilon}, \qquad \|u_{\epsilon}\|_{L^{\infty}(\tilde\Omega \times [0, \infty) \times \mathbb{T})} \le C_{\epsilon}.
    \end{equation*}
    So by dominated convergence theorem,
    \begin{equation*}
    \tilde{u}_{R} \to u_{\epsilon}, \qquad \text{$\tilde{\mathbb{P}}$-almost surely and strongly in $L^{2}([0, T] \times \mathbb{T})$},
    \end{equation*}
  Hence,  we immediately conclude from the weak convergence \eqref{skorohodR} and Corollary \ref{lidentify}:
    \begin{equation}\label{ulweak}
    \tilde{u}_{R}\frac{\partial_{x}\tilde{\rho}_{R}}{\tilde{\rho}_{R}} \rightharpoonup u_{\epsilon}\frac{\partial_{x}\rho_{\epsilon}}{\rho_{\epsilon}}, \qquad \text{$\tilde{\mathbb{P}}$-almost surely and weakly in $ L^{2}([0, T] \times \mathbb{T})$.}
    \end{equation}
%{\cred we can apply vitali to get strong convergence of $u$ in $L^2$. This will simplify these steps and also fix the red comment below for which we require strong convergence of u}
    Finally, we establish the convergence in \eqref{q2rho}. By integration by parts, we compute that
    \begin{equation}\label{q2rho1}
    \int_{0}^{T} \int_{\mathbb{T}} \frac{(\partial_{x}\tilde{q}_{R})^{2}}{\tilde{\rho}_{R}} dx dt = -\int_{0}^{T} \int_{\mathbb{T}} \tilde{u}_{R}\partial_{x}^{2}\tilde{q}_{R} dx dt + \int_{0}^{T} \int_{\mathbb{T}} \tilde{u}_{R} \frac{\partial_{x}\tilde{\rho}_{R}}{\tilde{\rho}_{R}} \partial_{x}\tilde{q}_{R} dx dt.
    \end{equation}
    We can then pass to the limit in each term. Using the strong convergence from \eqref{ustrong} and the weak convergence from \eqref{skorohodR} of $\partial_{x}^{2}\tilde{q}_{R} \to \partial_{x}^{2} q_{\epsilon}$ in $L^{2}([0, T] \times \mathbb{T})$, $\tilde{\mathbb{P}}$-almost surely, we obtain:
    \begin{equation}\label{q2rho2}
    -\int_{0}^{T} \int_{\mathbb{T}} \tilde{u}_{R} \partial_{x}^{2}\tilde{q}_{R} dx dt \to -\int_{0}^{T} \int_{\mathbb{T}} u_{\epsilon}\partial_{x}^{2}q_{\epsilon} dx dt, \qquad \text{$\tilde{\mathbb{P}}$-almost surely.}
    \end{equation}
    The convergence \eqref{ulweak} implies that, %the weak convergence in \eqref{ulweak} and the strong convergence from \eqref{skorohodR} of $\partial_{x}^{2}\tilde{q}_{R} \to \partial_{x}^{2} q_{\epsilon}$ in $L^{2}([0, T] \times \mathbb{T})$, $\tilde{\mathbb{P}}$-almost surely, we obtain:
    \begin{equation}\label{q2rho3}
    \int_{0}^{T} \int_{\mathbb{T}} \tilde{u}_{R} \frac{\partial_{x}\tilde{\rho}_{R}}{\tilde{\rho}_{R}} dx dt \to \int_{0}^{T} \int_{\mathbb{T}} u_{\epsilon} \frac{\partial_{x}\rho_{\epsilon}}{\rho_{\epsilon}} dx dt, \qquad \text{$\tilde{\mathbb{P}}$-almost surely.}
    \end{equation}
    Using the convergences \eqref{q2rho2} and \eqref{q2rho3} in \eqref{q2rho1}, we obtain the desired convergence in \eqref{q2rho}.
\end{proof}

    We now have all of the components needed to pass to the limit as $R \to \infty$ in the entropy equality \eqref{entropyR} to obtain a limiting entropy equality for $(\rho_{\epsilon}, q_{\epsilon})$.

{ \begin{proposition}
    For every $\epsilon>0$, the limiting random variable $(\rho_{\epsilon},q_{\epsilon})$ with continuous paths in $\mathcal{X}$ almost surely, is a stationary martingale solution to \begin{equation}\label{approxsystem_ep}
\begin{cases}
    \partial_{t}\rho_{\epsilon} + \partial_{x}q_{\epsilon} = \epsilon \Delta \rho_{\epsilon}, \\
    dq_\epsilon + \partial_{x}\left(\frac{q_{\epsilon}^2}{\rho_{\epsilon}}\right) +  \partial_{x}(\kappa \rho_{\epsilon}^{\gamma}) = \bd{\Phi}^{\epsilon}(\rho_{\epsilon}, q_{\epsilon}) dW_{\epsilon} + \epsilon \Delta q_{\epsilon}.
\end{cases}
\end{equation}
Moreover, the solution $(\rho_{\epsilon},q_{\epsilon})$ satisfies the entropy equality for all entropy-flux pairs $(\eta, H)$ generated by subpolynomial $g \in \tilde{\mathcal{G}}$, and for all deterministic nonnegative test functions $\varphi(x) \in C^{2}(\mathbb{T})$ and $\psi(t) \in C_{c}^{\infty}(0, \infty)$ with $\psi \ge 0$:

\begin{multline}\label{entropyeq}
\int_{0}^{\infty} \left(\int_{\mathbb{T}} \eta(\bd{U}_{\epsilon}(t)) \varphi(x) dx\right) \partial_{t}\psi(t) dt + \int_{0}^{\infty} \left(\int_{\mathbb{T}} H(\bd{U}_{\epsilon}) \partial_{x}\varphi(x) dx\right) \psi(t) dt \\
- \int_{0}^{\infty} \left(\int_{\mathbb{T}} \alpha q_{\epsilon} \partial_{q}\eta(\bd{U}_{\epsilon}) \varphi(x) dx\right) \psi(t) dt + \int_{0}^{\infty}\left(\int_{\mathbb{T}} \partial_{q}\eta(\bd{U}_{\epsilon}) \Phi(\bd{U}_{\epsilon}) \varphi(x) dx\right) \psi(t) dW_\epsilon(t) \\
+ \int_{0}^{\infty}\left(\int_{\mathbb{T}} \frac{1}{2}\partial_{q}^{2}\eta(\bd{U}_{\epsilon}) G^{2}(\bd{U}_{\epsilon}) \varphi(x) dx\right) \psi(t) dt = \epsilon \int_{0}^{\infty} \left(\int_{\mathbb{T}} \langle D^{2}\eta(\bd{U}_{\epsilon}) \partial_{x}\bd{U}_{\epsilon}, \partial_{x}\bd{U}_{\epsilon}\rangle \varphi(x) dx\right) \psi(t) dt \\
- \epsilon \int_{0}^{\infty} \left(\int_{\mathbb{T}} \eta(\bd{U}_{\epsilon}) \partial_{x}^{2}\varphi dx\right) \psi(t) dt.
\end{multline}

\end{proposition}
\begin{proof}
 This follows from passing to the limit $R\to\infty$ in \eqref{approxsystem} and using the fact that the notion of stationarity
is stable under strong (and weak) convergence as proven in Lemmas A.5 (and A.4) in \cite{BFHM19}. We only make comments about passing to the limit in the dissipation term, as $R \to \infty$, namely:
\begin{equation}\label{dissipationRconv}
\int_{0}^{\infty} \left(\int_{\mathbb{T}} \langle D^{2}\eta(\tilde{\bd{U}}_{R}) \partial_{x}\tilde{\bd{U}}_{R}, \partial_{x}\tilde{\bd{U}}_{R} \rangle \varphi(x) dx\right) \psi(t) dt \to \int_{0}^{\infty}\left(\int_{\mathbb{T}} \langle D^{2}\eta(\bd{U}_{\epsilon}) \partial_{x}\bd{U}_{\epsilon}, \partial_{x}\bd{U}_{\epsilon} \rangle \varphi(x) dx\right) \psi(t) dt,
\end{equation}
almost surely as $R \to \infty$. For all other terms in the entropy equality, we can pass to the limit as $R \to \infty$ using the uniform $L^{\infty}$ bounds in Corollary \ref{stationaryReps}, the almost everywhere convergence in Proposition \ref{aeR}, the entropy bounds in Proposition \ref{generaletabound} (combined with the $L^{\infty}$ bounds in Corollary \ref{stationaryReps}  and the fact that the vacuum is of measure zero as proved in in Propositions \ref{logbounds} and \ref{rhoepspositive}), and the dominated convergence theorem.

Hence, it suffices to show the convergence \eqref{dissipationRconv}, $\tilde{\mathbb{P}}$-almost surely. We can verify this convergence by using the generalized dominated convergence theorem (see Theorem 11 in Section 4.4 of \cite{RF23}). Fix some deterministic nonnegative functions $\varphi(x) \in C^{\infty}(\mathbb{T})$ and $\psi \in C_{c}^{\infty}([0, \infty))$, and fix some $g \in \tilde{\mathcal{G}}$ that defines some entropy $\eta$. Recall from Corollary \ref{stationaryReps} and by equivalence of laws and Proposition \ref{rhoepspositive}, that for some deterministic constant $C_{\epsilon} > 0$,
\begin{equation*}
\frac1R \le \tilde{\rho}_{R}(\tilde\omega, t, x) \le C_{\epsilon}, \qquad 0 < \rho_{\epsilon}(\tilde\omega, t, x) \le C_{\epsilon}, \qquad \text{ for almost every }(\tilde\omega, t, x) \in \tilde\Omega \times [0, T] \times \mathbb{T}.
\end{equation*}
Hence, from Proposition \ref{generaletabound}, 
\begin{equation*}
\|D^{2}\eta(\tilde{\bd{U}}_{R})\|_{L^{\infty}([0, T] \times \mathbb{T})} \le C_{g, \epsilon}\tilde{\rho}_{R}^{-1}, \qquad \|D^{2}\eta(\bd{U}_{\epsilon})\|_{L^{\infty}([0, T] \times \mathbb{T})} \le C_{g, \epsilon}\rho_{\epsilon}^{-1}, \qquad \text{$\tilde{\mathbb{P}}$-almost surely.}
\end{equation*}

Hence, choosing $T$ so that $\text{supp}(\psi(t)) \subset [0, T]$, we note that for almost every $(\tilde\omega, t, x) \in \tilde\Omega \times [0, T] \times \mathbb{T}$: 
\begin{equation}\label{generalLebesgue1}
|\langle D^{2}\eta(\tilde{\bd{U}}_{R})\partial_{x}\tilde{\bd{U}}_{R}, \partial_{x}\tilde{\bd{U}}_{R} \rangle \varphi(x) \psi(t) dt| \le 2C_{\epsilon, g}\|\phi\|_{L^{\infty}}\|\psi\|_{L^{\infty}} \frac{\partial_{x}\tilde{\bd{U}}_{R}}{\tilde{\rho}_{R}} \cdot \partial_{x}\tilde{\bd{U}}_{R},
\end{equation}
and similarly for the limiting $\bd{U}_{\epsilon}$. By the weak convergence
\begin{equation*}
\frac{\partial_{x}\tilde{\rho}_{R}}{\tilde{\rho}_{R}} \rightharpoonup \frac{\partial_{x}\rho_{\epsilon}}{\rho_{\epsilon}}, \qquad \text{$\tilde{\mathbb{P}}$-almost surely in $L^{2}([0, T] \times \mathbb{T})$}
\end{equation*}
and the strong convergence
\begin{equation*}
\partial_{x}\tilde{\rho}_{R} \to \partial_{x}\rho_{\epsilon}, \qquad \text{$\tilde{\mathbb{P}}$-almost surely in $L^{2}([0, T] \times \mathbb{T}$)}
\end{equation*}
from \eqref{skorohodR}, and also by \eqref{q2rho} from Proposition \ref{weakstrongR}, we conclude that
\begin{equation}\label{generalLebesgue2}
2C_{\epsilon, g} \|\phi\|_{L^{\infty}} \|\varphi\|_{L^{\infty}} \int_{0}^{T} \int_{\mathbb{T}} \frac{\partial_{x}\tilde{\bd{U}}_{R}}{\tilde{\rho}_{R}} \cdot \partial_{x} \tilde{\bd{U}}_{R} dx dt \to 2C_{\epsilon, g} \|\varphi\|_{L^{\infty}}\|\psi\|_{L^{\infty}(\mathbb{T})} \int_{0}^{T} \int_{\mathbb{T}} \frac{\partial_{x}\bd{U}_{\epsilon}}{\rho_{\epsilon}}\cdot \partial_{x}\bd{U}_{\epsilon} dx dt, 
\end{equation}
$\tilde{\mathbb{P}}$-almost surely. So by applying the generalized dominated convergence theorem (see Theorem 11 in Section 4.4 of \cite{RF23}) in the $(t, x)$ variables (pathwise in outcome $\tilde\omega$) using \eqref{generalLebesgue1} and \eqref{generalLebesgue2}, we obtain the desired $\tilde{\mathbb{P}}$-almost sure convergence in \eqref{dissipationRconv}.

\end{proof}

\section{Uniform bounds: $\epsilon_{N}$ level}
Our goal in this section will be to pass to the limit as $\epsilon \to 0$ in the statistically stationary solutions to the approximate problem \eqref{approxsystem_ep} with artificial viscosity. At this stage, we have a statistically stationary solution $\bd{U}_{\epsilon} := (\rho_{\epsilon}, q_{\epsilon})$ defined on $[0, \infty)$ that satisfies the \textbf{approximate entropy equality} \eqref{entropyeq}.
\iffalse 
for all entropy-flux pairs $(\eta, H)$, and for all deterministic test functions $\varphi(x) \in C^{2}(\mathbb{T})$ and nonnegative $\psi(t) \in C_{c}^{\infty}(0, \infty)$ with $\psi \ge 0$:

\begin{multline*}
\int_{0}^{\infty} \left(\int_{\mathbb{T}} \eta(\bd{U}_{\epsilon}(t)) \varphi(x) dx\right) \partial_{t}\psi(t) dt + \int_{0}^{\infty} \left(\int_{\mathbb{T}} H(\bd{U}_{\epsilon}) \partial_{x}\varphi(x) dx\right) \psi(t) dt \\
- \int_{0}^{\infty} \left(\int_{\mathbb{T}} \alpha q_{\epsilon} \partial_{q}\eta(\bd{U}_{\epsilon}) \varphi(x) dx\right) \psi(t) dt + \int_{0}^{\infty}\left(\int_{\mathbb{T}} \partial_{q}\eta(\bd{U}_{\epsilon}) \Phi(\bd{U}_{\epsilon}) \varphi(x) dx\right) \psi(t) dW(t) \\
+ \int_{0}^{\infty}\left(\int_{\mathbb{T}} \frac{1}{2}\partial_{q}^{2}\eta(\bd{U}_{\epsilon}) G^{2}(\bd{U}_{\epsilon}) \varphi(x) dx\right) \psi(t) dt = \epsilon \int_{0}^{\infty} \left(\int_{\mathbb{T}} \langle D^{2}\eta(\bd{U}_{\epsilon}) \partial_{x}\bd{U}_{\epsilon}, \partial_{x}\bd{U}_{\epsilon}\rangle \varphi(x) dx\right) \psi(t) dt \\
- \epsilon \int_{0}^{\infty} \left(\int_{\mathbb{T}} \eta(\bd{U}_{\epsilon}) \partial_{x}^{2}\varphi dx\right) \psi(t) dt.
\end{multline*}
\fi 
Using stationarity (the equivalence of laws of the process $\bd{U}_{\epsilon}(t)$ at all times $t \ge 0$), we can deduce uniform bounds, where usual dissipative terms in the energy inequality become $L^{\infty}$ in time bounds. One particular challenge here is that the $\rho\partial_{q}\eta(\bd{U}_{\epsilon})$ term for entropies of the form $\eta_{m}$, does not immediately give any bounds on higher powers of $\rho_\epsilon$. We will hence have to use a Bogovskii operator technique to obtain higher integrability, along with carefully managing numerology of powers of $\rho_\epsilon$ in order to close the resulting estimate. This is done in Proposition \ref{uniformeps2}. From there, we will deduce further uniform bounds independent of $\epsilon$ in Section \ref{additionalest}, that will be important for the $\epsilon_{N}$ limit passage in the next section. Note that as discussed in Section \ref{noiseregularize}, it is easiest to define $\epsilon$-level approximations of the noise using a specific sequence $\{\epsilon_{N}\}_{N = 1}^{\infty}$ with $\epsilon_{N} \searrow 0$, but for simplicity of notation, we will omit the subscript of $N$ on the $\epsilon$ parameters in this section.

\subsection{Entropy bounds uniform in $\epsilon$.}\label{uniformentropyeps}

In this subsection, we deduce uniform entropy bounds independently of $\epsilon$. This is the crucial component of the proof, as uniform bounds that we derived previously, such as the uniform $L^{\infty}(\mathbb{T})$ bounds on $(\rho_{\epsilon}, q_{\epsilon})$, are $\epsilon$ dependent.  This is done as follows. From stationarity, we have a bound on the dissipation rate uniformly in $\epsilon$. We can then use these uniform bounds to recover a moment estimate on the powers of the density. Note that the dissipative term in the approximate entropy inequality involves $q\partial_{q}\eta$, and since this term does not include any terms with just powers of $\rho$ by themselves (see Proposition \ref{etamboundalg}), we do not immediately have these bounds for all powers of the density. However, we can use a Bogovskii-type approach to recover these higher moment bounds. This is similar in spirit to what is done for compressible Navier-Stokes equations, except for this case, without the additional in $\epsilon$ bounds on $\|u_{\epsilon}\|_{H^{1}(\mathbb{T})}$.

\begin{proposition}\label{uniformeps2}
The approximate solutions $\bd{U}_{\epsilon}$ satisfy the following uniform bounds for all positive integers $m$:
\begin{align}\label{etaboundsm}
    \bE\int_{\mathbb{T}} \eta_{m}(\boldsymbol{U}_\epsilon) dx \leq C_m,
\end{align}
for a constant $C_{m}$ that is independent of $\epsilon$ and $t \ge 0$, which depends only on $m$ and the damping parameter $\alpha > 0$.
Consequently, the bounds above imply that,
\begin{equation}\label{etambounds}
\bE \int_{\mathbb{T}} \rho_\epsilon\Big(u_\epsilon^{2m} + \rho_\epsilon^{(m - 1)(\gamma-1)}u_\epsilon^{2}\Big) \le C_{m}, \qquad {\mathbb{E}} \int_{\mathbb{T}} \rho_\epsilon^{  1+ m(\gamma-1)} \le C_{m},
\end{equation}
\begin{equation}\label{gradmbounds}
\epsilon {\mathbb{E}} \int_{\mathbb{T}} \Big(u_\epsilon^{2(m-1)} + \rho_\epsilon^{(m-1)(\gamma-1)}\Big)\rho_\epsilon^{\gamma - 2}|\partial_{x}\rho_\epsilon|^{2} + \epsilon {\mathbb{E} }\int_{\mathbb{T}} \Big(u_\epsilon^{2(m-1)} + \rho_\epsilon^{(m-1)(\gamma-1)}\Big) \rho_\epsilon |\partial_{x}u_\epsilon|^{2} \le C_{m}.
\end{equation}

\end{proposition}

\begin{proof}
We use the approximate entropy equality with the entropy $\eta_{m}$ for arbitrary positive integers $m$ and test functions $\varphi(x) = 1$ and $\psi(t) = \mathbbm{1}_{[T, T + 1]}(t)$. We obtain after taking expectations:
\begin{align*}
\bE \int_{\mathbb{T}} \eta_{m}(\bd{U}_{\epsilon}(T + 1)) &+ \bE \int_{T}^{T + 1} \int_{\mathbb{T}} \alpha q_{\epsilon}\partial_{q}\eta_{m}(\bd{U}_{\epsilon}) dx dt + \bE \int_{T}^{T + 1} \int_{\mathbb{T}} \epsilon \langle D^{2}\eta_m(\bd{U}_{\epsilon})\partial_{x}\bd{U}_{\epsilon}, \partial_{x}\bd{U}_{\epsilon}\rangle dx dt \\
&= \bE \int_{\mathbb{T}} \eta_{m}(\bd{U}_{\epsilon}(T)) + \mathbb{E} \int_{T}^{T + 1} \int_{\mathbb{T}} \frac{1}{2} \partial_{q}^{2} \eta_m(\bd{U}_{\epsilon})G^{2}(\bd{U}_{\epsilon}) dx dt.
\end{align*}
Then, by using stationarity and the fact that $D^{2}\eta_{m}$ is positive semidefinite, we obtain using the assumption on the noise \eqref{noiseassumption} and Proposition \ref{etamboundalg} that: 
\begin{align}\label{directbound}
\nonumber \alpha c_{m}\bE \int_{\mathbb{T}} \rho_{\epsilon}\Big(u_\epsilon^{2m} + \rho_\epsilon^{(m - 1)(\gamma-1)}u_\epsilon^{2}\Big) dx &\le \mathbb{E} \int_{\mathbb{T}} \alpha q_{\epsilon} \partial_{q}\eta_{m}(\bd{U}_{\epsilon}) dx \\
&\le \mathbb{E} \int_{\mathbb{T}} \frac{1}{2} \partial^{2}_{q}\eta_{m}(\bd{U}_{\epsilon}) G^{2}(\bd{U}_{\epsilon}) dx \le C_{m} \mathbb{E} \int_{\mathbb{T}} \eta_{m - 1}(\bd{U}_\epsilon),
\end{align}
and similarly,
\begin{equation}\label{epsbound}
\epsilon \mathbb{E} \int_{\mathbb{T}} \langle D^{2}\eta_{m}(\bd{U}_{\epsilon})\partial_{x}\bd{U}_{\epsilon}, \partial_{x}\bd{U}_{\epsilon} \rangle \le C_{m} \mathbb{E} \int_{\mathbb{T}} \eta_{m - 1}(\bd{U}_{\epsilon}). 
\end{equation}

Note, due to \eqref{directbound} and \eqref{epsbound} and the explicit expression for $\langle D^{2}\eta_{m}(\bd{U}_{\epsilon}) \partial_{x}\bd{U}_{\epsilon}, \partial_{x}\bd{U}_{\epsilon} \rangle$, that \eqref{etambounds} and \eqref{gradmbounds} will follow for every integer $m\geq 0$ if we prove the bound in \eqref{etaboundsm} for every integer $m\geq 0$.
This will be done via induction.
For the inductive step, we assume that for some positive integer $m$, we have 
\begin{align}\label{inductive}
    \mathbb{E} \int_{\mathbb{T}} \eta_{j}(\boldsymbol{U}_\epsilon) \leq C_m, \qquad \forall 0\leq j \leq m
\end{align}
and then prove that the same holds when $m$ is replaced by $m+1$. Note that the base case $m=0$ of \eqref{inductive} is true since $\eta_{0}(\bd{U}) = \rho$ and so by conservation of mass, $\displaystyle \mathbb{E} \int_{\mathbb{T}} \eta_{0}(\bd{U}) = M$.

First, thanks to  \eqref{directbound} and \eqref{epsbound} and the algebraic bound on the entropy dissipation in Proposition \ref{dissipationalg}, the assumption \eqref{inductive} also implies that we have the following bounds for all $0\leq j\leq m$:
\begin{multline}\label{inductive1}
\mathbb{E} \int_{\mathbb{T}} \rho_\epsilon \Big(u_\epsilon^{2(j+1)} + \rho_\epsilon^{j(\gamma-1)}u_\epsilon^{2}\Big) dx \le C_{m} ,\qquad
\mathbb{E} \int_{\mathbb{T}} \rho_\epsilon^{1+j(\gamma-1)} dx \le C_{m}\\
\epsilon \mathbb{E} \int_{\mathbb{T}} \Big(u_\epsilon^{2j} + \rho_\epsilon^{j(\gamma-1)}\Big)\rho_\epsilon^{\gamma - 2}|\partial_{x}\rho_\epsilon|^{2} + \epsilon \mathbb{E} \int_{\mathbb{T}} \Big(u_\epsilon^{2j} + \rho_\epsilon^{j(\gamma-1)}\Big) \rho_\epsilon |\partial_{x}u_\epsilon|^{2} \le C_{m}.
\end{multline}

Before we proceed, we observe that Proposition \ref{etamboundalg} and  \eqref{directbound} applied recursively gives us,
\begin{align*}
\bE\int_{\mathbb{T}}  \eta_{m+1} (\bd{U}_\epsilon) &\leq C_m\bE\int_{\mathbb{T}} \rho_{\epsilon}u_\epsilon^{2(m+1)} +\rho_\epsilon^{1+(m+1)(\gamma-1)} \le C_{m} \mathbb{E} \int_{\mathbb{T}} \eta_{m}(\bd{U}_\epsilon) + C_m\bE\int_{\mathbb{T}} \rho_\epsilon^{1+(m+1)(\gamma-1)}\\
  & \le C_{m} \mathbb{E} \int_{\mathbb{T}} \eta_{0}(\bd{U}_\epsilon) + C_m\sum_{k=1}^{m+1}\bE\int_{\mathbb{T}} \rho_\epsilon^{1+k(\gamma-1)} 
   \le C_{m,M}\left(1+ \bE\int_{\mathbb{T}} \rho_\epsilon^{1+(m+1)(\gamma-1)} \right).
\end{align*}
Hence, our {\bf aim} above boils down to proving that  $\displaystyle \bE\int_{\mathbb{T}}\rho_\epsilon^{1+(m+1)(\gamma-1)}$ is bounded under the assumption \eqref{inductive} (and thus \eqref{inductive1}).

%{\color{red}{Recall that the case of $\gamma=1$ follows immediately from conservation of mass.}} Hence, in what follows, we will assume that $\gamma>1$. 

\iffalse Note also that \eqref{inductive1}- \eqref{inductive2} holds for the base case of $m = 1, 2$ by Proposition \ref{uniformeps1} {\cred any reason why $m=2$ was included in the previous prop when we're repeating it in this one as well? Seems like the two props can be combined?}. The goal is to show that under the assumptions of \eqref{inductive1} and \eqref{inductive2}, these bounds hold with $m + 1$ in place of $m$ also, namely:
\begin{equation}\label{mplus1}
\mathbb{E}\int_{\mathbb{T}} \rho_\epsilon\Big(u_\epsilon^{2(m + 1)} + \rho_\epsilon^{m(\gamma-1)} u_\epsilon^{2}\Big) dx \le C_{m}.
\end{equation}

{\cred doesn't this follow straight from \eqref{directbound} and that our inductive hypothesis says that $\bE\int_{\mathbb{T}} \eta_{m}$ is bounded? For better exposition, maybe our objective should just be to prove that $\eta_{m+1}$ is bounded if $\eta_m$ is bounded? Everything else is equivalent to this anyways..
We go straight to the Bogovski step after this.
}
\fi
The idea behind showing the inductive step % i.e. bounds for $\displaystyle \bE\int_{\mathbb{T}}\rho_\epsilon^{1+(m+1)(\gamma-1)}$ 
is to test the weak formulation with a special test function $\varphi$ (see \eqref{testbogovskiim}) constructed using the Bogovskii operator. We use the entropy-flux pair: $\eta = \rho u$, $H = \rho u^{2} + \kappa\rho^{\gamma}$ in \eqref{entropyeq} to obtain the following weak formulation for any deterministic test function $\varphi \in C^{2}(\mathbb{T})$:
\begin{multline*}
\mathbb{E} \int_{\mathbb{T}} \rho_{\epsilon} u_{\epsilon}(T + 1) \varphi - \mathbb{E} \int_{\mathbb{T}} \rho_{\epsilon} u_{\epsilon}(T) \varphi \\
=  \mathbb{E} \int_{T}^{T + 1} \int_{\mathbb{T}} \rho_{\epsilon} u_{\epsilon}^{2} \partial_{x}\varphi dx dt + \mathbb{E} \int_{T}^{T + 1} \int_{\mathbb{T}} {\kappa} \rho_{\epsilon}^{\gamma} \partial_{x}\varphi - \mathbb{E} \int_{T}^{T + 1} \int_{\mathbb{T}} \alpha \rho_{\epsilon} u_{\epsilon} \varphi(x) dx dt + \epsilon \mathbb{E} \int_{T}^{T + 1} \int_{\mathbb{T}} \rho_{\epsilon} u_{\epsilon} \partial^{2}_{x}\varphi.
\end{multline*}
Then, by stationarity, we have that for all $t \ge 0$:
\begin{equation}\label{weakbogovskii}
\mathbb{E} \int_{\mathbb{T}} \kappa \rho_{\epsilon}^{\gamma} \partial_{x}\varphi = -\mathbb{E} \int_{\mathbb{T}} \rho_{\epsilon} u_{\epsilon}^{2} \partial_{x}\varphi dx dt + \mathbb{E} \int_{\mathbb{T}} \alpha \rho_{\epsilon} u_{\epsilon} \varphi(x) dx - \epsilon \mathbb{E} \int_{\mathbb{T}} \rho_{\epsilon} u_{\epsilon} \partial^{2}_{x}\varphi.
\end{equation}
To obtain a higher moment estimate on the density, we will use the Bogovskii operator in 1D and substitute the following test function:
\begin{equation}\label{testbogovskiim}
\varphi = \int_0^x \left(\rho_{\epsilon}^{m(\gamma-1)} - \int_{\mathbb{T}} \rho_{\epsilon}^{m(\gamma-1)} dx \right)dx.
\end{equation}
 Since
\begin{equation*}
\mathbb{E} \int_{\mathbb{T}} \kappa \rho_{\epsilon}^{\gamma} \partial_{x}\left(\int_0^x\rho_{\epsilon}^{m(\gamma-1)}\right) = \kappa \mathbb{E} \int_{\mathbb{T}} \rho_{\epsilon}^{\gamma + m(\gamma-1)},
\end{equation*}
this gives us, for $s=m(\gamma-1)$
\begin{equation}\label{Idefm}
\begin{split}
   \kappa \mathbb{E} \int_{\mathbb{T}} \rho_{\epsilon}^{1 + (m+1)(\gamma-1)} dx&=\kappa \mathbb{E} \int_{\mathbb{T}} \rho_{\epsilon}^{\gamma + m(\gamma-1)} dx\\
&= \kappa \mathbb{E} \int_{\mathbb{T}} \rho_{\epsilon}^{\gamma} \left(\int_{\mathbb{T}} \rho_{\epsilon}^{s}\right) dx - \mathbb{E} \int_{\mathbb{T}} \rho_{\epsilon}^{1 + s} u_\epsilon^{2}  dx
\\
&+ \mathbb{E} \int_{\mathbb{T}} \rho_{\epsilon} u_{\epsilon}^{2} \left(\int_{\mathbb{T}} \rho_{\epsilon}^{s}\right) dx + \mathbb{E} \int_{\mathbb{T}} \alpha \rho_{\epsilon} u_{\epsilon} \varphi(x) dx - \epsilon \mathbb{E} \int_{\mathbb{T}} \rho_\epsilon u_\epsilon \partial_{x}(\rho_\epsilon^{s}) dx \\
&:= I_{1} + I_{2} + I_{3} + I_{4} + I_{5}.
\end{split}
\end{equation}

We estimate the terms on the right-hand side of \eqref{Idefm} under the assumption \eqref{inductive} as follows.

\medskip

\noindent \textbf{Term I1.} 
\iffalse We estimate this in the same way as in Proposition \ref{uniformeps1}, see the calculation \eqref{I1calc} and the estimates that follow that. We obtain in the same way that
\begin{equation*}
|I_{1}| \le \delta \mathbb{E} \int_{\mathbb{T}} \rho_\epsilon^{\gamma + s} dx + C_{\delta}.
\end{equation*}
\fi
For Term I1, we claim that for any $s>0$ we have the estimate:
\begin{equation}\label{I1est}
|I_{1}| \le \delta \mathbb{E} \int_{\mathbb{T}} \rho_\epsilon^{\gamma + m(\gamma-1)} dx + C_{m,\delta}.
\end{equation}
If $0 < s \le 1$, we can use the fact that $\displaystyle \int_{\mathbb{T}} \rho_\epsilon^{s} dx \le M^{s}$ almost surely, by conservation of mass and Young's inequality to immediately deduce this inequality. 

In the case where $s > 1$, we establish \eqref{I1est} by considering $r>1$ to be chosen later, and estimating $I_1$ as follows:
\begin{multline*}
\mathbb{E} \int_{\mathbb{T}} \rho_\epsilon^{\gamma} \left(\int_{\mathbb{T}} \rho_\epsilon^{s} \right) dx \le M^{(r - 1)/r} \mathbb{E} \int_{\mathbb{T}} \rho_\epsilon^{\gamma} \left(\int_{\mathbb{T}} \rho_\epsilon^{{ sr - (r-1)}}\right)^{1/r} dx = C_{r} \mathbb{E}\left[\left(\int_{\mathbb{T}} \rho_\epsilon^{\gamma}\right) \left(\int_{\mathbb{T}} \rho_\epsilon^{sr - (r-1)}\right)^{1/r}\right] \\
\le \frac{\delta}2 \mathbb{E} \left(\int_{\mathbb{T}} \rho_\epsilon^{\gamma}\right)^{\frac{p}{p - 1}} + C_{\delta} {{\mathbb{E}\left(\int_{\mathbb{T}} \rho_\epsilon^{sr-{(r-1)}} dx\right)^{p/r}}}.
\end{multline*}
Let $\displaystyle p = \frac{\gamma + s}{s}$ and $\displaystyle \frac{p}{p - 1} = \frac{\gamma + s}{\gamma}$ and then choose $r$ so that $1 < \displaystyle r < p = \frac{\gamma + s}{s}$. We thus conclude that
\begin{equation*}
|I_{1}| \le \frac{\delta}2 \mathbb{E} \int_{\mathbb{T}} \rho_\epsilon^{\gamma + s} + C_{\delta} \mathbb{E} \int_{\mathbb{T}} \rho_\epsilon^{\gamma + s - \frac{p}{r}{(r-1)}} \le \delta \mathbb{E} \int_{\mathbb{T}} \rho_\epsilon^{\gamma + s} + C_{\delta},
\end{equation*}
by Young's inequality $\rho_\epsilon^{\gamma + s - \frac{p{(r-1)}}{r}} \le \frac{\delta}2\rho_\epsilon^{\gamma + s} + C_{\delta}$, where we note that this inequality applies since $s,r > 1$ and hence the exponent is positive: $\displaystyle \gamma + s - \frac{p}{r} {(r-1)}\ge \gamma - \frac{\gamma}{s} + s - 1 > 0$.
\medskip

\noindent \textbf{Term I2.} Observe that, since $s=m(\gamma-1)$, we have
\begin{align*}
\mathbb{E} \int_{\mathbb{T}} \rho_\epsilon^{1 + s}u_\epsilon^{2} dx &= \mathbb{E} \int_{\mathbb{T}} \rho_\epsilon^{\frac{1}{m + 1}} u_\epsilon^{2} \rho_\epsilon^{s + \frac{m}{m + 1}} \le C_\delta\mathbb{E} \int_{\mathbb{T}} \rho_\epsilon u_\epsilon^{2(m + 1)} dx + { \delta}\mathbb{E} \int_{\mathbb{T}} \rho_\epsilon^{1 + \frac{(m + 1)s}{m}} dx  \\
&=\mathbb{E} \int_{\mathbb{T}} \rho_\epsilon u_\epsilon^{2(m + 1)} dx + { \delta}\mathbb{E} \int_{\mathbb{T}} \rho_\epsilon^{m(\gamma-1)+\gamma} dx\\
&\le C_{m,\delta}+ { \delta}\mathbb{E} \int_{\mathbb{T}} \rho_\epsilon^{m(\gamma-1)+\gamma} dx,
\end{align*}
thanks to our induction assumption \eqref{inductive1}.
\iffalse
Since $(m-1)(\gamma-1) < s \le m(\gamma-1)$, we have that 
\begin{equation}\label{numerologym}
1 + m(\gamma-1) < 1 + (\gamma-1) + s = \gamma + s, \qquad 1 + \frac{(m + 1)s}{m} = 1 + s + \frac{s}{m} \le 1 + s + (\gamma-1) = \gamma + s.
\end{equation}
Therefore, we can apply Young's inequality to obtain:
\begin{equation*}
|I_{2}| \le \delta \mathbb{E} \int_{\mathbb{T}} \rho_\epsilon^{\gamma + s} dx + C_{\delta}.
\end{equation*}
\fi
\medskip

\noindent \textbf{Term I3.} We estimate again thanks to \eqref{inductive1},
\begin{align*}
\mathbb{E} \int_{\mathbb{T}} \rho_\epsilon u_\epsilon^{2} \left(\int_{\mathbb{T}} \rho^{s}_\epsilon \right) dx &\le \mathbb{E} \left(\left(\int_{\mathbb{T}} \rho_\epsilon^{s} dx\right) \left(\int_{\mathbb{T}} \rho_\epsilon dx\right)^{\frac{m}{m + 1}} \left(\int_{\mathbb{T}} \rho_\epsilon u_\epsilon^{2(m + 1)} dx\right)^{\frac{1}{m + 1}}\right) \\
&\le M^{\frac{m}{m + 1}} \left(\mathbb{E} \int_{\mathbb{T}} \rho_\epsilon u_\epsilon^{2(m + 1)} dx + \mathbb{E} \left(\int_{\mathbb{T}} \rho_\epsilon^{s} dx\right)^{\frac{m + 1}{m}}\right) \\
&{{\le C\left(1 + \mathbb{E} \int_{\mathbb{T}} \rho_\epsilon^{\frac{(m + 1)s}{m}} dx\right).}}
\end{align*}
Since 
\begin{align}\label{numerologym}
    \frac{(m + 1)s}{m} = (m+1)(\gamma-1)= m(\gamma-1)+\gamma-1< m(\gamma-1)+\gamma,
\end{align}
we have by Young's inequality:
\begin{equation*}
|I_{3}| \le \delta \mathbb{E} \int_{\mathbb{T}} \rho_\epsilon^{\gamma + m(\gamma-1)} dx + C_{\delta}.
\end{equation*}

\medskip

\noindent \textbf{Term I4.} We compute for $s=m(\gamma-1)$ that
\begin{align*}
\left|\mathbb{E} \int_{\mathbb{T}} \rho_\epsilon u_\epsilon \varphi(x) dx\right| &\le \mathbb{E} \left(\left(\int_{\mathbb{T}} \rho_\epsilon u_\epsilon^{m + 1} dx\right)^{\frac{1}{m + 1}} \left(\int_{\mathbb{T}} \rho_\epsilon dx\right)^{\frac{m}{m + 1}} \|\varphi\|_{L^{\infty}(\mathbb{T})}\right) \\
&\le M^{\frac{m}{m + 1}} \mathbb{E}\left(\left(\int_{\mathbb{T}} \rho_\epsilon u_\epsilon^{m + 1} dx\right)^{\frac{1}{m + 1}} \|\varphi\|_{L^{\infty}(\mathbb{T})}\right) \\
&{{\le M^{\frac{m}{m + 1}}\left(\mathbb{E} \int_{\mathbb{T}} \rho_\epsilon u_\epsilon^{m + 1} dx + \mathbb{E}\Big(\|\varphi\|^{\frac{m + 1}{m}}_{W^{1, (m + 1)/m}(\mathbb{T})}\Big)\right)}} \\
&\le C_{m}\left(1 + \mathbb{E} \int_{\mathbb{T}} \rho_\epsilon^{\frac{(m + 1)s}{m}} dx\right),
\end{align*}
where we used the inductive assumption \eqref{inductive1} and the following
\begin{align*}
    \|\varphi\|_{W^{1,(m+1)/m}}^{\frac{m + 1}{m}} = \left\|\rho_\epsilon^{m(\gamma-1)}-\int_{\mathbb{T}}\rho_\epsilon^{m(\gamma-1)}\right\|_{L^\frac{m + 1}{m}(\mathbb{T})}^{\frac{m + 1}{m}}\leq C_m \int_{\mathbb{T}}\rho_\epsilon^{(m+1)(\gamma-1)}.
\end{align*}
%which applies here since $m + 1 \le 2m$ for all positive integers $m$ and we have a bound by the inductive assumption on $\displaystyle \mathbb{E} \int_{\mathbb{T}} \rho_\epsilon u_\epsilon^{2m} dx \le C_{m}$. 
Again, since we have $\displaystyle \frac{(m + 1)s}{m} < \gamma + s$ by \eqref{numerologym}, we conclude by Young's inequality that
\begin{equation*}
|I_{4}| \le \delta \mathbb{E} \int_{\mathbb{T}} \rho_\epsilon^{\gamma + m(\gamma-1)} dx + C_{\delta}.
\end{equation*}

\noindent \textbf{Term I5.} Finally, we use the gradient bound to conclude for  $s=m(\gamma-1)$ that
\begin{align*}
\epsilon \left|\mathbb{E} \int_{\mathbb{T}} \rho_\epsilon u_\epsilon \partial_{x}(\rho_\epsilon^{s}) dx\right| &= s\epsilon \left|\mathbb{E} \int_{\mathbb{T}} \rho_\epsilon^{s}u_\epsilon \partial_{x}\rho_\epsilon dx\right| = \epsilon \left|\mathbb{E} \int_{\mathbb{T}} \Big(\rho_\epsilon^{1 - \frac{\gamma}{2}} \rho_\epsilon^{s - (m - 1)\frac{(\gamma-1)}{2}}\Big) \Big((\partial_{x}\rho_\epsilon) \rho_\epsilon^{\frac{\gamma}{2} - 1} \rho_\epsilon^{(m - 1)\frac{(\gamma-1)}{2}}u_{\epsilon}\Big)\right| \\
&\le \delta \mathbb{E} \int_{\mathbb{T}} \rho_\epsilon^{2 - \gamma+2s - (m - 1)(\gamma-1)} dx + C_\delta \mathbb{E} \int_{\mathbb{T}} \rho_\epsilon^{(m - 1)(\gamma-1)}u_\epsilon^{2}\Big(\rho_\epsilon^{\gamma - 2}|\partial_{x}\rho_{\epsilon}|^{2}\Big) dx.
\end{align*}
In the first integral we observe that for $s=m(\gamma-1)$, the exponent ${{2 - \gamma + 2s - (m - 1)(\gamma-1)}=1+m(\gamma-1)}$.
In the second integral, we estimate $\rho_\epsilon^{(m - 1)(\gamma-1)}u_\epsilon^{2} \le C_{m}(\rho_\epsilon^{m(\gamma-1)} + u_\epsilon^{2m})$ by using Young's inequality with exponents $m$ and $m/(m - 1)$. Thanks to \eqref{inductive1} we hence obtain,
\begin{align*}
\epsilon\left|\mathbb{E} \int_{\mathbb{T}} \rho_\epsilon u_\epsilon \partial_{x}(\rho_\epsilon^{s}) dx\right| &\le \delta \mathbb{E} \int_{\mathbb{T}} \rho_\epsilon^{1 +m(\gamma-1)} ds + C_{m,\delta} \mathbb{E} \int_{\mathbb{T}} \Big(\rho_\epsilon^{m (\gamma-1)} + u_\epsilon^{2m}\Big) \rho_\epsilon^{\gamma - 2} |\partial_{x}\rho_\epsilon|^{2} dx \\
&\le C_{m,\delta} + \delta \mathbb{E} \int_{\mathbb{T}} \rho_\epsilon^{1 + m(\gamma-1)} dx.
\end{align*}
\iffalse
Since $(m - 1)(\gamma-1) < s \le m(\gamma-1)$, we have that $s - (m - 1)(\gamma-1) < \gamma - 1 $. Hence, 
\begin{equation*}
s < 2\gamma + 2(m - 1)(\gamma-1) - 2 \quad \Longrightarrow \quad 2 - \gamma + 2s - 2(m - 1)\theta < \gamma + s.
\end{equation*}
Since $1 + m(\gamma-1) < \gamma + s$ also, we can apply Young's inequality to obtain:
\fi 
Thus, for $\gamma>1$, we have
\begin{equation*}
|I_{5}| \le \delta \mathbb{E} \int_{\mathbb{T}} \rho_{\epsilon}^{\gamma + m(\gamma-1)} + C_{m,\delta}.
\end{equation*}

\medskip

\noindent \textbf{Conclusion of the proof.} Combining all of the estimates of $I_{1}$ through $I_{5}$ in \eqref{Idefm}, we hence obtain:
\begin{equation*}
\mathbb{E} \int_{\mathbb{T}} \rho_{\epsilon}^{1 + (m+1)(\gamma-1)} dx = \mathbb{E} \int_{\mathbb{T}} \rho_{\epsilon}^{\gamma + m(\gamma-1)} dx \le \delta \mathbb{E} \int_{\mathbb{T}} \rho_{\epsilon}^{\gamma + m(\gamma-1)} dx + C_{m,\delta}.
\end{equation*}
Hence, taking $\delta$ sufficiently small, we conclude the proof.
\end{proof}

\subsection{Additional estimates that are uniform in $\epsilon$.}\label{additionalest} Using the fundamental uniform entropy estimates in Proposition \ref{uniformeps2}, we can derive additional estimates that are uniform in $\epsilon$, which will be important in the limit passage as $\epsilon_{N} \to 0$. The first estimate is a uniform moment bound on the entropy and flux $(\eta_{m}, H_{m})$. 

\begin{proposition}\label{fluxmomenteps}
    Let $(\rho_{\epsilon}, q_{\epsilon})$ be the $\epsilon$ level statistically stationary solutions to the approximate $\epsilon$ level problem in \eqref{approxsystem_ep}. Then, for every positive integer $m$, every $1 \le p < \infty$, and for all $t \ge 0$:
    \begin{equation*}
    \mathbb{E} \int_{\mathbb{T}} |\eta_{m}(\bd{U}_{\epsilon})(t)|^{p} \le C(m, p), \qquad \mathbb{E} \int_{\mathbb{T}} |H_{m}(\bd{U}_{\epsilon})(t)|^{p} \le C(m, p),
    \end{equation*}
    for a constant $C(m, p)$ depending only on $m$ and $p$, and \textit{independent of both} $\epsilon > 0$ and $t \ge 0$. 
\end{proposition}

\begin{proof}
\iffalse
By the uniform estimates in Proposition \ref{uniformeps2} that are $\epsilon$ independent, we deduce that for all nonnegative integers $m$ and for all $t \ge 0$:
\begin{equation}\label{etamexpectation}
\mathbb{E} \int_{\mathbb{T}} \eta_{m}(\bd{U}_{\epsilon}(t)) \le C_{m},
\end{equation}
for a constant $C_{m}$ that is independent of $\epsilon > 0$ and $t \ge 0$, and which hence depends only on $m$. 
\fi 
This follows directly from combining the bounds \eqref{etaboundsm} which hold for every positive integer $m$, obtained in Proposition \ref{uniformeps2}, with the algebraic identities in Lemma \ref{momentidentity}. 
\end{proof}

As a direct corollary, we can obtain uniform boundedness of the moments of the state variables, the density and momentum, independently of $\epsilon$. 

\begin{corollary}\label{momentrhoq}
    There exists a constant $C_{p}$ depending only $p$ (and independent of $\epsilon$ and $t \ge 0$), such that for all $1 \le p < \infty$ and for all $t \ge 0$:
    \begin{equation*}
    \mathbb{E} \int_{\mathbb{T}} |\rho_{\epsilon}(t)|^{p} \le C_{s}, \qquad \mathbb{E} \int_{\mathbb{T}} |q_{\epsilon}(t)|^{p} \le C_{p}.
    \end{equation*}
\end{corollary}

\begin{proof}
    Using the algebraic bounds in Proposition \ref{etamboundalg}: 
    \begin{equation*}
    0 \le \rho u^{2m} + \rho^{1 + m(\gamma-1)} \le c_m \eta_{m}(\bd{U}),
    \end{equation*}
    for a positive constant $c_{m}$. So the following bounds follow immediately from the previous Proposition \ref{fluxmomenteps} %. To show the corresponding result for the momentum $q_{\epsilon} := \rho_{\epsilon}u_{\epsilon}$, note that 
    for all $1 \le p< \infty$:
    \begin{equation*}
    \mathbb{E} \int_{\mathbb{T}} |\rho_{\epsilon} u_{\epsilon}^{2p}(t)| \le C_{p}, \quad \text{ and } \quad \mathbb{E} \int_{\mathbb{T}} |\rho_{\epsilon}(t)|^{p} \le C_{p}, \qquad \text{ for all } t \ge 0,
    \end{equation*}
    where the constant $C_{p}$ is independent of $t \ge 0$ and $\epsilon > 0$. %The first inequality here follows from Proposition \ref{fluxmomenteps} applied with $\eta$ being the energy functional. 
    Then, using H\"{o}lder's inequality, for all $1 \le p < \infty$: 
    \begin{equation*}
    \mathbb{E} \int_{\mathbb{T}} |\rho_{\epsilon}^{2}u_{\epsilon}^{2}|^{p} \le C_{p} \qquad \text{ for all } t \ge 0,
    \end{equation*}
    for some uniform constant $C_{p}$ independent of $\epsilon$ and $t \ge 0$, which establishes the desired result. 
\end{proof}

We next establish uniform local-in-time H\"{o}lder continuity bounds on the approximate solutions $\bd{U}_{\epsilon} = (\rho_{\epsilon}, q_{\epsilon})$. To do this, we will use the uniform moment bounds in Proposition \ref{fluxmomenteps} and Corollary \ref{momentrhoq} to deduce an extension of the result in Proposition \ref{uniformeps2} to a maximum in time locally.

\begin{proposition}\label{momentsup}
    { For every $m\geq 1$}, there exists a constant $C_{ m}(t_2 - t_1)$ that is independent of $\epsilon$ such that
    \begin{equation}\label{momentsupbase}
    \mathbb{E}\sup_{t \in [t_1, t_2]} \int_{\mathbb{T}} \eta_{m}(\bd{U}_{\epsilon})(t) \le C_{ m}(t_2 - t_1).
    \end{equation}
    In addition, there exists a constant $C(t_2 - t_1, p)$ that is independent of $\epsilon$ such that for all $1 \le p < \infty$: 
    \begin{equation*}
    \mathbb{E}\sup_{t \in [t_1, t_2]} \int_{\mathbb{T}} |\eta_{m}(\bd{U}_{\epsilon})|^p \le C(t_2 - t_1, p), \qquad \mathbb{E} \sup_{t \in [t_1, t_2]} \int_{\mathbb{T}} |H_m(\bd{U}_{\epsilon})|^p \le C(t_2 - t_1, p),
    \end{equation*}
    \begin{equation*}
    \mathbb{E} \sup_{t \in [t_{1}, t_{2}]} \int_{\mathbb{T}} |\rho_{\epsilon}(t)|^p \le C(t_2 - t_1, p), \qquad \mathbb{E} \sup_{t \in [t_1, t_2]} \int_{\mathbb{T}} |q_{\epsilon}(t)|^p \le C(t_2 - t_1, p).
    \end{equation*}
\end{proposition}

\begin{proof}
    It suffices to show the estimate \eqref{momentsupbase}, since the remaining estimates follow using the algebraic identities on the entropy in Proposition \ref{etamboundalg} {and Lemma \ref{momentidentity}}, as in the proofs of Proposition \ref{fluxmomenteps} and Corollary \ref{momentrhoq}. So we prove the estimate in \eqref{momentsupbase}. To do this, we use the entropy formulation \eqref{entropyeq} with a test function $\varphi = 1$, and we hence obtain for $t \ge t_1$:
    \begin{align*}
    \int_{\mathbb{T}} \eta_{m}(\bd{U}_{\epsilon})(t) = \int_{\mathbb{T}} \eta_{m}(\bd{U}_{\epsilon})(t_{1}) &- \int_{t_{1}}^{t} \int_{\mathbb{T}} \alpha q_{\epsilon}\partial_{q}\eta(\bd{U}_{\epsilon}) + \int_{t_{1}}^{t} \int_{\mathbb{T}} \partial_{q}\eta(\bd{U}_{\epsilon})\bd{\Phi}^{\epsilon}(\bd{U}_{\epsilon}) dW(t) \\
    &+ \int_{t_{1}}^{t} \int_{\mathbb{T}} \frac{1}{2} \partial_{q}^{2}\eta(\bd{U}_{\epsilon})G_{\epsilon}^{2}(\bd{U}_{\epsilon}) + \epsilon \int_{t_{1}}^{t} \int_{\mathbb{T}} \langle D^{2}\eta(\bd{U}_{\epsilon})\partial_{x}\bd{U}_{\epsilon}, \partial_{x}\bd{U}_{\epsilon} \rangle.
    \end{align*}
    We then take a supremum over $t \in [t_1, t_2]$ and expectation, and obtain the following terms.

    \begin{itemize}
        \item First, we note that by Proposition \ref{uniformeps2}:
        \begin{equation*}
        \mathbb{E}\int_{\mathbb{T}} \eta_{m}(\bd{U}_{\epsilon})(t_{1}) \le C_{m}.
        \end{equation*}
        \item Next, by the algebraic bounds in Proposition \ref{etamboundalg} and Young's inequality, we have 
        \begin{equation*}
        |q_{\epsilon}\partial_q\eta_{m}(\bd{U}_{\epsilon})| \le C_{m}\rho(u^{2m} + \rho^{(m - 1)(\gamma-1)} u^{2}) \le C_{m}\rho(u^{2m} + \rho^{m(\gamma-1)} )  \le C_{m}\eta_{m}(\bd{U}), 
        \end{equation*}
        \begin{equation*}
        |\partial^{2}_{q}\eta(\bd{U}_{\epsilon})G^{2}(\bd{U}_{\epsilon})| \le { A_0|\partial^{2}_{q}\eta(\bd{U}_{\epsilon})\rho_\epsilon^2|  } \leq  C_{m}\eta_{m - 1}(\bd{U}_{\epsilon}).
        \end{equation*}
        Therefore, we estimate: 
        \begin{align*}
        \mathbb{E}\int_{t_{1}}^{t_{2}} \int_{\mathbb{T}} |\alpha q_{\epsilon}\partial_{q}\eta(\bd{U}_{\epsilon})| + |\partial^{2}_{q}\eta(\bd{U}_{\epsilon})G^{2}_{\epsilon}(\bd{U}_{\epsilon})| &\le C_{m}\int_{t_{1}}^{t_{2}} \left(\mathbb{E} \int_{\mathbb{T}} \eta_{m - 1}(\bd{U}_{\epsilon}) + \eta_{m}(\bd{U}_{\epsilon})\right) \\
        &\le C_{m}(t_{2} - t_{1}),
        \end{align*}
        by Proposition \ref{uniformeps2}.
        \iffalse {\cred By stationarity and \eqref{directbound} we have

        \begin{align*}
 \mathbb{E}\int_{t_1}^{t_2} \int_{\mathbb{T}} \alpha q_{\epsilon} \partial_{q}\eta_{m}(\bd{U}_{\epsilon}) dx &+ \mathbb{E} \int_{t_1}^{t_2}\int_{\mathbb{T}}  \partial^{2}_{q}\eta_{m}(\bd{U}_{\epsilon}) G^{2}(\bd{U}_{\epsilon}) dx \le 2 \int_{t_1}^{t_2}\mathbb{E} \int_{\mathbb{T}} \partial^{2}_{q}\eta_{m}(\bd{U}_{\epsilon}) G^{2}(\bd{U}_{\epsilon}) dx\\
 &\le C_{m}\int_{t_1}^{t_2} \mathbb{E} \int_{\mathbb{T}} \eta_{m - 1}(\bd{U}_\epsilon) \leq C_m(t_2-t_1).
\end{align*}
 }
 \fi 
        \item Recall that by stationarity (for example, by combining inequality \eqref{epsbound} with Proposition \ref{uniformeps2}), we have that $\displaystyle \epsilon \mathbb{E} \int_{\mathbb{T}} \langle D^{2}\eta_{m}(\bd{U}_{\epsilon})\partial_{x}\bd{U}_{\epsilon}, \partial_{x}\bd{U}_{\epsilon} \rangle(t) \le C_{m}$ for a constant $C_{m}$ that is independent of $t \ge 0$ and $\epsilon$. Therefore,
        \begin{equation*}
        \epsilon \bE\int_{t_{1}}^{t_{2}} \langle D^{2}\eta_{m}(\bd{U}_{\epsilon})\partial_{x}\bd{U}_{\epsilon}, \partial_{x}\bd{U}_{\epsilon}\rangle \le C_m(t_2 - t_1).
        \end{equation*}
        \item Finally, using the Burkholder-Davis-Gundy inequality, we estimate that
        \begin{equation*}
        \mathbb{E}\sup_{t \in [t_1, t_2]} \left|\int_{t_{1}}^{t} \left(\int_{\mathbb{T}} \partial_{q}\eta(\bd{U}_{\epsilon})\bd{\Phi}^{\epsilon}(\bd{U}_{\epsilon}) dx\right)dW(t)\right| \le \mathbb{E} \left|\int_{t_{1}}^{t_{2}} \sum_{k = 1}^{\infty} \left(\int_{\mathbb{T}} \partial_{q}\eta_{m}(\bd{U}_{\epsilon}) G_{k}^{\epsilon}(\bd{U}_{\epsilon}) dx\right)^{2} dt\right|^{1/2}.
        \end{equation*}
        Using the algebraic bounds in Proposition \ref{etamboundalg}:
        \begin{align*}
        |\partial_{q}\eta_{m}(\bd{U}_{\epsilon})G_{k}^{\epsilon}(\bd{U}_{\epsilon})| &\le C_{m}\alpha_{k}\rho_{\epsilon}(|u_{\epsilon}|^{2m - 1} + \rho_{\epsilon}^{2(m - 1)\theta} |u_{\epsilon}|) \\
        &\le C_{m}\alpha_{k}\rho_{\epsilon}(1 + |u_{\epsilon}|^{2m} + \rho_{\epsilon}^{m(\gamma-1)}) \le C_{m}\alpha_{k}\Big(\eta_{0}(\bd{U}_{\epsilon}) + \eta_{m}(\bd{U}_{\epsilon})\Big).
        \end{align*}
        Hence, using \eqref{Gepsbound}:
        \begin{equation*}
        \mathbb{E} \sup_{t \in [t_1, t_2]} \left|\int_{t_{1}}^{t} \int_{\mathbb{T}} \partial_{q}\eta(\bd{U}_{\epsilon}) \bd{\Phi}^{\epsilon}(\bd{U}_{\epsilon}) dW\right| \le C_{m}A_{0}\mathbb{E} \int_{t_{1}}^{t_{2}} \int_{\mathbb{T}} \Big[\Big(\eta_{0}(\bd{U}_{\epsilon})\Big)^{2} + \Big(\eta_{m}(\bd{U}_{\epsilon})\Big)^{2}\Big] \le C_m(t_2 - t_1),
        \end{equation*}

        by the uniform moment bounds on the entropies in Proposition \ref{fluxmomenteps}.
    \end{itemize}
    Hence, we deduce that
    \begin{equation*}
    \mathbb{E} \sup_{t \in [t_1, t_2]} \int_{\mathbb{T}} \eta_{m}(\bd{U}_{\epsilon}) \le C_{m}\Big(1 + (t_2 - t_1)\Big). 
    \end{equation*}
\end{proof}

Finally, we conclude this section on uniform in $\epsilon$ bounds on the statistically stationary solutions by deriving the following uniform bound on H\"{o}lder continuity in time, for the approximate statistically stationary solutions $(\rho_{\epsilon}, q_{\epsilon})$.

\begin{proposition}\label{Hnegative2}
    For each $\beta \in (0, 1/4)$ and every $T > 0$, there exists a constant $C_{\beta, T}$ independent of $\epsilon$ such that
    \begin{equation*}
    \mathbb{E}\|(\rho_{\epsilon}, q_{\epsilon})\|_{C^{\beta}(0, T; H^{-2}(\mathbb{T}))} \le C_{\beta, T},
    \end{equation*}
where $H^{-2}(\mathbb{T})$ denotes the dual of $H^{2}(\mathbb{T})$.
\end{proposition}

\begin{proof}
    We can establish this bound by using the Kolmogorov continuity criterion. Namely, it suffices to show that for each deterministic $\varphi \in H^{2}(\mathbb{T})$ with $\|\varphi\|_{H^{2}(\mathbb{T})} \le 1$, and for all times $0 \le t_1 \le t_2 \le T$:
    \begin{equation}\label{increments}
    \mathbb{E} |\langle \rho_{\epsilon}(t_2) - \rho_{\epsilon}(t_1), \varphi \rangle|^{4} \le C_T|t_1 - t_2|^{2}, \qquad \mathbb{E} |\langle q_{\epsilon}(t_2) - q_{\epsilon}(t_1), \varphi\rangle|^{4} \le C_T|t_1 - t_2|^{2}.
    \end{equation}
    
    Using the weak formulation (which follows from the entropy equality \eqref{entropyeq} for the approximate system \eqref{approxsystem_ep}):
    \begin{equation*}
    \mathbb{E}\left|\int_{\mathbb{T}} \Big(\rho_{\epsilon}(t_2) - \rho_{\epsilon}(t_1)\Big)\varphi\right|^{4} \le C\left(\mathbb{E}\left|\int_{t_1}^{t_2} \int_{\mathbb{T}} q_{\epsilon} \partial_{x}\varphi\right|^{4} + \mathbb{E} \left|\int_{t_1}^{t_2} \int_{\mathbb{T}} \rho_{\epsilon} \partial_{x}^{2}\varphi\right|^{4}\right) \le C|t_2 - t_1|^{4} \le CT^2|t_2 - t_1|^{2},
    \end{equation*}
  using Proposition \ref{momentsup} and the fact that $\|\varphi\|_{H^{2}(\mathbb{T})} \le 1$. Similarly, we estimate using the weak formulation, that
    \begin{multline*}
    \mathbb{E} \left|\int_{\mathbb{T}} \Big(q_{\epsilon}(t_2) - q_{\epsilon}(t_1)\Big)\varphi\right|^{4} \\
    \le C\Bigg(\mathbb{E} \left|\int_{t_1}^{t_2} \int_{\mathbb{T}} \frac{q_{\epsilon}^{2}}{\rho_{\epsilon}} \partial_{x}\varphi\right|^{4} + \mathbb{E} \left|\int_{t_1}^{t_2} \int_{\mathbb{T}} \kappa \rho_{\epsilon}^{\gamma} \partial_{x}\varphi\right|^{4} + \mathbb{E} \left|\int_{t_1}^{t_2} \int_{\mathbb{T}} \epsilon q_{\epsilon} \partial_{x}^{2}\varphi\right|^{4} + \mathbb{E} \left|\int_{t_1}^{t_2} \int_{\mathbb{T}} \bd{\Phi}^{\epsilon}(\rho_{\epsilon}, q_{\epsilon}) \varphi dx dW_\epsilon(t)\right|^{4}\Bigg).
    \end{multline*}
    We can then estimate, using Proposition \ref{momentsup}:
    \begin{align*}
    \mathbb{E} \left|\int_{t_1}^{t_2} \int_{\mathbb{T}} \frac{q_{\epsilon}^{2}}{\rho_{\epsilon}} \partial_{x}\varphi\right|^{4} &\le \mathbb{E} \left|\int_{t_{1}}^{t_{2}} \|\eta_{1}(\bd{U}_{\epsilon})\|_{L^{2}(\mathbb{T})}\right|^{4} \le |t_2 - t_1|^{4} \cdot \mathbb{E} \sup_{t \in [t_1, t_2]} \left(\int_{\mathbb{T}} |\eta_{1}(\bd{U}_{\epsilon})|^{2}\right)^{2} \\
    &\le |t_2 - t_1|^{4}\left(\mathbb{E} \sup_{t \in [t_1, t_2]} \int_{\mathbb{T}} |\eta_{1}(\bd{U}_{\epsilon})|^{4}\right) \le C_T |t_2 - t_1|^{4}.
    \end{align*}
    We can similarly use the higher moment bounds on the supremum in time of $(\rho_{\epsilon}, q_{\epsilon})$ in Proposition \ref{momentsup} to deduce that
    \begin{equation*}
    \mathbb{E}\left|\int_{t_{1}}^{t_{2}} \int_{\mathbb{T}} \kappa \rho_{\epsilon}^{\gamma} \partial_{x}\varphi\right|^{4} + \mathbb{E}\left|\int_{t_{1}}^{t_{2}} \int_{\mathbb{T}} \epsilon q_{\epsilon} \partial_{x}^{2}\varphi\right|^{4} \le C_T|t_2 - t_1|^{4}.
    \end{equation*}
    Finally, we use the BDG inequality, \eqref{Gepsbound}, \eqref{GN}, and $\|\varphi\|_{L^{\infty}(\mathbb{T})} \le C\|\varphi\|_{H^{2}(\mathbb{T})} \le C$, to estimate:
    \begin{align*}
    \mathbb{E} \left|\int_{t_{1}}^{t_{2}} \int_{\mathbb{T}} \bd{\Phi}^{\epsilon}(\rho_{\epsilon}, q_{\epsilon}) \varphi dx dW_\epsilon(t)\right|^{4} &\le \mathbb{E}\left|\int_{t_{1}}^{t_{2}} \sum_{k = 1}^{\infty} \left(\int_{\mathbb{T}} G_{k}^{\epsilon}(\rho_{\epsilon}, q_{\epsilon})\varphi dx\right)^{2} dt\right|^{2} \le C\mathbb{E} \left(\int_{t_{1}}^{t_{2}} \sum_{k = 1}^{\infty} \int_{\mathbb{T}} |G_{k}^{\epsilon}(\rho_{\epsilon}, q_{\epsilon})|^{2}\right)^{2} \\
    &\le C A_{0} |t_2 - t_1|^{2} \cdot \mathbb{E}  \sup_{t \in [t_1, t_2]} \int_{\mathbb{T}} \rho_{\epsilon}^{2} \le C_{T}|t_2 - t_1|^{2}.
    \end{align*}

    This establishes both estimates in \eqref{increments}, and completes the proof of the proposition via the Kolmogorov continuity criterion.
\end{proof}

\section{Passage to the limit: $\epsilon_{N} \to 0$}

We now carry out the final limit passage in the stationary solutions $(\rho_{\epsilon}, q_{\epsilon}, W_{\epsilon})$ as $\epsilon_{N} \to 0$. Note that the laws of $(\rho_{\epsilon}, q_{\epsilon})$ are tight in $C_{loc}(\R^{+}; H^{-3}(\mathbb{T}))$ by the uniform equicontinuity estimate in Proposition \ref{Hnegative2}. However, having almost sure strong convergence in the very weak space $C(0, T; H^{-3}(\mathbb{T}))$ is not enough to pass to the limit in the approximate entropy equality \eqref{entropyeq}, in order to obtain the entropy inequality \eqref{entropyineq} for the limiting solution $(\rho, q)$ to the original problem. In particular, to pass to the limit as $\epsilon_{N} \to 0$ in the entropy equality \eqref{entropyeq}, we must be able to pass to the limit in the composition of continuous functions with the solution itself, as we must pass to the limit in $\eta(\bd{U}_{\epsilon})$ and $H(\bd{U}_{\epsilon})$ for entropy-flux pairs $(\eta, H)$. To do this, we will appeal to theory of \textit{Young measures}, which gives a way of passing to the limit in compositions with continuous functions, at the price of making the solution measure-valued, so at every point $(t, x) \in [0, \infty) \times \mathbb{T}$, the solution is a measure on the state space $[0, \infty) \times \R$ for the density and momentum. %Therefore, we will have to consider measure-valued solutions, and the limiting solution will be measure-valued. 
However, we will be able to reduce the measure-valued limit to an actual real-valued function using a reduction of Young measure argument that is standard in the literature for the compressible isentropic Euler equations in 1D. See Section 5.2 in \cite{BerthelinVovelle} and Section I.5 in \cite{LPS96}. This will then allow us to pass to the limit as $\epsilon_{N} \to 0$ in the entropy equality and complete the proof.
%{\cred shouldn't we already have genuine functions? We have tightness in $(L^\infty(0,T;L^p),w^*)$ and $C_w(0,T;L^p)$ for any $p$ since it is compactly embedded in $L^\infty(0,T;L^p)\cap C^\alpha(0,T;H^{-\alpha})$?}

In this section, we will hence give an exposition on (random) Young measures in the context of the current problem. Then, we will use these results to apply the Skorohod representation theorem to the approximate solutions in $\epsilon$ to obtain a limiting (measure-valued) solution in the limit. We will then appeal to usual reduction of Young measure arguments for the compressible isentropic Euler equations, which involves using a functional equation for the limiting solution, in order to show that our limiting stationary solution to the original problem is genuinely function-valued. Finally, we pass to the limit in the entropy equality as $\epsilon_{N} \to 0$ to conclude the proof, and hence obtain a statistically stationary weak martingale solution to the original problem \eqref{maineqn}.

\subsection{An exposition on Young measures}

Let $(X, \lambda)$ be a $\sigma$-finite measure space, and let $(E, \mathcal{B}(E))$ be a topological space $E$ equipped with the Borel $\sigma$-algebra, namely the $\sigma$-algebra generated by all open and closed sets in the topology of $E$. Then, recall that a \textit{measurable function} $f: X \to E$ is a function for which $f^{-1}(A)$ is a measurable subset of $X$ for every measurable subset $A \subset E$. For concreteness and to elucidate the notation, we remark that the approximate solutions $(\rho_{\epsilon}, q_{\epsilon})$ are measurable functions from $X = [0, \infty) \times \mathbb{T}$ (spacetime) to $E = [0, \infty) \times \mathbb{R}$ (the set of admissible values for the density and momentum). We hence make the distinction between a measurable function and a Young measure, which unlike a genuine real-valued function, is probability measure-valued at each point $x \in X$. Specifically, we have the following definition.

\begin{definition}
    A \textbf{Young measure} is a map $\nu: X \to \mathcal{P}(E)$, where $\mathcal{P}(E)$ is the set of probability measures on $E$ equipped with the weak-star topology. A Young measure is required to be \textit{weakly-star measurable} in the sense that for every continuous and bounded function $\varphi \in C_{b}(E)$, the function
    \begin{equation*}
    f_{\varphi}(x) = \langle \varphi, \nu_{x}\rangle := \int_{E} \varphi(p) d\nu_{x}(p)
    \end{equation*}
    is a measurable map from $X$ to $E$, where we denote the (probability measure) value of $\nu$ at the point $x \in X$ by $\nu_{x}$. 
    \iffalse 
    {\color{red}{Given a Young measure $\nu: X \to \mathcal{P}(E)$, we can also identify the Young measure with an associated probability measure $\nu \in \mathcal{P}(X \times E)$ (with total mass equal to the measure of $X$), as in , whose action on continuous and bounded functions $\psi(x, p) \in C_{b}(X \times E)$ is given as follows:
    \begin{equation}\label{youngidentify}
    \langle \psi(x, p), \nu \rangle = \int_{X} \int_{E} \psi(x, p) d\nu(p) d\lambda(x).
    \end{equation}}}
    \fi 
\end{definition}

For more information, we direct the reader to Sections 4.1 and 4.2 of \cite{BerthelinVovelle}.

Note that a Young measure is a generalization of a pointwise-defined function. In the case where $f: X \to E$ is a pointwise-defined measurable function, we can associate to $f$ the corresponding Young measure $\nu^f$ such that:
\begin{equation*}
\nu^{f}_{x} = \delta_{f(x)},
\end{equation*}
where $\delta_{f(x)}$ is a Dirac-delta probability measure on $E$ centered at $f(x) \in E$. In this case, we can view the action of integrating against the Young measure as function composition, since
\begin{equation*}
\langle \varphi, \nu^{f}_{x} \rangle = \int_{E} \varphi(p) d\delta_{f(x)}(p) = \varphi(f(x)), \qquad \text{ for all } x \in X \text{ and } \varphi \in C_{b}(E). 
\end{equation*}

We can define the space of Young measures $\mathcal{V}$ on the space $X$ taking values in $\mathcal{P}(E)$, and we equip this space with the (vague) weak-star topology of convergence. Namely, we have the following definition of convergence.

\begin{definition}\label{vague}
A \textbf{sequence of Young measure $\{\nu_{n}\}_{n = 1}^{\infty}: X \to \mathcal{P}(E)$ converges weakly-star to a limiting Young measure $\nu$} if 
\begin{equation*}
\int_{X} \psi(x) \int_{E} \varphi(p) d\nu_{x}(p) d\lambda(x) \to \int_{X} \psi(x) \int_{E} \varphi(p) d\nu_x(p) d\lambda(x), \quad \text{ for all } \varphi \in C_{b}(E) \text{ and } \psi \in C_{c}^{\infty}(X).
\end{equation*}
\end{definition}

We refer the reader to Section 2.8 of \cite{BFH18} for more information about Young measures, and for more exposition about the ideas above. %{\color{red}{In the case where $X$ is a space of finite measure (for example, a compact subset of Euclidean space), we remark that the notion of convergence in Definition \ref{vague} coincides with the notion of weak convergence of probability measures on $X \times E$, when considering the sequence of Young measures $\{\nu_{n}\}_{n = 1}^{\infty}$ as elements in $\mathcal{P}(X \times E)$ via the identification \eqref{youngidentify}. See Section 4.2 in \cite{BerthelinVovelle} for more details.}} 

Finally, we note that since we are considering a stochastic problem, we must also consider a notion of random Young measures.

\begin{definition}\label{Vdef}
Let $(\Omega, \mathcal{F}, \mathbb{P})$ be a probability space. A \textbf{random (probabilistic) Young measure} is a measurable map from $\Omega \to \mathcal{V}$, where $\mathcal{V}$ is the space of Young measures on $X$ taking values in $\mathcal{P}(E)$, equipped with the (vague) weak-star topology of convergence.
\end{definition}

Next, we will apply the theory of probabilistic Young measures to the current problem, in terms of expressing the approximate solutions $(\rho_{\epsilon}, q_{\epsilon})$ in terms of Young measures. For $\epsilon > 0$ and any $\omega\in \Omega$, these are genuine pointwise-defined functions from $(t, x) \in [0, \infty) \times \mathbb{T}$ to $(\rho_{\epsilon}, q_{\epsilon}) \in [0, \infty) \times \R$. For the purpose of using probabilistic Young measures, it will often be more convenient to consider these approximate solutions in terms of the density and the \textit{fluid velocity} (instead of the fluid momentum), so that the approximate solutions are $(\rho_{\epsilon}, u_{\epsilon})$, since the entropy functions $\eta_{m}$ are polynomials in $\rho$ and $u$ (rather than $\rho$ and $q$).

A technical difficulty here is defining the fluid velocity, $u_{\epsilon} = q_{\epsilon}/\rho_{\epsilon}$ when there is vacuum $\rho_{\epsilon} = 0$. For this purpose, we will define the phase space $[0, \infty) \times \R$ with a different topology that identifies all points with $\{\rho = 0\}$ as the same. %(since at all of these points, regardless of the value of $u_{\epsilon}$, the density and momentum are both equal to zero). 
This is in the spirit of Section 2.3 in \cite{LFW07}, and we introduce the following definitions that follow the notation of \cite{LFW07}.

\begin{definition}\label{Htopology}
Let $H := (0, \infty) \times \R \subset \R^{2}$ be equipped with the usual topology of Euclidean space (the subspace topology). Consider $\overline{\mathcal{H}} := [0, \infty) \times \R$ as a compactification of $H$, and define the space $C_{b}(\overline{\mathcal{H}})$ of functions $f$ on $\overline{\mathcal{H}}$ such that:
    \begin{itemize}
        \item $f$ is a continuous function on $[0, \infty) \times \R$ in the usual sense.
        \item $f$ is constant along the vacuum set $V := \{(\rho, u) : \rho = 0\}$.
        \item The function $f_{\infty}(z) := \lim_{r \to \infty} f(rz)$ is a bounded continuous function on $S^{1} \cap ([0, \infty) \times \R)$, where $S^{1}$ is the unit circle in $\R^{2}$. 
    \end{itemize}
We then endow the space $\overline{\mathcal{H}}$ with the following topology of convergence, where $x_{n} \to x$ in $\overline{\mathcal{H}}$ if $f(x_n) \to f(x)$ for all $f \in C_{b}(\overline{\mathcal{H}})$. Note that with this topology, $C_{b}(\overline{\mathcal{H}})$ coincides with the continuous bounded functions on $\overline{\mathcal{H}}$. 
\end{definition}

\begin{remark}
    This compactification $\overline{\mathcal{H}}$ can be thought of as a closed half 2-sphere with the space $H$ being the open half 2-sphere, the vacuum set $V$, being reduced to a single point on the circular boundary of the half 2-sphere, and all upper half plane points at infinity being the rest of the circular boundary of the half 2-sphere. Hence, topologically, the compactification reduces all points in $V$ to be the ``same" point, which makes sense, as points in $V$ are indistinguishable (since the fluid velocity $u$ does not have a well-defined value and momentum is 0 at vacuum). 
\end{remark}

The Young measures for the approximate solutions $(\rho_{\epsilon}, u_{\epsilon})$ will be defined on $(t, x) \in [0, \infty) \times \mathbb{T}$, and will take values in $\mathcal{P}(\overline{\mathcal{H}})$. In the next subsection, we explicitly define these Young measures and then state results about compactness criteria for these approximate Young measures.

\subsection{Young measures for approximate solutions and compactness results}

We consider the approximate solutions $(\rho_{\epsilon}, u_{\epsilon})$ for each parameter $\epsilon > 0$, defined for $(t, x) \in [0, \infty) \times \mathbb{T}$. These are random (genuinely pointwise-defined) functions taking values in $\overline{\mathcal{H}}$, where we recall that vacuum points with $\rho = 0$ in $\overline{\mathcal{H}}$ are considered to be indistinguishable. Thus, we can define an associated probabilistic Young measure for each approximate solution $(\rho_{\epsilon}, u_{\epsilon})$, by
\begin{equation*}
\nu_{\epsilon} := \delta_{(\rho_{\epsilon}, u_{\epsilon})}.
\end{equation*}
These are random Young measures on the sigma-finite measure space $[0, \infty) \times \mathbb{T}$, with respect to the (range) space $\overline{\mathcal{H}}$. The goal will be to show that this sequence of random Young measures is tight, as random processes taking values in the space $\mathcal{V}$ of Young measures from $[0, \infty) \times \mathbb{T}$ to ${\mathcal{P}(\overline{\mathcal{H}})}$.

This requires compactness results for Young measures. Namely, recall that (deterministic) Young measures $\nu_{n}$ converge to $\nu$ in the (vague) weak-star topology of $\mathcal{V}$ if
\begin{multline}\label{Hbarconv}
\int_{[0, \infty) \times \mathbb{T}} \psi(t, x) \int_{\overline{\mathcal{H}}} \varphi(p) d(\nu_{n})_{t, x}(p) d\lambda(t, x) \to \int_{[0, \infty) \times \mathbb{T}} \psi(t, x) \int_{\overline{\mathcal{H}}} \varphi(p) d\nu_{t, x}(p) d\lambda(t, x), \\
\qquad \text{ for all } \psi \in C_{c}^{\infty}([0, \infty) \times \mathbb{T}) \text{ and } \varphi \in C_{b}(\overline{\mathcal{H}}),
\end{multline}
where we recall the definition of $C_{b}(\overline{\mathcal{H}})$ from Definition \ref{Htopology}. We have the following result on compactness of deterministic Young measures, adapted to the specific function spaces and range spaces we are considering.

\begin{proposition}\label{youngcompact}
    Let $\{C_{m}\}_{m = 1}^{\infty}$ be a monotonically increasing sequence of positive constants. The set of Young measures from $[0, \infty) \times \mathbb{T}$ with range space $ \overline{\mathcal{H}}$ satisfying:
    \begin{equation}\label{setm}
    \bigcap_{m = 1}^{\infty} \left\{\nu\in \mathcal{V}:\int_{[0, m] \times \mathbb{T}} \int_{\overline{\mathcal{H}}} \eta_{E}(p) d\nu_{t, x}(p) d\lambda(t, x) \le C_{m}\right\},
    \end{equation}
    where $\eta_{E}(\rho, u) := \displaystyle \frac{1}{2} \rho u^{2} + \frac{\kappa}{\gamma - 1} \rho^{\gamma}$ is the energy functional (the first entropy), is compact in $\mathcal{V}$. 
\end{proposition}

\begin{proof}
    Consider a sequence $\{\nu_{n}\}_{n = 1}^{\infty}$ of Young measures in the set \eqref{setm}, and consider the restriction of the Young measures to each set $[0, m] \times \mathbb{T}$, which we denote by $\nu_{n}|_{[0, m] \times \mathbb{T}}$ . We claim that for each fixed $m$, there exists a subsequence along which $\nu_{n_{k}}|_{[0, m] \times \mathbb{T}}$ converges weakly-star to some $\nu_{[0, m] \times \mathbb{T}}$ in the sense that
    \begin{multline}\label{weaklystar0m}
    \int_{[0, m] \times \mathbb{T}} \psi(t, x) \int_{\overline{\mathcal{H}}} \varphi(p) d(\nu_{n})_{t, x}(p) d\lambda(t, x) \to \int_{[0, m] \times \mathbb{T}} \psi(t, x) \int_{\overline{\mathcal{H}}} \varphi(p) d\nu_{t, x}(p) d\lambda(t, x), \\
    \text{ for all } \psi \in C([0, m] \times \mathbb{T}) \text{ and } \varphi \in C_{b}(\overline{\mathcal{H}}). 
    \end{multline}
    To show this, it suffices by {Prokhorov's} theorem to show that the $\lambda\ltimes\nu_{n}|_{[0, m] \times \mathbb{T}}$, %considered as probability measures on $[0, m] \times \mathbb{T} \times \overline{\mathcal{H}}$ via \eqref{youngidentify}, 
    are tight as finite measures on $[0, m] \times \mathbb{T}\times \overline{\mathcal{H}}$. To do this, consider the set:
    \begin{equation*}
    K_{R} := \eta_{E}^{-1}([0, R]) = \{(\rho, u) \in \overline{\mathcal{H}} : 0 \le \eta_{E}(\rho, u) \le R\}. 
    \end{equation*}
    Note that this set $K_{R}$ is compact in $\overline{\mathcal{H}}$ under the topology of $K_{R}$ defined in Definition \ref{Htopology} (even though $K_{R}$ is not compact with respect to the usual subspace topology of $[0, \infty) \times \mathbb{T}$ since $\eta_{E}(0, u) = 0$ for all $u \in \R$). Hence, by \eqref{setm}: %{\cred all these measures to be changed}
    \begin{equation*}
    \lambda\ltimes\nu_{n}|_{[0, m] \times \mathbb{T}}\Big(([0, m] \times \mathbb{T} \times K_{R})^{c}\Big) \le \frac{C_{m}}{R},
    \end{equation*}
    which shows tightness of the measures $\lambda\ltimes\nu_{n}|_{[0, m] \times \mathbb{T}}$, and hence shows that along a subsequence, $\lambda\ltimes\nu_{n_{k}}|_{[0, m] \times \mathbb{T}}$ converges to some limiting $\lambda\ltimes\nu|_{[0, m] \times \mathbb{T}}$ weakly-star in the sense of \eqref{weaklystar0m}.

    Hence, we can show compactness of the original set \eqref{setm} in the original sigma finite (but infinite measure) space of $[0, \infty) \times \mathbb{T}$ by using a diagonalization argument, namely, extract a convergence subsequence converging as Young measures on $[0, 1] \times [0, \infty)$, and then in general successively refine the subsequence to get convergence on larger and larger domains $[0, m] \times [0, \infty)$. We can use a diagonalization procedure to extract a convergent subsequence for the whole space, and it is easy to see using the definition of convergence \eqref{weaklystar0m} that the limiting Young measures on $[0, m] \times \mathbb{T}$ must agree for different values of $m$ on their overlap. This concludes the proof. 
\end{proof}

We can extend this to a tightness result for \textit{probabilistic} Young measures on $[0, \infty) \times \mathbb{T}$ taking values in $\mathcal{P}(\overline{\mathcal{H}})$. 

\begin{proposition}\label{randomyoung}
    Let $\{\nu_{n}\}_{n = 1}^{\infty}$ be a sequence of probabilistic Young measures on $[0, \infty) \times \mathbb{T}$ taking values in $\mathcal{P}(\overline{\mathcal{H}})$, satisfying for all positive integers $m$:
    \begin{equation*}
    \mathbb{E} \int_{[0, m] \times \mathbb{T}} \int_{\overline{\mathcal{H}}} \eta_{E}(p) d(\nu_{n})_{t, x}(p) d\lambda(t, x) \le \alpha_{m},
    \end{equation*}
    for some positive constants $\alpha_{m}$. Then, the sequence $\{\nu_{n}\}$ is tight in $\mathcal{V}$. 
\end{proposition}

\begin{proof}
    This follows by combining the deterministic compactness result in Proposition \ref{youngcompact} with Chebychev's inequality. Consider an arbitrary $\varepsilon > 0$. For each $m$, we define $C_{m} := \alpha_{m}2^{m} \varepsilon^{-1}$ so that
    \begin{equation*}
    \mathbb{P}\left(\int_{[0, m] \times \mathbb{T}} \int_{\overline{\mathcal{H}}} \eta_{E}(p) d(\nu_{n})_{t, x}(p){ d\lambda(t,x)} > C_{m}\right) < \varepsilon 2^{-m} \qquad \text{ for all positive integers } m \text{ and } n,
    \end{equation*}
    by Chebychev's inequality. Then, the set
    \begin{equation*}
    K_{\varepsilon} := {\bigcap_{m = 1}^{\infty}} \left\{\nu \in \mathcal{V} : \int_{[0, m] \times \mathbb{T}} \int_{\overline{\mathcal{H}}} \eta_{E}(p) d\nu_{t, x}(t, x){d\lambda(x,t)} \le C_{m}\right\}
    \end{equation*}
    is a compact set in $\mathcal{V}$ by Proposition \ref{youngcompact}, and furthermore,
    \begin{equation*}
    \mathbb{P}(\nu_{n} \in K_{\varepsilon}^{c}) \le \sum_{m = 1}^{\infty} \varepsilon 2^{-m} = \varepsilon. 
    \end{equation*}
    This establishes the tightness claim.
\end{proof}

Finally, we end this section by discussing the convergence of continuous and bounded functionals of Young measures, in both the deterministic and stochastic settings, which is referred to as \textit{momentum convergence}. Specifically, we have the following deterministic result on momentum convergence.

\begin{proposition}
    Let $\nu_{n} \to \nu$ in $\mathcal{V}$ be a sequence of Young measures $[0, \infty) \times \mathbb{T} \to \mathcal{P}(\overline{\mathcal{H}})$ satisfying
    \begin{equation}\label{supn}
    \sup_{n \ge 1} \int_{[0, T] \times \mathbb{T}} \int_{\overline{\mathcal{H}}} |\eta(p)|^{s} d(\nu_{n})_{t,x}(p) d\lambda(t, x) \le C_{T},
    \end{equation}
    for some finite $T > 0$, and for some continuous and bounded function $\eta: [0, \infty) \times \R \to \R$ (with respect to the usual Euclidean topology) satisfying one of the following two cases:
    \begin{itemize}
        \item \textbf{Case 1.} For all $(\rho, u) \in (0, \infty) \times \R$, $\lim_{r \to \infty} \eta(r(\rho, u))$ is either $-\infty$, $0$, or $\infty$, and $\eta(0, u) = 0$ for all $u \in \R$.
        \item \textbf{Case 2.} $\eta \in C_{b}(\overline{\mathcal{H}})$ in the sense of Definition \ref{Htopology}.
    \end{itemize}
    Then,
    \begin{equation}\label{absetas}
    \int_{[0, T] \times \mathbb{T}} \int_{\overline{\mathcal{H}}} |\eta(p)|^{s} d\nu_{t, x}(p) d\lambda(t, x) \le C_{T},
    \end{equation}
    for the same constant $C_{T}$, and furthermore, for all $\varphi \in L^{\frac{s}{s - r}}([0, T] \times \mathbb{T})$ and $1 \le r < s$:
    \begin{equation}\label{momentumconv}
    \int_{[0, T] \times \mathbb{T}} \int_{\overline{\mathcal{H}}} \varphi(t, x) (\eta(p))^{r} d(\nu_{n})_{t, x}(p) d\lambda(t, x) \to \int_{[0, T] \times \mathbb{T}} \int_{\overline{\mathcal{H}}} \varphi(t, x) (\eta(p))^{r} d\nu_{t, x}(p) d\lambda(t, x).
    \end{equation}
\end{proposition}

\begin{proof}
    We follow the proof of Proposition 4.3 in \cite{BerthelinVovelle}, and make some additional comments. First, we use the fact that $\nu_{n} \to \nu$ in $\mathcal{V}$, in the sense of \eqref{Hbarconv}. This immediately gives us convergence against functionals in $C_{b}(\overline{\mathcal{H}})$ in the sense of Definition \ref{Htopology}, namely the result for Case 2. \textit{However, $\eta$ might not be in $C_{b}(\overline{\mathcal{H}})$, in the first case above where $\lim_{r \to \infty} \eta(r(\rho, u)) = -\infty, 0, \infty$ for each $(\rho, u) \in (0, \infty) \times \R$ and $\eta(0, u) = 0$ for all $u \in \R$.}
    
    However, using a radially symmetric compactly supported smooth truncation function $\chi \in C_{c}^{\infty}(\R)$ that is decreasing radially, even, with $\chi(z) = 1$ on $[-1, 1]$ and $\chi(z) = 0$ for $|z| \ge 2$, we note that $\chi_{R}(\eta(p)) \eta(p)$ for $\chi_{R}(z) := \chi(p/R)$ is indeed in $C_{b}(\overline{\mathcal{H}})$, since the limit radially at infinity is just zero in all directions in the first case above for $\eta$. Hence, for any $1 \le r < s$:
    \begin{multline}\label{momentumconvmid}
    \int_{[0, T] \times \mathbb{T}} \int_{\overline{\mathcal{H}}} \varphi(t, x) \chi_{R}(\eta(p)) (\eta(p))^{r} d(\nu_{n})_{t, x}(p) d\lambda(t, x) \to \int_{[0, T] \times \mathbb{T}} \int_{\overline{\mathcal{H}}} \varphi(t, x) \chi_{R}(p) (\eta(p))^{r} d\nu_{t, x}(p) d\lambda(t, x), \\
    \qquad \text{ for all $R$},
    \end{multline}
    and setting $\varphi = 1$:
    \begin{multline}\label{chiRabsconv}
    \int_{[0, T] \times \mathbb{T}} \int_{\overline{\mathcal{H}}} \chi_{R}(\eta(p)) |\eta(p)|^{s} d(\nu_{n})_{t, x}(p) d\lambda(t, x) \to \int_{[0, T] \times \mathbb{T}} \int_{\overline{\mathcal{H}}} \chi_{R}(p) |\eta(p)|^{s} d\nu_{t, x}(p) d\lambda(t, x), \quad \text{ for all $R$}.
    \end{multline}
    So using \eqref{supn} and \eqref{chiRabsconv}, by monotone convergence:
    \begin{equation*}
    \int_{[0, T] \times \mathbb{T}} \int_{\overline{\mathcal{H}}} |\eta(p)|^{r} d\nu_{t, x}(p) d\lambda(t, x) = \lim_{R \to \infty} \int_{[0, T] \times \mathbb{T}} \int_{\overline{\mathcal{H}}} \chi_{R}(\eta(p)) |\eta(p)|^{r} d\nu_{t, x}(p)d\lambda(t, x) \le C_{T},
    \end{equation*}
    which establishes \eqref{absetas}. To show \eqref{momentumconv}, we use \eqref{momentumconvmid}. We can estimate that for $1 \le r < s$:
    \begin{align*}
   & \left|\int_{[0, T] \times \mathbb{T}} \int_{\overline{\mathcal{H}}} \varphi(t, x) (1 - \chi_{R}(\eta(p))) (\eta(p))^{r} d(\nu_{n})_{t, x}(p) d\lambda(t, x) \right| \\
    &\le \frac{1}{R^{s - r}} \int_{[0, T] \times \mathbb{T}} \int_{\overline{\mathcal{H}}} |\varphi(t, x)| \cdot |\eta(p)|^{s} d(\nu_{n})_{t, x}(p) d\lambda(t, x) \to 0,\quad{ \text{ as }R\to\infty},
    \end{align*}
    \textit{uniformly in $n$}, by the given bounds and the fact that $\varphi \in C_{b}([0, T] \times \mathbb{T})$. A similar convergence holds for the limiting Young measure $\nu$, which gives us the desired convergence \eqref{momentumconv} for $\varphi \in C_{b}([0, T] \times \mathbb{T})$, using \eqref{supn} and \eqref{absetas}. We can then extend the convergence \eqref{momentumconv} to $\varphi \in L^{\frac{s}{s - r}}([0, T] \times \mathbb{T})$ using a density argument.
\end{proof}

We can extend this momentum convergence result to probabilistic Young measures, as follows.

\begin{proposition}\label{momentumconvrandom}
    Let $\nu_{n} \to \nu$ almost surely, as probabilistic Young measures from $[0, T] \times \mathbb{T}$ to $\mathcal{P}_{1}(\overline{\mathcal{H}})$, satisfying the following uniform bound:
    \begin{equation*}
    \sup_{n \ge 1} \mathbb{E} \int_{[0, T] \times \mathbb{T}} \int_{\overline{\mathcal{H}}} (\eta(p))^{s} d(\nu_{n})_{t, x}(p) d\lambda(t, x) \le C_{T},
    \end{equation*}
    for some finite $T > 0$ and for some continuous bounded function $\eta: [0, \infty) \times \R \to \R$ such that either
    \begin{itemize}
        \item \textbf{Case 1.} For all $(\rho, u) \in (0, \infty) \times \R$, $\lim_{r \to \infty} \eta(r(\rho, u))$ is either $-\infty$, $0$, or $\infty$, and $\eta(0, u) = 0$ for all $u \in \R$.
        \item \textbf{Case 2.} $\eta \in C_{b}(\overline{\mathcal{H}})$ in the sense of Definition \ref{Htopology}. 
    \end{itemize}
    Then, for the same constant $C_{T}$ as above:
    \begin{equation*}
    \mathbb{E} \int_{[0, T] \times \mathbb{T}} \int_{\overline{\mathcal{H}}} (\eta(p))^{s} d\nu_{t, x}(p) d\lambda(t, x) \le C_{T},
    \end{equation*}
    and for all $\varphi \in L^{\frac{s}{s - r}}([0, T] \times \mathbb{T})$ and for all $1 \le r < s$ and $1 \le \delta < s/r$:
    \begin{equation*}
    \lim_{n \to \infty} \mathbb{E} \left|\int_{[0, T] \times \mathbb{T}} \int_{\overline{\mathcal{H}}} \varphi(t, x) (\eta(p))^{r} d(\nu_{n})_{t, x}(p) d\lambda(t, x) - \int_{[0, T] \times \mathbb{T}} \int_{\overline{\mathcal{H}}} \varphi(t, x) (\eta(p))^{r} d\nu_{t, x}(p) d\lambda(t, x)\right|^{\delta} = 0.
    \end{equation*}
\end{proposition}

\begin{proof}
We refer the reader to the proof of Proposition 4.5 in \cite{BerthelinVovelle}.
\end{proof}

\subsection{The Skorohod argument for the $\epsilon_{N} \to 0$ limit passage}

Next, we pass to the limit in the $\epsilon$-level approximate solutions using a Skorohod representation theorem argument. We define the phase space: 
\begin{equation*}
\mathcal{X}_{} := \mathcal{V} \times C_{loc}([0, \infty); H^{-3}(\mathbb{T}))\cap (C_{w,loc}([0,\infty);L^\gamma(\mathbb{T})) \times C_{w,loc}([0,\infty);L^{\frac{2\gamma}{\gamma+1}}(\mathbb{T})))\times C_{loc}([0, \infty); \mathcal{U}),
\end{equation*}
where $\mathcal{V}$ is the space of Young measures on $[0,\infty)\times\mathbb{T}$ taking values in $\mathcal{P}(\bar{\mathcal{H}})$, and we consider the laws $\mu_{\epsilon}$ of the approximate solutions 
\begin{equation*}
(\nu_{\epsilon}, \bd{U}_{\epsilon}, W_{\epsilon}), \qquad \text{ for } \bd{U}_{\epsilon} := (\rho_{\epsilon}, q_{\epsilon}),
\end{equation*}
in the phase space $\mathcal{X}$. Here we recall that $\nu_\epsilon =\delta_{(\rho_\epsilon,u_\epsilon)}$. We claim that we have the following tightness result.

\begin{proposition}\label{tighteps}
    The laws $\{\mu_{\epsilon}\}_{\epsilon > 0}$ are tight as probability measures on the phase space $\mathcal{X}_{}$. 
\end{proposition}

\begin{proof}
    By the result in Proposition \ref{momentsup}, we have that
    \begin{equation*}
    \mathbb{E} \int_{0}^{m} \int_{\mathbb{T}} \eta_{E}(\rho_{\epsilon}, u_{\epsilon}) dx dt \le C_{m},
    \end{equation*}
    for some constant $C_{m}$, for each positive integer $m$, since ${{\eta_{1} = 2\eta_{E}}}$. Since the Young measure for the approximate solution is $\nu_{\epsilon} = \delta_{(\rho_{\epsilon}, u_{\epsilon})}$, we conclude by Proposition \ref{randomyoung} that the laws of the probabilistic Young measures $\{\nu_{\epsilon}\}_{\epsilon > 0}$ are tight in $\mathcal{V}$. Furthermore, by the compact embedding
    \begin{equation*}
    C^{\alpha}_{loc}([0, \infty); H^{-2}(\mathbb{T})) \subset \subset C_{loc}([0, \infty); H^{-3}(\mathbb{T})),
    \end{equation*}
    which is a consequence of the Arzela-Ascoli compactness theorem, and the following embedding stated in Theorem 1.8.5 in \cite{BFH18}:
    \begin{align}\label{cw}
C_{loc}^\alpha([0,\infty);H^{-2}(\mathbb{T}))\cap L^\infty([0,\infty);L^p(\mathbb{T}))\subset\subset C_{w,loc}([0,\infty);L^p(\mathbb{T})),\quad\text{for any } \alpha>0, 1<p<\infty.
    \end{align}
    we have tightness of the laws of $\bd{U}_{\epsilon} := (\rho_{\epsilon}, q_{\epsilon})$ in $C_{loc}([0, \infty); H^{-3}(\mathbb{T}))\cap (C_{w,loc}([0,\infty);L^\gamma(\mathbb{T})) \times C_{w,loc}([0,\infty);L^{\frac{2\gamma}{\gamma+1}}(\mathbb{T})))$  by the estimate in Proposition \ref{Hnegative2} and Corollary \ref{momentrhoq}. The tightness of the laws of $\{W_{\epsilon}\}_{\epsilon > 0}$ in $C_{loc}([0, \infty); \mathcal{U}))$ is immediate.
\end{proof}
\begin{remark}
    Note that due to the embedding \eqref{cw}, we can construct statistically stationary solutions in the state space $L^{p}(\mathbb{T}) \times L^{p}(\mathbb{T})$ for any $1 < p < \infty$, as described in Remark \ref{rem:anyp}.
\end{remark}
The tightness result in Proposition \ref{tighteps} allows us to use the Skorohod representation theorem since the path space $\mathcal{X}$ is a Jakubowski space. 

\begin{proposition}
    There exists a probability space $(\tilde{\Omega}, \tilde{\mathcal{F}}, \tilde{\mathbb{P}})$ and random variables $(\tilde{\nu}_{\epsilon}, \tilde{\bd{U}}_{\epsilon}, \tilde{W}_{\epsilon})$ and $({\nu}, {\bd{U}}, {W})$ taking values in $\mathcal{X}_{}$, where we will denote $\tilde{\bd{U}}_{\epsilon} := (\tilde{\rho}_{\epsilon}, \tilde{q}_{\epsilon})$ and ${\bd{U}} := ({\rho}, {q})$, such that
    \begin{equation}\label{youngconv}
    (\tilde{\nu}_{\epsilon}, \tilde{\bd{U}}_{\epsilon}, \tilde{W}_{\epsilon}) =_{d} (\nu_{\epsilon}, \bd{U}_{\epsilon}, W_{\epsilon}),
    \end{equation}
    and
    \begin{equation*}
    (\tilde{\nu}_{\epsilon}, \tilde{\bd{U}}_{\epsilon}, \tilde{W}_{\epsilon}) \to ({\nu}, {\bd{U}}, {W}), \qquad \tilde{\mathbb{P}}\text{-almost surely in the topology of $\mathcal{X}$}.
    \end{equation*}
    Furthermore, the limiting process $\{{W}_t\}_{t \ge 0}$ is a $\mathcal{U}$-Wiener process with respect to the filtration $\{{\mathcal{F}}_{t}\}_{t \ge 0}$ defined by:
    \begin{equation*}
    {\mathcal{F}}_{t} := \sigma\{\bd{U}_{}(s), {W}_{}(s) : 0 \le s \le t\}.
    \end{equation*}
\end{proposition}

Using the Young measures will allow us to pass to the limit in $\eta(\tilde{\bd{U}}_{\epsilon})$ to $\eta({\bd{U}})$, but the issue is that we do not know that the limiting solution $({\rho}, {q})$ is function-valued, since it is represented either by a limiting Young measure ${\nu}$, or a very weak distributional space $C_{loc}([0, \infty); H^{-3}(\mathbb{T}))$. We therefore will need to carry out a \textbf{reduction of the Young measure}, which is an argument in which we will show that the limiting Young measure corresponds to a genuine function, rather than a measure-valued solution.

\subsection{Reduction of the Young measure}

In this subsection, we will reduce the limiting Young measure ${\nu}$ {to a Dirac mass}, which means that we will show that $$\text{this Young measure $\nu$ corresponds to a genuine function, except potentially when }{\rho} = 0.
$$ In this sense, the limiting density and momentum $({\rho}, {q}_{})$ is function valued, since we recall that the Young measure is defined for density and velocity, and hence, the momentum ${q}$ is unambiguously equal to zero in the vacuum (even if the value of ${u}$ is not necessarily well-defined). 

The main ingredient in the argument for the reduction of the Young measure is the following \textit{key functional equation}, which holds for all entropy flux pairs $(\eta, H)$ and $(\hat{\eta}, \hat{H})$ arising from subquadratic $g \in \mathcal{G}$ (see Definition \ref{tildeG}), and almost surely for all almost every $(t, x) \in [0, \infty) \times \mathbb{R}$:
\begin{equation}\label{functionaleqn}
\langle \eta, {\nu} \rangle \langle \hat{H}, {\nu} \rangle - \langle \hat{\eta}, {\nu}\rangle \langle H, {\nu}\rangle = \langle \eta\hat{H} - \hat{\eta}H, {\nu} \rangle, \qquad \text{ where } \langle f, {\nu} \rangle := \int_{\overline{\mathcal{H}}} f(p) d{\nu}_{t, x}(p).
\end{equation}
This functional equation is the key ingredient for an argument for the reduction of the Young measure, which will show that the Young measure is a Dirac delta function at each $(t, x) \in [0, \infty) \times \mathbb{T}$ almost surely, except potentially on vacuum when $\rho(t, x) = 0$. Since this argument is a standard argument in the literature, we do not provide the argument here and instead refer the reader to Section 5.2 in \cite{BerthelinVovelle}, and also Section I.5 in \cite{LPS96} (see also the discussion on pg.~604 in \cite{LPS96}).

We make a few comments on how the functional equation \eqref{functionaleqn} is obtained. Since the approximate Young measures $\tilde{\nu}_{\epsilon}$ are Dirac-delta functions (namely, they arise from the approximate statistically stationary solutions $(\tilde{\rho}_{\epsilon}, \tilde {q}_{\epsilon})$ which are function valued), the functional equation \eqref{functionaleqn} holds trivially for the approximate Young measures:
\begin{equation}\label{functionaleps}
\langle \eta, \tilde{\nu}_{\epsilon}(t, x) \rangle \langle \hat{H}, \tilde{\nu}_{\epsilon}(t, x) \rangle - \langle \hat{\eta}, \tilde{\nu}_{\epsilon}(t, x) \rangle \langle H, \tilde{\nu}_{\epsilon}(t, x) \rangle = \langle \eta \hat{H} - \hat{\eta} H, \tilde{\nu}_{\epsilon}(t, x) \rangle.
\end{equation}
Hence, we can obtain \eqref{functionaleqn} by passing to the limit as $\epsilon \to 0$ in \eqref{functionaleps}, and this can be done using the div-curl lemma and Murat's lemma, exactly as in Section 5.1 in \cite{BerthelinVovelle}, since the only uniform in $\epsilon$ estimates required for this argument for obtaining \eqref{functionaleqn}, are the estimates that we have established in Proposition \ref{uniformeps2}. Hence, we will not provide the details here.

\subsection{The limiting martingale solution}

Finally, we take the limit as $\epsilon_{N} \to 0$ in the approximate entropy equality \eqref{entropyeq} at the $\epsilon$ level. Note that since $\eta$ is convex and the test functions $\varphi(x) \in C^{2}(\mathbb{T})$ and $\psi(t) \in C_{c}^{\infty}(0, \infty)$ in the entropy equality are nonnegative, we have that
\begin{equation*}
\epsilon \int_{0}^{\infty}\left(\int_{\mathbb{T}} \langle D^{2}\eta(\tilde{\bd{U}}_{\epsilon}) \partial_{x}\tilde{\bd{U}}_{\epsilon}, \partial_{x}\tilde{\bd{U}}_{\epsilon} \rangle \varphi(x) dx\right) \psi(t) dt \ge 0,
\end{equation*}
and hence, we have that $\tilde{\mathbb{P}}$-almost surely:
\begin{multline}\label{entropyineqeps}
\int_{0}^{\infty} \left(\int_{\mathbb{T}} \eta(\tilde{\bd{U}}_{\epsilon}(t)) \varphi(x) dx\right) \partial_{t}\psi(t) dt + \int_{0}^{\infty} \left(\int_{\mathbb{T}} H(\tilde{\bd{U}}_{\epsilon}) \partial_{x}\varphi(x)\right) \psi(t) dt \\
- \alpha \int_{0}^{\infty} \left(\int_{\mathbb{T}} q_{\epsilon} \partial_{q}\eta(\tilde{\bd{U}}_{\epsilon}) \varphi(x) dx\right) \psi(t) dt + \int_{0}^{\infty} \left(\int_{\mathbb{T}} \partial_{q}\eta(\tilde{\bd{U}}_{\epsilon}) \Phi(\tilde{\bd{U}}_{\epsilon}) \varphi(x) dx\right) \psi(t) d\tilde{W}_\epsilon(t) \\
+ \int_{0}^{\infty} \left(\int_{\mathbb{T}} \frac{1}{2} \partial_{q}^{2}\eta(\tilde{\bd{U}}_{\epsilon}) G^{2}(\tilde{\bd{U}}_{\epsilon}) \varphi(x) dx\right) \psi(t) dt + \epsilon \int_{0}^{\infty} \left(\int_{\mathbb{T}} \eta(\tilde{\bd{U}}_{\epsilon}) \partial_{x}^{2}\varphi dx\right) \psi(t) dt \ge 0,
\end{multline}
for all $\varphi \in C^{2}(\mathbb{T})$ and $\psi \in C_{c}^{\infty}(0, \infty)$, with $\varphi, \psi \ge 0$.

Our goal is to pass to the limit as $\epsilon \to 0$ in the entropy inequality \eqref{entropyineqeps}. To do this, we combine the results on higher moment entropy bounds in Proposition \ref{fluxmomenteps} with the probabilistic momentum convergence results in Proposition \ref{momentumconvrandom}. As of now, we already know after the reduction of the Young measure, the convergence \eqref{youngconv}, and the definition of (vague) convergence of Young measures in Definition \ref{vague} that
\begin{equation*}
\int_{[0, T] \times \mathbb{T}} \int_{\overline{\mathcal{H}}} \psi(t, x) S(p) d(\tilde{\nu}_{\epsilon})_{t,x}(p) d\lambda(t, x) \to \int_{[0, T] \times \mathbb{T}} \int_{\overline{\mathcal{H}}} \psi(t, x) S(p) d\nu_{t,x}(p) d\lambda(t, x),
\end{equation*}
for all $\psi \in C_{c}^{\infty}([0, T] \times \mathbb{T})$ and $S \in C_{b}(\overline{\mathcal{H}})$ in the sense of Definition \ref{Htopology}. We want to extend this convergence to nonlinear continuous functions of the form found in \eqref{entropyineqeps}, such as those which are potentially unbounded. In particular, we have the following convergence result.

\begin{proposition}\label{L2omegaxt}
For $\eta_{m}$ defined for $g(z) = z^{2m}$ via \eqref{etamformula}, $\eta_{m}(\tilde{\bd{U}}_{\epsilon}) \to \eta_{m}(\bd{U})$ in $L^{2}(\tilde{\Omega} \times [0, T] \times \mathbb{T})$ for all $T > 0$. 
\end{proposition}

\begin{proof}
    We follow the approach of the analogous result in Proposition 5.11 in \cite{BerthelinVovelle}. First, note that since the Young measures $\tilde{\nu}_{\epsilon}$ correspond to genuine functions and since the reduction of Young measure argument shows that $\nu$ is function-valued (except potentially on vacuum which is fine, since $\eta_{m}$ is zero on vacuum), we have that
    \begin{equation*}
    \eta_m(\tilde{\bd{U}}_{\epsilon}) = \int_{\overline{\mathcal{H}}} \eta_m(p) d(\tilde{\nu}_{\epsilon})_{t, x}(p), \qquad \eta_m(\bd{U}) = \int_{\overline{\mathcal{H}}} \eta_m(p) d\nu_{t, x}(p). 
    \end{equation*}
Note that by the entropy moment bounds in Proposition \ref{fluxmomenteps}, we have that for all $1 \le s < \infty$:
    \begin{equation}\label{etams}
    {\mathbb{E}} \int_{0}^{T} \int_{\mathbb{T}} |\eta_m(\tilde{\bd{U}}_{\epsilon})|^{s} dx dt \le C_{T},
    \end{equation}
    independently of $\epsilon$. Note that by the algebraic estimates in Proposition \ref{etamboundalg}, $\eta_m$ satisfies Case 1 in Proposition \ref{momentumconvrandom}. Hence, by Proposition \ref{momentumconvrandom}, we conclude that for all $\varphi \in L^{2}([0, T] \times \mathbb{T})$
    \begin{equation*}
    {\mathbb{E}} \left|\int_{[0, T]} \int_{\mathbb{T}} \eta_m(\bd{U}) \varphi(t, x) dx dt - \int_{[0, T]} \int_{\mathbb{T}} \eta_m(\tilde{\bd{U}}_{\epsilon}) \varphi(t, x) dx dt\right|^{2} \to 0,
    \end{equation*}
    and by setting $\varphi = 1$ and $r = 2$ in Proposition \ref{momentumconvrandom}:
    \begin{equation*}
    {\mathbb{E}} \Big(\|\eta_m(\bd{U})\|_{L^{2}([0, T] \times \mathbb{T})}^{2} - \|\eta_m(\tilde{\bd{U}}_{\epsilon})\|_{L^{2}([0, T] \times \mathbb{T})}^{2}\Big) \to 0.
    \end{equation*}
    So we conclude that $\eta_m(\tilde{\bd{U}}_{\epsilon}) \rightharpoonup \eta_m(\bd{U})$ weakly in $L^{2}({ \tilde\Omega\times}[0, T] \times \mathbb{T})$, and we have convergence of the norms
    $\displaystyle \|\eta_m(\tilde{\bd{U}}_{\epsilon})\|_{L^{2}({ \tilde\Omega\times}[0, T] \times \mathbb{T})} \to \|\eta_m(\bd{U})\|_{L^{2}({ \tilde\Omega\times}[0, T] \times \mathbb{T})}$. Hence, by combining weak convergence with convergence of the norms:
    \begin{equation*}
    \eta_m(\tilde{\bd{U}}_{\epsilon}) \to \eta_m(\bd{U}), \qquad \text{ strongly in $L^{2}({ \tilde\Omega\times}[0, T] \times \mathbb{T})$. }%$\tilde{\mathbb{P}}$-almost surely}.
    \end{equation*}
   
\end{proof}

We also have almost everywhere convergence of the approximate solutions $\tilde{\bd{U}}_{\epsilon}$ along a subsequence.

\begin{proposition}\label{aeeps}
    Along a subsequence $\{\epsilon_{N_{k}}\}_{k = 1}^{\infty}$, we have that $(\tilde{\rho}_{\epsilon}, \tilde{q}_{\epsilon}) \to (\rho, q)$ almost everywhere on $\tilde\Omega \times [0, \infty) \times \mathbb{T}$.
\end{proposition}

\begin{proof}
    Note that the proof in Proposition \ref{L2omegaxt} also works to show that for the functions $\eta(\rho u) = \rho$ and $\eta(\rho, u) = \rho u$, we have that $\eta(\tilde{\bd{U}}_{\epsilon}) \to \eta(\bd{U})$ in $L^{2}(\tilde\Omega \times [0, T] \times \mathbb{T})$ for all $T > 0$. This is because we still have uniform bounds:
    \begin{equation*}
    {\mathbb{E}} \int_{0}^{T} \int_{\mathbb{T}} \Big(\tilde\rho_\epsilon + |\tilde\rho_\epsilon \tilde u_\epsilon|\Big)^{s} dx dt \le C_{s, T},
    \end{equation*}
    for all $1 \le s < \infty$, using $\rho + |\rho u| \le C\Big(\eta_{0}(\bd{U}) + \eta_{1}(\bd{U})\Big)$ and Proposition \ref{fluxmomenteps}. This establishes the result, since convergence in $L^{p}$ for $1 \le p < \infty$ implies convergence almost everywhere along a subsequence.
\end{proof}

Now, we have all of the necessary ingredients needed to pass to the limit as $\epsilon_{N} \to 0$ in the entropy inequality \eqref{entropyineqeps}. The goal is to obtain the limiting entropy inequality \eqref{entropyineq}, which is expected to hold for all entropy-flux pairs $(\eta, H)$ generated by all subpolynomial functions $g \in \tilde{\mathcal{G}}$. 

\begin{proposition}
    The entropy inequality \eqref{entropyineq} holds $\tilde{\mathbb{P}}$-almost surely for the limiting solution $(\rho, q)$ and for all entropy-flux pairs $(\eta, H)$ generated from $g \in \tilde{\mathcal{G}}$.
\end{proposition}

\begin{proof}
    We pass to the limit in each term in the approximate $\epsilon$-level entropy inequality \eqref{entropyineqeps}. To do this, consider some $(\eta, H)$ generated by $g \in \tilde{\mathcal{G}}$ satisfying \eqref{subpoly} for some positive integer $m$, and recall the algebraic bounds in Proposition \ref{generaletaH} and \ref{otheretaH}:
    \begin{equation*}
    |q\partial_{q}\eta(\bd{U})| \le C_{g}\Big(\eta_{0}(\bd{U}) + \eta_{m}(\bd{U})\Big), \qquad |G(\bd{U})\partial_{q}\eta(\bd{U})|\le C_{g}\Big(\eta_{0}(\bd{U}) + \eta_{m}(\bd{U})\Big),
    \end{equation*}
    \begin{equation}\label{algfinal}
    |G^{2}(\bd{U})\partial_{q}^{2}\eta(\bd{U})| \le C_{g}\Big(\eta_{0}(\bd{U}) + \eta_{m - 1}(\bd{U})\Big), \qquad |H(\bd{U})| \le C_{g}\Big(\eta_{0}(\bd{U}) + \eta_{m + 1}(\bd{U})\Big).
    \end{equation}
    By the continuity of $\eta$ and its derivatives, $H$, and $\bd{G}(\bd{U})$, we hence deduce that for all $T > 0$:
    \begin{equation}\label{Sconv}
    S(\tilde{\bd{U}}_{\epsilon}) \to S(\bd{U}) \quad \text{in $L^{2}(\tilde\Omega \times [0, T] \times \mathbb{T})$ }, \ \  \text{for $S = \{q\partial_{q}\eta, \bd{G}(\bd{U})\partial_{q}\eta, \rho\partial_{q}\eta, \bd{G}^{2}(\bd{U})\partial^{2}_{q}\eta, H\}$}.
    \end{equation}
    This follows by the generalized dominated convergence theorem (see Theorem 11 in Section 4.4 of \cite{RF23}), combined with the almost everywhere convergence in Proposition \ref{aeeps}, and the convergence in Proposition \ref{L2omegaxt} combined with the algebraic bounds \eqref{algfinal}. For example, for $q\partial_{q}\eta(\bd{U}_{\epsilon})$, we note that $|q\partial_{q}\eta(\tilde{\bd{U}}_{\epsilon})| \le C_{g}\Big(\eta_{0}(\tilde{\bd{U}}_{\epsilon}) + \eta_{m}(\tilde{\bd{U}}_{\epsilon})\Big)$,
    \begin{equation*}
    {\mathbb{E}} \int_{0}^{T} \int_{\mathbb{T}} \Big(\eta_{0}(\tilde{\bd{U}}_{\epsilon}) + \eta_{m}(\tilde{\bd{U}}_{\epsilon})\Big)^{2} dx dt \to {\mathbb{E}} \int_{0}^{T} \int_{\mathbb{T}} \Big(\eta_{0}(\bd{U}) + \eta_{m}(\bd{U})\Big)^{2} dx dt,
    \end{equation*}
    and $|q\partial_{q}\eta(\bd{U}) - q\partial_{q}\eta(\tilde{\bd{U}}_{\epsilon})| \to 0$ almost everywhere on $\tilde\Omega \times [0, T] \times \mathbb{T}$ by Proposition \ref{aeeps}.

    The convergence in \eqref{Sconv} is sufficient to pass to the limit in the terms in the approximate entropy inequality \eqref{entropyineqeps}. We only explicitly comment on the passage of the stochastic integral term here. Consider some $\psi \in C_{c}^{\infty}(0, \infty)$, and note that it has support in $[0, T]$ for sufficiently large $T$. By standard results on convergence of stochastic integrals (see \cite{Ben}, Lemma 2.1 in \cite{DGHT}, and Lemma 2.6.6 in \cite{BFH18}), it suffices to show that
    \begin{equation}\label{convinprob}
    \left(\int_{\mathbb{T}} \partial_{q}\eta(\tilde{\bd{U}}_{\epsilon}) \bd{G}^{\epsilon}(\tilde{\bd{U}}_{\epsilon}) \varphi(x) dx\right) \psi(t) \to \left(\int_{\mathbb{T}} \partial_{q}\eta(\bd{U}) \bd{G}(\bd{U}) \varphi(x) dx\right) \psi(t), \text{ in probability in $L^{2}(0, T)$}.
    \end{equation}
    To show this, note that
    \begin{multline*}
    {\mathbb{E}} \int_{0}^{T} \left(\int_{\mathbb{T}} \Big(\partial_{q}\eta(\bd{U}) \bd{\Phi}(\bd{U}) - \partial_{q}\eta(\tilde{\bd{U}}_{\epsilon}) \bd{\Phi}(\tilde{\bd{U}}_{\epsilon})\Big) \varphi(x) dx\right)^{2} (\psi(t))^{2} dt \\
    \le \|\psi\|_{L^{\infty}(0, T)}^{2} \|\varphi\|_{L^{2}(\mathbb{T})}^{2} {\mathbb{E}} \int_{0}^{T} \int_{\mathbb{T}} \Big(\partial_{q}\eta(\bd{U})\bd{G}(\bd{U}) - \partial_{q}\eta(\tilde{\bd{U}}_{\epsilon}) \bd{G}^{\epsilon}(\tilde{\bd{U}}_{\epsilon})\Big)^{2} dx dt \to 0,
    \end{multline*}
    as $\epsilon \to 0$. This follows by \eqref{Sconv} and the fact that $\displaystyle {\mathbb{E}} \int_{0}^{T} \int_{\mathbb{T}} \Big(\partial_{q}\eta(\bd{U}_{\epsilon}) (\bd{G}(\tilde{\bd{U}}_{\epsilon}) - \bd{G}^{\epsilon}(\tilde{\bd{U}}_{\epsilon})\Big)^{2} \to 0$ by generalized dominated convergence with the dominating function $\tilde{\rho}_{\epsilon}\partial_{q}\eta(\tilde{\bd{U}}_{\epsilon})$ using \eqref{Gbound} and \eqref{A0eps} (see Theorem 11 of Section 4.4 in \cite{RF23}). This completes the proof of the main theorem in Theorem \ref{mainthm}.
\end{proof}

\section*{Acknowledgments}

J. Kuan was partially supported by the National Science Foundation under the Mathematical Sciences Postdoctoral Research Fellowship (DMS-2303177).
K. Tawri was partially supported by the National Science Foundation grant DMS-2407197. K. Trivisa gratefully acknowledges the support by the National Science Foundation under the grants DMS-2231533 and DMS-2008568.
\printbibliography
\end{document}